\documentclass{article}

\usepackage[main, final]{neurips_2026}

\usepackage[utf8]{inputenc} %
\usepackage[T1]{fontenc}    %
\usepackage[hypertexnames=false]{hyperref}
\usepackage{url}            %
\usepackage{booktabs}       %
\usepackage{amsfonts}       %
\usepackage{nicefrac}       %
\usepackage{microtype}      %
\usepackage{xcolor}         %
\usepackage{thm-restate}

\usepackage{tcolorbox}

\newtcolorbox{theorembox}{
  colback=gray!20,
  colframe=gray!20,
  boxrule=0.8pt,
  before skip=10pt,
  after skip=10pt
}

\usepackage{graphicx}
\usepackage{apptools}
\usepackage[flushleft]{threeparttable}
\usepackage{array,booktabs,makecell}
\usepackage{multirow}
\usepackage{nicefrac}
\usepackage{amsthm}
\usepackage{amsmath,amsfonts,bm}
\usepackage{wrapfig}
\usepackage{caption}
\usepackage{siunitx}
\usepackage{thm-restate}
\usepackage{nccmath}
\usepackage{empheq}
\usepackage{bbm}
\usepackage{tabularx}

\usepackage{algorithmic}
\usepackage{algorithm}

\usepackage{color}
\usepackage{colortbl}
\definecolor{bgcolor}{rgb}{0.76,0.88,0.50}
\definecolor{bgcolor0}{rgb}{0.93,0.99,1}
\definecolor{bgcolor1}{rgb}{0.8,1,1}
\definecolor{bgcolor2}{rgb}{0.8,1,0.8}
\definecolor{bgcolor3}{rgb}{0.50,0.90,0.50}
\usepackage{tcolorbox}
\usepackage{pifont}
\definecolor{mydarkgreen}{rgb}{39,130,67}
\definecolor{mydarkred}{rgb}{192,25,25}

\usepackage[scaled=0.86]{helvet}
\newcommand{\algname}[1]{#1}

\newcommand{\norm}[1]{\left\| #1 \right\|}
\newcommand{\sqnorm}[1]{\left\| #1 \right\|^2}
\newcommand{\lin}[1]{\left\langle #1\right\rangle} %
\newcommand{\inp}[2]{\left\langle#1,#2\right\rangle} %

\newcommand{\flr}[1]{\left\lfloor #1\right\rfloor} %

\newcommand{\R}{\mathbb{R}} %
\newcommand{\N}{\mathbb{N}} %

\newcommand{\Q}{\mathbb{Q}}
\newcommand{\U}{\mathbb{U}}

\newcommand{\Exp}[1]{{\mathbb{E}}\left[#1\right]}
\newcommand{\ExpSub}[2]{{\mathbb{E}}_{#1}\left[#2\right]}
\newcommand{\ExpCond}[2]{{\mathbb{E}}\left[\left.#1\right\vert#2\right]}

\newcommand{\Ind}[1]{\mathbf{1}_{\left\{#1\right\}}}

\newcommand{\cC}{\mathcal{C}}

\newcommand{\cO}{\mathcal{O}}

\newcommand{\cS}{\mathcal{S}}

\newcommand{\cZ}{\mathcal{Z}}

\newcommand{\mH}{\mathbf{H}}
\newcommand{\mA}{\mathbf{A}}
\newcommand{\mI}{\mathbf{I}}

\theoremstyle{plain}
\newtheorem{theorem}{Theorem}[section]
\newtheorem*{theorem*}{Theorem}
\newtheorem{proposition}[theorem]{Proposition}
\newtheorem{lemma}[theorem]{Lemma}

\theoremstyle{definition}
\newtheorem{definition}[theorem]{Definition}
\newtheorem{assumption}[theorem]{Assumption}
\newtheorem{example}[theorem]{Example}
\theoremstyle{remark}
\newtheorem{remark}[theorem]{Remark}

\newcommand{\eqdef}{:=}

\makeatletter
\newcommand{\vast}{\bBigg@{4}}

\usepackage{subcaption}

\hypersetup{
  colorlinks   = true, %
  urlcolor     = blue, %
  linkcolor    = {red!75!black}, %
  citecolor   = {blue!50!black} %
}

\DeclareSymbolFont{extraup}{U}{zavm}{m}{n}
\DeclareMathSymbol{\varheart}{\mathalpha}{extraup}{86}
\DeclareMathSymbol{\vardiamond}{\mathalpha}{extraup}{87}

\definecolor{myred}{RGB}{180,30,30}

\allowdisplaybreaks

\title{Scalable Distributed Stochastic Optimization via Bidirectional Compression: Beyond Pessimistic Limits}

\author{
  Grigory Begunov \\
  AXXX\textsuperscript{1}, MSU\textsuperscript{2}
  \And
  Alexander Tyurin \\
  AXXX\textsuperscript{1}, Applied AI Institute\textsuperscript{3}
}

\newcommand{\reducesize}{}

\begin{document}

\maketitle
\footnotetext[1]{AXXX, Moscow, Russia}
\footnotetext[2]{Lomonosov Moscow State University, Moscow, Russia}
\footnotetext[3]{Applied AI Institute, Moscow, Russia}

\begin{abstract}
In centralized, distributed, and federated learning with stochastic gradients and $n$ workers, it was recently shown that it is infeasible to find an $\varepsilon$--stationary point faster than $\textstyle \tilde{\Omega}(\min\{{\color{myred}\nicefrac{d \kappa L \Delta}{\varepsilon}} + \nicefrac{h L \Delta}{\varepsilon} + \nicefrac{h \sigma^2 L \Delta}{n \varepsilon^2}, {\color{myred} \nicefrac{h \sigma^2 L \Delta}{\varepsilon^2}} + \nicefrac{h L \Delta}{\varepsilon}\})$ seconds in both homogeneous and heterogeneous settings under standard assumptions: $L$--smoothness, $\sigma^2$-bounded unbiased stochastic gradients, and lower boundedness of the function, i.e., $f(x) \geq f^*$ for all $x \in \R^d$, where $\Delta = f(x^0) - f^*$, $h$ is the computation time, $\kappa$ is the communication speed between the workers and the server, and $d$ is the dimension of the iterates and gradients. This result is pessimistic since it does not allow a complexity in which both ${\color{myred}\nicefrac{d \kappa L \Delta}{\varepsilon}}$ and ${\color{myred} \nicefrac{h \sigma^2 L \Delta}{\varepsilon^2}}$ improve with $n$, even when using random sparsification techniques; moreover, this lower bound can be matched by either non-distributed SGD or vanilla Synchronous SGD, which reduces the impact of recent progress in the design of compression-based methods. In this work, we challenge this limitation and propose new compressed methods, \ref{eq:inkheart} and \ref{eq:mthree}, and show that under an additional structural assumption, which is necessary due to the lower bound and which does not restrict the class of considered problems, we achieve new state-of-the-art time complexities that break this pessimistic barrier and allow scaling with the number of workers $n$.
\end{abstract}

\section{Introduction}
\begin{wrapfigure}{r}{0.33\textwidth}
  \centering
  \includegraphics[width=0.3\textwidth]{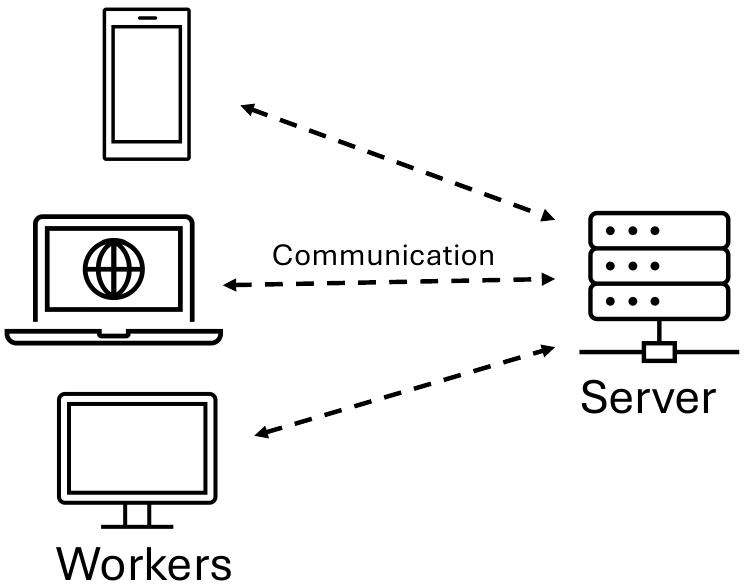}
\end{wrapfigure}
We consider a stochastic nonconvex optimization problem
\begin{align}
\label{eq:main_problem}
\reducesize \min \limits_{x \in \R^d} f(x),
\end{align}
where $f\,:\,\R^d \rightarrow \R$ and $d$ is the dimension of $f$. We assume that $d$ is huge, which is indeed the case in modern machine learning and large language model training \citep{GPT3,touvron2023llama}. This problem is solved in the classical federated learning setup, where $n$ workers, such as CPUs, GPUs, servers, or mobile devices, are connected to a central server via a communication channel \citep{konevcny2016federated, mcmahan2017communication}. The workers compute stochastic gradients in parallel and send them to the server, which aggregates the information and communicates it back to the workers. We start by considering the homogeneous setup, but we also consider the heterogeneous setup in Section~\ref{sec:heter}.

In the nonconvex setting, the objective is to obtain an $\varepsilon$-stationary point, that is, a (random) point $\bar{x} \in \R^d$ satisfying ${\rm \mathbb{E}}[\|\nabla f(\bar{x})\|^2] \leq \varepsilon$ under the following standard optimization assumptions:
\begin{assumption}
    \label{ass:lipschitz_constant}
    Function $f$ is $L$--smooth:
    \begin{align*}
      \norm{\nabla f(x) - \nabla f(y)} \leq L \norm{x - y} \qquad \forall x, y \in \R^d.
    \end{align*}
 \end{assumption}

 \begin{assumption}
    \label{ass:lower_bound}
    There exists $f^* \in \R$ such that $f(x) \geq f^*$ for all $x \in \R^d$.
    We define $\Delta \eqdef f(x^0) - f^*,$ where $x^0$ is the starting point of optimization methods.
 \end{assumption}

\begin{assumption}
    \label{ass:stochastic_variance_bounded}
     The workers have access to unbiased stochastic gradients $\nabla f(x; \xi)$ with $\sigma^2$--bounded variance: $\ExpSub{\xi}{\nabla f(x;\xi)} = \nabla f(x)$ and ${\mathbb{E}}_{\xi}[\|\nabla f(x;\xi) - \nabla f(x)\|^2] \leq \sigma^2$ for all $x \in \R^d,$ where $\sigma > 0$ is some constant.
\end{assumption} 
In order to compare methods, explain the main goal, and present the new theoretical results we consider the following assumption about the optimization environment. Note that the newly presented methods are valid and converge even without it.
\begin{assumption}
  \label{ass:time}
  Each of the $n$ workers requires at most $h$ seconds to compute a stochastic gradient, and communication \emph{from the server to any worker} (s2w communication) takes at most $\kappa$ seconds per coordinate, and communication \emph{from any worker to the server} (w2s communication) takes at most $\tau$ seconds per coordinate.
\end{assumption}
We consider settings with bidirectional communication costs, where transmitting information in either direction requires time. Under Assumption~\ref{ass:time}, sending a vector $v \in \R^d$ from the server to any worker takes $d \times \kappa$ seconds, while communication from any worker to the server requires $d \times \tau$ seconds. In contrast, many existing works, especially in the early development of federated learning algorithms, assume that communication \emph{from the server to the workers} is free, i.e., $\kappa = 0,$ which is arguably not realistic in practice, since communication over the Internet or 4G/5G networks can be costly in both directions \citep{huang2012close, narayanan2021variegated}.

\subsection{Related work}
\textbf{Synchronous method.} The centralized setting under the computation and communication costs assumptions is well studied. The standard baseline is the Synchronous SGD method: $x^{k+1} = x^k - \frac{\gamma}{n} \sum_{i=1}^{n} \nabla f(x^k;\xi^k_i),$ where $\gamma$ is the step size and every worker computes one stochastic gradient (takes $h$ seconds), sends it to the server (takes $d \tau$ seconds), which updates the point and sends it back to the workers (takes $d \kappa$ seconds). One can easily show that the iteration complexity of this method with a proper $\gamma$ is $\cO(\nicefrac{L \Delta}{\varepsilon} + \nicefrac{\sigma^2 L \Delta}{n \varepsilon^2})$ \citep{lan2020first} and the time complexity is $\cO\left(\left(h + d \tau + d \kappa\right) \left(\nicefrac{L \Delta}{\varepsilon} + \nicefrac{\sigma^2 L \Delta}{n \varepsilon^2}\right)\right).$ Using the minibatching technique, i.e., calculating $\sum_{j=1}^{b}\nabla f(x^k;\xi^k_{ij})$ instead of $\nabla f(x^k;\xi^k_i)$ with a proper $b$, this complexity can be improved to:
\begin{align}
  \label{eq:kOFvscLJ}
  \reducesize \cO\left({\color{myred} \left(d \kappa + d \tau\right) \frac{L \Delta}{\varepsilon}} + h \frac{L \Delta}{\varepsilon} + h \frac{\sigma^2 L \Delta}{n \varepsilon^2}\right),
\end{align}
where the last ``statistical term'' decreases as the number of workers $n$ grows. This provides a theoretical explanation for the benefit of distributed optimization and the use of a large number of workers.

\textbf{Methods with only worker-to-server compression.} However, the first ``communication term'' in \eqref{eq:kOFvscLJ} does not scale with $n.$ It turns out that the term $d \tau$ corresponding to the w2s communication can be improved using compressed communication techniques \citep{Seide2014, alistarh2017qsgd}. There are many efficient methods with compressed communication, including \algname{DIANA}~\citep{DIANA}, \algname{Accelerated DIANA}~\citep{ADIANA}, \algname{MARINA}~\citep{gorbunov2021marina}, and \algname{DASHA}~\citep{tyurin2022dasha}, which rely on \emph{unbiased compressors}.
\begin{definition}\label{def:unbiased_compressor}
  A stochastic mapping $\cC: \R^d \rightarrow \R^d$ is an unbiased compressor if there exists $\omega \geq 0$
  such that $\Exp{\cC(x)} = x \textnormal{ and }\Exp{\|\cC(x) - x\|^2} \leq \omega \norm{x}^2.$ We denote the family of such compressors by $\mathbb{U}(\omega)$. Unless otherwise stated, we assume that all compressors are statistically independent.
\end{definition}
One of the most common examples of an unbiased compressor is Rand$K$ $\in \mathbb{U}(\nicefrac{d}{K} - 1)$, which operates by selecting $K$ coordinates of the input vector $x$ uniformly at random, rescaling them by $\nicefrac{d}{K}$, and setting all other coordinates to zero (see Definition~\ref{def:rand_k}). Beyond Rand$K$, a variety of other unbiased compressors have been proposed and studied in the literature \citep{beznosikov2020biased,xu2021grace,horvoth2022natural,szlendak2021permutation}.

Using the seminal ideas \citep{Seide2014,alistarh2017qsgd}, we can construct the following compressed method:
\begin{align}
\label{eq:mTyVY}
\reducesize x^{k+1} = x^{k} - \frac{\gamma}{n b m} \sum\limits_{i=1}^n \sum\limits_{k=1}^m \cC_{ik}\left(\sum\limits_{j=1}^{b} \nabla f(x^k;\xi^{k}_{ij})\right),
\end{align}
where worker $i$ computes $b$ stochastic gradients, and then sends $m$ compressed vectors $\{\cC_{ik}(\cdot)\}_{k \in [m]}$ to the server\footnote{Or, almost equivalently, use one Rand$m$ compressor, since both constructions have variance parameter at most $d/m$.}, which aggregates and calculates $x^{k+1}.$ Using Rand$K$ and a proper choice of parameters, we can improve \eqref{eq:kOFvscLJ} to the time complexity
\begin{align}
  \label{eq:shadowheart}
  \reducesize \cO\left({\color{myred} \frac{d \kappa L \Delta}{\varepsilon}} + \tau \left(\frac{d}{n} + 1\right) \frac{L \Delta}{\varepsilon} + \sqrt{\frac{d \tau h \sigma^2}{n \varepsilon}} \frac{L \Delta}{\varepsilon} + h \left(1 + \frac{\sigma^2}{n \varepsilon}\right) \frac{L \Delta}{\varepsilon} \right).
\end{align}
However, all the listed methods, including \eqref{eq:mTyVY}, still depend on $\nicefrac{d \kappa L \Delta}{\varepsilon}$, which does not improve with $n$, because they send full vectors from the server to the workers. Moreover, if $\kappa = \tau$, then the compression technique does not help at all, since \eqref{eq:shadowheart} equals $\cO\left(\nicefrac{\kappa d L \Delta}{\varepsilon} + h \left(\nicefrac{L \Delta}{\varepsilon} + \nicefrac{\sigma^2 L \Delta}{n \varepsilon^2}\right)\right)$, reducing to \eqref{eq:kOFvscLJ}, as in Synchronous SGD with minibatching that does not use compression at all. The ``communication term'' $\nicefrac{d \kappa L \Delta}{\varepsilon}$ does not decrease as the number of workers grows.

\textbf{Bidirectionally compressed methods.} A large body of work employs communication compression to reduce the s2w communication cost \citep{zheng2019communication,liu2020double,philippenko2021preserved,fatkhullin2021ef21,yue2023core,gruntkowska2023ef21,tyurin20232direction}. However, none of these works provide theoretical guarantees in which the s2w communication term scales with $n$. A notable exception is the work by \citet{gruntkowska2024improving}, where the authors design the M3 method for the heterogeneous setting that achieves scaling under an additional assumption; however, they consider the deterministic setting, unlike our stochastic setting, which is much more challenging.
\subsection{Recent pessimistic lower bound and motivation}
A recent work by \citet{tyurin2025proving} analyzes our stochastic setting under Assumptions~\ref{ass:lipschitz_constant}, \ref{ass:lower_bound}, \ref{ass:stochastic_variance_bounded}, and \ref{ass:time}, and proves a pessimistic result that it is (informally) impossible to achieve a time complexity better than
\begin{align}
  \label{eq:BuypTEDpDiOaSX}
  \reducesize \tilde{\Omega}\left(\min\left\{{\color{myred}\frac{d \kappa L \Delta}{\varepsilon}} + h  \left(1 + \frac{\sigma^2}{n \varepsilon}\right) \frac{L \Delta}{\varepsilon} + \tau  \left(1 + \frac{d}{n}\right)  \frac{L \Delta}{\varepsilon} + 
  \sqrt{\frac{d \tau h \sigma^2}{n \varepsilon}}  \frac{L \Delta}{\varepsilon}, {\color{myred} \frac{h \sigma^2 L \Delta}{\varepsilon^2}} + \frac{h L \Delta}{\varepsilon}\right\} \right)
\end{align}
up to logarithmic factors, using random sparsification in both the homogeneous and heterogeneous settings. In particular, if $\kappa \simeq \tau,$ then the lower bound is 
\begin{align}
  \label{eq:BuypTEDpDiOaSX2}
  \reducesize \tilde{\Omega}\left(\min\left\{{\color{myred}\frac{d \kappa L \Delta}{\varepsilon}} + \frac{h L \Delta}{\varepsilon} + \frac{h \sigma^2 L \Delta}{n \varepsilon^2}, {\color{myred} \frac{h \sigma^2 L \Delta}{\varepsilon^2}} + \frac{h L \Delta}{\varepsilon}\right\} \right),  
\end{align}
which can be matched by Synchronous SGD (see \eqref{eq:kOFvscLJ}) or by the non-distributed \algname{SGD} method. This means that using methods with random sparsification for compression in the distributed stochastic centralized setting does not yield any advantage, and it is infeasible to scale both the communication and statistical terms.
\begin{quote}
  Despite the significant progress in centralized stochastic optimization, classical Synchronous SGD with minibatching or non-distributed SGD remain the state-of-the-art method in practical scenarios where the server-to-worker (s2w) communication time is non-negligible. In particular, the lower bound \eqref{eq:BuypTEDpDiOaSX2} of \citet{tyurin2025proving} proves that there is no hope for improvement under Assumptions~\ref{ass:lipschitz_constant}, \ref{ass:lower_bound}, \ref{ass:stochastic_variance_bounded}, and \ref{ass:time} when $\kappa \simeq \tau.$ \it Is the situation truly this pessimistic, or is it possible to design more advanced methods with provably better time complexity guarantees under an additional structural assumption?
\end{quote}

\subsection{Contributions}
In this paper, we develop two new methods, \ref{eq:inkheart} and \ref{eq:mthree}, that achieve new state-of-the-art time complexities in centralized stochastic optimization. Despite significant prior progress on this topic, this is the first result to improve the convergence of the naive Synchronous SGD method in the practical setting where communication from the server to the workers is not negligible and the workers have access only to stochastic gradients.

In order to achieve this result, due to \citep{tyurin2025proving}, it is necessary to introduce an additional structural assumption about $f.$ We consider Assumption~\ref{ass:AB_assumption} \citep{gruntkowska2024improving,wang2023new}. \emph{Notice that Assumption~\ref{ass:AB_assumption} follows from Assumption~\ref{ass:lipschitz_constant} and vice versa due to Proposition~\ref{prop:one}.} This assumption captures an additional structural property of $f$ through the structural parameters $L_A$ and $L_B.$ It is always possible to take $L_A = L$ and $L_B = 0.$ However, our theoretical improvements are most significant when $L_A$ is small. For instance, $L_A = 0$ when $f$ is a quadratic function, and $L_A$ is small when $f$ has a ``slowly changing'' Hessian (see Propositions~\ref{prop:quadratic_inequality} and \ref{theorem:a_b_hess}); we also observe that $L_A$ is small in a small-scale practical task (see Section~\ref{sec:emp_l_a}). We empirically verify our improvement in Section~\ref{sec:experiments}.

\newcommand{\ratehomog}{\max\left\{h, \tau, \kappa, \frac{d \kappa}{\sqrt{n}}, \frac{d \tau}{n},
    \left(\frac{d^3 \tau \kappa^2}{n}\right)^\frac{1}{3}, \sqrt{\frac{d\sigma^2 h \tau}{n \varepsilon}}, \frac{\sigma^2 h}{n \varepsilon}\right\} \frac{L_{\max} \Delta}{\varepsilon} + \frac{d \kappa L_A \Delta}{\varepsilon}}
\newcommand{\rateheter}{\max\left\{h, \tau, \kappa, \frac{d \kappa}{\sqrt{n}}, \frac{d \tau}{\sqrt{n}}, \left(\frac{d^3 \tau^2 \kappa}{n}\right)^\frac{1}{3}, \left(\frac{d^2 \tau^2 h \sigma^2}{n \varepsilon}\right)^\frac{1}{3}, \frac{\sigma^2 h}{n \varepsilon}\right\} \frac{L_{\max} \Delta}{\varepsilon} + \frac{d \kappa L_A \Delta}{\varepsilon}}
\textbf{New state of the art in stochastic centralized optimization. Sections~\ref{sec:homog} and \ref{sec:heter}.} Under an additional assumption (necessary due to the lower bound), in both homogeneous and heterogeneous scenarios, we develop \ref{eq:inkheart} and \ref{eq:mthree} that achieve new state-of-the-art time complexities
\begin{align}
  \label{eq:asasdasdas}
  \reducesize\cO\left(\ratehomog\right)
\end{align}
and 
\begin{align}
  \label{eq:asdafsasfas}
  \reducesize\cO\left(\rateheter\right),
\end{align}
respectively (Theorems~\ref{thm:main_time_complexity} and \ref{thm:main_time_complexity_heter}). In particular, \eqref{eq:asasdasdas} is never worse than \eqref{eq:BuypTEDpDiOaSX} due to Proposition~\ref{prop:one}, since we can always take $L_A = L$ and $L_B = 0$ and use that $\left(\nicefrac{d^3 \tau \kappa^2}{n}\right)^{1/3} \leq \nicefrac{d \tau}{n} + d \kappa.$ \emph{However, \eqref{eq:asasdasdas} can be arbitrarily better in the regime when $n$ is large and $L_A$ is small, since in this case $\eqref{eq:asasdasdas} \to \max\{h, \tau, \kappa\} \nicefrac{L_{\max} \Delta}{\varepsilon},$ which does not depend on $d$ and $\nicefrac{\sigma^2}{\varepsilon}.$} We obtain similar conclusions in the heterogeneous scenario for \eqref{eq:asdafsasfas}.

\textbf{Extension to the setting with heterogeneous times. Section~\ref{sec:ext}.} We also extend our result to the scenario when the computation and communication times are heterogeneous (Assumption~\ref{ass:time_heter}). This is a standard assumption in the analysis of parallel and asynchronous methods \citep{mishchenko2022asynchronous,tyurin2023optimal}. In this setting, previous works have the same issue: both computation and communication terms do not scale in the complexities \citep{tyurin2024shadowheart}. In Section~\ref{sec:ext}, we analyze a new method, \ref{eq:inkheart_heter}, based on \ref{eq:inkheart} and prove a new state-of-the-art time complexity.
\subsection{Preliminaries: Functional $(L_A, L_B)$ Inequality}\label{sec:ab_homo}
In order to improve the pessimistic lower bound, we introduce the Functional $(L_A, L_B)$ Inequality:
\begin{assumption}[Functional $(L_A, L_B)$ Inequality; \citet{gruntkowska2024improving,wang2023new}]\label{ass:AB_assumption}
  There exist constants $L_A,L_B \geq 0$ such that
  \begin{align}\label{eq:functional}
      &\reducesize \norm{\frac{1}{n} \sum\limits_{i=1}^n (\nabla f(x+u_i) - \nabla f(x))}^2 \leq L_A^2\left(\frac{1}{n} \sum\limits_{i=1}^n \norm{u_i}^2\right) + L_B^2 \norm{\frac{1}{n} \sum\limits_{i=1}^n u_i}^2
  \end{align}
  for all $n \geq 1$ and $x, u_1, \dots, u_n \in \R^d.$
\end{assumption}
Before we present our new methods and theoretical results, let us discuss the properties of Assumption~\ref{ass:AB_assumption} and its connection to Assumption~\ref{ass:lipschitz_constant}. Note that Assumption~\ref{ass:AB_assumption} only refines the functional smoothness properties and \emph{does not} restrict the class of considered functions; this is formalized by the following proposition.
\begin{proposition}
  \label{prop:one}
  If Assumption~\ref{ass:lipschitz_constant} holds, then Assumption~\ref{ass:AB_assumption} is satisfied with $L_A = L$ and 
  $L_B = 0$. If Assumption~\ref{ass:AB_assumption} holds, then $f$ is $L$-smooth with $L = \sqrt{L_A^2 + L_B^2}$.
\end{proposition}
The next easily verified proposition says that we can choose $L_A = 0$ when $f$ is a quadratic function.
\begin{restatable}{proposition}{QUADRATICINEQUALITY}
\label{prop:quadratic_inequality}
Let $f: \R^d\rightarrow \R$ be a quadratic function defined as $f(x) = \frac{1}{2} x^\top \mH x + \lin{b, x} + c$,
where $\mH \in \mathbb{S}^d, b \in \R^d$ and $c \in \R$. 
Then, Assumption~\ref{ass:AB_assumption} holds with $L_A = 0$ and $L_B = \norm{\mH}_2$.
\end{restatable}
While quadratic optimization problems are popular in many applications, it is important to understand when $L_A$ is small for non-quadratic functions. We also consider the proposition below.
\begin{restatable}[\citet{gruntkowska2024improving}]{proposition}{THEOREMLHESS}
    \label{theorem:a_b_hess}
    Let $f: \R^d\rightarrow \R$ be twice continuously differentiable, $L$--smooth, and $D \in \R$ be the smallest constant such that
        $\reducesize \sup_{z_1,\ldots,z_n\in\R^d} \norm{\nabla^2 f(z_1) - \frac{1}{n} \sum_{j=1}^n \nabla^2 f(z_j)} \leq D$
    for all $n \geq 1.$ Then, Assumption~\ref{ass:AB_assumption} holds with $L_A = \sqrt{2} D$ and $L_B = \sqrt{2} L.$ Moreover, $D$ is always finite, and $D \leq 2 L.$
\end{restatable}
This proposition states that if the Hessian does not change significantly, then $L_A$ is small. In the worst case, $L_A = \sqrt{2} D \leq 2 \sqrt{2} L$; however, it can be arbitrarily smaller than $L$, for instance, $L_A = 0$ when $f$ is a quadratic function, or consider the following example of an additive non-quadratic function with a small smoothness constant.

\begin{example}\label{thm:quad_ab}
Let $f: \R^d\rightarrow \R$ such that $f(x) = \frac{1}{2} x^\top \mH x + \lin{b, x} + c + g(x)$,
where $\mH \in \mathbb{S}^d, b \in \R^d, c \in \R$ and $g: \R^d \rightarrow \R$ is $L_{\textnormal{small}}$--smooth. Then, $f$ satisfies 
Assumption~\ref{ass:AB_assumption} with $L_A = 2 \sqrt{2} L_{\textnormal{small}}$ and $L_B = \sqrt{2} (\norm{\mH} + L_{\textnormal{small}}).$
\end{example}
\textbf{Numerical experiments.} In Section~\ref{sec:emp_l_a}, we also consider a machine learning task with a convolutional neural network, where we provide empirical evidence that $L_A$ is relatively small and can be $20 \sim 200$ times smaller than $L.$

\section{Inkheart SGD: A New Algorithm in a Homogeneous Setting}
\label{sec:homog}
\begin{figure}[t]
\begin{theorembox}
\centerline{The Inkheart SGD Method}
Initialize vector $x^0 \in \R^d$ and take $x^0_i = x^0$ for all $i \in [n]$, step size $\gamma > 0$, params $b, m, \ell \in \N,$ compressor $\cC^k_{ij} \in \mathbb{U}(\omega),$ and $\cC^k_{\textnormal{s},ij} \in \mathbb{U}(\omega_s)$ for all $i,j,k \geq 0.$ Then, iterate the following steps for $k = 0, 1, \dots$:
\begin{equation}
  \tag{Inkheart SGD}\label{eq:inkheart}
\begin{aligned}
    g^{k} &= \reducesize \frac{1}{n} \sum \limits_{i=1}^{n} \frac{1}{b m} \sum\limits_{j=1}^{m} \cC^k_{ij} \left(\sum\limits_{r=1}^{b} \nabla f(x_i^{k};\xi_{ir}^{k})\right), \\ 
    x^{k+1} &= x^k - \gamma g^k, \\
    c^k &\sim \text{Bernoulli}(p), \quad \text{where } p \in (0, 1] \text{ is a parameter}, \\
    x^{k+1}_i &= 
    \begin{cases}
        x^{k+1} & \text{if } c^k = 1,\\
        \reducesize x_i^k + \frac{1}{\ell} \sum\limits_{j=1}^{\ell} \cC^k_{\textnormal{s},ij}(x^{k+1} - x^k) & \text{if } c^k = 0
    \end{cases} \\
    &\textnormal{for all } i \in [n].
\end{aligned}
\end{equation}
\end{theorembox}
\end{figure}
We now present our new method, \ref{eq:inkheart}, which combines ideas from the classical QSGD method \citep{alistarh2017qsgd} and a recent paper \citep{gruntkowska2024improving}. Let us explain how \ref{eq:inkheart} works. At the beginning of every iteration $k$, each worker $i \in [n]$ computes a mini-batch of size $b.$ These estimates are summed locally and then compressed using the operators $\{\cC^k_{ij}\}_{j \in [m]}$. Then, they are transmitted to the server,
which averages them across all workers to form $g^k$. Next, the server performs a standard gradient descent step $x^{k+1} = x^k - \gamma g^k$. After that, the algorithm decides whether to synchronize all local models with the updated global iterate or to perform a compressed update. In particular, with probability $p \in (0, 1]$, all workers set $x_i^{k+1} = x^{k+1}$, corresponding to a full synchronization step. Otherwise, with probability $1-p$, each worker updates its local model using compressed vectors $\cC^k_{\textnormal{s},ij}$ of the model difference $x^{k+1} - x^k$. In this case, worker $i$ adds $\ell$ independent compression operators $\cC^k_{\textnormal{s},ij}$ to the local $x_i^k.$ Notice that the value of $p$ is small, non-compressed vectors are sent rarely, and this does not have an adverse effect on the final time complexity.

\subsection{Theoretical results}\label{sec:inkheart}

For \ref{eq:inkheart}, we can prove the following \emph{iteration} rate:

\begin{theorem}[Iteration complexity. Follows from Theorem~\ref{thm:main_iteration_complexity}] \label{thm:homog} Let Assumptions~\ref{ass:lipschitz_constant}, \ref{ass:lower_bound}, \ref{ass:stochastic_variance_bounded}, and \ref{ass:AB_assumption} be satisfied.
  Then, \ref{eq:inkheart} with $\gamma = \Theta\left(1 / \max\left\{L_{\max}, \frac{\sqrt{\omega \omega_s} L_{\max}}{\sqrt{p n m \ell}}, \frac{\sqrt{\omega_s} L_{\max}}{\sqrt{p \ell n}}, \frac{\sqrt{\omega_s} L_A}{\sqrt{p \ell}}\right\}\right),$
     and all $m$ and $b$ such that $\nicefrac{8 \omega}{m} \leq n$ and $\nicefrac{8 \omega \sigma^2}{m b \varepsilon} + \nicefrac{8 \sigma^2}{b \varepsilon}
   \leq n$ finds an $\varepsilon$--stationary point of \eqref{eq:main_problem} after at most
  \begin{align}
  \label{eq:nZpJEEdKVHXVhG}
  \reducesize K = \Theta\left(\max\left\{1, \frac{\sqrt{\omega \omega_s} }{\sqrt{p n m \ell}},  
  \frac{\sqrt{\omega_s}}{\sqrt{p \ell n}}\right\} \frac{L_{\max} \Delta}{\varepsilon} + \frac{\sqrt{\omega_s} L_A \Delta}{\sqrt{p \ell} \varepsilon}\right)
  \end{align}
  iterations, where $L_{\max} \eqdef \max\{L, L_A, L_B\}.$
\end{theorem}
Note that Theorem~\ref{thm:homog} is an \emph{auxiliary} result and does not yield a time complexity that allows us to determine whether the method can improve the pessimistic lower bounds \eqref{eq:BuypTEDpDiOaSX} and \eqref{eq:BuypTEDpDiOaSX2}. Theorem~\ref{thm:homog} works with all unbiased compressors that satisfy Definition~\ref{def:unbiased_compressor}. Without loss of generality, and for simplicity, we fix the Rand$1$ compressor with $\omega = \omega_s = d - 1$. Since this compressor sends only one coordinate, the time required to send one compressed vector from the workers to the server and vice versa is $\tau$ and $\kappa$, respectively, according to Assumption~\ref{ass:time}. There are three main bottlenecks in \eqref{eq:inkheart}: stochastic gradient computations and communication between the workers and the server (in both directions). The idea is to assign a time budget $t$ to each operation; thus, $b = \flr{\frac{t}{h}},\; m = \flr{\frac{t}{\tau}},$ and $\ell = \flr{\frac{t}{\kappa}}$. Substituting $\omega$ and $\omega_s$ into \eqref{eq:nZpJEEdKVHXVhG}, and noting that each iteration takes at most $\bar{t} \eqdef h b + \tau m + (p d + (1 - p) \ell)\kappa$ seconds on average, it remains to minimize the total time $K \times \bar{t}$ over $t$ to obtain the following result.
\begin{theorembox}
\begin{restatable}[Time Complexity of \ref{eq:inkheart}]{theorem}{THMHOMOTIME}\label{thm:main_time_complexity}
  Consider the assumptions and result of Theorem~\ref{thm:homog}. Additionally, assume that Assumption~\ref{ass:time} is satisfied and the workers and the server use Rand$1$ compressor (Definition~\ref{def:unbiased_compressor}). Let $b = \flr{\frac{t}{h}}, m = \flr{\frac{t}{\tau}},$ and $\ell = \flr{\frac{t}{\kappa}},$ where 
  $t = \max\left\{h, \tau, \kappa, \nicefrac{16 \omega \tau}{n}, \nicefrac{16 \sigma^2 h}{n \varepsilon}, \nicefrac{2 d \kappa}{\sqrt{n}},
  \sqrt{\nicefrac{32 d \sigma^2 h \tau}{n \varepsilon}},
  \left(\nicefrac{8 d^3 \tau \kappa^2}{n}\right)^\frac{1}{3}\right\}.$ Then, the expected time for \ref{eq:inkheart} to find an $\varepsilon$--stationary point of \eqref{eq:main_problem} is
  \begin{align}\reducesize \cO\left(\ratehomog\right).
    \label{eq:qgniEiqxzsNzEp}
  \end{align}
\end{restatable}
\end{theorembox}
Notice that, except for the last term $\nicefrac{d \kappa L_A \Delta}{\varepsilon}$, all other terms with $d$ and $\nicefrac{\sigma^2}{\varepsilon}$ decrease as the number of workers increases. In particular, in the case when $L_A \approx 0$, for instance, when $f$ is a quadratic function or $L_A$ is small due to Proposition~\ref{theorem:a_b_hess}, this complexity tends to $\cO\left(\max\left\{h, \tau, \kappa\right\} \nicefrac{L_{\max} \Delta}{\varepsilon}\right)$, which does not depend on $d$ or $\nicefrac{\sigma^2}{\varepsilon}$. To the best of our knowledge, this is the first result in distributed stochastic optimization. When $\kappa = \tau,$ the time complexity equals $\cO\left(\max\left\{h, \kappa, \sqrt{\nicefrac{d\sigma^2 h \kappa}{n \varepsilon}}, \nicefrac{d \kappa}{n^{1/3}}, \nicefrac{\sigma^2 h}{n \varepsilon}\right\} \nicefrac{L_{\max} \Delta}{\varepsilon} + \nicefrac{d \kappa L_A \Delta}{\varepsilon}\right).$ Due to Proposition~\ref{prop:one}, it is never worse than $\cO\left({\color{myred}\nicefrac{d \kappa L \Delta}{\varepsilon}} + \nicefrac{h L \Delta}{\varepsilon} + \nicefrac{h \sigma^2 L \Delta}{n \varepsilon^2}\right)$ achieved by Synchronous SGD with minibatching. However, the former can be arbitrarily smaller when $n$ is large and $L_A$ is small. Theorem~\ref{thm:main_time_complexity} is proved for Rand$1$. Similarly, one can extend it to any other unbiased compressor, where the only things that would change are the values of $\omega$ and $\omega_s$ and the amount required to transmit one compressed value.

\section{M4: A New Algorithm in a Heterogeneous Setting}\label{sec:ab_heter}
\label{sec:heter}
\newcommand{\explain}{M4 = MARINA-P + Momentum + Momentum + MARINA}
\begin{figure}[t]
\begin{theorembox}
\centerline{The M4 Method}
\centerline{(\algname{\explain}): }
Initialize vector $x^0 \in \R^d$ and vectors $v_i^0 \in \R^d$ for all $i \in [n]$ and take $g^0 = \nicefrac{1}{n} \sum_{i=1}^{n} v_i^0$ and $w_i^0 = x_i^0 = x^0$
for all $i \in [n]$, step size $\gamma > 0$, probabilities $0 < p_{\textnormal{s}}, p \leq 1$ and params $b \in \N,$ $\nu, \mu \in (0, 1],$ compressor $\cC^k_{i} \in \mathbb{U}(\omega),$ and $\cC^k_{\textnormal{s},i} \in \mathbb{U}(\omega_s)$ for all $i,k \geq 0.$ Then, iterate the following steps for $k = 0, 1, \dots$:
\begin{equation}
  \tag{M4}\label{eq:mthree}
\begin{aligned}
    x^{k+1} &= x^k - \gamma g^k, \\
    w^{k+1}_i &= 
    \begin{cases}
        x^{k+1} & \textnormal{with probability } p_{\textnormal{s}},\\
        w_i^k + \cC_{\textnormal{s},i}^k(x^{k+1} - x^k) & \textnormal{with probability } 1 - p_{\textnormal{s}},
    \end{cases} \quad \forall i \in [n]\\
    x_i^{k+1} &= (1 - \mu) x_i^{k} + \mu w_i^{k+1} \qquad \forall i \in [n] \\
    v_i^{k + 1} &= \reducesize (1 - \nu) v_i^k + \frac{\nu}{b} \sum_{r = 1}^{b} \nabla f_i (x_i^{k + 1}; \xi_{ir}^{k + 1}) \qquad \forall i \in [n] \\
    g^{k+1} &= 
    \begin{cases}
        \frac{1}{n} \sum_{i=1}^{n} v_i^{k+1} & \textnormal{with probability } p,\\
        g^{k} + \frac{1}{n} \sum_{i=1}^{n} \cC^k_{i}(v_i^{k+1} - v_i^{k}) & \textnormal{with probability } 1 - p
    \end{cases}
\end{aligned}
\end{equation}
where the first probabilistic choice is the same for all workers from set $[n]:$ one Bernoulli random variable is drawn for all workers. The coins for the first and second probabilistic choices with $p_{\textnormal{s}}$ and $p$ are independent.
\end{theorembox}
\end{figure}

In the heterogeneous setting, we consider the following nonconvex distributed optimization task:
\begin{align}  
  \label{eq:main_heter}
  \reducesize \min\limits_{x \in \R^d} \left\{f(x) \eqdef \frac{1}{n} \sum \limits_{i = 1}^n f_i(x)\right\}.
\end{align}
Instead of Assumptions~\ref{ass:lipschitz_constant}, \ref{ass:stochastic_variance_bounded}, and \ref{ass:AB_assumption}, we consider the following assumptions in the heterogeneous setting.
\begin{assumption}
\label{ass:hetero_function}
The function $f$ is $L$--smooth, and the function $f_i$ is $L_i$ smooth for all $i \in [n]$. We define $\hat{L}^2 \eqdef \frac{1}{n} \sum_{i = 1}^n L_i^2$.
\end{assumption}
\begin{assumption}[Heterogeneous setting]
  \label{ass:stochastic_variance_bounded_heter}
  For all $i \in [n]$ and $x \in \R^d,$ the stochastic gradients $\nabla f_i(x; \xi)$ are unbiased and have $\sigma^2$--bounded variance: $\ExpSub{\xi}{\nabla f_i(x;\xi)} = \nabla f_i(x)$ and ${\mathbb{E}}_{\xi}[\|\nabla f_i(x;\xi) - \nabla f_i(x)\|^2] \leq \sigma^2$ for all $x \in \R^d,$ where $\sigma > 0$ is some constant.
\end{assumption}
\begin{assumption}[Functional $(L_A, L_B)$ similarity]\label{ass:AB_assumption_heter}
  There exist constants $L_A,L_B \geq 0$ such that
  \begin{align}\label{eq:similarity}
      &\reducesize \norm{\frac{1}{n} \sum\limits_{i=1}^n (\nabla f_i(x+u_i) - \nabla f_i(x))}^2 
      \leq L_A^2\left(\frac{1}{n} \sum\limits_{i=1}^n \norm{u_i}^2\right) 
        + L_B^2 \norm{\frac{1}{n} \sum\limits_{i=1}^n u_i}^2
  \end{align}
  for all $n \geq 1$ and $x, u_1, \dots, u_n \in \R^d.$ We define $L_{\max} \eqdef \max\left\{\max_{i \in [n]} L_i, L_A, L_B\right\}$.
\end{assumption}

We are ready to present our new method, \ref{eq:mthree}. The design of this method is based on the ideas from \citep{gorbunov2021marina,fatkhullin2023momentum,gruntkowska2024improving}. At each iteration $k$, the server first performs a gradient descent step $x^{k+1} = x^k - \gamma g^k$. Then, with probability $p_{\textnormal{s}}$, all workers synchronize by setting $w_i^{k+1} = x^{k+1}$, while with probability $1 - p_{\textnormal{s}}$, they update $w_i^{k+1}$ using a compressed version of the model difference $x^{k+1} - x^k$ with compressor $\cC^k_{\textnormal{s},i}$. The workers receive either $x^{k+1}$ or the compressed vectors from the server. Next, each worker forms its local iterate $x_i^{k+1}$ as a convex combination of $x_i^k$ and $w_i^{k+1}$ with weight $\mu$, and updates its local gradient estimator $v_i^{k+1}$ using a mini-batch of size $b$ and momentum parameter $\nu$. Finally, the server updates the global estimator $g^{k+1}$: with probability $p$, it computes the exact average $\frac{1}{n} \sum_{i=1}^n v_i^{k+1},$ receiving the full vectors from the workers, and with probability $1 - p$, it performs a compressed update based on compressor $\cC^k_{i}$ applied to the differences $v_i^{k+1} - v_i^k$ that the workers send to the server. Similarly to \ref{eq:inkheart}, since $p$ and $p_{\textnormal{s}}$ are small, the synchronization of full vectors does not slow down the optimization procedure on average. Without loss of generality, we assume that the workers and the server send one compressed vector, since one can easily show that the average of unbiased compressors is an unbiased compressor.

\subsection{Theoretical results}\label{sec:m4}

\begin{restatable}[Iteration complexity]{theorem}{THEOREMHETERITER}
  \label{thm:heter}
  Let Assumptions~\ref{ass:hetero_function}, \ref{ass:lower_bound}, \ref{ass:stochastic_variance_bounded_heter}, and \ref{ass:AB_assumption_heter} be satisfied. 
  Then, \ref{eq:mthree} with momentum parameters
  $\nu = \mu = \eta \eqdef \min\left\{
    \frac{1}{6} \sqrt{\frac{b n \varepsilon}{\omega (\omega + 1) \sigma^2}}, 
    \frac{b n \varepsilon}{6 \sigma^2},
    \left(\frac{n}{\omega (\omega + 1) \omega_{\textnormal{s}}}\right)^\frac{1}{3}, 1\right\},$
  $p_{\textnormal{s}} = \frac{1}{\omega_{\textnormal{s}} + 1}, p = \frac{1}{\omega + 1},$ $v_i^0 = \frac{1}{b_{\textnormal{init}}} 
\sum_{b = 1}^{b_{\textnormal{init}}} \nabla f_i(x_i^0; \xi_{i, b}^0)$ where $b_{\textnormal{init}} = \Theta\left(\sqrt{\frac{b}{\nu} \left(1 + \frac{\sigma^2}{n \varepsilon}\right)}\right),$ and step size 
$\gamma = \Theta \left(\sqrt{\left(\omega_{\textnormal{s}} (\omega_{\textnormal{s}} + 1) L_A^2
  + \frac{\omega_{\textnormal{s}}}{n} (\omega_\textnormal{s} + 1) L_B^2
  + \left(\frac{\omega (\omega + 1)}{n} + \frac{1}{\eta^2}\right) L_{\max}^2\right)}\right)^{-1}$
  finds an $\varepsilon$--stationary point of \eqref{eq:main_heter} after at most 
  \begin{align*}
    \reducesize K = \cO\left(\frac{\Delta}{\gamma \varepsilon} + \sqrt{\max\left\{
    \sqrt{\frac{\omega (\omega + 1) \sigma^2}{b n \varepsilon}}, 
    \frac{\sigma^2}{b n \varepsilon},
    \left(\frac{\omega (\omega + 1) \omega_{\textnormal{s}}}{n}\right)^\frac{1}{3}, 1\right\} \frac{1}{b} \left(1 + \frac{\sigma^2}{n \varepsilon}\right)}\right)
  \end{align*}
  iterations.
\end{restatable}
In the proof, we define the Lyapunov function $\Psi^k$ as follows:
\begin{align*}
    &\textstyle \Psi^k = f(x^k) - f^* + \lambda_A \sqnorm{g^k - v^k} 
  + \lambda_B \sqnorm{v^k - \frac{1}{n} \sum \limits_{i = 1}^n \nabla f_i(x_i^k)}
  + \frac{\lambda_C}{n} \sum \limits_{i = 1}^n \sqnorm{v_i^{k} - \nabla f_i (x_i^{k})} \\
  &\textstyle + \frac{\lambda_D}{n} \sum\limits_{i=1}^n \norm{w_i^k - x_i^k}^2
  + \frac{\lambda_E}{n} \sum\limits_{i=1}^n \norm{w_i^k - x^k}^2 + \lambda_F \norm{\frac{1}{n} \sum\limits_{i=1}^n \left(w_i^k - x_i^k\right)}^2 
      + \lambda_G \norm{\frac{1}{n} \sum\limits_{i=1}^n w_i^k - x^k}^2,
\end{align*}
  where $\lambda_A, \lambda_B, \lambda_C, \lambda_D, \lambda_E, \lambda_F, \lambda_G$ are defined in Theorem~\ref{thm:main_hetero}. Using this choice of coefficients, we prove the inequality $\frac{1}{K} \sum_{k = 0}^{K - 1} \Exp{\|\nabla f(x^k)\|^2} 
  \leq \nicefrac{2 \Psi^{0}}{\gamma K} + \frac{\varepsilon}{2},$ and then, choosing $w_i^0 = x_i^0 = x^0$ and $g^0 = \nicefrac{1}{n} \sum_{i=1}^{n} v_i^0,$ we obtain $\frac{1}{K} \sum_{k = 0}^{K - 1} \Exp{\|\nabla f(x^k)\|^2} \leq \nicefrac{2 \Delta}{\gamma K} + \nicefrac{\varepsilon}{2} + \nicefrac{2}{\gamma K} \left(\lambda_B \sqnorm{v^0 - \frac{1}{n} \sum_{i = 1}^n \nabla f_i(x_i^0)}
  + \lambda_C \frac{1}{n} \sum_{i = 1}^n \sqnorm{v_i^{0} - \nabla f_i (x_i^{0})}\right).$ In order to reduce the dependence on the last two terms, we cannot set $v_i^{0} = \nabla f_i (x_i^{0})$ in the stochastic setting; instead, we initialize it with a minibatch of size $b_{\textnormal{init}}.$

  As in the homogeneous case, Theorem~\ref{thm:heter} is an auxiliary result, and we now present our main theoretical result in the heterogeneous setting.

  \begin{theorembox}
  \begin{restatable}[Time Complexity]{theorem}{THMHETERTIME}\label{thm:main_time_complexity_heter}
    Consider the assumptions and result of Theorem~\ref{thm:heter}. Additionally, assume that Assumption~\ref{ass:time} is satisfied and the workers and server use Rand$K$ compressor with
    $K = \flr{\frac{t}{\tau}}$ and $K = \flr{\frac{t}{\kappa}},$ respectively, and use batch size $b = \flr{\frac{t}{h}},$ where $t = \max \left\{h, \tau, \kappa, \sqrt[3]{\nicefrac{d^2 \tau^2 h \sigma^2}{n \varepsilon}}\right\}.$ Then, the expected time for \ref{eq:mthree} to find an $\varepsilon$--stationary point of \eqref{eq:main_heter} after at most\
    \begin{align}\reducesize \cO\left(\rateheter\right)
      \label{eq:zIoPJkvWZWTboVa}
    \end{align}
    seconds, including the initial computation of minibatches of size $b_{\textnormal{init}}.$
  \end{restatable}
  \end{theorembox}
  The dependence here is slightly worse than in \eqref{eq:qgniEiqxzsNzEp}, the homogeneous setting. Nevertheless, the conclusion is the same, and \ref{eq:mthree} enjoys scaling with $n$ when the structural parameter $L_A$ is small, unlike all previous compressed and non-compressed methods in stochastic optimization.
\section{Extension to Optimization with Heterogeneous Times}
\label{sec:ext}
In this section, we extend the result from Section~\ref{sec:homog} to the case where computation and communication times are heterogeneous, an assumption used to compare asynchronous methods \citep{mishchenko2022asynchronous,tyurin2023optimal}.
\begin{assumption}
  \label{ass:time_heter}
  For all $i \in [n],$ worker $i$ requires at most $h_i$ seconds to compute a stochastic gradient, and communication from the server to any worker $i$ takes at most $\kappa_i$ seconds per coordinate, and communication from worker $i$ to the server takes at most $\tau_i$ seconds per coordinate.
\end{assumption}
In Section~\ref{sec:heter_inkheart}, we extend \ref{eq:inkheart} to this setting. There are two essential changes: i) the number of computed stochastic gradients and sent compressed vectors varies across workers; ii) the received compressed vectors are aggregated with a particular optimal choice of weights \eqref{eq:weights} to minimize the variance of $g^k.$

\subsection{Theoretical results}\label{sec:incheart_heter}
\begin{theorembox}
\begin{restatable}[Time Complexity of \ref{eq:inkheart_heter}]{theorem}{THMHOMOTIMEASYNC}\label{thm:main_time_complexity_async}
  Let Assumptions~\ref{ass:lipschitz_constant}, \ref{ass:lower_bound}, \ref{ass:stochastic_variance_bounded}, and \ref{ass:AB_assumption} be satisfied. Additionally, assume that Assumption~\ref{ass:time_heter} is satisfied and the workers and the server use Rand$1$ compressors. Let $b_i = \flr{\frac{t}{h_i}}, m_i = \flr{\frac{t}{\tau_i}},  \ell_i = \flr{\frac{t}{\kappa_i}},$ where $t = \max\left\{\max_{i \in [n]} \left\{h_i, \tau_i, \kappa_i\right\}, s^*\right\},$ $\kappa_{\max} = \max_{i \in [n]}{\kappa_i},$ and $s^*$ is the solution of the equation
\begin{align}\label{eq:weighted_equlibrium_time}
\left(\sum \limits_{i=1}^{n} \frac{1}{\frac{16 \omega \tau_i}{s} 
+ \frac{16 \sigma^2 h_i}{\varepsilon s}  
+\frac{32 \sigma^2 \omega h_i \tau_i}{\varepsilon s^2} 
+ \frac{4 d \omega_\textnormal{s} \kappa_{\max} \kappa_i}{s^2}
+ \frac{8 d \omega_\textnormal{s} \omega \kappa_{\max} \kappa_i \tau_i}{s^3}}\right)^{-1} = 1.
\end{align}
  Then, the expected time for \eqref{eq:inkheart_heter} to find an $\varepsilon$--stationary point of \eqref{eq:main_problem} is at most
  \begin{align}\reducesize \cO\left(\frac{t L_{\max}\Delta}{\varepsilon} + \frac{d\kappa_{\max} L_A \Delta}{\varepsilon}\right).
    \label{eq:AmLCh}
  \end{align}
\end{restatable}
\end{theorembox}
When the times are equal, this result reduces to Theorem~\ref{thm:main_time_complexity}, and we can obtain an explicit formula that exhibits scaling with $n$ when $L_A$ is small. In general, following \citep{tyurin2024shadowheart}, to find the parameters and derive the final time complexity, one has to solve \eqref{eq:weighted_equlibrium_time}. When the times $\kappa_i, h_i, \tau_i$ do not grow too fast, one can show that $s^* \to 0$ as $n \to \infty,$ and the limiting time complexity does not depend on $d,$ unlike the result by \citep{tyurin2024shadowheart}.

Notice that both $t$ and $s^*$ depend on $\max_{i \in [n]} \left\{h_i, \tau_i, \kappa_i\right\}$ and $\kappa_{\max}.$ This might be a problem when, for instance, the computation and communication time of one of the workers is huge. For instance, if $\kappa_{n} \to \infty,$ then $\kappa_{\max} \to \infty$ and $\eqref{eq:AmLCh} \to \infty.$

A more robust strategy is to take a subset of workers $S \subseteq [n]$, evaluate the time 
$t(S) \eqdef \max\left\{\max_{i \in S} \left\{h_i, \tau_i, \kappa_i\right\}, s^*(S)\right\}$, 
where $s^*(S)$ is the solution of \eqref{eq:weighted_equlibrium_time} when only the workers in 
$S$ participate in the optimization ($\kappa_{\max} \to \max_{i \in S}\kappa_i$ and $\sum_{i=1}^{n} 
\to \sum_{i \in S}$), 
and then minimize the time complexity~\eqref{eq:AmLCh} over all subsets $S \subseteq [n].$ 
This way, we find the fastest subset of workers $S^*$ and obtain a better time complexity
\begin{align*}\reducesize \cO\left(\frac{t(S^*) L_{\max}\Delta}{\varepsilon} + \frac{d \max\limits_{i \in S^*} 
  \kappa_{i} L_A \Delta}{\varepsilon}\right).
\end{align*}
In Section~\ref{sec:optimal}, we describe an efficient polynomial-time algorithm that minimizes the complexity 
over all subsets $S \subseteq [n]$.
\begin{restatable}{theorem}{THMOPTSET}\label{thm:worker_selection} 
A subset of workers $S^*$ that minimizes
\begin{align*}
\tilde{T}(S) \eqdef \max\left\{t(S) L_{\max}, d \max_{i \in S} \kappa_i L_A \right\},
\end{align*} 
can be found using Algorithm~\ref{alg:main_choose_subset}.
\end{restatable}

\section{Conclusion and Future Work}
The centralized stochastic optimization with compressed methods is a well-explored direction (e.g., \citep{errorSGD,DIANA,gorbunov2021marina,fatkhullin2021ef21,tyurin2022dasha,huang2022lower,tyurin20232direction}). Nevertheless, for the first time, we develop two new algorithms, \ref{eq:inkheart} and \ref{eq:mthree}, that provably improve upon simple baselines, non-distributed SGD and Synchronous SGD, in the stochastic centralized setting. A similar improvement is achieved in the presence of heterogeneous computation and communication times. To present this improvement, we rely on structural assumptions that are necessary due to the lower bound of \citep{tyurin2025proving} and do not restrict the class of functions. We show that this assumption is both theoretically sound and practically relevant. One interesting direction is to consider alternative assumptions that capture the structure of $f$ and could potentially improve the baselines in stochastic distributed optimization.

\section*{Acknowledgements}
The work was supported by the grant for research centers in the field of AI provided by the Ministry of Economic Development of the Russian Federation in accordance with the agreement 000000C313925P4F0002 and the agreement №139-10-2025-033.

\bibliography{neurips_2026}

@inproceedings{GPT3,
 author = {Brown, Tom and Mann, Benjamin and Ryder, Nick and Subbiah, Melanie and Kaplan, Jared D and Dhariwal, Prafulla and Neelakantan, Arvind and Shyam, Pranav and Sastry, Girish and Askell, Amanda and Agarwal, Sandhini and Herbert-Voss, Ariel and Krueger, Gretchen and Henighan, Tom and Child, Rewon and Ramesh, Aditya and Ziegler, Daniel and Wu, Jeffrey and Winter, Clemens and Hesse, Chris and Chen, Mark and Sigler, Eric and Litwin, Mateusz and Gray, Scott and Chess, Benjamin and Clark, Jack and Berner, Christopher and McCandlish, Sam and Radford, Alec and Sutskever, Ilya and Amodei, Dario},
 booktitle = {Advances in Neural Information Processing Systems},
 editor = {H. Larochelle and M. Ranzato and R. Hadsell and M. F. Balcan and H. Lin},
 pages = {1877--1901},
 publisher = {Curran Associates, Inc.},
 title = {Language Models are Few-Shot Learners},
 url = {https://proceedings.neurips.cc/paper/2020/file/1457c0d6bfcb4967418bfb8ac142f64a-Paper.pdf},
 volume = {33},
 year = {2020}
}

@InProceedings{ADIANA,
  author    = {Zhize Li and Dmitry Kovalev and Xun Qian and Peter Richt\'{a}rik},
  booktitle = {International Conference on Machine Learning},
  title     = {Acceleration for Compressed Gradient Descent in Distributed and Federated Optimization},
  year      = {2020},
}

@inproceedings{Seide2014,
  author    = {Seide,  Frank and Fu,  Hao and Droppo,  Jasha and Li,  Gang and Yu,  Dong},
  title     = {1-bit stochastic gradient descent and its application to data-parallel distributed training of speech {DNN}s},
  booktitle = {Fifteenth Annual Conference of the International Speech Communication Association},
  year      = {2014},
}

@inproceedings{gorbunov2021marina,
  title={MARINA: Faster non-convex distributed learning with compression},
  author={Gorbunov, Eduard and Burlachenko, Konstantin P and Li, Zhize and Richt{\'a}rik, Peter},
  booktitle={International conference on machine learning},
  pages={3788--3798},
  year={2021},
  organization={PMLR}
}

@InProceedings{errorSGD,
  author    = {Wu, Jiaxiang and Huang, Weidong and Huang, Junzhou and Zhang, Tong},
  title     = {Error Compensated Quantized {SGD} and its Applications to Large-scale Distributed Optimization},
  booktitle = {Proceedings of the 35th International Conference on Machine Learning},
  year      = {2018},
  editor    = {Dy, Jennifer and Krause, Andreas},
  volume    = {80},
  series    = {Proceedings of Machine Learning Research},
  pages     = {5325--5333},
  address   = {Stockholmsmässan, Stockholm Sweden},
  month     = {10--15 Jul},
  publisher = {PMLR},
}

@Article{DIANA,
  author  = {Mishchenko, Konstantin and Gorbunov, Eduard and Tak{\'a}{\v{c}}, Martin and Richt{\'a}rik, Peter},
  journal = {arXiv preprint arXiv:1901.09269},
  title   = {Distributed Learning with Compressed Gradient Differences},
  year    = {2019},
}

@article{beznosikov2020biased,
	title={On Biased Compression for Distributed Learning},
	author={Beznosikov,  Aleksandr and Horv{\'a}th,  Samuel and Richt{\'a}rik,  Peter and Safaryan,  Mher},
	journal={arXiv preprint arXiv:2002.12410},
	year={2020}
}

@inproceedings{mcmahan2017communication,
  title={Communication-efficient learning of deep networks from decentralized data},
  author={McMahan, Brendan and Moore, Eider and Ramage, Daniel and Hampson, Seth and y Arcas, Blaise Aguera},
  booktitle={Artificial intelligence and statistics},
  pages={1273--1282},
  year={2017},
  organization={PMLR}
}

@InProceedings{alistarh2017qsgd,
  author    = {Alistarh, Dan and Grubic, Demjan and Li, Jerry and Tomioka, Ryota and Vojnovic, Milan},
  title     = {{QSGD}: {C}ommunication-efficient {SGD} via gradient quantization and encoding},
  booktitle = {Advances in Neural Information Processing Systems (NIPS)},
  year      = {2017},
  pages     = {1709--1720},
}

@article{lecun2010mnist,
  title={MNIST handwritten digit database},
  author={LeCun, Yann and Cortes, Corinna and Burges, CJ},
  journal={ATT Labs [Online]. Available: http://yann.lecun.com/exdb/mnist},
  volume={2},
  year={2010}
}

@article{konevcny2016federated,
  title={Federated learning: Strategies for improving communication efficiency},
  author={Kone{\v{c}}n{\'y}, Jakub and McMahan, H Brendan and Yu, Felix X and Richt{\'a}rik, Peter and Suresh, Ananda Theertha and Bacon, Dave},
  journal={arXiv preprint arXiv:1610.05492},
  year={2016}
}

@article{tyurin2022dasha,
  title={{DASHA}: Distributed Nonconvex Optimization with Communication Compression, Optimal Oracle Complexity, and No Client Synchronization},
  author={Tyurin, Alexander and Richt{\'a}rik, Peter},
  journal={11th International Conference on Learning Representations (ICLR)},
  year={2023},
}

@inproceedings{horvoth2022natural,
  title={Natural compression for distributed deep learning},
  author={Horv{\'a}th, Samuel and Ho, Chen-Yu and Horv\'{a}th, \v{L}udov\'{i}t and Sahu, Atal Narayan and Canini, Marco and Richt{\'a}rik, Peter},
  booktitle={Mathematical and Scientific Machine Learning},
  pages={129--141},
  year={2022},
  organization={PMLR}
}

@article{fatkhullin2021ef21,
  title={{EF21} with bells \& whistles: Practical algorithmic extensions of modern error feedback},
  author={Fatkhullin, Ilyas and Sokolov, Igor and Gorbunov, Eduard and Li, Zhize and Richt{\'a}rik, Peter},
  journal={arXiv preprint arXiv:2110.03294},
  year={2021}
}

@inproceedings{xu2021grace,
  title={Grace: A compressed communication framework for distributed machine learning},
  author={Xu, Hang and Ho, Chen-Yu and Abdelmoniem, Ahmed M and Dutta, Aritra and Bergou, El Houcine and Karatsenidis, Konstantinos and Canini, Marco and Kalnis, Panos},
  booktitle={2021 IEEE 41st International Conference on Distributed Computing Systems (ICDCS)},
  pages={561--572},
  year={2021},
  organization={IEEE}
}

@article{fatkhullin2023momentum,
  title={Momentum Provably Improves Error Feedback!},
  author={Fatkhullin, Ilyas and Tyurin, Alexander and Richt{\'a}rik, Peter},
  journal={Advances in Neural Information Processing Systems},
  year={2023}
}

@article{huang2022lower,
  title={Lower Bounds and Nearly Optimal Algorithms in Distributed Learning with Communication Compression},
  author={Huang, Xinmeng and Chen, Yiming and Yin, Wotao and Yuan, Kun},
  journal={Advances in Neural Information Processing Systems},
  year={2022}
}

@article{philippenko2021preserved,
  title={Preserved central model for faster bidirectional compression in distributed settings},
  author={Philippenko, Constantin and Dieuleveut, Aymeric},
  journal={Advances in Neural Information Processing Systems},
  volume={34},
  pages={2387--2399},
  year={2021}
}

@inproceedings{liu2020double,
  title={A double residual compression algorithm for efficient distributed learning},
  author={Liu, Xiaorui and Li, Yao and Tang, Jiliang and Yan, Ming},
  booktitle={International Conference on Artificial Intelligence and Statistics},
  pages={133--143},
  year={2020},
  organization={PMLR}
}

@book{lan2020first,
  title={First-order and stochastic optimization methods for machine learning},
  author={Lan, Guanghui},
  year={2020},
  publisher={Springer}
}

@article{mishchenko2022asynchronous,
  title={Asynchronous {SGD} beats minibatch {SGD} under arbitrary delays},
  author={Mishchenko, Konstantin and Bach, Francis and Even, Mathieu and Woodworth, Blake},
  journal={Advances in Neural Information Processing Systems},
  year={2022}
}

@article{kingma2014adam,
  title={Adam: A method for stochastic optimization},
  author={Kingma, Diederik P and Ba, Jimmy},
  journal={International Conference on Learning Representations},
  year={2015}
}

@article{tyurin2023optimal,
  title={Optimal Time Complexities of Parallel Stochastic Optimization Methods Under a Fixed Computation Model},
  author={Tyurin, Alexander and Richt{\'a}rik, Peter},
  journal = {Advances in Neural Information Processing Systems},
  year = {2023},
}

@article{tyurin20232direction,
  title={{2Direction}: Theoretically Faster Distributed Training with Bidirectional Communication Compression},
  author={Tyurin, Alexander and Richt{\'a}rik, Peter},
  journal = {Advances in Neural Information Processing Systems},
  year = {2023},
}

@inproceedings{gruntkowska2023ef21,
  title={{EF21-P} and friends: Improved theoretical communication complexity for distributed optimization with bidirectional compression},
  author={Gruntkowska, Kaja and Tyurin, Alexander and Richt{\'a}rik, Peter},
  booktitle={International Conference on Machine Learning},
  pages={11761--11807},
  year={2023},
  organization={PMLR}
}

@inproceedings{szlendak2021permutation,
  title={Permutation Compressors for Provably Faster Distributed Nonconvex Optimization},
  author={Szlendak, Rafa{\l} and Tyurin, Alexander and Richt{\'a}rik, Peter},
  booktitle={International Conference on Learning Representations},
  year={2021}
}

@article{tyurin2024shadowheart,
  title={{Shadowheart {SGD}}: Distributed Asynchronous {SGD} with Optimal Time Complexity Under Arbitrary Computation and Communication Heterogeneity},
  author={Tyurin, Alexander and Pozzi, Marta and Ilin, Ivan and Richt{\'a}rik, Peter},
  journal={Advances in Neural Information Processing Systems},
  volume={37},
  year={2024}
}

@inproceedings{huang2012close,
  title={A close examination of performance and power characteristics of 4G LTE networks},
  author={Huang, Junxian and Qian, Feng and Gerber, Alexandre and Mao, Z Morley and Sen, Subhabrata and Spatscheck, Oliver},
  booktitle={Proceedings of the 10th international conference on Mobile systems, applications, and services},
  pages={225--238},
  year={2012}
}

@inproceedings{narayanan2021variegated,
  title={A variegated look at 5G in the wild: performance, power, and QoE implications},
  author={Narayanan, Arvind and Zhang, Xumiao and Zhu, Ruiyang and Hassan, Ahmad and Jin, Shuowei and Zhu, Xiao and Zhang, Xiaoxuan and Rybkin, Denis and Yang, Zhengxuan and Mao, Zhuoqing Morley and others},
  booktitle={Proceedings of the 2021 ACM SIGCOMM 2021 Conference},
  pages={610--625},
  year={2021}
}

@article{touvron2023llama,
  title={Llama: Open and efficient foundation language models},
  author={Touvron, Hugo and Lavril, Thibaut and Izacard, Gautier and Martinet, Xavier and Lachaux, Marie-Anne and Lacroix, Timoth{\'e}e and Rozi{\`e}re, Baptiste and Goyal, Naman and Hambro, Eric and Azhar, Faisal and others},
  journal={arXiv preprint arXiv:2302.13971},
  year={2023}
}

@article{gruntkowska2024improving,
  title={Improving the worst-case bidirectional communication complexity for nonconvex distributed optimization under function similarity},
  author={Gruntkowska, Kaja and Tyurin, Alexander and Richt{\'a}rik, Peter},
  journal={Advances in Neural Information Processing Systems},
  volume={37},
  pages={88807--88873},
  year={2024}
}

@article{zheng2019communication,
  title={Communication-efficient distributed blockwise momentum {SGD} with error-feedback},
  author={Zheng, Shuai and Huang, Ziyue and Kwok, James},
  journal={Advances in Neural Information Processing Systems},
  volume={32},
  year={2019}
}

@article{yue2023core,
  title={Core: Common random reconstruction for distributed optimization with provable low communication complexity},
  author={Yue, Pengyun and Zhao, Hanzhen and Fang, Cong and He, Di and Wang, Liwei and Lin, Zhouchen and Zhu, Song-chun},
  journal={arXiv preprint arXiv:2309.13307},
  year={2023}
}

@inproceedings{tyurin2025proving,
  title={Proving the Limited Scalability of Centralized Distributed Optimization via a New Lower Bound Construction},
  author={Tyurin, Alexander},
  booktitle={International Conference on Learning Representations (ICLR)},
  year={2026},
}

@inproceedings{wang2023new,
  title={A New Theoretical Perspective on Data Heterogeneity in Federated Optimization},
  author={Wang, Jiayi and Wang, Shiqiang and Chen, Rong-Rong and Ji, Mingyue},
  booktitle={Federated Learning and Analytics in Practice: Algorithms, Systems, Applications, and Opportunities},
  year={2023}
}

@article{krizhevsky2010cifar,
  title={Cifar-10 (canadian institute for advanced research)},
  author={Krizhevsky, Alex and Nair, Vinod and Hinton, Geoffrey},
  journal={URL http://www. cs. toronto. edu/kriz/cifar. html},
  volume={5},
  number={4},
  pages={1},
  year={2010}
}

@InProceedings{He_2016_CVPR,
author = {He, Kaiming and Zhang, Xiangyu and Ren, Shaoqing and Sun, Jian},
title = {Deep Residual Learning for Image Recognition},
booktitle = {Proceedings of the IEEE Conference on Computer Vision and Pattern Recognition (CVPR)},
month = {June},
year = {2016}
}
\bibliographystyle{apalike}

\appendix
\newpage
\tableofcontents

\newpage
\section{Notations}
\begin{table}[h]
\centering
\begin{tabular}{ll}
\toprule
\textbf{Symbol} & \textbf{Description} \\
\midrule
$[n]$ & Set $\{1, \dots, n\}$ \\
$\mathbb{R}^d$ & $d$-dimensional Euclidean space \\
$\mathbb{S}^d$ & Space of symmetric $d \times d$ matrices \\
$\|x\|$ & Euclidean norm (vectors) \\
$\|\mH\|$ & Spectral norm (matrices) \\
$n$ & Number of workers \\
$\cC_{\cdot}^k$ & Worker-to-server compressor (worker $i$, iteration $k$) \\
$\cC_{\textnormal{s},\cdot}^k$ & Server-to-worker compressor (worker $i$, iteration $k$) \\
$g = \cO(f)$ & There exists $C > 0$ such that $g(z) \le C \, f(z)$ for all $z \in \cZ$. \\
$g = \Omega(f)$ & There exists $C > 0$ such that $g(z) \ge C \, f(z)$ for all $z \in \cZ$. \\
$g = \Theta(f)$ & There exist $C_1, C_2 > 0$ such that $C_1 f(z) \le g(z) \le C_2 f(z)$ for all $z \in \cZ$. \\
$\tilde{\cO},\tilde{\Omega},$ and $\tilde{\Theta}$ & The same as $\cO$, $\Omega,$ and $\Theta,$ but up to logarithmic factors. \\
$\Delta$ & Initial optimality gap, $\Delta \eqdef f(x^0) - f^*$. \\
$h$ & Time to compute one stochastic gradient. \\
$\tau$ & Time to transmit one coordinate from a worker to the server. \\
$\kappa$ & Time to transmit one coordinate from the server to a worker. \\
$h_i, \tau_i, \kappa_i$ & Corresponding times for worker $i$ in the heterogeneous-time setting. \\
\bottomrule
\end{tabular}
\end{table}

\section{Inkheart SGD Method with Heterogeneous Computations and Communications}
\label{sec:heter_inkheart}
In this section, we provide an algorithm that extends \ref{eq:inkheart} from Section~\ref{sec:homog}. See the description in Section~\ref{sec:ext}.
\begin{figure}[h]
\begin{theorembox}
\centerline{The Heterogeneous-Time Inkheart SGD Method}
\centerline{(with Heterogeneous Computations and Communications)}
Initialize vector $x^0 \in \R^d$ and take $x^0_i = x^0$ for all $i \in [n]$, step size $\gamma > 0$, params $\{b_i, m_i, \ell_i\}_{i \in [n]},$ compressor $\cC^k_{ij} \in \mathbb{U}(\omega),$ and $\cC^k_{\textnormal{s},ij} \in \mathbb{U}(\omega_s)$ for all $i,j,k \geq 0.$ Choose the weights 
\begin{align}
  \textstyle \beta_i = \frac{w_{i}}{\sum\limits_{j=1}^n w_{j}}, \quad
\textnormal{where} \quad
w_i = \left(\frac{8 \omega}{m_i} + \frac{8 \sigma^2 \omega}{\varepsilon b_i m_i} 
  + \frac{8 \sigma^2}{\varepsilon b_i} 
  + \frac{\omega_\textnormal{s} \omega}{p m_i \ell_i }
  + \frac{\omega_\textnormal{s}}{p \ell_i}\right)^{-1} \quad \textnormal{ for all } i \in [n].
  \label{eq:weights}
\end{align}
Then, iterate the following steps for $k = 0, 1, \dots$:
\begin{equation}
  \tag{Heterogeneous-Time Inkheart SGD}\label{eq:inkheart_heter}
\begin{aligned}
    g^{k} &= \textstyle \sum \limits_{i=1}^{n} \frac{\beta_i}{b_i m_i} \sum\limits_{j=1}^{m_i} \cC^k_{ij} \left(\sum\limits_{r=1}^{b_i} \nabla f(x_i^{k};\xi_{ir}^{k})\right), \\ 
    x^{k+1} &= x^k - \gamma g^k, \\
    c^k &\sim \textnormal{Bernoulli}(p), \textnormal{ where } p \in (0, 1] \text{ is a parameter}, \\
    x^{k+1}_i &= 
    \begin{cases}
        x^{k+1} & \text{if } c^k = 1,\\
        \textstyle x_i^k + \frac{1}{\ell_i} \sum\limits_{j=1}^{\ell_i} \cC^k_{\textnormal{s},ij}(x^{k+1} - x^k) & \text{if } c^k = 0.
    \end{cases}
\end{aligned}
\end{equation}
\end{theorembox}
\end{figure}
\section{Properties of Functional Inequality and Proofs}
\QUADRATICINEQUALITY*
\begin{proof}
  Note that $\nabla f(x) = \mH x + b$ and 
\begin{align*}
  &\textstyle \norm{\frac{1}{n} \sum\limits_{i=1}^n (\nabla f(x+u_i) - \nabla f(x))}^2 \\
  &= \norm{\frac{1}{n} \sum\limits_{i=1}^n (\mH(x+u_i) + b - Hx- b)}^2
  = \norm{\frac{1}{n} \sum\limits_{i=1}^n \mH u_i}^2
  \leq \sqnorm{\mH} \norm{\frac{1}{n} \sum\limits_{i=1}^n u_i}^2.
\end{align*}
\end{proof}
In the proofs, we use the following auxiliary assumption which is equivalent to Assumption~\ref{ass:AB_assumption} due to Theorem~\ref{thm:ab_with_weights}.
\begin{assumption}[Weighted Functional $(L_A, L_B)$ Inequality]\label{ass:AB_weighted_assumption}
    There exist constants $L_A,L_B \geq 0$ such that
      \begin{align}\label{eq:wfunctional}
        &\textstyle \norm{\sum\limits_{i=1}^n \beta_i (\nabla f(x+u_i) - \nabla f(x))}^2 \leq L_A^2\left(\sum\limits_{i=1}^n \beta_i \norm{u_i}^2\right) + L_B^2 \norm{\sum\limits_{i=1}^n \beta_i u_i}^2
  \end{align}
  for all $n \geq 1,$ $x, u_1, \dots, u_n \in \R^d,$ and $\beta_1, \dots, \beta_n \in \R$ such that $0\leq \beta_i \leq 1$ for all $i \in [n]$ and $\sum \limits_{i = 1}^n \beta_i = 1.$
\end{assumption}

\begin{restatable}{theorem}{ABWEIGHTS}\label{thm:ab_with_weights}
  The function $f$ satisfies Assumption~\ref{ass:AB_assumption} if and only if $f$ satisfies Assumption~\ref{ass:AB_weighted_assumption}.
\end{restatable}

\begin{proof}
Assumption \ref{ass:AB_assumption} follows from Assumption~\ref{ass:AB_weighted_assumption} with $\beta_i = \nicefrac{1}{n}$ for all $i \in [n].$

We now prove the other direction. First, we fix arbitrary vectors $x \in \mathbb{R}^d$ and $u_1, \dots, u_n \in \mathbb{R}^d$. 
Assume that $\beta \eqdef (\beta_1, \dots, \beta_n) \in \mathbb{Q}^n$. Then there exists a common denominator $k$ such that 
$\beta = \left(\frac{k_1}{k}, \dots, \frac{k_n}{k}\right)$ and $\sum_{i = 1}^n k_i = k$. 
Let us define a sequence of vectors 
$\tilde{u} = (\underbrace{u_1, \dots, u_1}_{k_1}, \dots, \underbrace{u_n, \dots, u_n}_{k_n})$. 
Then,
\begin{align*}
&\norm{\sum\limits_{i=1}^n \beta_i (\nabla f(x+u_i) - \nabla f(x))}^2
= \norm{\frac{1}{k} \sum\limits_{i=1}^n k_i (\nabla f(x+u_i) - \nabla f(x))}^2 \\
&= \norm{\frac{1}{k} \sum\limits_{i=1}^n \sum\limits_{j = 1}^{k_i} (\nabla f(x+u_i) - \nabla f(x))}^2
= \norm{\frac{1}{k} \sum\limits_{i=1}^{k}  (\nabla f(x+\tilde{u}_i) - \nabla f(x))}^2.
\end{align*}
Due to Assumption~\ref{ass:AB_assumption} we get 
\begin{align*}
\norm{\frac{1}{k} \sum\limits_{i=1}^{k}  (\nabla f(x+\tilde{u}_i) - \nabla f(x))}^2 
&\leq L_A^2\left(\frac{1}{k} \sum\limits_{i=1}^k \norm{\tilde{u}_i}^2\right) 
  + L_B^2 \norm{\frac{1}{k} \sum\limits_{i=1}^k \tilde{u}_i}^2 \\
&= L_A^2\left(\frac{1}{k} \sum\limits_{i=1}^n k_i \norm{u_i}^2\right) 
  + L_B^2 \norm{\frac{1}{k} \sum\limits_{i=1}^n k_i u_i}^2.
\end{align*}
Since $\beta_i = \frac{k_i}{k}$, the proof of this case is completed.
To finish the proof for an arbitrary $\beta \in \R^n$ it remains to note that inequality has the form 
$g(\beta) < 0$, where $g$ is continuous function on $\R^d$. 
Since inequality holds for all $\beta \in \Q^n$, by continuity, it extends to all of $\R^d$.
\end{proof}

\section{Auxiliary Lemmas and Definitions}
\begin{definition}
    \label{def:rand_k}
    Assume that $S$ is a random subset from $[d],$ $|S| = K,$ $K \in [d].$ A stochastic mapping $\cC\,:\, \R^d \times \mathbb{S}_{\nu} \rightarrow \R^d$ is Rand$K$ if
    $$\cC(x;S) = \frac{d}{K} \sum_{j \in S} x_j e_j,$$ where $\{e_i\}_{i=1}^d$ is the standard unit basis.
\end{definition}

\begin{lemma}[Variance decomposition; Folklore result]\label{lemma:variance_decomposition}
Let $x \in \R^d$ be a random vector with finite mean and variance. Then for any deterministic vector
$c \in \R^d$, we have the identity
\begin{align*}
  \Exp{\sqnorm{x - \Exp{x}}} = \Exp{\sqnorm{x - c}} - \sqnorm{\Exp{x} - c}.
\end{align*}
\end{lemma}

\begin{lemma}\label{lemma:variance_equality}
  Let $x_i \in \R^d, i \in [k]$ be mutually independent random vectors with expectations $\Exp{x_i}$. 
  Then, for all $\beta_1, \dots, \beta_k \in \R,$
  \begin{align*}
    \Exp{\sqnorm{\sum \limits_{i = 1}^k \beta_i (x_i - \Exp{x_i})}} 
    = \sum \limits_{i = 1}^k \beta_i^2 \Exp{\sqnorm{x_i - \Exp{x_i}}}.
  \end{align*}
\end{lemma}
\begin{proof}
It suffices to show that the cross-terms vanish. For any $i, j \in [k]$ with $i \ne j$, by the independence of
$x_i$ and $x_j$, we obtain
\begin{align*}
\Exp{\lin{\beta_i (x_i - \Exp{x_i}), \beta_j (x_j - \Exp{x_j})}}
= \beta_i \beta_j \lin{\Exp{x_i - \Exp{x_i}}, \Exp{x_j - \Exp{x_j}}} = 0.
\end{align*}
\end{proof}

\begin{lemma}\label{lemma:averaging_compressors}
Let $\cS$ be a set of indices. Suppose compressors $\cC_{i} \in \U(\omega)$ for all $i \in \cS$ (Definition~\ref{def:unbiased_compressor}), 
and let $\beta_i \geq 0$ be scalars such that $\sum \limits_{i \in \cS} \beta_i = 1$. 
Then the compressor $\bar{\cC}(x) := \sum \limits_{i \in \cS} \beta_i \cC_{i}(x)$ belongs to $\U\left(\omega \sum \limits_{i \in \cS} \beta_i^2\right).$
\begin{proof}
First, we show that $\bar{\cC}(x)$ is unbiased
\begin{align*}
  \Exp{\bar{\cC}(x)} = \frac{1}{|\cS|} \sum \limits_{i \in \cS} \Exp{\cC_{i}(x)} = x.
\end{align*}
Then, since the compressors $\cC_{i}$ are mutually independent, we may use Lemma~\ref{lemma:variance_equality} and get
\begin{align*}
&\Exp{\sqnorm{\bar{\cC}(x) - x}} \\
&= \Exp{\sqnorm{\sum \limits_{i \in \cS} \beta_i \left(\cC_{i}(x) - x\right)}}
  = \sum \limits_{i \in \cS} \beta_i^2 \Exp{\sqnorm{\cC_{i}(x) - x}}
  \leq \left(\omega \sum \limits_{i \in \cS} \beta_i^2\right) \norm{x}^2.
\end{align*}
\end{proof}

\end{lemma}

\begin{lemma}\label{lemma:appropriate_gamma}
Consider $f(\gamma) = \frac{1}{2 \gamma} - c - d \gamma$, 
where $d = \sum \limits_{i=1}^{k} d_i$ and $c > 0, d_i > 0$ for all $i \in [k]$, 
then for any $\gamma \in \left(0, \frac{1}{(k + 1) \max\left\{2c, \sqrt{2 d_1}, \cdots , \sqrt{2 d_k}\right\}}\right]$
holds $f(\gamma) \geq 0$.
\end{lemma}
\begin{proof}
Clearly,
\begin{align*}
f(\gamma) = \frac{1 - 2c \gamma - 2d \gamma^2}{2 \gamma} > 0 \quad \iff \quad 
\frac{d \gamma^2 + c \gamma - \frac{1}{2}}{\gamma} < 0.
\end{align*}
We find the zeros of the numerator:
\begin{align*}
D = c^2 + 4 d \cdot \frac{1}{2} \quad \iff \quad \gamma = \frac{-c \pm \sqrt{c^2 + 2d}}{2 d}.
\end{align*}
Thus, the maximum $\gamma$ that satisfies the inequality is 
\begin{align*}
&\gamma_{\max} = \frac{-c + \sqrt{c^2 + 2d}}{2 d} = 
\frac{(-c + \sqrt{c^2 + 2d})(c + \sqrt{c^2 + 2d})}{2 d (c + \sqrt{c^2 + 2d})}
= \frac{1}{c + \sqrt{c^2 + 2d}}.
\end{align*}

and $f(\gamma) > 0$ for any $\gamma \in \left(0, \gamma_{\max}\right]$. 
To finish the proof, it remains to note that
\begin{align*}
&\gamma_{\max} \geq \frac{1}{2c + \sqrt{2d}} = \frac{1}{2c + \sqrt{2\sum \limits_{i=1}^{k} d_i}}
\geq \frac{1}{2c + \sum \limits_{i=1}^{k} \sqrt{2 d_i}} \geq
\frac{1}{(k + 1) \max\left\{2c, \sqrt{2 d_1}, \cdots , \sqrt{2 d_k}\right\}}.
\end{align*}
\end{proof}

\begin{lemma}\label{lemma:approximate_solution}
Consider the function $g(x) = ax^3 + bx^2 + cx - 1$. If $a, b, c > 0$ then there exists a unique 
$x_0$ such that $g(x_0) = 0$ and $x_0 \in 
\left[\frac{\bar{x}}{2}, \bar{x}\right]$, 
where $\bar{x} = \frac{1}{\max\left\{\sqrt[3]{a}, \sqrt{b}, c\right\}}$.
\end{lemma}
\begin{proof}
Since $g$ is strictly increasing, there exists only one solution to the equation $g(x) = 0$.
Then, we compute $g$ at the points $\bar{x}$ and $\frac{\bar{x}}{2}$. 
Assume $\sqrt{b} \geq \max\left\{\sqrt[3]{a}, c\right\}$, so $\bar{x} = \frac{1}{\sqrt{b}}$ (other cases are similar).
\begin{align}\label{eq:adksnfdkscjsk}
g(\bar{x}) = \frac{a}{b^\frac{3}{2}} + \frac{b}{b} + \frac{c}{\sqrt{b}} - 1 > 0
\quad \textnormal{and} \quad
g\left(\frac{\bar{x}}{2}\right) = \frac{a}{8 b^\frac{3}{2}} + \frac{b}{4 b} + \frac{c}{2 \sqrt{b}} - 1 < 0.
\end{align}
We obtain the last inequality since $\sqrt{b} \geq \max\left\{\sqrt[3]{a}, c\right\}$ and 
$\frac{1}{8} + \frac{1}{4} + \frac{1}{2} < 1$. The inequalities in \eqref{eq:adksnfdkscjsk} and the fact that 
$g$ is strictly increasing imply that $x_0 \in \left[\frac{\bar{x}}{2}, \bar{x}\right]$.
\end{proof}

\begin{lemma}\label{lemma:bounded_initial_gradient}
Assume function $f: \R^d \rightarrow \R$ satisfies Assumption~\ref{ass:lipschitz_constant}
and Assumption~\ref{ass:lower_bound}.
Then, we have $\sqnorm{\nabla f(x^0)} \leq 2 L \Delta$.
\end{lemma} 
\begin{proof}
Consider the property of $L$-smoothness:
\begin{align*}
f(y) - f(x) \leq \langle \nabla f(x), y - x \rangle + \frac{L}{2} \|y - x\|^2.
\end{align*}
The right-hand side is a quadratic function. Its minimum is attained at the point
\begin{align*}
y^* = x - \frac{1}{L} \nabla f(x).
\end{align*}
Substituting $y = y^*$ into the smoothness inequality yields
\begin{align*}
f(y^*) - f(x) &\leq -\frac{1}{L} \sqnorm{\nabla f(x)} + \frac{1}{2L} \sqnorm{\nabla f(x)} = -\frac{1}{2L} \sqnorm{\nabla f(x)}.
\end{align*}
Rearranging the terms and setting $x = x^0$, we obtain
\begin{align*}
\sqnorm{\nabla f(x^0)} \leq 2L (f(x^0) - f(y^*)) \leq 2 L \Delta.
\end{align*}
\end{proof}

\begin{lemma}\label{lemma:quadratic_simplex}
Let $w_i > 0$ for all $i = 1, \dots, n$. Consider the optimization problem
\begin{align*}
\begin{cases*}
\displaystyle \sum_{i=1}^{n} \beta_i^2 w_i^{-1} \rightarrow \min_{\beta} \\
\displaystyle \sum_{i = 1}^n \beta_i = 1, \quad 0 \leq \beta_i \leq 1.
\end{cases*}
\end{align*}
The unique optimal solution is given by
\begin{align*}
  \beta_i^* = \frac{w_i}{\sum_{j=1}^n w_j},
\end{align*}
and the corresponding minimum value equals $\left(\sum_{j=1}^n w_j\right)^{-1}$.
\end{lemma}

\begin{proof}
We first solve the problem subject only to the equality constraint $\sum_{i=1}^n \beta_i = 1$ using the method of Lagrange multipliers. Define the Lagrangian
\begin{align*}
  \mathcal{L}(\beta, \lambda) = \sum_{i=1}^{n} \frac{\beta_i^2}{w_i} - \lambda \left( \sum_{i=1}^n \beta_i - 1 \right).
\end{align*}
Setting the partial derivatives to zero yields
\begin{align*}
  \frac{\partial \mathcal{L}}{\partial \beta_i} = \frac{2\beta_i}{w_i} - \lambda = 0 \quad \Rightarrow \quad \beta_i = \frac{\lambda w_i}{2}, \quad i = 1, \dots, n.
\end{align*}
Substituting this expression into the equality constraint gives $\frac{\lambda}{2} \sum_{j=1}^n w_j = 1$, hence $\frac{\lambda}{2} = \left(\sum_{j=1}^n w_j\right)^{-1}$. This immediately implies
\begin{align*}
  \beta_i^* = \frac{w_i}{\sum_{j=1}^n w_j}.
\end{align*}
Since $w_i > 0$, we have $\beta_i^* > 0$. Furthermore, $\beta_i^* \leq 1$ because $w_i \leq \sum_{j=1}^n w_j$ for all $i$. 
Thus, the box constraints $0 \leq \beta_i \leq 1$ are automatically satisfied. 
Finally, evaluating the objective at $\beta^*$ gives
\begin{align*}
  \sum_{i=1}^{n} \frac{(\beta_i^*)^2}{w_i} = \sum_{i=1}^{n} \frac{w_i^2}{w_i \left(\sum_{j=1}^n w_j\right)^2} = \frac{\sum_{i=1}^n w_i}{\left(\sum_{j=1}^n w_j\right)^2} = \left(\sum_{j=1}^n w_j\right)^{-1},
\end{align*}
which completes the proof.
\end{proof}

\begin{lemma}
  \label{lemma:page_lemma}
  Suppose that Assumption~\ref{ass:lipschitz_constant} holds and let $x^{k+1} = x^{k} - \gamma g^{k}$. Then for any $g^{k} \in \R^d$ and $\gamma > 0$, we have
  \begin{eqnarray}
    \label{eq:page_lemma}
    f(x^{k + 1}) \leq f(x^k) - \frac{\gamma}{2}\norm{\nabla f(x^k)}^2 - \left(\frac{1}{2\gamma} - \frac{L}{2}\right)
    \norm{x^{k+1} - x^k}^2 + \frac{\gamma}{2}\norm{g^{k} - \nabla f(x^k)}^2.
  \end{eqnarray}
\end{lemma}

\begin{proof}
  Using $L-$smoothness, we have 
  \begin{align*}
    f(x^{k+1}) &\leq f(x^k) + \inp{\nabla f(x^k)}{x^{k+1} - x^{k}} + \frac{L}{2} \norm{x^{k+1} - x^{k}}^2 \\
    &= f(x^k) - \gamma \inp{\nabla f(x^k)}{g^k} + \frac{L}{2} \norm{x^{k+1} - x^{k}}^2.
  \end{align*}
  Next, due to $-\inp{x}{y} = \frac{1}{2}\norm{x - y}^2 - \frac{1}{2}\norm{x}^2 - \frac{1}{2}\norm{y}^2,$ we obtain 
  \begin{align*}
    f(x^{k+1}) \leq f(x^k) -\frac{\gamma}{2} \norm{\nabla f(x^k)}^2 - \left(\frac{1}{2\gamma} - \frac{L}{2}\right) \norm{x^{k+1} - x^{k}}^2 + \frac{\gamma}{2}\norm{g^k - \nabla f(x^k)}^2.
  \end{align*}
\end{proof}

\section{Proofs for \ref{eq:inkheart_heter} and \ref{eq:inkheart}}\label{sec:sh_proofs}

We provide proofs for \ref{eq:inkheart_heter}. Then, the results for \ref{eq:inkheart} are corollaries, since \ref{eq:inkheart_heter} is a generalization of \ref{eq:inkheart}.

We denote by $\ExpSub{k}{\cdot}$ the expectation conditional on all workers receiving $x_i^k$ and by
$\ExpSub{k, \xi}{\cdot}$ the expectation conditional on all workers computing the stochastic gradients at iteration $k$.

\begin{theorem}\label{thm:main_iteration_complexity} Let Assumptions~\ref{ass:lipschitz_constant},
\ref{ass:lower_bound}, \ref{ass:stochastic_variance_bounded} and 
\ref{ass:AB_assumption} be satisfied and suppose that 
$\cC_{ij}$ satisfies Definition~\ref{def:unbiased_compressor} with parameter $\omega$, 
$\cC_{\textnormal{s},ij}$ satisfies Definition~\ref{def:unbiased_compressor} with parameter $\omega_\textnormal{s}$.
Consider \ref{eq:inkheart_heter} with arbitrary $m_i, b_i, \ell_i > 0$ and weights $\beta_i$ (not necessarily defined as in \eqref{eq:weights}; it is sufficient to assume that $\beta_i \in [0, 1]$ and $\sum_{i=1}^{n} \beta_i = 1$)
chosen to satisfy
\begin{align}\label{eq:shadowheart_conditions}
\begin{cases}
8 \sum\limits_{i=1}^n \beta_i^2\frac{\omega}{m_i} \leq 1 \\
8 \sum\limits_{i=1}^n \beta_i^2 \left(\frac{\omega \sigma^2}{m_i b_i \varepsilon} + \frac{\sigma^2}{b_i \varepsilon}
  \right) \leq 1.
\end{cases}
\end{align}

Then, \ref{eq:inkheart_heter} with 
\begin{align}\label{eq:gamma}
\gamma = \frac{1}{6} \times \min \left\{\frac{1}{L_{\max}}, 
  \frac{1}{L_{\max} \sqrt{\sum \limits_{i=1}^{n} 
    \left(\frac{\omega \omega_\textnormal{s}}{p m_i \ell_i}
  + \frac{\omega_\textnormal{s}}{p \ell_i} \right) \beta_i^2}},
  \frac{1}{L_A \sqrt{\frac{1}{p} \sum \limits_{i=1}^{n} \frac{\omega_\textnormal{s} \beta_i}{\ell_i}}},\right\}
\end{align}
converges after at most 

\begin{align}\label{eq:iteration_complexity}
K = 48 \times \frac{\Delta}{\varepsilon} \max\left\{L_{\max}, 
 L_{\max} \sqrt{\sum \limits_{i=1}^{n} \left(\frac{\omega \omega_\textnormal{s}}{p m_i \ell_i}
  + \frac{\omega_\textnormal{s}}{p \ell_i} \right) \beta_i^2},
  L_A \sqrt{\frac{1}{p} \sum \limits_{i=1}^{n} \frac{\omega_\textnormal{s} \beta_i}{\ell_i}}\right\}
\end{align}
iterations.
\end{theorem}

\begin{proof}
Using Assumption~\ref{ass:lipschitz_constant} and Lemma~\ref{lemma:page_lemma}, we get
\begin{align*}
  f(x^{k+1}) 
  &\leq f(x^{k}) - \frac{\gamma}{2} \norm{\nabla f(x^k)}^2 - \left(\frac{1}{2 \gamma} - \frac{L}{2}\right) \norm{x^{k+1} - x^k}^2 + \frac{\gamma}{2} \sqnorm{g^k - \nabla f(x^k)}
\end{align*}
Then, we take conditional expectation from both parts and use Lemma~\ref{lemma:grad_estimator_residual} to bound the last term
\begin{align}\label{eq:asfkdjhajfndskj}
\ExpSub{k}{f(x^{k+1})}
&\leq f(x^{k}) - \frac{\gamma}{2} \norm{\nabla f(x^k)}^2 - \left(\frac{1}{2 \gamma} - \frac{L}{2}\right) \ExpSub{k}{\sqnorm{x^{k+1} - x^k}} \\
&\quad + 2 \gamma \left(\sum\limits_{i=1}^n \beta_i^2 \frac{\omega}{m_i}\right)\sqnorm{\nabla f(x^k)} 
  + \gamma \sum \limits_{i=1}^{n} \left(L_A^2 \beta_i  + 2 L^2 \beta_i^2 \frac{\omega}{m_i}\right) \sqnorm{x^k_i - x^k} \nonumber\\
&\quad + \gamma L_B^2 \sqnorm{\sum \limits_{i=1}^{n} \beta_i \left(x^k_i - x^k\right)} 
  + \gamma \sum\limits_{i=1}^n \beta_i^2 \left(\frac{\omega \sigma^2}{m_i b_i} + \frac{\sigma^2}{b_i} 
  \right), \nonumber
\end{align}
By Assumption of the theorem, the inequalities
$2 \gamma \left(\sum\limits_{i=1}^n \beta_i^2 \frac{\omega}{m_i}\right)\sqnorm{\nabla f(x^k)} 
\leq \frac{\gamma}{4}\sqnorm{\nabla f(x^k)}$ and
$\gamma \sum\limits_{i=1}^n \beta_i^2 \left(\frac{\omega \sigma^2}{m_i b_i} + \frac{\sigma^2}{b_i}
  \right) \leq \frac{\gamma \varepsilon}{8}$ hold. 
Then, we sum the inequalities \eqref{eq:idskndsfihg} and \eqref{eq:idkxkdifiwwg} from Lemma~\ref{lemma:wpointrec}  
multiplied by $\sum\limits_{i=1}^n \kappa_i$ and $\eta$, respectively:
\begin{align}\label{eq:kdfngfoksdofk}
&\sum\limits_{i=1}^n \kappa_i \beta_i \ExpSub{k}{\norm{x^{k+1}_i - x^{k+1}}^2}
+ \eta \ExpSub{k}{\norm{\sum\limits_{i=1}^n \beta_i \left(x^{k+1}_i - x^{k+1}\right)}^2} \\
&\leq (1 - p) \left[\sum\limits_{i=1}^n \frac{\omega_\textnormal{s} \kappa_i \beta_i}{\ell_i} 
  \ExpSub{k}{\norm{x^{k+1} - x^k}^2}
  + \sum\limits_{i=1}^n\beta_i \kappa_i \norm{x_i^k - x^k}^2\right] \nonumber\\
&\quad + (1 - p) \left[\sum\limits_{i=1}^n \frac{\omega_\textnormal{s} \eta \beta_i^2}{\ell_i} 
\ExpSub{k}{\norm{x^{k+1} - x^k}^2}
+ \eta \norm{\sum_{i=1}^{n} \beta_i \left(x^{k}_i - x^{k}\right)}^2\right]. \nonumber
\end{align}

Then, we add the inequality~\eqref{eq:kdfngfoksdofk}
to both parts of the inequality~\eqref{eq:asfkdjhajfndskj}
to construct a Lyapunov function.

\begin{align}\label{eq:knsmdfksdhsso}
&\ExpSub{k}{f(x^{k+1})} + \sum\limits_{i=1}^n \kappa_i \beta_i \ExpSub{k}{\norm{x^{k+1}_i - x^{k+1}}^2}
  + \eta \ExpSub{k}{\norm{\sum\limits_{i=1}^n \beta_i \left(x^{k+1}_i - x^{k+1}\right)}^2}\\
&\leq f(x^{k}) - \frac{\gamma}{4} \norm{\nabla f(x^k)}^2 \nonumber \\
&-\quad \underbrace{\left(\frac{1}{2 \gamma} - \frac{L}{2} 
    - (1 - p)\sum \limits_{i=1}^{n}\frac{\omega_\textnormal{s} \kappa_i \beta_i}{\ell_i} 
    - (1 - p)\sum \limits_{i=1}^{n}\frac{\omega_\textnormal{s} \eta \beta_i^2}{\ell_i}\right)}_{A}
      \ExpSub{k}{\sqnorm{x^{k+1} - x^k}} \nonumber \\
&\quad + \sum \limits_{i=1}^{n} \underbrace{\left(\gamma L_A^2 \beta_i  + 2 \gamma  L^2 \beta_i^2 \frac{\omega}{m_i} 
  + (1 - p)\kappa_i \beta_i\right)}_{B_i} \sqnorm{x^k_i - x^k} \nonumber \\
&\quad+ \underbrace{\left(\gamma L_B^2 + (1 - p) \eta \right)}_{C} \sqnorm{\sum \limits_{i=1}^{n} \beta_i \left(x^k_i - x^k\right)} 
  + \frac{\gamma \varepsilon}{8}. \nonumber
\end{align}

Next, we should choose $\kappa_i, \eta$ to achieve $\eta  = C$ and $\kappa_i \beta_i = B_i$ for all $i \in [n]$:
\begin{align*}
&\begin{cases*}
\kappa_i = 2 \gamma L^2 \beta_i\frac{\omega}{m_i}
  + \gamma L_A^2 + \kappa_i (1 - p) \\
\eta = \left(\gamma L_B^2 + \eta (1 - p)\right)
\end{cases*}
&&\begin{cases*}
\kappa_i = \frac{2 \gamma \omega L^2 \displaystyle \beta_i}
    {p m_i}
  + \frac{\gamma L_A^2}{p} \\
\eta = \frac{\gamma L_B^2}{p}
\end{cases*}.
\end{align*}

Then, we substitute $\kappa_i, i \in[n]$ and $\eta$ into $A$: 
\begin{align*}
A &= \frac{1}{2 \gamma} - \frac{L}{2} - (1 - p)\sum\limits_{i=1}^n \frac{\omega_\textnormal{s} \kappa_i \beta_i}{\ell_i} 
    - (1 - p) \sum\limits_{i=1}^n \frac{\omega_\textnormal{s} \eta \beta_i^2}{\ell_i} \\
&\geq \frac{1}{2 \gamma} - \frac{L}{2} - \sum\limits_{i=1}^n \frac{\omega_\textnormal{s} \kappa_i \beta_i}{\ell_i} 
    - \sum\limits_{i=1}^n \frac{\omega_\textnormal{s} \eta \beta_i^2}{\ell_i} \\
&= \frac{1}{2 \gamma} - \frac{L}{2} 
  - \gamma \frac{2 \omega \omega_\textnormal{s} L^2}{p}\sum \limits_{i=1}^{n} \frac{\beta_i^2}
    {\ell_i m_i}
  - \gamma \frac{L_A^2}{p} \sum \limits_{i=1}^{n} \frac{\omega_\textnormal{s} \beta_i}{\ell_i}
  - \gamma \frac{L_B^2}{p} \sum \limits_{i=1}^{n} \frac{\omega_\textnormal{s} \beta_i^2}{\ell_i}\\
&\geq \frac{1}{2 \gamma} - L
  - 2 \gamma \left(\underbrace{\frac{\omega \omega_\textnormal{s} L^2}{p}\sum \limits_{i=1}^{n} \frac{\beta_i^2}
    {\ell_i m_i}
  + \frac{L_B^2}{p} \sum \limits_{i=1}^{n} \frac{\omega_\textnormal{s} \beta_i^2}{\ell_i}}_{d_1}
  + \underbrace{\frac{L_A^2}{p} \sum \limits_{i=1}^{n} \frac{\omega_\textnormal{s} \beta_i}{\ell_i}}_{d_2}\right).
\end{align*}
Applying Lemma~\ref{lemma:appropriate_gamma} to the function 
$g(\gamma) = \frac{1}{2 \gamma} - L - (2 d_1 + 2 d_2) \gamma$,
we conclude that $A \geq 0$ holds whenever
\begin{align*}
\gamma \leq \frac{1}{3} \times \frac{1}{\max\left\{2L, \sqrt{2 \cdot 2 d_1}, \sqrt{2 \cdot 2 d_2}\right\}}
  = \frac{1}{6} \times \frac{1}{\max\left\{L, \sqrt{d_1}, \sqrt{d_2}\right\}}.
\end{align*}

By definition, $\max \left\{L, L_B\right\} \leq L_{\max}$, which implies
\begin{align*}
d_1 \leq L_{\max}^2 \left(\frac{\omega \omega_\textnormal{s}}{p}\sum \limits_{i=1}^{n} \frac{\beta_i^2}
    {\ell_i m_i}
  + \frac{1}{p} \sum \limits_{i=1}^{n} \frac{\omega_\textnormal{s} \beta_i^2}{\ell_i}\right).
\end{align*} 
Thus, our choice of $\gamma$ in \eqref{eq:gamma} guarantees $A \geq 0$.
Dropping the non-positive term $- A \cdot \ExpSub{k}{\sqnorm{x^{k+1} - x^k}}$ and using 
$\eta  = C$, $\kappa_i \beta_i = B_i$, we transform the inequality \eqref{eq:knsmdfksdhsso} into

\begin{align*}
&\ExpSub{k}{f(x^{k+1})} + \sum\limits_{i=1}^n \kappa_i \beta_i \ExpSub{k}{\norm{x^{k+1}_i - x^{k+1}}^2}
  + \eta \ExpSub{k}{\norm{\sum\limits_{i=1}^n \beta_i \left(x^{k+1}_i - x^{k+1}\right)}^2}\\
&\leq f(x^{k}) - \frac{\gamma}{4} 
  \norm{\nabla f(x^k)}^2
  + \sum \limits_{i=1}^{n} \kappa_i \beta_i \sqnorm{x^k_i - x^k} 
  + \eta \sqnorm{\sum \limits_{i=1}^{n} \beta_i \left(x^k_i - x^k\right)} 
  + \frac{\gamma \varepsilon}{8}.
\end{align*}

Taking the full expectation and summing for $k \in [0, \dots, K - 1]$, we get
\begin{align*}
&\Exp{f(x^K)} + \sum\limits_{i=1}^n \kappa_i \beta_i \Exp{\norm{x^{K}_i - x^{K}}^2}
  + \eta \Exp{\norm{\sum\limits_{i=1}^n \beta_i \left(x^{K}_i - x^{K}\right)}^2}\\
&\leq f(x^{0}) - \frac{\gamma}{4} \sum\limits_{k=0}^{K - 1} \Exp{\norm{\nabla f(x^k)}^2}
  + \sum \limits_{i=1}^{n} \kappa_i \beta_i \sqnorm{x^0_i - x^0} 
  + \eta \sqnorm{\sum \limits_{i=1}^{n} \beta_i \left(x^0_i - x^0\right)} 
  + K \frac{\gamma \varepsilon}{8}.
\end{align*}

From the construction of \ref{eq:inkheart_heter}, $x_i^0 = x^0$ for $i \in [n]$.
Ignoring non-negative terms we finally obtain
\begin{align*}
&\frac{\gamma}{4} \sum\limits_{k=0}^{K - 1} \Exp{\norm{\nabla f(x^k)}^2}
\leq f(x^{0}) - \Exp{f(x^K)} + K \frac{\gamma \varepsilon}{8}
\quad \iff \quad 
\sum\limits_{k=0}^{K - 1} \Exp{\norm{\nabla f(x^k)}^2}
\leq \frac{4 \Delta}{\gamma} + K\frac{\varepsilon}{2}.
\end{align*}

To achieve $\frac{1}{K} \sum\limits_{k=0}^{K - 1} \Exp{\norm{\nabla f(x^k)}^2}  \leq \varepsilon$ we should take 
$\frac{4 \Delta}{\gamma K} \leq \frac{\varepsilon}{2}$. 
Thus, it is sufficient to run \ref{eq:inkheart_heter} for $K = \frac{8 \Delta}{\gamma \varepsilon}$ iterations.
\end{proof}
\subsection{Time Complexity}\label{sec:homo_time_complexity}

In this section, we restrict our attention to the case where all compressors used by the server and the workers are $\text{Rand}1$. We begin by selecting the probability $p$.
The following lemma shows which probability $p$ we can choose.

\begin{lemma}\label{lemma:server_expectation}
Consider the probabilistic mechanism from \ref{eq:inkheart_heter} and assume that the compressors $\cC_{\textnormal{s},ij}$
are Rand$1$ (which means the server sends only one coordinate at a time).
Then, for $p = \min\left\{\frac{\ell_{\min}}{d}, 1\right\}$, where $\ell_{\min} \eqdef \flr{\frac{t}{\kappa_{\max}}}$ 
and $\kappa_{\max} \eqdef \displaystyle \max_{i \in [n]}{\kappa_i}$,
the condition $\Exp{t_{\textnormal{server}}^k} \leq 2t$ holds.
\end{lemma}
\begin{proof}
Sending one coordinate takes $\kappa$ seconds and the server either sends the full point ($d$ coordinates)
or sends a compressed point, which takes less than $t$ seconds by the definition of the time budget. 
When the server sends the full point, the workers synchronize and wait for the slowest worker.
Thus, we obtain
\begin{align*}
  \Exp{t_{\textnormal{server}}^k} 
  \leq p \kappa_{\max} d + (1 - p) t \leq p \kappa_{\max} d + t 
\end{align*}
Since $p = \min\left\{\frac{\ell_{\min}}{d}, 1\right\} \leq \frac{\ell_{\min}}{d}$, we get
\begin{align*}
\Exp{t_{\textnormal{server}}^k} \leq \frac{\ell_{\min}}{d} \kappa_{\max} d + t 
  = \flr{\frac{t}{\kappa_{\max}}} \kappa_{\max}  + t 
  \leq 2 t.
\end{align*}
\end{proof}

Below we obtain the optimal time budget $t$, the optimal weights $w_i$, and the time complexity for the scenario when all workers have different time performances $h_i, \tau_i$ and $\kappa_i$ for $i \in [n]$.

\THMHOMOTIMEASYNC*

\begin{proof}
First, we require $t \geq \displaystyle \max_{i \in [n]} \left\{h_i, \tau_i, \kappa_i\right\}$ to ensure that at 
least one gradient 
is computed and compressed by each worker, and the server sends a new point compressed by at least one compressor.  
Then, we show that with our choice of $t$
the conditions~\eqref{eq:shadowheart_conditions} are satisfied and \ref{eq:inkheart_heter} converges. Notice that

\begin{align*}
\begin{cases*}
8 \sum\limits_{i=1}^n \beta_i^2\frac{\omega}{m_i} \leq \sum\limits_{i=1}^n \beta_i^2 w_i^{-1} \leq 1 \\
8 \sum\limits_{i=1}^n \beta_i^2 \left(\frac{\omega \sigma^2}{m_i b_i \varepsilon} + \frac{\sigma^2}{b_i \varepsilon}
  \right) \leq \sum\limits_{i=1}^n \beta_i^2 w_i^{-1} \leq 1.\\
\end{cases*}
\end{align*}

The last inequalities follow from Lemma~\ref{lemma:weights_bound}. Moreover, using Lemma~\ref{lemma:weights_bound},
and our choice of $\beta_i$ and $t$, we have
$\sum \limits_{i=1}^{n} \left(\frac{\omega \omega_\textnormal{s}}{p m_i \ell_i}
  + \frac{\omega_\textnormal{s}}{p \ell_i} \right) \beta_i^2
\leq \sum\limits_{i=1}^n \beta_i^2 w_i^{-1} \leq 1$. Therefore,
\begin{align}\label{eq:knsdakfdiune}
L_{\max} \sqrt{\sum \limits_{i=1}^{n} \left(\frac{\omega \omega_\textnormal{s}}{p m_i \ell_i}
  + \frac{\omega_\textnormal{s}}{p \ell_i} \right) \beta_i^2} \leq L_{\max}.
\end{align}

Recall that the time of each iteration consists of computing gradients, sending compressed vectors by the workers
and sending compressed points by the server. Using Lemma~\ref{lemma:server_expectation}, we bound the time complexity
as $\Exp{T_\textnormal{time}} \leq 4t \times K$, where $K$ is the iteration complexity~\eqref{eq:iteration_complexity} 
from Theorem~\ref{thm:main_iteration_complexity}:
\begin{align*}
\Exp{T_\textnormal{time}} &\leq 192 t \times \frac{\Delta}{\varepsilon} \max\left\{L_{\max}, 
 L_{\max} \sqrt{\sum \limits_{i=1}^{n} \left(\frac{\omega \omega_\textnormal{s}}{p m_i \ell_i}
  + \frac{\omega_\textnormal{s}}{p \ell_i} \right) \beta_i^2},
  L_A \sqrt{\frac{1}{p} \sum \limits_{i=1}^{n} \frac{\omega_\textnormal{s} \beta_i}{\ell_i}}\right\}
\end{align*}

Then, using \eqref{eq:knsdakfdiune}, we obtain 
\begin{align}\label{eq:knaskfkcmiesn}
\Exp{T_\textnormal{time}} & \leq 192 \times \frac{\Delta}{\varepsilon} \max\left\{L_{\max} t,
  t \times L_A \sqrt{\frac{1}{p} \sum \limits_{i=1}^{n} \frac{\omega_\textnormal{s} \beta_i}{\ell_i}}\right\}.
\end{align}

To finish the proof we consider the last term from \eqref{eq:knaskfkcmiesn}
and use bounds $\frac{1}{\ell_i} \leq \frac{2 \kappa_i}{t}$ and 
$\frac{1}{p} = \frac{d}{\ell_{\min}} \leq \frac{d}{\frac{t}{2 \kappa_{\max}}} = \frac{2 d \kappa_{\max}}{t}$.
Finally, we obtain
\begin{align*}
t \times L_A \sqrt{\frac{1}{p} \sum \limits_{i=1}^{n} \beta_i \frac{\omega_\textnormal{s}}{\ell_i}}
  \leq t \times L_A \sqrt{\frac{2 d \kappa_{\max}}{t} \sum \limits_{i=1}^{n} \beta_i \frac{2 \omega_\textnormal{s} \kappa_i}{t}}
  = 2 L_A \sqrt{\kappa_{\max} d \omega_\textnormal{s} \sum \limits_{i=1}^{n} \beta_i \kappa_i} \leq 2 d L_A \kappa_{\max}.
\end{align*}
Here we used $\omega_\textnormal{s} \leq d - 1$ and $\sum \limits_{i=1}^{n} \beta_i = 1$.
\end{proof}

\begin{lemma}\label{lemma:weights_bound}
Consider the time budget $t$ defined in Theorem~\ref{thm:main_time_complexity_async}.
Then, the following inequality holds
\begin{align*}
  \sum \limits_{i=1}^{n} \beta_i^2 \underbrace{\left(
    \frac{8 \omega}{m_i} + \frac{8 \sigma^2}{\varepsilon b_i} + \frac{8 \sigma^2 \omega}{\varepsilon b_i m_i}
    + \frac{\omega_\textnormal{s}}{p \ell_i}
    + \frac{d \omega_\textnormal{s} \omega}{\ell_{\min} \ell_i m_i}\right)}_{w_i^{-1}} \leq 1.
\end{align*}
\end{lemma}

\begin{proof}
With our choice of $\beta_i$  defined in \eqref{eq:weights},
we obtain 
\begin{align}\label{eq:jasncieiwnmc}
\sum \limits_{i=1}^{n} \beta_i^2 w_i^{-1} = \frac{1} {\sum\limits_{j=1}^n w_i}
= \left(\sum\limits_{j=1}^n w_j \right)^{-1}.
\end{align}

Since $b_i = \flr{\frac{t}{h_i}} \geq \frac{t}{2h_i}$, 
$m_i = \flr{\frac{t}{\tau_i}} \geq \frac{t}{2\tau_i}$ and $\ell_i = \flr{\frac{t}{\kappa_i}} \geq \frac{t}{2\kappa_i}$,
we get 
\begin{align}\label{eq:sknckmewosl}
w_i &= \left(\frac{8\omega}{m_i} + \frac{8\sigma^2 \omega}{\varepsilon b_i m_i} 
  + \frac{8\sigma^2}{\varepsilon b_i} 
  + \frac{d \omega_\textnormal{s}}{\ell_{\min} \ell_i}
  + \frac{d \omega_\textnormal{s} \omega}{\ell_{\min} \ell_i m_i}\right)^{-1} \nonumber \\
&\geq \left(\frac{8 \omega}{\frac{t}{2\tau_i}} 
  + \frac{8 \sigma^2 \omega}{\varepsilon \frac{t}{2h_i} \frac{t}{2\tau_i}} 
  + \frac{8 \sigma^2}{\varepsilon \frac{t}{2h_i}} 
  + \frac{d \omega_\textnormal{s}}{\frac{t}{2\kappa_{\max}} \frac{t}{2\kappa_i}}
  + \frac{d \omega_\textnormal{s} \omega}{\frac{t}{2\kappa_{\max}} \frac{t}{2\kappa_i} \frac{t}{2\tau_i}}\right)^{-1} \nonumber \\
&= \left(\frac{16 \omega \tau_i}{t} + \frac{16 \sigma^2 h_i}{\varepsilon t}  
  +\frac{32 \sigma^2 \omega h_i \tau_i}{\varepsilon t^2} 
  + \frac{4 d \omega_\textnormal{s} \kappa_{\max} \kappa_i }{t^2}
  + \frac{8 d \omega_\textnormal{s} \omega \kappa_{\max} \kappa_i \tau_i}{t^3}\right)^{-1}.
\end{align}
Substituting \eqref{eq:sknckmewosl} into \eqref{eq:jasncieiwnmc} we obtain
\begin{align*}
\sum \limits_{i=1}^{n} \beta_i^2 w_i^{-1}
\leq \underbrace{\left(\sum \limits_{i=1}^{n} \frac{1}{\frac{16 \omega \tau_i}{t} 
  + \frac{16 \sigma^2 h_i}{\varepsilon t}  
  + \frac{32 \sigma^2 \omega h_i \tau_i}{\varepsilon t^2} 
  + \frac{4 d \omega_\textnormal{s} \kappa_{\max} \kappa_i }{t^2}
  + \frac{8 d \omega_\textnormal{s} \omega \kappa_{\max} \kappa_i \tau_i}{t^3}}\right)^{-1}}_{\delta(t)}.
\end{align*}
Recall that $\delta(t)$ is a strictly decreasing function of $t$ and by the definition of the time budget
$t \geq s^*$. Since $\delta(s^*) = 1$ the last inequality yields
$\sum \limits_{i=1}^{n} \beta_i^2 w_i^{-1} \leq 1$.
\end{proof}

\THMHOMOTIME*

\begin{proof}
Since \ref{eq:inkheart_heter} generalizes \ref{eq:inkheart}, we can use Theorem~\ref{thm:main_time_complexity_async}. Recall the definition of $s^*$ in Theorem~\ref{thm:main_time_complexity_async}:
\begin{align*}
&\left(\sum \limits_{i=1}^{n} \frac{1}{\frac{16 \omega \tau_i}{s} 
  + \frac{16 \sigma^2 h_i}{\varepsilon s}  
  + \frac{32 \sigma^2 \omega h_i \tau_i}{\varepsilon s^2} 
  + \frac{4 d \omega_\textnormal{s} \kappa_{\max} \kappa_i }{s^2}
  + \frac{8 d \omega_\textnormal{s} \omega \kappa_{\max} \kappa_i \tau_i}{s^3}}\right)^{-1} = 1.
\end{align*}
We rewrite this equation to the case when all the workers are identical. 
\begin{align*}
\left(\frac{1}{\frac{16 \omega \tau}{n s} 
  + \frac{16 \sigma^2 h}{\varepsilon n s}  
  +\frac{32 \sigma^2 \omega h \tau}{\varepsilon n s^2}
  + \frac{4 d \omega_\textnormal{s} \kappa^2 }{n s^2}
  + \frac{8 d \omega_\textnormal{s} \omega \kappa^2 \tau}{n s^3}}\right)^{-1} &= 1  \iff\\
\frac{16 \omega \tau}{n s} 
  + \frac{16 \sigma^2 h}{\varepsilon n s}  
  + \frac{32 \sigma^2 \omega h \tau}{\varepsilon n s^2}
  + \frac{4 d \omega_\textnormal{s} \kappa^2}{n s^2}
  + \frac{8 d \omega_\textnormal{s} \omega \kappa^2 \tau}{n s^3} &= 1.
\end{align*}

Then, we substitute $r = \frac{1}{s}$. Lemma~\ref{lemma:approximate_solution} yields

\begin{align*}
s^* &\leq \mathcal O\left(\max\left\{\frac{16 \omega \tau}{n}, \frac{16 \sigma^2 h}{n \varepsilon},
  \sqrt{\frac{32 \omega \sigma^2 h \tau}{n \varepsilon}},
  \sqrt{\frac{4 d \omega_\textnormal{s} \kappa^2}{n}},
  \left(\frac{8 d \omega_\textnormal{s} \omega \kappa^2 \tau}{n}\right)^\frac{1}{3}\right\}\right).
\end{align*}
Since we use Rand$1$ compressors, $\omega = \omega_\textnormal{s} = d - 1$ and 
\begin{align*}
s^* &\leq \mathcal O\left(\max\left\{\frac{16 \omega \tau}{n}, \frac{16 \sigma^2 h}{n \varepsilon}, \frac{2 d \kappa}{\sqrt{n}},
  \sqrt{\frac{32 d \sigma^2 h \tau}{n \varepsilon}},
  \left(\frac{8 d^3 \tau \kappa^2}{n}\right)^\frac{1}{3}\right\}\right).
\end{align*}

Applying Theorem~\ref{thm:main_time_complexity_async} completes the proof.
\end{proof}

\begin{remark}
Let us compare our results with those obtained in \citep{tyurin2024shadowheart}. 
Consider Corollary~A.3 which bounds the time complexity
  $T_{\textnormal{time}} \leq \frac{768 L \Delta}{\alpha \varepsilon} \times \left(\kappa 
  + 2t^* \right)$ and Example 6.5
which bounds the optimal time budget  
  $t^* \leq \max \left\{h, \tau, \frac{\tau \omega}{n}, 
  \frac{h \sigma^2}{n \varepsilon}, 
  \sqrt{\frac{\tau h \sigma^2 \omega}{n \varepsilon}}\right\}$.
Note that the previous work focuses on biased server's compressors. If we 
convert an unbiased Rand$1$ compressor $\cC(x)$ with $\omega = d - 1$ into a biased 
compressor by scaling it as
$\frac{1}{d} \cC(x)$, we obtain a biased compressor with $\alpha = 1 - \frac{d - 1}{d} = \frac{1}{d}$.
Thus, for the method described in \citep{tyurin2024shadowheart} the following upper 
bound holds
\begin{align*}
  T_\textnormal{time} &\leq \frac{768 L  d \Delta}{\varepsilon} \left(\kappa + 2
  \cdot \max \left\{h, \tau, \frac{\tau d}{n}, 
  \frac{h \sigma^2}{n \varepsilon}, 
  \sqrt{\frac{\tau h \sigma^2 d}{n \varepsilon}}\right\}\right)
\end{align*}
Note that the time complexity of their method contains terms involving $d$ that do not scale with $n$.
\end{remark}

\subsection{Optimal subset of the workers}
\label{sec:optimal}
Previously, we utilized all available workers in \ref{eq:inkheart_heter}. Below, we demonstrate that alternative strategies may achieve better performance.
Recall the time bound we obtained in Theorem~\ref{thm:main_time_complexity_async}. 
As we observe, it depends on $\kappa_{\max}$ and $\max_{i \in [n]} \left\{h_i, \tau_i, \kappa_i\right\}$.
Therefore, when we add a new worker, the time complexity does not necessarily decrease.
This may happen even if the parameters of the $i$th worker are relatively small but its $\kappa_i$ parameter is greater than the 
$\left\{\kappa_j\right\}$ of the other workers. Thus, we should find a subset of workers that yields the optimal time complexity. 
Algorithm~\ref{alg:main_choose_subset} provides a solution. 
We denote $M_i \eqdef \max\{h_i, \tau_i, \kappa_i\}$.

\begingroup
\setlength{\textfloatsep}{3pt}
\setlength{\intextsep}{3pt}
\setlength{\floatsep}{0pt}
\begin{algorithm}[H]
  \caption{Optimal Worker Subset Selection}
  \label{alg:main_choose_subset}
  \begin{algorithmic}[1]
    \STATE \textbf{Input:} Parameters $\{(h_i, \tau_i, \kappa_i)\}_{i=1}^n, \nicefrac{\sigma^2}{\varepsilon}, \omega, \omega_\textnormal{s}$.
    \STATE Compute $M_i \gets \max\{h_i, \tau_i, \kappa_i\}$ for all $i \in [n]$.
    \STATE Sort indices by $\kappa$ in non-decreasing order; let the permutation be $\pi$, so $\kappa_{\pi(1)} \le \dots \le \kappa_{\pi(n)}$.
    \STATE $T_{\min} \gets +\infty$, \quad $S^* \gets \emptyset$
    \FOR{$k = 1$ \TO $n$}
      \STATE Let $\tilde{S} \gets \{\pi(1), \dots, \pi(k)\}$.
      \STATE Sort workers in $\tilde{S}$ by $M_i$ in non-decreasing order; let the ordering be $\sigma^{(k)}(1), \dots, \sigma^{(k)}(k)$.
      \FOR{$m = 1$ \TO $k$}
        \STATE $\hat{S} \gets \{\sigma^{(k)}(1), \dots, \sigma^{(k)}(m)\}$
        \STATE Evaluate $\tilde{T}(\hat{S})$ using Theorem~\ref{thm:main_time_complexity_async}.
        \IF{$\tilde{T}(\hat{S}) < T_{\min}$}
          \STATE $T_{\min} \gets \tilde{T}(\hat{S})$
          \STATE $S^* \gets \hat{S}$
        \ENDIF
      \ENDFOR
    \ENDFOR
    \STATE \textbf{return} $S^*$
  \end{algorithmic}
\end{algorithm}
\endgroup

\begin{definition}
Let $[n]$ be a set of workers' indices. Similarly to Theorem~\ref{thm:main_time_complexity_async}, we consider
\begin{align}\label{eq:equlibrium_time_subset}
t(S) \eqdef \max\left\{M_{\max}(S), s^*(S)\right\},
\end{align}
for an arbitrary $S \subseteq [n]$, where $M_{\max}(S) \eqdef \max_{i \in S} M_i$ and $s^*(S)$ is the solution to the equation
\begin{align}\label{eq:weighted_equlibrium_time_subset}
\delta_S(s) \eqdef \left(\underbrace{\sum \limits_{i \in S} \frac{1}{\frac{16 \omega \tau_i}{s} 
  + \frac{16 \sigma^2 h_i}{\varepsilon s}  
  + \frac{32 \sigma^2 \omega h_i \tau_i}{\varepsilon s^2} 
  + \frac{4 d \omega_\textnormal{s} \kappa_{\max}(S) \kappa_i }{s^2}
  + \frac{8 d \omega_\textnormal{s} \omega \kappa_{\max}(S) \kappa_i \tau_i}{s^3}}}_{\psi_S(s)}\right)^{-1} = 1.
\end{align}
Here we denote $\kappa_{\max}(S) \eqdef \displaystyle \max_{i \in S} \kappa_i$. Theorem~\ref{thm:main_time_complexity_async} yields the time complexity when \ref{eq:inkheart} is executed on the subset of workers $S$:
\begin{align*}
\Exp{T_{\textnormal{time}}}
&= \cO \left(\frac{\Delta}{\varepsilon} \max\left\{t(S) L_{\max}, d\kappa_{\max}(S) L_A \right\}\right).
\end{align*}
Denote $\tilde{T}(S) \eqdef \max\left\{t(S) L_{\max}, d\kappa_{\max}(S) L_A \right\}$. Then, we define the optimal subset of workers as
\begin{align*}
S^* \in \underset{S \subseteq [n]}{\operatorname{argmin}} \: \tilde{T}(S).
\end{align*}
\end{definition}

Note that $S^*$ may not be unique.
Theorem~\ref{thm:worker_selection} provides guarantees for Algorithm~\ref{alg:main_choose_subset}.

\THMOPTSET*

\begin{proof}
Assume the subset $S$ minimizes $\tilde{T}(\cdot)$, i.e., there is no $S'$ such that $\tilde{T}(S') < \tilde{T}(S)$.
Let $\pi$ be a permutation of $\{1,\dots,n\}$ such that 
$\kappa_{\pi(1)} \leq \kappa_{\pi(2)} \leq \dots \le \kappa_{\pi(n)}$.
Define $i^\pi_S := \max\{i : \pi(i) \in S\}$.
Consider the initial segment of sorted indices $\tilde{S} := \{\pi(1), \dots, \pi(i^\pi_S)\}$;
by construction, $S \subseteq \tilde{S}$.

Next, let $\sigma$ be a permutation of the workers in $\tilde{S}$ such that
$M_{\sigma(1)} \leq M_{\sigma(2)} \leq \dots \leq M_{\sigma(i^\pi_S)}$.
Define $i^\sigma_S := \max\{i : \sigma(i) \in S\}$ and consider the initial segment
$\hat{S} := \{\sigma(1), \dots, \sigma(i^\sigma_S)\}$.
Again, $S \subseteq \hat{S}$ by construction.

Since the algorithm computes $\tilde{T}(\hat{S})$ when $k = i^\pi_S$ and $m = i^\sigma_S$, 
the case $S = \hat{S}$ is trivial.
Suppose instead that $S \subset \hat{S}$. Then there exists an index $j \leq i^\sigma_S$ such that $\sigma(j) \notin S$.
By Lemma~\ref{lemma:fast_worker}, adding the worker $\sigma(j)$ to $S$ does not increase the time complexity, i.e., $\tilde{T}(S \cup \{\sigma(j)\}) \leq \tilde{T}(S)$.
By iteratively using Lemma~\ref{lemma:fast_worker} and adding such workers from $\hat{S} \setminus S$, 
we can transform $S$ into $\hat{S}$, which completes the proof.
\end{proof}

\begin{lemma}[Fast worker improves performance]\label{lemma:fast_worker}
Let $[n]$ be a set of workers' indices. Consider an arbitrary subset $S \subseteq [n]$ and $j \in [n]$ such that 
$j \notin S$, $M_j \leq M_{\max}(S)$, and $\kappa_j \leq \kappa_{\max}(S)$. 
Then, $\tilde{T}(S \cup \{j\}) \leq \tilde{T}(S)$.
\end{lemma}

\begin{proof}
Since $M_j \leq M_{\max}(S)$ and $\kappa_j \leq \kappa_{\max}(S)$, we have
\begin{align}\label{eq:skmscoiesllsmdmf}
  \tilde{T}(S \cup \{j\}) &= \max\left\{M_{\max}(S \cup \{j\}) L_{\max}, s^*(S \cup \{j\}) L_{\max}, d\kappa_{\max}(S \cup \{j\}) L_A \right\} \\
  &\leq \max\left\{M_{\max}(S) L_{\max}, s^*(S \cup \{j\}) L_{\max}, d\kappa_{\max}(S) L_A \right\} \nonumber.
\end{align}

It remains to show $s^*(S \cup \{j\}) \leq s^*(S)$.
Recall that $s^*(S)$ is the solution to the equation $\delta_{S}(s) = 1$. 
Comparing $\delta_{S}(s)$ and $\delta_{S \cup \{j\}}(s)$, we observe:
\begin{align*}
\delta_{S \cup \{j\}}(s) &= \left(\sum \limits_{i \in S \cup \{j\}} \frac{1}{\frac{16 \omega \tau_i}{s} 
  + \frac{16 \sigma^2 h_i}{\varepsilon s}  
  + \frac{32 \sigma^2 \omega h_i \tau_i}{\varepsilon s^2} 
  + \frac{4 d \omega_\textnormal{s} \kappa_{\max}(S \cup \{j\}) \kappa_i }{s^2}
  + \frac{8 d \omega_\textnormal{s} \omega \kappa_{\max}(S \cup \{j\}) \kappa_i \tau_i}{s^3}}\right)^{-1} \\
&= \left(\sum \limits_{i \in S \cup \{j\}} \frac{1}{\frac{16 \omega \tau_i}{s} 
  + \frac{16 \sigma^2 h_i}{\varepsilon s}  
  + \frac{32 \sigma^2 \omega h_i \tau_i}{\varepsilon s^2} 
  + \frac{4 d \omega_\textnormal{s} \kappa_{\max}(S) \kappa_i }{s^2}
  + \frac{8 d \omega_\textnormal{s} \omega \kappa_{\max}(S) \kappa_i \tau_i}{s^3}}\right)^{-1} \\
&= \left(\psi_{S}(s) + \frac{1}{\frac{16 \omega \tau_j}{s} 
  + \frac{16 \sigma^2 h_j}{\varepsilon s}  
  + \frac{32 \sigma^2 \omega h_j \tau_j}{\varepsilon s^2} 
  + \frac{4 d \omega_\textnormal{s} \kappa_{\max}(S) \kappa_j }{s^2}
  + \frac{8 d \omega_\textnormal{s} \omega \kappa_{\max}(S) \kappa_j \tau_j}{s^3}}\right)^{-1}
  \leq \delta_{S}(s).
\end{align*}
We obtain that $\delta_{S \cup \{j\}}(s) \leq \delta_{S}(s)$ for all $s$. Thus, $\delta_{S \cup \{j\}}(s^*(S)) \leq \delta_{S}(s^*(S)) = 1$.
Since $\delta_{S \cup \{j\}}(s)$ is a decreasing function and 
$\delta_{S \cup \{j\}}(s^*(S)) \leq \delta_{S \cup \{j\}}(s^*(S \cup \{j\})) = 1$,
we conclude that $s^*(S \cup \{j\}) \leq s^*(S)$.
Substituting the last inequality into \eqref{eq:skmscoiesllsmdmf} yields $\tilde{T}(S \cup \{j\}) \leq \tilde{T}(S)$.  
\end{proof}

\subsection{Bounding Variances: Auxiliary Lemmas}
In this section we prove some inequalities to bound the residual between the gradient estimator $g^k$ from \eqref{eq:inkheart_heter}
and the true gradient and inequalities that control the distance between the point on the server $x^k$ and local points
$x_i^k$ at which workers compute stochastic gradients.

\begin{lemma}\label{lemma:grad_estimator_residual}
Consider the gradient estimator $g^k$ defined in \ref{eq:inkheart_heter}. 
Suppose that the function $f$ satisfies 
Assumptions~\ref{ass:lipschitz_constant}, \ref{ass:lower_bound}, \ref{ass:stochastic_variance_bounded} and
\ref{ass:AB_assumption} and the compressors $\left\{\cC_{ij}\right\} \in \U(\omega)$.
Then, $g^k$ is (in general) a biased estimator with
$\ExpSub{k}{g^k} = \sum_{i=1}^{n} \beta_i \nabla f(x^k_i)$
and the following inequality holds:
\begin{align*}
&\ExpSub{k}{\sqnorm{g^k - \nabla f(x^k)}} \\
&\leq 4 \left(\sum\limits_{i=1}^n \beta_i^2 \frac{\omega}{m_i}\right)\sqnorm{\nabla f(x^k)} 
  + \sum \limits_{i=1}^{n} \left(2 L_A^2 \beta_i + 4 L^2 \beta_i^2 \frac{\omega}{m_i}\right) \sqnorm{x^k_i - x^k} \\
&\quad + 2 L_B^2 \sqnorm{\sum \limits_{i=1}^{n} \beta_i \left(x^k_i - x^k\right)} 
  + \sum\limits_{i=1}^n 2 \beta_i^2 \left(\frac{\omega \sigma^2}{m_i b_i} + \frac{\sigma^2}{b_i}
  \right).
\end{align*}
\end{lemma}

\begin{proof} First, we compute the expectation.
Since the compressors and the stochastic gradients are unbiased we obtain
\begin{align*}
\ExpSub{k}{g^k} 
&= \ExpSub{k}{\sum \limits_{i=1}^{n} \frac{\beta_i}{b_i m_i} \sum\limits_{j=1}^{m_i}  \cC_{ij}\left(\sum\limits_{r=1}^{b_i} \nabla f(x^k_i;\xi_{ir}^k)\right)} \\
&= \sum \limits_{i=1}^{n} \frac{\beta_i}{b_i m_i} \sum\limits_{j=1}^{m_i}  \ExpSub{k}{\ExpSub{k, \xi}{\cC_{ij}\left(\sum\limits_{r=1}^{b_i} \nabla f(x^k_i;\xi_{ir}^k)\right)}} \\
&= \sum \limits_{i=1}^{n} \frac{\beta_i}{b_i m_i} \sum\limits_{j=1}^{m_i}  \sum\limits_{r=1}^{b_i} \ExpSub{k}{\nabla f(x^k_i;\xi_{ir}^k)}
= \sum_{i=1}^{n} \beta_i \nabla f(x^k_i).
\end{align*}

Then, we consider the residual between the gradient estimator and the true gradient
\begin{align*}
&\frac{1}{2} \ExpSub{k}{\sqnorm{g^k - \nabla f(x^k)}} 
  = \frac{1}{2} \ExpSub{k}{\sqnorm{g^k
  - \sum_{i=1}^{n} \beta_i \nabla f(x^k_i) + \sum_{i=1}^{n} \beta_i \nabla f(x^k_i) - \nabla f(x^k)}}
\end{align*}
Then, we apply Young's inequality 
\begin{align*}
&\frac{1}{2} \ExpSub{k}{\sqnorm{g^k - \nabla f(x^k)}} 
\leq \underbrace{\ExpSub{k}{\sqnorm{g^k  - \sum_{i=1}^{n} \beta_i \nabla f(x^k_i)}}}_{I_1}
 + \underbrace{\sqnorm{\sum_{i=1}^{n} \beta_i \nabla f(x^k_i) - \nabla f(x^k)}}_{I_2}.
\end{align*}
Using Lemma~\ref{lemma:weighted_gk_variance}, we bound $I_1$:
\begin{align}\label{eq:ksmfdosoifjewnoo}
I_1 \leq \sum\limits_{i=1}^n \beta_i^2 \left(2 \frac{\omega}{m_i} \sqnorm{\nabla f(x^k)} 
  + 2 L^2 \frac{\omega}{m_i} \sqnorm{x^k_i - x^k} 
  + \frac{\omega \sigma^2}{m_i b_i} + \frac{\sigma^2}{b_i}\right).
\end{align} 
To estimate $I_2$, we use Assumption~\ref{ass:AB_weighted_assumption}
\begin{align}\label{eq:dskmosfiojwoide}
I_2  \leq L_A^2 \sum\limits_{i=1}^n \beta_i \norm{x^k_i - x^k}^2
  + L_B^2 \norm{\sum\limits_{i=1}^n \beta_i \left(x^k_i - x^k\right)}^2.
\end{align}

Combining \eqref{eq:ksmfdosoifjewnoo} and \eqref{eq:dskmosfiojwoide}, we derive
\begin{align*}
&\frac{1}{2} \ExpSub{k}{\sqnorm{g^k - \nabla f(x^k)}} \\
&\leq \sum\limits_{i=1}^n \beta_i^2 \left(2 \frac{\omega}{m_i} \sqnorm{\nabla f(x^k)} 
  + 2 L^2 \frac{\omega}{m_i} \sqnorm{x^k_i - x^k} 
  + \frac{\omega \sigma^2}{m_i b_i} + \frac{\sigma^2}{b_i}\right) \\
&\quad + L_A^2 \sum\limits_{i=1}^n \beta_i \norm{x^k_i - x^k}^2
  + L_B^2 \norm{\sum\limits_{i=1}^n \beta_i \left(x^k_i - x^k\right)}^2 \\
&= \sum\limits_{i=1}^n \beta_i^2 \left(2 \frac{\omega}{m_i} \sqnorm{\nabla f(x^k)} 
  + \frac{\omega \sigma^2}{m_i b_i} + \frac{\sigma^2}{b_i}\right) \\
&\quad + \sum\limits_{i=1}^n (L_A^2  \beta_i + 2 L^2 \beta_i^2 \frac{\omega}{m_i}) \sqnorm{x^k_i - x^k} 
  + L_B^2 \norm{\sum\limits_{i=1}^n \beta_i \left(x^k_i - x^k\right)}^2.
\end{align*}
\end{proof}

The following lemma bounds the residual between the true gradients at the workers' local points and the gradient estimator.
\begin{lemma}\label{lemma:weighted_gk_variance}
Consider the gradient estimator $g^k$ defined in \ref{eq:inkheart_heter}. 
Suppose that the function $f$ satisfies 
Assumptions~\ref{ass:lipschitz_constant}, \ref{ass:lower_bound}, \ref{ass:stochastic_variance_bounded} and
the compressors $\left\{\cC_{ij}\right\} \in \U(\omega)$.
Then, we have
\begin{align*}
&\ExpSub{k}{\sqnorm{g^k - \sum_{i=1}^{n} \beta_i \nabla f(x^k_i)}}
\leq \sum\limits_{i=1}^n \beta_i^2 \left(2 \frac{\omega}{m_i} \sqnorm{\nabla f(x^k)} 
  + 2 L^2 \frac{\omega}{m_i} \sqnorm{x^k_i - x^k} 
  + \frac{\omega \sigma^2}{m_i b_i} + \frac{\sigma^2}{b_i}\right).
\end{align*}
\end{lemma}

\begin{proof}
Lemma~\ref{lemma:grad_estimator_residual} yields $\ExpSub{k}{g^k} = \sum_{i=1}^{n} \beta_i \nabla f(x^k_i)$.
By the definition of the gradient estimator $g^k$ in \eqref{eq:inkheart_heter}, we obtain
\begin{align*}
\ExpSub{k}{\sqnorm{g^k - \sum_{i=1}^{n} 
    \beta_i \nabla f(x^k_i)}}
&=\ExpSub{k}{\sqnorm{\sum \limits_{i=1}^{n} \frac{\beta_{i}}{b_i m_i} \sum\limits_{j=1}^{m_i}  \cC_{ij}\left(\sum\limits_{r=1}^{b_i} \nabla f(x^k_i;\xi_{ir}^k)\right)
- \sum_{i=1}^{n} \beta_i \nabla f(x^k_i)}}\\ 
&=\sum \limits_{i=1}^{n} \ExpSub{k}{\sqnorm{\frac{\beta_{i}}{b_i m_i} \sum\limits_{j=1}^{m_i}  
    \cC_{ij}\left(\sum\limits_{r=1}^{b_i} \nabla f(x^k_i;\xi_{ir}^k)\right) - \beta_i \nabla f(x^k_i)}}
\end{align*}
In the last equality we use the independence across workers.
Next, we bound each term separately using the tower property and variance decomposition:
\begin{align*}
&\ExpSub{k}{\sqnorm{\frac{\beta_{i}}{b_i m_i} \sum\limits_{j=1}^{m_i}  
    \cC_{ij}\left(\sum\limits_{r=1}^{b_i} \nabla f(x^k_i;\xi_{ir}^k)\right) - \beta_i \nabla f(x^k_i)}} \\
&= \underbrace{\ExpSub{k}{\sqnorm{\frac{\beta_{i}}{b_i m_i}\sum\limits_{j=1}^{m_i}  
  \cC_{ij}\left(\sum\limits_{r=1}^{b_i} \nabla f(x^k_i;\xi_{ir}^k)\right)
  - \frac{\beta_{i}}{b_i m_i} \sum\limits_{j=1}^{m_i} 
  \sum\limits_{r=1}^{b_i} \nabla f(x^k_i;\xi_{ir}^k)}}}_{I_1} \\
&\quad + \underbrace{\ExpSub{k}{\sqnorm{\frac{\beta_{i}}{b_i m_i} \sum\limits_{j=1}^{m_i} 
  \sum\limits_{r=1}^{b_i} \nabla f(x^k_i;\xi_{ir}^k)
  - \beta_i \nabla f(x^k_i)}}}_{I_2}
\end{align*}

We analyze each term separately. We use the bounded variance property of compressors and variance decomposition.
\begin{align*}
I_1 &= \sum\limits_{j=1}^{m_i} \left(\frac{\beta_{i}}{b_i m_i}\right)^2 \ExpSub{k}{\sqnorm{
    \cC_{ij}\left(\sum\limits_{r=1}^{b_i} \nabla f(x^k_i;\xi_{ir}^k)\right)
    - \sum\limits_{r=1}^{b_i} \nabla f(x^k_i;\xi_{ir}^k)}} \\
    &\leq \sum\limits_{j=1}^{m_i} \left(\frac{\beta_{i}}{b_i m_i}\right)^2 \omega \ExpSub{k}{\sqnorm{
    \sum\limits_{r=1}^{b_i} \nabla f(x^k_i;\xi_{ir}^k)}} 
    \leq \frac{\beta_{i}^2 \omega}{b_i^2 m_i} \left(b_i^2 \sqnorm{\nabla f(x^k_i)} + b_i \sigma^2\right).
\end{align*}

For each $r$, the stochastic gradients are independent. Therefore, 
\begin{align*}
I_2  &=  \sum\limits_{r=1}^{b_i}  \ExpSub{k}{\sqnorm{\frac{\beta_{i}}{b_i m_i} \sum\limits_{j=1}^{m_i} 
  (\nabla f(x^k_i;\xi_{ir}^k) - \nabla f(x^k_i))}} \\
&= \sum\limits_{r=1}^{b_i}  \ExpSub{k}{\sqnorm{\frac{\beta_{i}}{b_i} (\nabla f(x^k_i;\xi_{ir}^k) - \nabla f(x^k_i))}}
\leq \frac{\beta_i^2 \sigma^2}{b_i}.
\end{align*}

Thus, we obtain
\begin{align*}
&\ExpSub{k}{\sqnorm{g^k - \sum_{i=1}^{n} \beta_i \nabla f(x^k_i)}}
\leq \sum\limits_{i=1}^n \beta_i^2 \left(\frac{\omega}{m_i} \sqnorm{\nabla f(x^k_i)} 
  + \frac{\omega \sigma^2}{m_i b_i} + \frac{\sigma^2}{b_i}\right) \\
&\leq \sum\limits_{i=1}^n \beta_i^2 \left(2 \frac{\omega}{m_i} \sqnorm{\nabla f(x^k)} 
  + 2 \frac{\omega}{m_i} \sqnorm{\nabla f(x^k_i) - \nabla f(x^k)} 
  + \frac{\omega \sigma^2}{m_i b_i} + \frac{\sigma^2}{b_i}\right) \\
&\leq \sum\limits_{i=1}^n \beta_i^2 \left(2 \frac{\omega}{m_i} \sqnorm{\nabla f(x^k)} 
  + 2 L^2 \frac{\omega}{m_i} \sqnorm{x^k_i - x^k} 
  + \frac{\omega \sigma^2}{m_i b_i} + \frac{\sigma^2}{b_i}\right).
\end{align*}
\end{proof}

In the previous lemmas, we bounded the residual between the gradient estimate and the true gradient. 
In the following, we analyze the distance between the point at which the server performs the update 
and the local points at which the workers compute the stochastic gradients. 

\begin{lemma}\label{lemma:wpointrec} Consider the point updates in \ref{eq:inkheart_heter}. 
Suppose that the compressors $\cC_{\textnormal{s},ij} \in \U(\omega_{\textnormal{s}})$, then 
the following inequalities hold:
\begin{align}\label{eq:idskndsfihg}
&\sum\limits_{i=1}^n \beta_i \ExpSub{k}{\norm{x^{k+1}_i - x^{k+1}}^2}
&\leq (1 - p) \left[\sum\limits_{i=1}^n \frac{\omega_{\textnormal{s}} \beta_i}{\ell_i} 
\ExpSub{k}{\norm{x^{k+1} - x^k}^2}
+ \sum\limits_{i=1}^n\beta_i \norm{x_i^k - x^k}^2\right]
\end{align}

and 

\begin{align}\label{eq:idkxkdifiwwg}
&\ExpSub{k}{\norm{\sum\limits_{i=1}^n \beta_i \left(x^{k+1}_i - x^{k+1}\right)}^2}
&\leq (1 - p) \left[\sum\limits_{i=1}^n \frac{\omega_{\textnormal{s}} \beta_i^2}{\ell_i} 
\ExpSub{k}{\norm{x^{k+1} - x^k}^2}
+ \norm{\sum_{i=1}^{n} \beta_i \left(x^{k}_i - x^{k}\right)}^2\right].
\end{align}
\end{lemma}

\begin{proof}
We start by analyzing the evolution of the quantity $\sqnorm{x_i^k - x^k}$.
\begin{align*}
&\ExpSub{k}{\norm{x^{k+1}_i - x^{k+1}}^2} \\
&= \ExpSub{k}{\Ind{c^k = 1} \sqnorm{x^{k+1} - x^{k+1}}
+ \Ind{c^k = 0} \norm{x_i^k + \frac{1}{\ell_i} \sum_{j=1}^{\ell_i} \cC_{\textnormal{s},ij}(x^{k+1} - x^k) - x^{k+1}}^2} 
\end{align*}
Note that the first term vanishes. In the second term, $c^k$ is independent of the other random variables 
and of the sigma-algebra associated with the conditional expectation $\ExpSub{k}{\cdot}$. 
Thus, we can rewrite the expression as:
\begin{align*}
\ExpSub{k}{\norm{x^{k+1}_i - x^{k+1}}^2} 
  &= \Exp{\Ind{c^k = 0}} \cdot \ExpSub{k}{\norm{x_i^k + \frac{1}{\ell_i} \sum_{j=1}^{\ell_i} \cC_{\textnormal{s},ij}(x^{k+1} - x^k) - x^{k+1}}^2} \\
&= (1 - p) \ExpSub{k}{\ExpCond{\sqnorm{\frac{1}{\ell_i} \sum_{j=1}^{\ell_i} \cC_{\textnormal{s},ij}(x^{k+1} - x^k) 
  - (x^{k+1} - x_i^k)}}{x^{k + 1}}}
\end{align*}
Using Lemma~\ref{lemma:variance_decomposition} for independent compressors we get
\begin{align*}
&\ExpSub{k}{\norm{x^{k+1}_i - x^{k+1}}^2} \\
& = (1 - p) \ExpSub{k}{\ExpCond{\norm{\frac{1}{\ell_i} 
  \sum_{j=1}^{\ell_i} \cC_{\textnormal{s},ij}(x^{k+1} - x^k) - (x^{k+1} - x^k)}^2}{x^{k + 1}}}\\
&\phantom{=}+(1 - p) \ExpSub{k}{
\norm{(x^{k + 1} - x_i^k) - (x^{k+1} - x^k)}^2}
\end{align*}
Then, we apply Lemma~\ref{lemma:averaging_compressors} to 
$\bar{\cC}(x) = \frac{1}{\ell_i} \sum \limits_{j = 1}^{\ell_i} \cC_{\textnormal{s}, ij}(x)$
\begin{align*}
&\ExpSub{k}{\norm{x^{k+1}_i - x^{k+1}}^2}
\leq (1 - p) \frac{\omega}{\ell_i} \ExpSub{k}{
 \norm{x^{k+1} - x^k}^2} +(1 - p) \norm{x_i^k - x^k}^2.
\end{align*}

Proof for the second inequality is almost the same.
\begin{align*}
&\ExpSub{k}{\norm{\sum\limits_{i=1}^n \beta_i \left(x^{k+1}_i - x^{k+1}\right)}^2} 
  = (1 - p) \ExpSub{k}{\norm{\sum\limits_{i=1}^n \beta_i 
    \left(\frac{1}{\ell_i} \sum_{j=1}^{\ell_i} \cC_{\textnormal{s},ij}(x^{k+1} - x^k) - (x^{k+1} - x_i^k)\right)}^2}
\end{align*}

Using Lemma~\ref{lemma:variance_decomposition} for independent compressors, we get
\begin{align*}
&\ExpSub{k}{\norm{\sum\limits_{i=1}^n \beta_i \left(x^{k+1}_i - x^{k+1}\right)}^2} \\
&= (1 - p) \ExpSub{k}{\ExpCond{\norm{\sum\limits_{i=1}^n \beta_i 
    \left(\frac{1}{\ell_i} \sum_{j=1}^{\ell_i} \cC_{\textnormal{s},ij}(x^{k+1} - x^k)
    - (x^{k+1} - x^k)\right)}^2}{x^{k + 1}}} \\
&\quad + (1 - p) \ExpSub{k}{\sqnorm{\sum\limits_{i=1}^n \beta_i (x^{k+1} - x_i^k) 
  - \sum\limits_{i=1}^n \beta_i (x^{k+1} - x^k)}}
\end{align*}
Then, we apply Lemma~\ref{lemma:averaging_compressors} to 
$\bar{\cC} = \sum\limits_{i=1}^n  \sum \limits_{j=1}^{\ell_i} \frac{\beta_i}{\ell_i} \cC_{\textnormal{s},ij}$
and get
\begin{align*}
&\ExpSub{k}{\norm{\sum\limits_{i=1}^n \beta_i \left(x^{k+1}_i - x^{k+1}\right)}^2} \\
&\leq (1 - p) \sum\limits_{i=1}^n 
     \sum_{j=1}^{\ell_i} \frac{\beta_i^2}{\ell_i^2} \ExpSub{k}{\norm{x^{k+1} - x^k}^2}
  + (1 - p) \ExpSub{k}{\sqnorm{\sum\limits_{i=1}^n \beta_i (x^{k+1} - x^k)}}.
\end{align*}
\end{proof}

\section{Proofs for \ref{eq:mthree}}

\subsection{Known lemmas}
We take several auxiliary lemmas from \citep{gruntkowska2024improving}. 
\begin{lemma}[\citet{gruntkowska2024improving}]\label{lemma:w_k_1_x_k}
Let $\left\{\cC_{\textnormal{s}, i}\right\}_{i = 1}^n \in \U(\omega_{\textnormal{s}}) $. Then, $w_i^{k + 1}$ in \ref{eq:mthree} satisfies
\begin{align*}
\ExpSub{k}{\sqnorm{w_i^{k + 1} - x_i^{k}}} 
& \leq \ExpSub{k}{\sqnorm{(x^{k + 1} - x^{k}) + p_{\textnormal{s}}(w_i^{k} - x^{k}) + (w_i^{k} - x_i^{k})}}\\
&\quad + \ExpSub{k}{p_{\textnormal{s}} \sqnorm{w_i^{k} - x^{k}}}
  + \omega_\textnormal{s}\ExpSub{k}{\sqnorm{x^{k + 1} - x^{k}}}
\end{align*}
for all $i \in [n]$ and 
\begin{align*}
\ExpSub{k}{\sqnorm{\frac{1}{n} \sum\limits_{i=1}^n(w_i^{k + 1} - x_i^{k})}} 
& \leq \ExpSub{k}{\sqnorm{(x^{k + 1} - x^{k}) + p_{\textnormal{s}} \left(\frac{1}{n} \sum\limits_{i=1}^nw_i^{k} - x^{k}\right)
  + \frac{1}{n}\sum\limits_{i=1}^n (w_i^{k} - x_i^{k})}}\\
&\quad + p_{\textnormal{s}} \ExpSub{k}{\sqnorm{\frac{1}{n}\sum\limits_{i=1}^n w_i^{k} - x^{k}}}
  + \frac{\omega_{\textnormal{s}}}{n} \ExpSub{k}{\sqnorm{x^{k + 1} - x^{k}}}.
\end{align*}
\end{lemma}

\begin{lemma}[\citet{gruntkowska2024improving}]\label{lemma:worker_point_momentum_residual}
Let $\left\{\cC_{\textnormal{s}, i}\right\}_{i = 1}^n \in \U(\omega_{\textnormal{s}}) $. Then, $x_i^{k + 1}$ in \ref{eq:mthree} satisfies
\begin{align*}
\ExpSub{k}{\sqnorm{x_i^{k + 1} - w_i^{k + 1}}} 
&\leq \left(1 - \frac{\mu}{2}\right) \ExpSub{k}{\sqnorm{w_i^{k} - x_i^{k}}} \\ 
&\quad + 4 \left(\frac{1}{\mu} + \omega_{\textnormal{s}}\right) \ExpSub{k}{\sqnorm{x^{k + 1} - x^{k}}} \\ 
&\quad + 4 p_{\textnormal{s}}\left(1 + \frac{p_{\textnormal{s}}}{\mu}\right)\ExpSub{k}{\sqnorm{w_i^{k} - x^k}}
\end{align*}
and 
\begin{align*}
\ExpSub{k}{\sqnorm{\frac{1}{n} \sum_{i=1}^{n} x_i^{k + 1} - w_i^{k + 1}}} 
&\leq \left(1 - \frac{\mu}{2}\right) \ExpSub{k}{\sqnorm{\frac{1}{n} \sum_{i=1}^{n} w_i^{k} - x_i^{k}}} \\ 
&\quad + 4 \left(\frac{1}{\mu} + \frac{\omega_{\textnormal{s}}}{n}\right) \ExpSub{k}{\sqnorm{x^{k + 1} - x^{k}}} \\ 
&\quad + 4 p_{\textnormal{s}}\left(1 + \frac{p_{\textnormal{s}}}{\mu}\right)
  \ExpSub{k}{\sqnorm{\frac{1}{n} \sum_{i=1}^{n} w_i^k - x^k}}
\end{align*}
\end{lemma}

\begin{lemma}[\citet{gruntkowska2024improving}]\label{lemma:server_point_momentum_residual}
Let $\left\{\cC_{\textnormal{s}, i}\right\}_{i = 1}^n \in \U(\omega_{\textnormal{s}}) $. Then, $w_i^{k + 1}$ in \ref{eq:mthree} satisfies
\begin{align}
  \frac{1}{n} \sum_{i=1}^{n}\ExpSub{k}{\norm{w^{k+1}_i - x^{k+1}}^2} 
  &\leq (1 - p_{\textnormal{s}}) \frac{1}{n} \sum_{i=1}^{n} \ExpSub{k}{\norm{w^{k}_i - x^{k}}^2}
    + (1 - p_{\textnormal{s}}) \omega_{\textnormal{s}} \ExpSub{k}{\norm{x^{k+1} - x^k}^2}.
\end{align}
and
\begin{align}
  \ExpSub{k}{\norm{\frac{1}{n} \sum_{i=1}^{n} w^{k+1}_i - x^{k+1}}^2} 
  &\leq (1 - p_{\textnormal{s}}) \ExpSub{k}{\norm{\frac{1}{n} \sum_{i=1}^{n} w^{k}_i - x^{k}}^2}
    + (1 - p_{\textnormal{s}}) \frac{\omega_{\textnormal{s}}}{n} \ExpSub{k}{\norm{x^{k+1} - x^k}^2}.
\end{align}
\end{lemma}

\subsection{Auxiliary lemmas}
\begin{lemma}\label{lemma:workers_point_iteration}
Consider \ref{eq:mthree}. Let $\left\{\cC_{\textnormal{s}, i}\right\}_{i = 1}^n \in \U(\omega_{\textnormal{s}}) $. Then,  
\begin{align*}
\ExpSub{k}{\sqnorm{x_i^{k} - x_i^{k + 1}}} 
& \leq 3 \mu^2 \ExpSub{k}{\sqnorm{w_i^{k} - x_i^{k}}} 
  + 4 \mu^2 p_{\textnormal{s}} \ExpSub{k}{\sqnorm{w_i^{k} - x^{k}}}\\
&\quad + \mu^2 (\omega_\textnormal{s} + 3) \ExpSub{k}{\sqnorm{x^{k + 1} - x^{k}}}
\end{align*}
and 
\begin{align*}
\ExpSub{k}{\norm{\frac{1}{n} \sum\limits_{i=1}^n x_i^{k + 1} - x_i^k}^2}
&\leq 3 \mu^2 \ExpSub{k}{\norm{\frac{1}{n} \sum\limits_{i=1}^n \left(w_i^{k} - x_i^k\right)}^2}
  + 4 \mu^2 p_{\textnormal{s}} \ExpSub{k}{\norm{\frac{1}{n} \sum\limits_{i=1}^n w_i^{k} - x^k}^2} \\
&\quad + \mu^2 \left(\frac{\omega_{\textnormal{s}}}{n} + 3\right)\ExpSub{k}{\norm{x^{k + 1} - x^k}^2}
\end{align*}
\end{lemma}
\begin{proof}
By the definition of $x^k$ in \ref{eq:mthree}, we obtain
\begin{align*}
\ExpSub{k}{\norm{\frac{1}{n} \sum\limits_{i=1}^n x_i^{k + 1} - x_i^k}^2}
&= \ExpSub{k}{\norm{\frac{1}{n} \sum\limits_{i=1}^n \left((1 - \mu) x_i^k + \mu w_i^{k + 1} - x_i^k\right)}^2} \\
&= \mu^2 \ExpSub{k}{\norm{\frac{1}{n} \sum\limits_{i=1}^n \left(w_i^{k + 1} - x_i^k\right)}^2}
\end{align*}

Then, we use Lemma~\ref{lemma:w_k_1_x_k} and Jensen's inequality to obtain
\begin{align*}
&\ExpSub{k}{\norm{x_i^{k + 1} - x_i^k}^2} \\
&\leq \mu^2  \ExpSub{k}{3 \sqnorm{w_i^{k} - x_i^{k}}
  + 3 \sqnorm{p_{\textnormal{s}}(w_i^{k} - x^{k})}
  + 3 \sqnorm{x^{k + 1} - x^{k}}} \\
&\quad + p_{\textnormal{s}} \ExpSub{k}{\sqnorm{w_i^{k} - x^{k}}}
  + \omega_\textnormal{s} \ExpSub{k}{\sqnorm{x^{k + 1} - x^{k}}}. 
\end{align*}

Using the inequality $p_{\textnormal{s}}^2 \leq p_{\textnormal{s}}$ and simplifying we obtain the final bound.
The proof of the second inequality is analogous.
\end{proof}

\begin{lemma}\label{lemma:gradient_momentum_iteration}
Consider \ref{eq:mthree}. Let $\left\{\cC_i\right\}_{i = 1}^n \in \U(\omega)$ and 
$\left\{\cC_{\textnormal{s}, i}\right\}_{i = 1}^n \in \U(\omega_{\textnormal{s}}) $.
Let Assumptions~\ref{ass:hetero_function}, \ref{ass:stochastic_variance_bounded_heter}, and \ref{ass:AB_assumption_heter}
be satisfied.
Then,  
\begin{align*}
\frac{1}{n} \sum\limits_{i=1}^n \ExpSub{k}{\sqnorm{v_i^{k + 1} - v_i^k}} 
&\leq \frac{2 \nu^2}{n} \sum_{i=1}^n \ExpSub{k}{\sqnorm{v_i^k - \nabla f_i (x_i^{k})}} \\
&\quad + 6 \nu^2 \mu^2 L_{\max}^2 \frac{1}{n} \sum\limits_{i=1}^n \ExpSub{k}{\sqnorm{w_i^{k} - x_i^{k}}} \\
&\quad + 8 \nu^2 \mu^2 p_{\textnormal{s}} L_{\max}^2 \frac{1}{n} \sum\limits_{i=1}^n \ExpSub{k}{\sqnorm{w_i^{k} - x^{k}}} \\
&\quad + 2 \nu^2 \mu^2  (\omega_\textnormal{s} + 3) \hat{L}^2 \ExpSub{k}{\sqnorm{x^{k + 1} - x^{k}}} + \nu^2 \frac{\sigma^2}{b}
\end{align*}
\end{lemma}
\begin{proof}
By the definition of $v_i^{k + 1}$ defined in \ref{eq:mthree}, we get
\begin{align*}
&\ExpSub{k}{\sqnorm{(1 - \nu)v_i^k + \nu \nabla f_i (x_i^{k + 1}; \xi_i^{k + 1}) - v_i^k}} \\
&= \nu^2 \ExpSub{k}{\sqnorm{\nabla f_i (x_i^{k + 1}; \xi_i^{k + 1}) - v_i^k}} \\
&\overset{(i)}{=} \nu^2 \ExpSub{k}{\sqnorm{\nabla f_i (x_i^{k + 1}) - \nabla f_i (x_i^{k + 1}; \xi_i^{k + 1})}}
  + \nu^2 \ExpSub{k}{\sqnorm{v_i^k - \nabla f_i (x_i^{k + 1})}} \\
&\overset{(ii)}{\leq} \nu^2 \frac{\sigma^2}{b} + \nu^2 \ExpSub{k}{\sqnorm{v_i^k - \nabla f_i (x_i^{k + 1})}},
\end{align*}
where $(i)$ follows from the unbiasedness of stochastic gradients 
$\ExpSub{k + 1}{\nabla f_i (x_i^{k + 1}; \xi_i^{k + 1})} = \nabla f_i (x_i^{k + 1})$,
the tower property and the variance decomposition (Lemma~\ref{lemma:variance_decomposition}), 
while $(ii)$ relies on Assumption~\ref{ass:stochastic_variance_bounded_heter}.
Consider the last term separately. Due to Young's inequality we obtain

\begin{align*} 
&\ExpSub{k}{\sqnorm{v_i^k - \nabla f_i (x_i^{k + 1})}} 
\leq 2 \ExpSub{k}{\sqnorm{v_i^k - \nabla f_i (x_i^{k})}} 
  + 2 L_i^2 \ExpSub{k}{\sqnorm{x_i^{k + 1} - x_i^{k}}}
  \end{align*}
Then, we apply Lemma~\ref{lemma:workers_point_iteration} to the last term and get 
\begin{align*}
\ExpSub{k}{\sqnorm{v_i^{k + 1} - v_i^k}} 
&\leq 2 \nu^2 \ExpSub{k}{\sqnorm{v_i^k - \nabla f_i (x_i^{k})}} 
  + 6 \nu^2 L_i^2 \mu^2  \ExpSub{k}{\sqnorm{w_i^{k} - x_i^{k}}} \\
&\quad + 8 \nu^2 L_i^2 \mu^2 p_{\textnormal{s}}  \ExpSub{k}{\sqnorm{w_i^{k} - x^{k}}} 
  + 2 \nu^2 L_i^2 \mu^2  (\omega_\textnormal{s} + 3) \ExpSub{k}{\sqnorm{x^{k + 1} - x^{k}}} + \nu^2 \frac{\sigma^2}{b}.
\end{align*}
Averaging over $i \in [n]$ and using 
\begin{align*}
\frac{1}{n} \sum\limits_{i=1}^n L_i^2 \ExpSub{k}{\sqnorm{w_i^{k} - x_i^{k}}} 
&\leq L_{\max}^2 \frac{1}{n} \sum\limits_{i=1}^n \ExpSub{k}{\sqnorm{w_i^{k} - x_i^{k}}}, \\
\frac{1}{n} \sum\limits_{i=1}^n L_i^2 \ExpSub{k}{\sqnorm{w_i^{k} - x^{k}}} 
&\leq L_{\max}^2 \frac{1}{n} \sum\limits_{i=1}^n \ExpSub{k}{\sqnorm{w_i^{k} - x^{k}}}, \\
\textnormal{and} \quad \frac{1}{n} \sum\limits_{i=1}^n L_i^2 \ExpSub{k}{\sqnorm{x^{k + 1} - x^{k}}} 
&= \hat{L}^2 \ExpSub{k}{\sqnorm{x^{k + 1} - x^{k}}}
\end{align*}
we complete the proof.
\end{proof}

\subsection{Bounding variances}
First, we bound the deviation between the true gradient and the gradient estimator $g^k$ from \ref{eq:mthree} in terms of simpler quantities.

\begin{lemma}\label{lemma:heterogeneus_estimator_residual}
Consider \ref{eq:mthree}. Let Assumption~\ref{ass:AB_assumption_heter} be satisfied. Then, 
\begin{align*}
\ExpSub{k}{\sqnorm{g^k - \nabla f(x^k)}} 
&\leq 3 \ExpSub{k}{\sqnorm{g^k - v^k}}
  + 3\ExpSub{k}{\sqnorm{v^k - \frac{1}{n} \sum \limits_{i = 1}^n \nabla f_i(x_i^k)}} \\
&\quad + 6 L_A^2\left(\frac{1}{n} \sum\limits_{i=1}^n \ExpSub{k}{\norm{w_i^k - x_i^k}^2}
  + \frac{1}{n} \sum\limits_{i=1}^n \ExpSub{k}{\norm{w_i^k - x^k}^2}\right) \\
&\quad + 6 L_B^2 \left(\ExpSub{k}{\norm{\frac{1}{n} \sum\limits_{i=1}^n \left(w_i^k - x_i^k\right)}^2}
      + \ExpSub{k}{\norm{\frac{1}{n} \sum\limits_{i=1}^n w_i^k - x^k}^2} \right).
\end{align*}
\end{lemma}
\begin{proof}

Using simple algebra and Young's inequality we get
\begin{align*}
&\ExpSub{k}{\sqnorm{g^k - \nabla f(x^k)}} \\
&= \ExpSub{k}{\sqnorm{g^{k} - v^{k} 
  + v^{k} 
  - \frac{1}{n} \sum \limits_{i = 1}^n \nabla f_i(x_i^k)
  + \frac{1}{n} \sum \limits_{i = 1}^n \nabla f_i(x_i^k)
  - \nabla f(x^k)}} \\
&\leq 3 \ExpSub{k}{\sqnorm{g^k - v^k}}
  + 3\ExpSub{k}{\sqnorm{v^k - \frac{1}{n} \sum \limits_{i = 1}^n \nabla f_i(x_i^k)}} \\
  &\quad + 3\underbrace{\ExpSub{k}{\sqnorm{\nabla f(x^k) - \frac{1}{n} \sum \limits_{i = 1}^n \nabla f_i(x_i^k)}}}_{I}.
\end{align*}

We consider $I$ separately. 
\begin{align*}
I &= \norm{\frac{1}{n} \sum\limits_{i=1}^n (\nabla f_i(x_i^k) - \nabla f_i(x^k))}^2 
  \leq L_A^2\left(\frac{1}{n} \sum\limits_{i=1}^n \norm{x_i^k - x^k}^2\right) 
    + L_B^2 \norm{\frac{1}{n} \sum\limits_{i=1}^n x_i^k - x^k}^2 \\
&\leq 2 L_A^2\left(\frac{1}{n} \sum\limits_{i=1}^n \norm{w_i^k - x_i^k}^2
  + \frac{1}{n} \sum\limits_{i=1}^n \norm{w_i^k - x^k}^2\right) \\
&\quad + 2 L_B^2 \left(\norm{\frac{1}{n} \sum\limits_{i=1}^n \left(w_i^k - x_i^k\right)}^2 
      + \norm{\frac{1}{n} \sum\limits_{i=1}^n w_i^k - x^k}^2 \right).
\end{align*}
Here we use Assumption~\ref{ass:AB_assumption_heter} and Young's inequality twice.
\end{proof}

The following lemma bounds the error  induced by the worker compressors on average across workers.
\begin{lemma}\label{lemma:estimator_momentum_residual}
Consider \ref{eq:mthree}. Let $\left\{\cC_i\right\}_{i = 1}^n \in \U(\omega)$, 
$\left\{\cC_{\textnormal{s}, i}\right\}_{i = 1}^n \in \U(\omega_{\textnormal{s}}) $. 
Let Assumptions~\ref{ass:hetero_function}, \ref{ass:stochastic_variance_bounded_heter}, and \ref{ass:AB_assumption_heter}
be satisfied. Then,  
\begin{align*}
&\ExpSub{k}{\sqnorm{g^{k + 1} - v^{k + 1}}} \\
&\leq (1 - p) \ExpSub{k}{\sqnorm{g^{k} - v^{k}}} \\
&\quad + \frac{2 \nu^2 \omega}{n} \frac{1}{n} \sum \limits_{i = 1}^n \ExpSub{k}{\sqnorm{v_i^k - \nabla f_i (x_i^{k})}} \\
&\quad + \frac{6 \nu^2  \mu^2 \omega L_{\max}^2}{n} \frac{1}{n} \sum \limits_{i = 1}^n \ExpSub{k}{\sqnorm{w_i^{k} - x_i^{k}}} \\
&\quad + \frac{8 \nu^2 \mu^2 p_{\textnormal{s}} \omega L_{\max}^2}{n} \frac{1}{n} \sum \limits_{i = 1}^n \ExpSub{k}{\sqnorm{w_i^{k} - x^{k}}} \\
&\quad + \frac{2 \nu^2 \mu^2  (\omega_\textnormal{s} + 3) \omega \hat{L}^2}{n} \ExpSub{k}{\sqnorm{x^{k + 1} - x^{k}}} 
  + \frac{\omega \nu^2 \sigma^2}{n b}. 
\end{align*}
\end{lemma}
\begin{proof}
By the definition of $g^{k + 1}$ in \ref{eq:mthree}, we get
\begin{align*}
  &\ExpSub{k}{\sqnorm{g^{k + 1} - v^{k + 1}}} \\
  &= (1 - p) \ExpSub{k}{\sqnorm{\frac{1}{n} \sum \limits_{i = 1}^n \left(g_i^k + \cC_i(v_i^{k + 1} - v_i^k) - v_i^{k + 1}\right)}} \\
  &\overset{(i)}{=} (1 - p) \ExpSub{k}{\sqnorm{\frac{1}{n} \sum \limits_{i = 1}^n \cC_i(v_i^{k + 1} - v_i^k) - \left(v_i^{k + 1} - v_i^k\right)}}
    + (1 - p) \ExpSub{k}{\sqnorm{\frac{1}{n} \sum \limits_{i = 1}^n (g_i^k - v_i^{k})}} \\
  &\overset{(ii)}{=} (1 - p) \frac{1}{n^2} \sum \limits_{i = 1}^n \ExpSub{k}{\sqnorm{\cC_i(v_i^{k + 1} - v_i^k) - \left(v_i^{k + 1} - v_i^k\right)}}
    + (1 - p) \ExpSub{k}{\sqnorm{g^{k} - v^{k}}} \\
  &\leq (1 - p) \ExpSub{k}{\sqnorm{g^{k} - v^{k}}}
    + \frac{(1 - p) \omega}{n} \frac{1}{n} \sum \limits_{i = 1}^n \ExpSub{k}{\sqnorm{v_i^{k + 1} - v_i^k}},
\end{align*}

where $(i)$ follows from the unbiasedness of the compressors,
the tower property and the variance decomposition (Lemma~\ref{lemma:variance_decomposition}), 
while $(ii)$ relies on the fact that the compressors are independent.
Then, we use Lemma~\ref{lemma:gradient_momentum_iteration} and obtain
\begin{align*}
&\ExpSub{k}{\sqnorm{g^{k + 1} - v^{k + 1}}} \\
&\leq (1 - p) \ExpSub{k}{\sqnorm{g^{k} - v^{k}}} \\
&\quad + (1 - p) \frac{2 \nu^2 \omega}{n} \frac{1}{n} \sum \limits_{i = 1}^n \ExpSub{k}{\sqnorm{v_i^k - \nabla f_i (x_i^{k})}} \\
&\quad + (1 - p) \frac{6 \nu^2  \mu^2 \omega L_{\max}^2}{n} \frac{1}{n} \sum \limits_{i = 1}^n \ExpSub{k}{\sqnorm{w_i^{k} - x_i^{k}}} \\
&\quad + (1 - p) \frac{8 \nu^2 \mu^2 p_{\textnormal{s}} \omega L_{\max}^2}{n} \frac{1}{n} \sum \limits_{i = 1}^n \ExpSub{k}{\sqnorm{w_i^{k} - x^{k}}} \\
&\quad + (1 - p) \frac{2 \nu^2 \mu^2  (\omega_\textnormal{s} + 3) \omega \hat{L}^2}{n} \frac{1}{n} \sum \limits_{i = 1}^n \ExpSub{k}{\sqnorm{x^{k + 1} - x^{k}}} 
  + \frac{(1 - p) \omega \nu^2 \sigma^2}{n b}.
\end{align*}
Applying the inequality $1 - p \leq 1$ we complete the proof.
\end{proof}

The following lemma controls the bias caused by the momentum defined in \ref{eq:mthree}. 

\begin{lemma}\label{lemma:gradient_momentum_bias}
Let $\left\{\cC_i\right\}_{i = 1}^n \in \U(\omega)$, 
$\left\{\cC_{\textnormal{s}, i}\right\}_{i = 1}^n \in \U(\omega_{\textnormal{s}}) $. 
Let Assumptions~\ref{ass:hetero_function}, \ref{ass:stochastic_variance_bounded_heter}, and \ref{ass:AB_assumption_heter}
be satisfied. Then,   
\begin{align*}
\frac{1}{n} \sum_{i = 1}^n \ExpSub{k}{\sqnorm{v_i^{k + 1} - \nabla f_i(x_i^{k + 1})}} 
&\leq (1 - \nu) \frac{1}{n} \sum_{i = 1}^n \ExpSub{k}{\sqnorm{v_i^k - \nabla f_i(x_i^{k})}} \\
&\quad + \frac{9 L_{\max}^2 \mu^2}{\nu} \frac{1}{n} \sum_{i = 1}^n \ExpSub{k}{\sqnorm{w_i^{k} - x_i^{k}}}\\
&\quad + \frac{12 L_{\max}^2 \mu^2 p_{\textnormal{s}} }{\nu} \frac{1}{n} \sum_{i = 1}^n \ExpSub{k}{\sqnorm{w_i^{k} - x^{k}}}\\
&\quad + \frac{3 \hat{L}^2 \mu^2 (\omega_\textnormal{s} + 3)}{\nu} \ExpSub{k}{\sqnorm{x^{k + 1} - x^{k}}}
  + \nu^2 \frac{\sigma^2}{b} 
\end{align*}
and 
\begin{align*}
&\ExpSub{k}{\sqnorm{v^{k + 1} - \frac{1}{n} \sum \limits_{i = 1}^n \nabla f_i(x_i^{k + 1})}} \\
&\leq (1 - \nu) \ExpSub{k}{\sqnorm{\frac{1}{n} 
    \sum \limits_{i = 1}^n \left(v_i^k - \nabla f_i(x_i^{k})\right)}} \\
&\quad + \frac{9 \mu^2 L_A^2}{\nu} \frac{1}{n} \sum\limits_{i=1}^n \ExpSub{k}{\sqnorm{w_i^{k} - x_i^{k}}}\\
&\quad + \frac{12 \mu^2 p_{\textnormal{s}} L_A^2}{\nu} \frac{1}{n} \sum\limits_{i=1}^n \ExpSub{k}{\sqnorm{w_i^{k} - x^{k}}}\\
&\quad + \frac{9 \mu^2 L_B^2}{\nu} \ExpSub{k}{\norm{\frac{1}{n} \sum\limits_{i=1}^n \left(w_i^{k} - x_i^k\right)}^2} \\
&\quad + \frac{12 \mu^2 p_{\textnormal{s}} L_B^2}{\nu} \ExpSub{k}{\norm{\frac{1}{n} \sum\limits_{i=1}^n w_i^{k} - x^k}^2} \\
&\quad + \frac{3}{\nu} \left(\mu^2 (\omega_\textnormal{s} + 3) L_A^2 + \mu^2 (\frac{\omega_{\textnormal{s}}}{n} + 3)L_B^2\right) \ExpSub{k}{\norm{x^{k + 1} - x^k}^2}
  + \frac{\nu^2 \sigma^2}{n b}.
\end{align*}
\end{lemma}
\begin{proof}
We start with the first inequality. By the definition of $v_i^{k + 1}$ in \eqref{eq:mthree} we get 
\begin{align*}
&\ExpSub{k}{\sqnorm{v_i^{k + 1} - \nabla f_i(x_i^{k + 1})}} \\
&= \ExpSub{k}{\sqnorm{(1 - \nu) v_i^k + \nu \nabla f_i (x_i^{k + 1}; \xi_i^{k + 1}) - \nabla f_i(x_i^{k + 1})}} \\
&= \ExpSub{k}{\sqnorm{(1 - \nu) \left(v_i^k - \nabla f_i(x_i^{k + 1})\right) 
  + \nu \left(\nabla f_i (x_i^{k + 1}; \xi_i^{k + 1}) - \nabla f_i(x_i^{k + 1})\right)}} \\
&\overset{(i)}{=} (1 - \nu)^2 \ExpSub{k}{\sqnorm{v_i^k - \nabla f_i(x_i^{k + 1})}} 
  + \nu^2 \ExpSub{k}{\sqnorm{\nabla f_i (x_i^{k + 1}; \xi_i^{k + 1}) - \nabla f_i(x_i^{k + 1})}} \\
&\overset{(ii)}{\leq} (1 - \nu)^2 (1 + \rho)\ExpSub{k}{\sqnorm{v_i^k - \nabla f_i(x_i^{k})}} \\
&\quad + (1 - \nu)^2 (1 + \frac{1}{\rho}) \ExpSub{k}{\sqnorm{\nabla f_i(x_i^{k}) - \nabla f_i(x_i^{k + 1})}} 
  + \nu^2 \frac{\sigma^2}{b} \\
&\leq (1 - \nu)^2 (1 + \rho)\ExpSub{k}{\sqnorm{v_i^k - \nabla f_i(x_i^{k})}} 
  + (1 - \nu)^2 (1 + \frac{1}{\rho}) L_i^2 \ExpSub{k}{\sqnorm{x_i^{k} - x_i^{k + 1}}} 
  + \nu^2 \frac{\sigma^2}{b} \\
&\overset{(iii)}{\leq} (1 - \nu)\ExpSub{k}{\sqnorm{v_i^k - \nabla f_i(x_i^{k})}} 
  + \frac{3 L_i^2}{\nu} \ExpSub{k}{\sqnorm{x_i^{k} - x_i^{k + 1}}} 
  + \nu^2 \frac{\sigma^2}{b},
\end{align*}

where $(i)$ follows from the variance decomposition (Lemma~\ref{lemma:variance_decomposition}) 
and Assumption~\ref{ass:stochastic_variance_bounded_heter}. $(ii)$ applies Young's inequality with parameter $\rho > 0$.
$(iii)$ is obtained by setting $\rho = \frac{\nu}{2}$. Then, use Lemma~\ref{lemma:workers_point_iteration}
\begin{align*}
&\ExpSub{k}{\sqnorm{v_i^{k + 1} - \nabla f_i (x_i^{k + 1})}} \\
&\leq (1 - \nu)\ExpSub{k}{\sqnorm{v_i^k - \nabla f_i(x_i^{k})}} 
  + \frac{9 L_i^2 \mu^2}{\nu} \ExpSub{k}{\sqnorm{w_i^{k} - x_i^{k}}}
  + \frac{12 L_i^2 \mu^2 p_{\textnormal{s}} }{\nu} \ExpSub{k}{\sqnorm{w_i^{k} - x^{k}}}\\
&\quad + \frac{3 L_i^2 \mu^2 (\omega_\textnormal{s} + 3)}{\nu} \ExpSub{k}{\sqnorm{x^{k + 1} - x^{k}}}
  + \nu^2 \frac{\sigma^2}{b}.
\end{align*} 

Averaging and using 
\begin{align*}
\frac{1}{n} \sum\limits_{i=1}^n L_i^2 \ExpSub{k}{\sqnorm{w_i^{k} - x_i^{k}}} 
&\leq L_{\max}^2 \frac{1}{n} \sum\limits_{i=1}^n \ExpSub{k}{\sqnorm{w_i^{k} - x_i^{k}}}, \\
\frac{1}{n} \sum\limits_{i=1}^n L_i^2 \ExpSub{k}{\sqnorm{w_i^{k} - x^{k}}} 
&\leq L_{\max}^2 \frac{1}{n} \sum\limits_{i=1}^n \ExpSub{k}{\sqnorm{w_i^{k} - x^{k}}}, \\
\textnormal{and} \quad \frac{1}{n} \sum\limits_{i=1}^n L_i^2 \ExpSub{k}{\sqnorm{x^{k + 1} - x^{k}}} 
&= \hat{L}^2 \ExpSub{k}{\sqnorm{w_i^{k} - x^{k}}},
\end{align*}
we complete the proof.

Then, we consider the second inequality.
\begin{align*}
&\ExpSub{k}{\sqnorm{v^{k + 1} - \frac{1}{n} \sum \limits_{i = 1}^n \nabla f_i(x_i^{k + 1})}} \\
&= \ExpSub{k}{\sqnorm{\frac{1}{n} \sum \limits_{i = 1}^n \left((1 - \nu)v_i^k + \nu \nabla f_i (x_i^{k + 1}; \xi_i^{k + 1})\right)
  - \frac{1}{n} \sum \limits_{i = 1}^n \nabla f_i(x_i^{k + 1})}} \\
&= \ExpSub{k}{\sqnorm{\nu \frac{1}{n} \sum \limits_{i = 1}^n 
    \left(\nabla f_i (x_i^{k + 1}; \xi_i^{k + 1}) - \nabla f_i (x_i^{k + 1})\right)
  + (1 - \nu) \frac{1}{n} \sum \limits_{i = 1}^n \left(v_i^k - \nabla f_i(x_i^{k + 1})\right)}} \\
&\overset{(i)}{=} (1 - \nu)^2 \ExpSub{k}{\sqnorm{\frac{1}{n} \sum \limits_{i = 1}^n \left(v_i^k - \nabla f_i(x_i^{k + 1})\right)}} \\
&\quad + \nu^2 \ExpSub{k}{\sqnorm{\frac{1}{n} \sum \limits_{i = 1}^n 
    \left(\nabla f_i (x_i^{k + 1}; \xi_i^{k + 1}) - \nabla f_i (x_i^{k + 1})\right)}} \\
&\overset{(ii)}{=} (1 - \nu)^2 \ExpSub{k}{\sqnorm{\frac{1}{n} 
  \sum \limits_{i = 1}^n \left(v_i^k - \nabla f_i(x_i^{k}) + \nabla f_i(x_i^{k}) - \nabla f_i(x_i^{k + 1})\right)}} \\
&\quad + \nu^2 \frac{1}{n^2} \sum \limits_{i = 1}^n \ExpSub{k}{\sqnorm{
    \left(\nabla f_i (x_i^{k + 1}; \xi_i^{k + 1}) - \nabla f_i (x_i^{k + 1})\right)}} \\
&\overset{(iii)}{\leq} (1 - \nu)^2 (1 + \rho) \ExpSub{k}{\sqnorm{\frac{1}{n} 
    \sum \limits_{i = 1}^n \left(v_i^k - \nabla f_i(x_i^{k})\right)}} \\
&\quad + (1 - \nu)^2 \left(1 + \frac{1}{\rho}\right) \ExpSub{k}{\sqnorm{\frac{1}{n} 
    \sum \limits_{i = 1}^n \left(\nabla f_i(x_i^{k}) - \nabla f_i(x_i^{k + 1})\right)}}
  + \nu^2 \frac{\sigma^2}{n b} \\
&\overset{(iv)}{\leq} (1 - \nu) \ExpSub{k}{\sqnorm{\frac{1}{n} 
    \sum \limits_{i = 1}^n \left(v_i^k - \nabla f_i(x_i^{k})\right)}} \\
&\quad + \frac{3}{\nu} \ExpSub{k}{\underbrace{\sqnorm{\frac{1}{n} 
    \sum \limits_{i = 1}^n \left(\nabla f_i(x_i^{k}) - \nabla f_i(x_i^{k + 1})\right)}}_{I}}
  + \nu^2 \frac{\sigma^2}{n b},
\end{align*}

where $(i)$ follows from the variance decomposition (Lemma~\ref{lemma:variance_decomposition}) 
and Assumption~\ref{ass:stochastic_variance_bounded_heter}. $(ii)$ follows from
Assumption~\ref{ass:stochastic_variance_bounded_heter} and mutual independence of $\xi_i^{k + 1}$ for all $i \in [n]$.
$(iii)$ applies Young's inequality with parameter $\rho > 0$.
$(iv)$ is obtained by setting $\rho = \frac{\nu}{2}$. We consider $I$ separately. 
Using Assumption~\ref{ass:AB_assumption_heter} and Lemma~\ref{lemma:workers_point_iteration},
we get

\begin{align*}
I 
&\leq L_A^2\left(\frac{1}{n} \sum\limits_{i=1}^n \norm{x_i^{k + 1} - x_i^k}^2\right) 
    + L_B^2 \norm{\frac{1}{n} \sum\limits_{i=1}^n x_i^{k + 1} - x_i^k}^2 \\
&\leq 3 \mu^2 L_A^2 \frac{1}{n} \sum\limits_{i=1}^n \ExpSub{k}{\sqnorm{w_i^{k} - x_i^{k}}}\\
&\quad + 4 \mu^2 p_{\textnormal{s}} L_A^2 \frac{1}{n} \sum\limits_{i=1}^n \ExpSub{k}{\sqnorm{w_i^{k} - x^{k}}}\\
&\quad + 3 \mu^2 L_B^2 \ExpSub{k}{\norm{\frac{1}{n} \sum\limits_{i=1}^n \left(w_i^{k} - x_i^k\right)}^2} \\
&\quad + 4 \mu^2 p_{\textnormal{s}} L_B^2 \ExpSub{k}{\norm{\frac{1}{n} \sum\limits_{i=1}^n w_i^{k} - x^k}^2} \\
&\quad + \left(\mu^2 (\omega_\textnormal{s} + 3) L_A^2 + \mu^2 (\frac{\omega_{\textnormal{s}}}{n} + 3)L_B^2\right) \ExpSub{k}{\norm{x^{k + 1} - x^k}^2}.
\end{align*}
\end{proof}

\begin{theorem}[General theorem]\label{thm:main_hetero}
Assume that the function $f$ satisfies Assumptions~\ref{ass:hetero_function}, \ref{ass:stochastic_variance_bounded_heter} and \ref{ass:AB_assumption_heter}.
Suppose that $\left\{\cC_i\right\}_{i = 1}^n \in \U(\omega)$ and
$\left\{\cC_{\textnormal{s}, i}\right\}_{i = 1}^n \in \U(\omega_{\textnormal{s}}) $.
Let $\gamma > 0$ be such that
\begin{align}\label{eq:general_gamma_hetero}
\gamma \leq \left(L + \sqrt{c \left(\tilde{C}_A L_A^2 + \tilde{C}_B L_B^2 + \tilde{C}_{\max} L_{\max}^2\right)}\right)^{-1},
\end{align}
where $c = 1416$,
\begin{align*}
\tilde{C}_A &= \left(\frac{\omega_{\textnormal{s}} (p_{\textnormal{s}} + \mu) + 1}{\nu^2}  + \frac{\omega_{\textnormal{s}} p_{\textnormal{s}} + 1}{\mu^2} 
  + \frac{\omega_{\textnormal{s}}}{p_{\textnormal{s}}}\right), \\
\tilde{C}_B &= \left(\frac{\frac{\omega_{\textnormal{s}}}{n} (p_{\textnormal{s}} + \mu) + 1}{\nu^2}  + \frac{\frac{\omega_{\textnormal{s}}}{n} p_{\textnormal{s}} + 1}{\mu^2} 
  + \frac{\omega_{\textnormal{s}}}{n p_{\textnormal{s}}}\right), \\
and \quad \tilde{C}_{\max} &= \left(\frac{\mu \omega \omega_\textnormal{s}}{n p} + \frac{\omega (1 + \omega_s p_{\textnormal{s}})}{n p}\right).
\end{align*}

Let
\begin{align*}
\Psi^k &= f(x^k) - f^* + \lambda_A \sqnorm{g^k - v^k} 
  + \lambda_B \sqnorm{v^k - \frac{1}{n} \sum \limits_{i = 1}^n \nabla f_i(x_i^k)}
  + \lambda_C \frac{1}{n} \sum \limits_{i = 1}^n \sqnorm{v_i^{k} - \nabla f_i (x_i^{k})} \\
&\quad + \lambda_D \frac{1}{n} \sum\limits_{i=1}^n \norm{w_i^k - x_i^k}^2
  + \lambda_E \frac{1}{n} \sum\limits_{i=1}^n \norm{w_i^k - x^k}^2 \\
&\quad + \lambda_F \norm{\frac{1}{n} \sum\limits_{i=1}^n \left(w_i^k - x_i^k\right)}^2 
      + \lambda_G \norm{\frac{1}{n} \sum\limits_{i=1}^n w_i^k - x^k}^2,
\end{align*}

where $\lambda_A = \frac{3 \gamma}{2 p}$,
$\lambda_B = \frac{3 \gamma}{2 \nu}$,
$\lambda_C = \frac{3 \nu \omega \gamma}{n p}$,
$\lambda_D = \gamma \left(\frac{27 \mu L_A^2}{\nu^2} + \frac{72 \mu L_{\max}^2 \omega}{n p} + \frac{6 L_A^2}{\mu}\right)$, 
$\lambda_E = \gamma \left(\frac{336 \omega L_{\max}^2}{n p}\left(\mu + p_{\textnormal{s}}\right)
  + \frac{126 L_A^2}{\nu^2} (\mu + p_{\textnormal{s}}) 
  + 48 L_A^2\left(\frac{1}{p_{\textnormal{s}}} + \frac{p_{\textnormal{s}}}{\mu^2}\right)\right)$, 
$\lambda_F = \gamma \left(\frac{27 \mu L_B^2}{\nu^2} + \frac{6 L_B^2}{\mu}\right)$ and 
$\lambda_G = \gamma \left(
  \frac{128 L_B^2}{\nu^2} (p_{\textnormal{s}} + \mu)
  + 48 L_B^2\left(\frac{p_{\textnormal{s}}}{\mu^2} + \frac{1}{p_{\textnormal{s}}}\right)
\right)$. Then, \ref{eq:mthree} with $\nu$ such that 
\begin{align*}
  \frac{9 \nu^2 \omega \sigma^2}{n p b} + \frac{3 \nu \sigma^2}{n b} \leq \frac{\varepsilon}{2}
\end{align*} 
ensures that 
\begin{align*}
\frac{1}{K} \sum\limits_{k=0}^{K - 1} \Exp{\norm{\nabla f(x^k)}^2}  \leq \frac{2 \Exp{{\Psi}^0}}{\gamma K} + \frac{\varepsilon}{2}.
\end{align*}

\end{theorem}
\begin{proof}
Using Assumption~\ref{ass:lipschitz_constant} and Lemma~\ref{lemma:page_lemma},
\begin{align*}
  f(x^{k+1}) 
  &\leq f(x^{k}) - \frac{\gamma}{2} \norm{\nabla f(x^k)}^2 - \left(\frac{1}{2 \gamma} - \frac{L}{2}\right) \norm{x^{k+1} - x^k}^2 + \frac{\gamma}{2} \sqnorm{g^k - \nabla f(x^k)}.
\end{align*}
Using Lemma~\ref{lemma:heterogeneus_estimator_residual} and taking expectation $\ExpSub{k}{\cdot}$ from both parts,
we obtain
\begin{align}\label{eq:skjdnfwkdmcmoidw}
\ExpSub{k}{f(x^{k+1})}
&\leq f(x^{k}) - \frac{\gamma}{2} \norm{\nabla f(x^k)}^2 - \left(\frac{1}{2 \gamma} - \frac{L}{2}\right) \ExpSub{k}{\norm{x^{k+1} - x^k}^2} \\
&\quad + \frac{3 \gamma}{2} \ExpSub{k}{\sqnorm{g^k - v^k}} + \frac{3 \gamma}{2} \ExpSub{k}{\sqnorm{v^k - \frac{1}{n} \sum \limits_{i = 1}^n \nabla f_i(x_i^k)}} \nonumber\\
&\quad +  \frac{6 \gamma L_A^2}{2} \ExpSub{k}{\frac{1}{n} \sum\limits_{i=1}^n \norm{w_i^k - x_i^k}^2}
  + \frac{6 \gamma L_A^2}{2} \ExpSub{k}{\frac{1}{n} \sum\limits_{i=1}^n \norm{w_i^k - x^k}^2} \nonumber\\
&\quad + \frac{6 \gamma L_B^2}{2} \ExpSub{k}{\norm{\frac{1}{n} \sum\limits_{i=1}^n \left(w_i^k - x_i^k\right)}^2}
      + \frac{6 \gamma L_B^2}{2} \ExpSub{k}{\norm{\frac{1}{n} \sum\limits_{i=1}^n w_i^k - x^k}^2}. \nonumber
\end{align}

To simplify the derivation, we define the seven norm-based terms in the Lyapunov function $\Psi^k$ as 
follows:

\begin{align*}
&\delta^k = f(x^k) - f^* &&X^k = \ExpSub{k}{\norm{x^{k+1} - x^k}^2} \\
&A^k = \ExpSub{k}{\sqnorm{g^k - v^k}}  &&B^k = \ExpSub{k}{\sqnorm{v^k - \frac{1}{n} \sum \limits_{i = 1}^n \nabla f_i(x_i^k)}} \\
&C^k = \ExpSub{k}{\frac{1}{n} \sum \limits_{i = 1}^n \sqnorm{v_i^k - \nabla f_i(x_i^k)}}  &&D^k = \ExpSub{k}{\frac{1}{n} \sum \limits_{i = 1}^n \norm{w_i^k - x_i^k}^2} \\
&E^k = \ExpSub{k}{\frac{1}{n} \sum \limits_{i = 1}^n \norm{w_i^k - x^k}^2}  &&F^k = \ExpSub{k}{\norm{\frac{1}{n} \sum \limits_{i = 1}^n \left(w_i^k - x_i^k\right)}^2} \\
&G^k = \ExpSub{k}{\norm{\frac{1}{n} \sum \limits_{i = 1}^n w_i^k - x^k}^2} 
\end{align*}

We define $\tilde{A}^{k+1} = \ExpSub{k}{A^{k+1}}, \dots, \tilde{G}^{k+1} = \ExpSub{k}{G^{k+1}}$.
Using the notation introduced above, inequality~\eqref{eq:skjdnfwkdmcmoidw} takes the form
\begin{align}\label{eq:start_of_the_proof_hetero_complexity}
\ExpSub{k}{f^{k + 1}} + \frac{\gamma}{2} \norm{\nabla f(x^k)}^2
&\leq \delta^k
  + \frac{3 \gamma}{2} A^k 
  + \frac{3 \gamma}{2} B^k 
  + \frac{6 \gamma L_A^2}{2} D^k \\
&\quad + \frac{6 \gamma L_A^2}{2} E^k
  +\frac{6 \gamma L_B^2}{2} F^k
  +\frac{6 \gamma L_B^2}{2} G^k
  - \left(\frac{1}{2 \gamma} - \frac{L}{2}\right) X^k. \nonumber
\end{align}
Adding $\ExpSub{k}{\Psi^{k + 1} - f^{k + 1}}$ to both sides of inequality~\eqref{eq:start_of_the_proof_hetero_complexity} yields
\begin{align}\label{eq:kslamckoibfbkd}
\ExpSub{k}{\Psi^{k + 1}} + \frac{\gamma}{2} \norm{\nabla f(x^k)}^2
&\leq \ExpSub{k}{\Psi^{k + 1} - f^{k + 1}} + \delta^k
  + \frac{3 \gamma}{2} A^k 
  + \frac{3 \gamma}{2} B^k 
  + \frac{6 \gamma L_A^2}{2} D^k \\
&\quad + \frac{6 \gamma L_A^2}{2} E^k
  +\frac{6 \gamma L_B^2}{2} F^k
  +\frac{6 \gamma L_B^2}{2} G^k
  - \left(\frac{1}{2 \gamma} - \frac{L}{2}\right) X^k. \nonumber
\end{align}

In this notation, the recursion lemmas take the following form: 
\begin{enumerate}
  \item Lemma~\eqref{lemma:estimator_momentum_residual} yields
\begin{align}\label{eq:A_bound}
\tilde{A}^{k + 1} 
&\leq (1 - p) A^k
+ \frac{2 \nu^2 \omega}{n} C^k
+ \frac{6 \nu^2 \mu^2 \omega L_{\max}^2}{n} D^k 
+ \frac{8 \nu^2 \mu^2 p_{\textnormal{s}} \omega L_{\max}^2}{n} E^k \\
& \quad + \frac{2 \nu^2 \mu^2  (\omega_\textnormal{s} + 3) \omega \hat{L}^2}{n} X^k
  + \frac{\omega \nu^2 \sigma^2}{n b}. \nonumber
\end{align}

  \item Lemma~\ref{lemma:gradient_momentum_bias} implies
\begin{align}\label{eq:C_bound}
\tilde{C}^{k + 1}
\leq (1 - \nu) C^k
+ \frac{9 L_{\max}^2 \mu^2}{\nu} D^k
+ \frac{12 L_{\max}^2 \mu^2 p_{\textnormal{s}} }{\nu} E^k
+ \frac{3 \hat{L}^2 \mu^2 (\omega_\textnormal{s} + 3)}{\nu} X^k
  + \nu^2 \frac{\sigma^2}{b}.
\end{align}

  \item Lemma~\ref{lemma:gradient_momentum_bias} yields
\begin{align}\label{eq:B_bound}
\tilde{B}^{k + 1}
&\leq (1 - \nu) B^k 
  + \frac{9 \mu^2 L_A^2}{\nu} D^k 
  + \frac{12 \mu^2 p_{\textnormal{s}} L_A^2}{\nu} E^k
  + \frac{9 \mu^2 L_B^2}{\nu} F^k 
  + \frac{12 \mu^2 p_{\textnormal{s}} L_B^2}{\nu} G^k \\
&\quad + \frac{3}{\nu}\left(\mu^2 (\omega_\textnormal{s} + 3) L_A^2 + \mu^2 (\frac{\omega_{\textnormal{s}}}{n} + 3)L_B^2\right) X^k
  + \frac{\nu^2 \sigma^2}{n b}. \nonumber
\end{align}

  \item Lemma~\ref{lemma:worker_point_momentum_residual} gives
\begin{align}\label{eq:D_bound}
\tilde{D}^{k + 1}
&\leq \left(1 - \frac{\mu}{2}\right) D^k
  + 4 p_{\textnormal{s}} \left(1 + \frac{p_{\textnormal{s}}}{\mu}\right) E^k
  + 4 \left(\frac{1}{\mu} + \omega_{\textnormal{s}}\right) X^k.
\end{align}

  \item Lemma~\ref{lemma:worker_point_momentum_residual} implies
\begin{align}\label{eq:F_bound}
\tilde{F}^{k + 1}
&\leq \left(1 - \frac{\mu}{2}\right) F^k  
  + 4 p_{\textnormal{s}}\left(1 + \frac{p_{\textnormal{s}}}{\mu}\right) G^k
  + 4 \left(\frac{1}{\mu} + \frac{\omega_{\textnormal{s}}}{n}\right) X^k.
\end{align}

  \item Lemma~\ref{lemma:server_point_momentum_residual} where we additionally use the inequality $1 - p_{\textnormal{s}} \leq 1$ gives 
\begin{align}\label{eq:E_bound}
  \tilde{E}^{k + 1}
  &\leq (1 - p_{\textnormal{s}}) E^k + \omega_{\textnormal{s}} X^k.
\end{align}

  \item Lemma~\ref{lemma:server_point_momentum_residual} where we additionally use the inequality $1 - p_{\textnormal{s}} \leq 1$ yields
\begin{align}\label{eq:G_bound}
  \tilde{G}^{k + 1}
  &\leq (1 - p_{\textnormal{s}}) G^k + \frac{\omega_{\textnormal{s}}}{n} X^k.
\end{align}
\end{enumerate}

Substituting inequalities \eqref{eq:A_bound}--\eqref{eq:G_bound} into the definition of $\Psi^{k+1}$, we obtain:

\begin{align*}
&\ExpSub{k}{\Psi^{k + 1} - f^{k + 1}} \\
&= \lambda_A \tilde{A}^{k + 1} 
  + \lambda_B \tilde{B}^{k + 1} 
  + \lambda_C \tilde{C}^{k + 1} 
  + \lambda_D \tilde{D}^{k + 1} 
  + \lambda_E \tilde{E}^{k + 1} 
  + \lambda_F \tilde{F}^{k + 1} 
  + \lambda_G \tilde{G}^{k + 1} \\
&\leq \lambda_A \left[ (1 - p) A^k
  + \frac{2 \nu^2 \omega}{n} C^k
  + \frac{6 \nu^2 \mu^2 \omega L_{\max}^2}{n} D^k 
  + \frac{8 \nu^2 \mu^2 p_{\textnormal{s}} \omega L_{\max}^2}{n} E^k 
  + \frac{2 \nu^2 \mu^2  (\omega_\textnormal{s} + 3) \omega \hat{L}^2}{n} X^k
  + \frac{\omega \nu^2 \sigma^2}{n b} \right] \\
&\quad + \lambda_B \left[ (1 - \nu) B^k 
  + \frac{9 \mu^2 L_A^2}{\nu} D^k 
  + \frac{12 \mu^2 p_{\textnormal{s}} L_A^2}{\nu} E^k
  + \frac{9 \mu^2 L_B^2}{\nu} F^k 
  + \frac{12 \mu^2 p_{\textnormal{s}} L_B^2}{\nu} G^k \right] \\
&\quad + \lambda_B \left[
  \frac{3}{\nu}\left(\mu^2 (\omega_\textnormal{s} + 3) L_A^2 + \mu^2 (\frac{\omega_{\textnormal{s}}}{n} + 3)L_B^2\right) X^k
  + \frac{\nu^2 \sigma^2}{n b} \right] \\
&\quad + \lambda_C \left[ (1 - \nu) C^k
  + \frac{9 L_{\max}^2 \mu^2}{\nu} D^k
  + \frac{12 L_{\max}^2 \mu^2 p_{\textnormal{s}} }{\nu} E^k
  + \frac{3 \hat{L}^2 \mu^2 (\omega_\textnormal{s} + 3)}{\nu} X^k
  + \frac{\nu^2 \sigma^2}{b} \right] \\
&\quad + \lambda_D \left[ \left(1 - \frac{\mu}{2}\right) D^k
  + 4 p_{\textnormal{s}} \left(1 + \frac{p_{\textnormal{s}}}{\mu}\right) E^k
  + 4 \left(\frac{1}{\mu} + \omega_{\textnormal{s}}\right) X^k \right] \\
&\quad + \lambda_E \left[ (1 - p_{\textnormal{s}}) E^k + \omega_{\textnormal{s}} X^k \right] \\
&\quad + \lambda_F \left[ \left(1 - \frac{\mu}{2}\right) F^k  
  + 4 p_{\textnormal{s}}\left(1 + \frac{p_{\textnormal{s}}}{\mu}\right) G^k
  + 4 \left(\frac{1}{\mu} + \frac{\omega_{\textnormal{s}}}{n}\right) X^k \right] \\
&\quad + \lambda_G \left[ (1 - p_{\textnormal{s}}) G^k + \frac{\omega_{\textnormal{s}}}{n} X^k \right].
\end{align*}

Next, we group the terms corresponding to the quantities ($A^k, B^k, \dots, G^k, X^k$) and the constant terms
multiplied by $\sigma^2$.
The resulting inequality can be written as:

\begin{align}\label{eq:lskmcklmcsnls}
&\ExpSub{k}{\Psi^{k + 1} - f^{k + 1}} \\ &\leq \tilde{\lambda}_A A^k + \tilde{\lambda}_B B^k + \tilde{\lambda}_C C^k 
  + \tilde{\lambda}_D D^k + \tilde{\lambda}_E E^k + \tilde{\lambda}_F F^k + \tilde{\lambda}_G G^k
  + \tilde{\lambda}_X X^k + \tilde{c}, \nonumber
\end{align}

where the coefficients $\tilde{\lambda}_A, \dots, \tilde{\lambda}_G, \tilde{\lambda}_X, \tilde{c}$ are defined as follows:

\begin{align}
\tilde{\lambda}_A &= \lambda_A (1 - p), \label{tilde_A} \\
\tilde{\lambda}_B &= \lambda_B (1 - \nu), \label{tilde_B} \\
\tilde{\lambda}_C &= \lambda_C (1 - \nu) + \lambda_A \frac{2 \nu^2 \omega}{n}, \label{tilde_C}\\
\tilde{\lambda}_D &= \lambda_D \left(1 - \frac{\mu}{2}\right) 
+ \lambda_A \frac{6 \nu^2 \mu^2 \omega L_{\max}^2}{n} 
+ \lambda_B \frac{9 \mu^2 L_A^2}{\nu} 
+ \lambda_C \frac{9 L_{\max}^2 \mu^2}{\nu}, \label{tilde_D}\\
\tilde{\lambda}_E &= \lambda_E (1 - p_{\textnormal{s}}) 
+ \lambda_A \frac{8 \nu^2 \mu^2 p_{\textnormal{s}} \omega L_{\max}^2}{n} 
+ \lambda_B \frac{12 \mu^2 p_{\textnormal{s}} L_A^2}{\nu} 
+ \lambda_C \frac{12 L_{\max}^2 \mu^2 p_{\textnormal{s}} }{\nu} 
+ \lambda_D 4 p_{\textnormal{s}} \left(1 + \frac{p_{\textnormal{s}}}{\mu}\right), \label{tilde_E}\\
\tilde{\lambda}_F &= \lambda_F \left(1 - \frac{\mu}{2}\right) 
+ \lambda_B \frac{9 \mu^2 L_B^2}{\nu}, \\
\tilde{\lambda}_G &= \lambda_G (1 - p_{\textnormal{s}}) 
+ \lambda_B \frac{12 \mu^2 p_{\textnormal{s}} L_B^2}{\nu} 
+ \lambda_F 4 p_{\textnormal{s}}\left(1 + \frac{p_{\textnormal{s}}}{\mu}\right), \label{tilde_F}\\
\tilde{\lambda}_X &= \lambda_A \frac{2 \nu^2 \mu^2  (\omega_\textnormal{s} + 3) \omega \hat{L}^2}{n} 
+ \lambda_B \frac{3}{\nu}\left(\mu^2 (\omega_\textnormal{s} + 3) L_A^2 + \mu^2 (\frac{\omega_{\textnormal{s}}}{n} + 3)L_B^2\right) \label{tilde_X}\\
&\quad + \lambda_C \frac{3 \hat{L}^2 \mu^2 (\omega_\textnormal{s} + 3)}{\nu} 
+ \lambda_D 4 \left(\frac{1}{\mu} + \omega_{\textnormal{s}}\right) 
+ \lambda_E \omega_{\textnormal{s}} 
+ \lambda_F 4 \left(\frac{1}{\mu} + \frac{\omega_{\textnormal{s}}}{n}\right) 
+ \lambda_G \frac{\omega_{\textnormal{s}}}{n}, \\
\tilde{c} &= \lambda_A \frac{\omega \nu^2 \sigma^2}{n b} + \lambda_B \frac{\nu^2 \sigma^2}{n b} + \lambda_C \frac{\nu^2 \sigma^2}{b} \label{tilde_c}.
\end{align}
Specifically, the term $A^k$ appears in \eqref{eq:A_bound};
$B^k$ appears in \eqref{eq:B_bound};
terms involving $C^k$ are collected from inequalities \eqref{eq:A_bound} and \eqref{eq:C_bound};
$F^k$ appears in  \eqref{eq:B_bound} and \eqref{eq:F_bound};
the coefficient of $E^k$ is gathered from inequalities \eqref{eq:A_bound},  \eqref{eq:C_bound}, \eqref{eq:B_bound}, \eqref{eq:D_bound} and \eqref{eq:E_bound};
$G^k$ appears in  \eqref{eq:B_bound},  \eqref{eq:F_bound} and \eqref{eq:G_bound}.
Then, we collect the constant terms appearing in \eqref{eq:A_bound}, \eqref{eq:B_bound} and \eqref{eq:C_bound}.
Finally, the coefficients associated with $X^k$ appear in 
\eqref{eq:A_bound}, \eqref{eq:B_bound}, \eqref{eq:C_bound}, \eqref{eq:D_bound}, \eqref{eq:E_bound}, \eqref{eq:F_bound},
and \eqref{eq:G_bound}.

Substituting \eqref{eq:lskmcklmcsnls} into inequality~\eqref{eq:kslamckoibfbkd} yields
\begin{align}\label{eq:kslamckosdx}
&\ExpSub{k}{\Psi^{k + 1}} + \frac{\gamma}{2} \norm{\nabla f(x^k)}^2 \\
&\leq \delta^k
  + \left(\tilde{\lambda}_A + \frac{3 \gamma}{2}\right) A^k 
  + \left(\tilde{\lambda}_B + \frac{3 \gamma}{2}\right) B^k 
  + \tilde{\lambda}_C C^k
  + \left(\tilde{\lambda}_D + \frac{6 \gamma L_A^2}{2}\right) D^k \nonumber\\
&\quad + \left(\tilde{\lambda}_E + \frac{6 \gamma L_A^2}{2}\right) E^k 
  + \left(\tilde{\lambda}_F + \frac{6 \gamma L_B^2}{2}\right) F^k
  + \left(\tilde{\lambda}_G + \frac{6 \gamma L_B^2}{2}\right) G^k \nonumber \\
&\quad - \left(\frac{1}{2 \gamma} - \frac{L}{2} - \tilde{\lambda}_X\right) X^k + \tilde{c}. \nonumber
\end{align}

Let $\hat{\lambda}_A, \dots, \hat{\lambda}_G$ denote the coefficients of $A^k, \dots, G^k$ in inequality \eqref{eq:kslamckosdx}, respectively.
Recall the chosen values for $\lambda_A, \cdots, \lambda_G$:

\begin{align}
&\lambda_A = \frac{3 \gamma}{2 p} \label{eq:lambda_choice_A}\\
&\lambda_B = \frac{3 \gamma}{2 \nu} \label{eq:lambda_choice_B} \\
&\lambda_C = \frac{3 \nu \omega \gamma}{n p} \label{eq:lambda_choice_C} \\
&\lambda_D = \gamma \left(
  \frac{27 \mu L_A^2}{\nu^2} 
  + \frac{72 \mu L_{\max}^2 \omega}{n p}
  + \frac{6 L_A^2}{\mu}\right) \label{eq:lambda_choice_D} \\
&\lambda_E = \gamma \left(\frac{336 \omega L_{\max}^2}{n p}\left(\mu + p_{\textnormal{s}}\right)
  + \frac{126 L_A^2}{\nu^2} (\mu + p_{\textnormal{s}}) 
  + 48 L_A^2\left(\frac{1}{p_{\textnormal{s}}} + \frac{p_{\textnormal{s}}}{\mu^2}\right)\right) \label{eq:lambda_choice_E}\\
&\lambda_F = \gamma \left(\frac{27 \mu L_B^2}{\nu^2} + \frac{6 L_B^2}{\mu}\right) \label{eq:lambda_choice_F} \\
&\lambda_G = \gamma \left(
  \frac{128 L_B^2}{\nu^2} (p_{\textnormal{s}} + \mu)
  + 48 L_B^2\left(\frac{p_{\textnormal{s}}}{\mu^2} + \frac{1}{p_{\textnormal{s}}}\right)
\right). \label{eq:lambda_choice_G}
\end{align}

We now demonstrate that this choice of $\lambda_A, \dots, \lambda_G$ ensures
$\hat{\lambda}_A \leq \lambda_A, \dots, \hat{\lambda}_G \leq \lambda_G$.
We start with $\hat{\lambda}_A$:
\begin{align*}
\hat{\lambda}_A = \tilde{\lambda}_A + \frac{3 \gamma}{2} = (1 - p) \lambda_A + \frac{3 \gamma}{2}.
\end{align*}

Observe that selecting $\lambda_A$ as in \eqref{eq:lambda_choice_A} ensures $p \lambda_A = \frac{3 \gamma}{2}$. 
Therefore, $\hat{\lambda}_A = \lambda_A$. We apply the same logic to the other coefficients.
Then,
\begin{align*}
\hat{\lambda}_B = \tilde{\lambda}_B + \frac{3 \gamma}{2} = (1 - \nu) \lambda_B + \frac{3 \gamma}{2}
= (1 - \nu) \lambda_B + \nu \lambda_B = \lambda_B.
\end{align*}

Since $\lambda_C = \frac{6 \nu \omega \gamma}{2 n p}$, we obtain
\begin{align*}
\hat{\lambda}_C = (1 - \nu) \lambda_C + \frac{2 \nu^2 \omega}{n} \lambda_A 
= (1 - \nu) \lambda_C + \frac{2 \nu^2 \omega}{n} \frac{3 \gamma}{2 p} 
= (1 - \nu) \lambda_C + \frac{6 \nu^2 \omega \gamma}{2 n p}
= \lambda_C.
\end{align*}

Due to our choice of $\lambda_D$ in \eqref{eq:lambda_choice_D}:
\begin{align*}
\hat{\lambda}_D &= (1 - \frac{\mu}{2}) \lambda_D
+ \frac{6 \nu^2 \mu^2 \omega L_{\max}^2}{n} \lambda_A
  + \frac{9 \mu^2 L_A^2}{\nu} \lambda_B
  + \frac{9 L_{\max}^2 \mu^2}{\nu} \lambda_C 
  + \frac{6 \gamma L_A^2}{2} \\
&= (1 - \frac{\mu}{2}) \lambda_D + \frac{6 \nu^2 \mu^2 \omega L_{\max}^2}{n} \frac{3 \gamma}{2 p}
  + \frac{9 \mu^2 L_A^2}{\nu} \frac{3 \gamma}{2 \nu}
  + \frac{9 L_{\max}^2 \mu^2}{\nu} \frac{6 \nu \omega}{2 n p}
  + \frac{6 \gamma L_A^2}{2} \\
&= (1 - \frac{\mu}{2}) \lambda_D + \frac{\mu \gamma}{2} \left(\frac{18 \nu^2 \mu \omega L_{\max}^2}{n p}
  + \frac{27 \mu L_A^2}{\nu^2} 
  + \frac{54 L_{\max}^2 \mu \omega}{n p}
  + \frac{6 L_A^2}{\mu}\right) \\
&\leq (1 - \frac{\mu}{2}) \lambda_D + \frac{\mu \gamma}{2} \left(
  \frac{27 \mu L_A^2}{\nu^2} 
  + \frac{72 \mu L_{\max}^2 \omega}{n p}
  + \frac{6 L_A^2}{\mu}\right)  = \lambda_D.
\end{align*}

The final inequality relies on $\nu^2 \leq 1$. Next, 
\begin{align*}
&\hat{\lambda}_F = (1 - \frac{\mu}{2}) \lambda_F + \frac{9 \mu^2 L_B^2}{\nu} \lambda_B
  + \frac{6 \gamma L_B^2}{2} \\
&= (1 - \frac{\mu}{2}) \lambda_F + \frac{9 \mu^2 L_B^2}{\nu} \frac{3 \gamma}{2 \nu} + \frac{6 \gamma L_B^2}{2}
= (1 - \frac{\mu}{2}) \lambda_F + \frac{\mu \gamma}{2} \left( \frac{27 \mu L_B^2}{\nu^2}
  + \frac{6 L_B^2}{\mu}\right) = \lambda_F.
\end{align*}

Then,
\begin{align*}
&\hat{\lambda}_G = (1 - p_{\textnormal{s}}) \lambda_G 
  + \underbrace{\frac{12 \mu^2 p_{\textnormal{s}} L_B^2}{\nu} \lambda_B
  + 4 p_{\textnormal{s}}\left(1 + \frac{p_{\textnormal{s}}}{\mu}\right) \lambda_F
  + \frac{6 \gamma L_B^2}{2}}_{I}. \\  
\end{align*}
We consider $I$ separately:
\begin{align*}
&I = \frac{12 \mu^2 p_{\textnormal{s}} L_B^2}{\nu} \frac{3 \gamma}{2 \nu}
  + 4 p_{\textnormal{s}}\left(1 + \frac{p_{\textnormal{s}}}{\mu}\right) 
    \gamma \left(\frac{27 \mu L_B^2}{\nu^2} + \frac{6 L_B^2}{\mu}\right)
  + \frac{6 \gamma L_B^2}{2}\\
&=  p_{\textnormal{s}} \gamma  \left(
  \frac{18 \mu^2 L_B^2}{\nu^2}
  + 4 \left(1 + \frac{p_{\textnormal{s}}}{\mu}\right) 
    \left(\frac{27 \mu L_B^2}{\nu^2} + \frac{6 L_B^2}{\mu}\right)
  + \frac{3 L_B^2}{p_{\textnormal{s}}}
\right) \\
&= p_{\textnormal{s}} \gamma \left(
  \frac{18 \mu^2 L_B^2}{\nu^2}
  + \frac{4 p_{\textnormal{s}}}{\mu} \left(\frac{27 \mu L_B^2}{\nu^2} + \frac{6 L_B^2}{\mu}\right)
  + 4 \left(\frac{27 \mu L_B^2}{\nu^2} + \frac{6 L_B^2}{\mu}\right)
  + \frac{3 L_B^2}{p_{\textnormal{s}}}
\right) \\
&\leq p_{\textnormal{s}} \gamma \left(
  4 p_{\textnormal{s}}\left(\frac{27 L_B^2}{\nu^2} + \frac{6 L_B^2}{\mu^2}\right)
  + 4 \left(\frac{32 \mu L_B^2}{\nu^2} + \frac{6 L_B^2}{\mu}\right)
  + \frac{3 L_B^2}{p_{\textnormal{s}}}
\right) \\
&\leq p_{\textnormal{s}} \gamma \left(
  \frac{128 p_{\textnormal{s}} L_B^2}{\nu^2} + \frac{128 \mu L_B^2}{\nu^2} 
  + \frac{24 p_{\textnormal{s}} L_B^2}{\mu^2} + \frac{24 L_B^2}{\mu} + \frac{24 L_B^2}{p_{\textnormal{s}}}
\right) \\
&= p_{\textnormal{s}} \gamma \left(
  \frac{128 L_B^2}{\nu^2} (p_{\textnormal{s}} + \mu)
  + 24 L_B^2\left(\frac{p_{\textnormal{s}}}{\mu^2} + \frac{1}{\mu} + \frac{1}{p_{\textnormal{s}}}\right)
\right) \\
&\leq p_{\textnormal{s}} \gamma \left(
  \frac{128 L_B^2}{\nu^2} (p_{\textnormal{s}} + \mu)
  + 48 L_B^2\left(\frac{p_{\textnormal{s}}}{\mu^2} + \frac{1}{p_{\textnormal{s}}}\right)
\right) \leq p_{\textnormal{s}} \lambda_G.
\end{align*}
Here, we apply the inequality $ab \leq a^2 + b^2$ with $a = \sqrt{\frac{p_{\textnormal{s}}}{\mu^2}}$,
$b = \sqrt{\frac{1}{p_{\textnormal{s}}}}$ and our choice of $\lambda_G$ defined in \eqref{eq:lambda_choice_G}. 
Therefore, we obtain $\hat{\lambda}_G \leq \lambda_G$. Then,

\begin{align*}
&\hat{\lambda}_E = (1 - p_{\textnormal{s}}) \lambda_E \\
&\quad + \underbrace{\frac{8 \nu^2 \mu^2 p_{\textnormal{s}} \omega L_{\max}^2}{n} \lambda_A
  + \frac{12 \mu^2 p_{\textnormal{s}} L_A^2}{\nu} \lambda_B
  + \frac{12 L_{\max}^2 \mu^2 p_{\textnormal{s}} }{\nu} \lambda_C
  + 4 p_{\textnormal{s}}\left(1 + \frac{p_{\textnormal{s}}}{\mu}\right) \lambda_D
  + \frac{6 \gamma L_A^2}{2}}_{I}.
\end{align*}
As before, we consider $I$ separately:
\begin{align*}
I &= p_{\textnormal{s}} \gamma \left(\frac{8 \nu^2 \mu^2 \omega L_{\max}^2}{n} \frac{3}{2 p}
  + \frac{12 \mu^2 L_A^2}{\nu} \frac{3}{2 \nu}
  + \frac{12 \mu^2 L_{\max}^2}{\nu} \frac{6 \nu \omega}{2 n p}
  + \frac{3L_A^2}{p_{\textnormal{s}}}\right) \\
&\quad + 4 p_{\textnormal{s}} \gamma \left(1 + \frac{p_{\textnormal{s}}}{\mu}\right) \left(\frac{27 \mu L_A^2}{\nu^2} 
  + \frac{72 \mu L_{\max}^2 \omega}{n p} + \frac{6 L_A^2}{\mu}\right) \\
&\leq p_{\textnormal{s}} \gamma \left(\frac{12 \mu \omega L_{\max}^2}{n p}
  + \frac{18 \mu L_A^2}{\nu^2}
  + \frac{36 \mu \omega L_{\max}^2}{n p}
  + \frac{3L_A^2}{p_{\textnormal{s}}}\right)
  + 4 p_{\textnormal{s}} \gamma \left(\frac{27 \mu L_A^2}{\nu^2} + \frac{72 \mu L_{\max}^2 \omega}{n p} + \frac{6 L_A^2}{\mu}\right) \\
&\quad + 4 \gamma \frac{p_{\textnormal{s}}^2}{\mu} \left(\frac{27 \mu L_A^2}{\nu^2} + \frac{72 \mu L_{\max}^2 \omega}{n p} 
  + \frac{6 L_A^2}{\mu}\right) \\
&= p_{\textnormal{s}} \gamma \left(\frac{48 \mu \omega L_{\max}^2}{n p}
  + \frac{18 \mu L_A^2}{\nu^2}
  + \frac{3L_A^2}{p_{\textnormal{s}}}\right)
  + 4 p_{\textnormal{s}} \gamma \left(\frac{72 \mu \omega L_{\max}^2}{n p} + \frac{27 \mu L_A^2}{\nu^2} + \frac{6 L_A^2}{\mu}\right) \\
&\quad + 4 p_{\textnormal{s}} \gamma \left(\frac{27 p_{\textnormal{s}} L_A^2}{\nu^2} + \frac{72 p_{\textnormal{s}} \omega L_{\max}^2}{n p} 
  + \frac{6 p_{\textnormal{s}} L_A^2}{\mu^2}\right) \\
&= p_{\textnormal{s}} \gamma \left(\frac{336 \mu \omega L_{\max}^2}{n p} + \frac{288 p_{\textnormal{s}} \omega L_{\max}^2}{n p} \right)
  + p_{\textnormal{s}} \gamma \left(\frac{24 L_A^2}{\mu} + \frac{108 p_{\textnormal{s}} L_A^2}{\nu^2}
  + \frac{24 p_{\textnormal{s}} L_A^2}{\mu^2} + \frac{126 \mu L_A^2}{\nu^2}
  + \frac{3L_A^2}{p_{\textnormal{s}}}\right) \\
&\leq \frac{336 p_{\textnormal{s}} \gamma \omega L_{\max}^2}{n p}\left(\mu + p_{\textnormal{s}}\right)
  + p_{\textnormal{s}} \gamma \left(\frac{24 L_A^2}{\mu} + \frac{126 p_{\textnormal{s}} L_A^2}{\nu^2}
  + \frac{24 p_{\textnormal{s}} L_A^2}{\mu^2} + \frac{126 \mu L_A^2}{\nu^2}
  + \frac{3L_A^2}{p_{\textnormal{s}}}\right) \\
&= \frac{336 p_{\textnormal{s}} \gamma \omega L_{\max}^2}{n p}\left(\mu + p_{\textnormal{s}}\right)
  + p_{\textnormal{s}} \gamma \left(\frac{126 L_A^2}{\nu^2} (\mu + p_{\textnormal{s}}) 
  + 24 L_A^2\left(\frac{1}{p_{\textnormal{s}}} + \frac{1}{\mu} + \frac{p_{\textnormal{s}}}{\mu^2}\right)\right) \\
&\leq p_{\textnormal{s}} \gamma \left(\frac{336 \omega L_{\max}^2}{n p}\left(\mu + p_{\textnormal{s}}\right)
  + \frac{126 L_A^2}{\nu^2} (\mu + p_{\textnormal{s}}) 
  + 48 L_A^2\left(\frac{1}{p_{\textnormal{s}}} + \frac{p_{\textnormal{s}}}{\mu^2}\right)\right). 
\end{align*}

Here, the lower-order terms are absorbed using the bounds $\mu \leq 1$ and $\nu \leq 1$.
To obtain the last inequality, we use 
$\frac{1}{\mu}  \leq \left(\frac{1}{\sqrt{p_{\textnormal{s}}}}\right)^2
  + \left(\frac{\sqrt{p_{\textnormal{s}}}}{\mu}\right)^2 
  = \frac{1}{p_{\textnormal{s}}} + \frac{p_{\textnormal{s}}}{\mu^2}$. Our choice of $\lambda_E$ in \eqref{eq:lambda_choice_E}
ensures $\hat{\lambda}_E \leq \lambda_E$.

Since we have shown $\hat{\lambda}_A \leq \lambda_A, \dots, \hat{\lambda}_G \leq \lambda_G$,
inequality \eqref{eq:kslamckosdx} transforms into
\begin{align}\label{eq:kasmcmisdbnd}
\ExpSub{k}{\Psi^{k + 1}} + \frac{\gamma}{2} \norm{\nabla f(x^k)}^2 \leq \ExpSub{k}{\Psi^{k}} 
  - \underbrace{\left(\frac{1}{2 \gamma} - \frac{L}{2} - \tilde{\lambda}_X\right)}_{\hat{\lambda}_X} X^k + \tilde{c}.
\end{align}

Next, we require the noise term to be less than $\frac{\gamma \varepsilon}{4}$:
\begin{align}\label{eq:varience_leq_epsilon_h}
\tilde{c} &= \frac{\omega \nu^2 \sigma^2}{n b} \lambda_A + \frac{\nu^2 \sigma^2}{n b} \lambda_B 
  + \nu^2 \frac{\sigma^2}{b} \lambda_C \\
&= \frac{\omega \nu^2 \sigma^2}{n b} \frac{3 \gamma}{2 p} + \frac{\nu^2 \sigma^2}{n b} \frac{3 \gamma}{2 \nu}
  + \nu^2 \frac{\sigma^2}{b} \frac{3 \nu \omega \gamma}{n p} \nonumber\\
&= \gamma \left(\frac{3 \omega \nu^2 \sigma^2}{2 n p b} + \frac{3 \nu \sigma^2}{2 n b}
  + \frac{3 \nu^3 \omega \sigma^2}{n p b}\right) \nonumber\\
&\leq \frac{\gamma}{2} \left(\frac{9 \nu^2 \omega \sigma^2}{n p b} + \frac{3 \nu \sigma^2}{n b}\right)
  \leq \frac{\gamma \varepsilon}{4}. \nonumber
\end{align}

The last inequality holds due to our choice of the momentum parameter $\nu$. 
Finally, to establish convergence, we choose $\gamma$ such that $\hat{\lambda}_X \geq 0$.

\begin{align}\label{eq:convergence_condition}
\hat{\lambda}_X  = \frac{1}{2 \gamma} &- \frac{L}{2} 
  - \frac{2 \nu^2 \mu^2  (\omega_\textnormal{s} + 3) \omega \hat{L}^2}{n} \lambda_A
  - \frac{3}{\nu}\left(\mu^2 (\omega_\textnormal{s} + 3) L_A^2 + \mu^2 (\frac{\omega_{\textnormal{s}}}{n} + 3)L_B^2\right) \lambda_B \nonumber \\
&- \frac{3 \hat{L}^2 \mu^2 (\omega_\textnormal{s} + 3)}{\nu} \lambda_C
  - 4 \left(\frac{1}{\mu} + \omega_{\textnormal{s}}\right) \lambda_D
  - 4 \left(\frac{1}{\mu} + \frac{\omega_{\textnormal{s}}}{n}\right) \lambda_F \\
&- \omega_{\textnormal{s}} \lambda_E
  - \frac{\omega_{\textnormal{s}}}{n} \lambda_G \geq 0 \nonumber.
\end{align}

Substituting the values of $\lambda_A, \cdots, \lambda_G$ from \eqref{eq:lambda_choice_A}--\eqref{eq:lambda_choice_G}
into \eqref{eq:convergence_condition}, we obtain
$\hat{\lambda}_X = \frac{1}{2 \gamma} - \frac{L}{2} - \tilde{A} \gamma$,
where $\tilde{A}$ is the aggregate coefficient defined as:
\begin{align*}
\tilde{A} = C_{\hat{L}} \hat{L}^2 + C_{\max} L_{\max}^2 + C_{A} L_A^2 + C_{B} L_B^2.
\end{align*}

By grouping terms, we derive each coefficient. We start with $C_{\hat{L}}:$
\begin{align}\label{eq:C_hat_relax}
C_{\hat{L}} &= \frac{1}{\gamma} \left(\frac{2 \nu^2 \mu^2  (\omega_\textnormal{s} + 3) \omega}{n} \lambda_A
  + \frac{3 \mu^2 (\omega_\textnormal{s} + 3)}{\nu} \lambda_C\right) \nonumber \\
&= \frac{2 \nu^2 \mu^2  (\omega_\textnormal{s} + 3) \omega}{n} \frac{3}{2 p}
  + \frac{3 \mu^2 (\omega_\textnormal{s} + 3)}{\nu} \frac{3 \nu \omega}{n p} \nonumber \\
&= \frac{3 \mu^2 \omega (\omega_\textnormal{s} + 3)}{n p} (\nu^2 + 3)
\leq \frac{36 \mu^2 \omega (\omega_\textnormal{s} + 1)}{n p}. \nonumber \\
\end{align}

Next, we derive the coefficient for $L_A$, which appears not only in \eqref{eq:convergence_condition} but also in the expressions
for $\lambda_D$ \eqref{eq:lambda_choice_D} and $\lambda_E$ \eqref{eq:lambda_choice_E}. 
\begin{align}\label{eq:C_A_relax}
C_{A} &= \frac{9 \mu^2 (\omega_\textnormal{s} + 3)}{2 \nu^2} 
+ 4 \left(\frac{1}{\mu} + \omega_{\textnormal{s}}\right) \left(
    \frac{27 \mu}{\nu^2} 
    + \frac{6}{\mu}\right)
+ \omega_{\textnormal{s}} \left(
  \frac{128}{\nu^2} (p_{\textnormal{s}} + \mu) + 48 \left(\frac{p_{\textnormal{s}}}{\mu^2} 
  + \frac{1}{p_{\textnormal{s}}}\right)\right) \\
&=  \omega_{\textnormal{s}} \left(\frac{9 \mu^2}{2 \nu^2} 
  + 4 \left(\frac{27 \mu}{\nu^2} + \frac{6}{\mu}\right)
  + \frac{128}{\nu^2} (p_{\textnormal{s}} + \mu) + 48 \left(\frac{p_{\textnormal{s}}}{\mu^2} 
  + \frac{1}{p_{\textnormal{s}}}\right)\right) \nonumber \\
&\quad + \frac{27 \mu^2}{2 \nu^2} 
+ 4 \frac{1}{\mu} \left(\frac{27 \mu}{\nu^2} + \frac{6}{\mu}\right) \nonumber\\
&=  \omega_{\textnormal{s}} \left(\frac{9 \mu^2}{2 \nu^2} 
  + \frac{108 \mu}{\nu^2} + \frac{24}{\mu}
  + \frac{128}{\nu^2} (p_{\textnormal{s}} + \mu) + 48 \left(\frac{p_{\textnormal{s}}}{\mu^2} 
  + \frac{1}{p_{\textnormal{s}}}\right)\right)
  + \frac{27 \mu^2}{2 \nu^2} + \frac{108}{\nu^2} + \frac{24}{\mu^2} \nonumber\\
&\leq  \omega_{\textnormal{s}} \left(\frac{241 \mu}{\nu^2} + \frac{24}{\mu}
  + \frac{128}{\nu^2} p_{\textnormal{s}} + 48 \left(\frac{p_{\textnormal{s}}}{\mu^2} 
  + \frac{1}{p_{\textnormal{s}}}\right)\right)
  + \frac{122}{\nu^2} + \frac{24}{\mu^2} \nonumber\\
&\leq  \omega_{\textnormal{s}} \left(\frac{241}{\nu^2} (\mu + p_{\textnormal{s}})
  + 48 \left(\frac{p_{\textnormal{s}}}{\mu^2} + \frac{1}{\mu}
  + \frac{1}{p_{\textnormal{s}}}\right)\right)
  + \frac{122}{\nu^2} + \frac{24}{\mu^2} \nonumber\\
&\leq  \omega_{\textnormal{s}} \left(\frac{241}{\nu^2} (\mu + p_{\textnormal{s}})
  + 96 \left(\frac{p_{\textnormal{s}}}{\mu^2} + \frac{1}{p_{\textnormal{s}}}\right)\right)
  + \frac{122}{\nu^2} + \frac{24}{\mu^2} \nonumber\\
&\leq  244 \omega_{\textnormal{s}} \left(\frac{1}{\nu^2} (\mu + p_{\textnormal{s}})
  + \frac{p_{\textnormal{s}}}{\mu^2} + \frac{1}{p_{\textnormal{s}}}\right)
  + 122 \left(\frac{1}{\nu^2} + \frac{1}{\mu^2}\right). \nonumber
\end{align}
Next, we derive the coefficient for $L_B$, which appears in \eqref{eq:convergence_condition}
and in the expressions for $\lambda_F$ \eqref{eq:lambda_choice_F} and $\lambda_G$ \eqref{eq:lambda_choice_G}:
\begin{align}\label{eq:C_B_relax}
C_{B} &= \frac{9 \mu^2 (\frac{\omega_{\textnormal{s}}}{n} + 3)}{2 \nu^2} 
      + 4 \left(\frac{1}{\mu} + \frac{\omega_{\textnormal{s}}}{n}\right) \left(\frac{27 \mu}{\nu^2} + \frac{6}{\mu}\right) 
      + \frac{\omega_{\textnormal{s}}}{n} \left(
  \frac{128}{\nu^2} (p_{\textnormal{s}} + \mu) + 48 \left(\frac{p_{\textnormal{s}}}{\mu^2} 
  + \frac{1}{p_{\textnormal{s}}}\right)\right) \\
&\leq  244 \frac{\omega_{\textnormal{s}}}{n} \left(\frac{1}{\nu^2} (\mu + p_{\textnormal{s}})
  + \frac{p_{\textnormal{s}}}{\mu^2} + \frac{1}{p_{\textnormal{s}}}\right)
  + 122 \left(\frac{1}{\nu^2} + \frac{1}{\mu^2}\right). \nonumber
\end{align}

We obtain the last inequality similarly to $C_A$. It remains to determine $C_{\max}$. Since $L_{\max}^2$ appears in $\lambda_D$ \eqref{eq:lambda_choice_D} 
and $\lambda_E$ \eqref{eq:lambda_choice_E}, substituting them into \eqref{eq:convergence_condition} yields
\begin{align}\label{eq:C_max_relax}
C_{\max} &= 4 \left(\frac{1}{\mu} + \omega_{\textnormal{s}}\right) \frac{72 \mu \omega}{n p} 
   + \omega_{\textnormal{s}} \frac{336 \omega}{n p}\left(\mu + p_{\textnormal{s}}\right) \\
&= \frac{\omega}{n p} \left( 288 \left(1 + \mu \omega_{\textnormal{s}}\right) 
  + 336 \left(\mu + p_{\textnormal{s}}\right) \omega_{\textnormal{s}}\right) \nonumber \\
&\leq \frac{336 \omega}{n p} \left(1 + \mu \omega_{\textnormal{s}}
  +  \mu \omega_{\textnormal{s}} + p_{\textnormal{s}} \omega_{\textnormal{s}}\right)
= \frac{672 \omega}{n p} \left(1 + (\mu + p_{\textnormal{s}})\omega_{\textnormal{s}}\right). \nonumber
\end{align}

Lemma~\ref{lemma:appropriate_gamma} implies that the condition $\hat{\lambda}_X \geq 0$ holds for any $\gamma$ less than
\begin{align}\label{eq:heterogeneus_gaamma}
\tilde{\gamma} \eqdef \left(L + \sqrt{2 \left(C_{\hat{L}} \hat{L}^2 + C_A L_A^2 + C_B L_B^2 +
  C_{\max} L_{\max}^2\right)}\right)^{-1}.
\end{align}

Combining \eqref{eq:C_hat_relax}, \eqref{eq:C_A_relax}, \eqref{eq:C_B_relax} and \eqref{eq:C_max_relax}, 
we obtain a lower bound on $\tilde{\gamma}$:

\begin{align}\label{eq:aksdmskoisdf}
&C_A L_A^2 + C_B L_B^2 + C_{\max} L_{\max}^2 + C_{\hat{L}} \hat{L}^2 \\
&\leq 244 \omega_{\textnormal{s}} \left(\frac{1}{\nu^2} (p_{\textnormal{s}} + \mu) + \left(\frac{p_{\textnormal{s}}}{\mu^2} 
  + \frac{1}{p_{\textnormal{s}}}\right)\right) L_A^2
  + 244 \frac{\omega_{\textnormal{s}}}{n} \left(\frac{1}{\nu^2} (p_{\textnormal{s}} + \mu) + \left(\frac{p_{\textnormal{s}}}{\mu^2} 
  + \frac{1}{p_{\textnormal{s}}}\right)\right) L_B^2 \nonumber \\
&\quad + 122 \left(\frac{1}{\nu^2} + \frac{1}{\mu^2} \right) \left(L_B^2 + L_A^2\right)
  + 672 \frac{\omega}{n p} \left(1 + (\mu + p_{\textnormal{s}})\omega_{\textnormal{s}}\right)L_{\max}^2
  + 36 \frac{\mu^2 \omega (\omega_\textnormal{s} + 1)}{n p} \hat{L}^2. \nonumber
\end{align}

Consider the last two terms separately. Using $\mu \leq 1$ and $\hat{L}^2 \leq L_{\max}^2$, we get
\begin{align}\label{eq:mskdsmiwif}
C_{\max} L_{\max}^2 + C_{\hat{L}} \hat{L}^2
&\leq \left(672 \frac{\omega}{n p} \left(1 + (\mu + p_{\textnormal{s}})\omega_{\textnormal{s}}\right)
  + 36 \frac{\mu^2 \omega (\omega_\textnormal{s} + 1)}{n p}\right) L_{\max}^2 \\
&\leq \left(672 \frac{\omega (1 + p_{\textnormal{s}}\omega_{\textnormal{s}})}{n p} + 672 \frac{\mu \omega \omega_{\textnormal{s}}}{n p}
  + 36 \frac{\mu \omega \omega_\textnormal{s}}{n p}
  + 36 \frac{\omega}{n p}\right) L_{\max}^2 \\
&\leq 708 \left(\frac{\mu \omega \omega_\textnormal{s}}{n p} + \frac{\omega (1 + \omega_s p_{\textnormal{s}})}{n p}\right) L_{\max}^2.
\end{align}

Thus, substituting \eqref{eq:mskdsmiwif} into \eqref{eq:aksdmskoisdf} yields
\begin{align}\label{eq:slkmcmksinfks}
&C_{\hat{L}} \hat{L}^2 + C_A L_A^2 + C_B L_B^2 + C_{\max} L_{\max}^2 \\
&\leq \tilde{C}_A L_A^2 + \tilde{C}_B L_B^2 + \tilde{C}_{\max} L_{\max}^2 \nonumber \\
&= 708 \underbrace{\left(\frac{\omega_{\textnormal{s}} (p_{\textnormal{s}} + \mu) + 1}{\nu^2}  + \frac{\omega_{\textnormal{s}} p_{\textnormal{s}} + 1}{\mu^2} 
  + \frac{\omega_{\textnormal{s}}}{p_{\textnormal{s}}}\right)}_{\tilde{C}_A} L_A^2
  + 708 \underbrace{\left(\frac{\frac{\omega_{\textnormal{s}}}{n} (p_{\textnormal{s}} + \mu) + 1}{\nu^2}  + \frac{\frac{\omega_{\textnormal{s}}}{n} p_{\textnormal{s}} + 1}{\mu^2} 
  + \frac{\omega_{\textnormal{s}}}{n p_{\textnormal{s}}}\right)}_{\tilde{C}_B } L_B^2 \nonumber \\
&\quad + 708 \underbrace{\left(\frac{\mu \omega \omega_\textnormal{s}}{n p} 
  + \frac{\omega (1 + \omega_s p_{\textnormal{s}})}{n p}\right)}_{\tilde{C}_{\max}} L_{\max}^2. \nonumber
\end{align}

Thus, combining \eqref{eq:heterogeneus_gaamma} and \eqref{eq:slkmcmksinfks}, we can take 
$\gamma$ such that 
\begin{align*}
\gamma = \left(L + \sqrt{1416 \left(\tilde{C}_A L_A^2 + \tilde{C}_B L_B^2 +
  \tilde{C}_{\max} L_{\max}^2\right)}\right)^{-1}.
\end{align*}

Note that $\gamma \leq \tilde{\gamma}$, implying $\hat{\lambda}_X \geq 0$ holds. Recall that our choice of $\nu$
ensures $\tilde{c} \leq \frac{\gamma \varepsilon}{4}$ in \eqref{eq:varience_leq_epsilon_h}.
Combining these results, inequality \eqref{eq:kasmcmisdbnd} transforms into 
\begin{align*}
\ExpSub{k}{\Psi^{k + 1}} + \frac{\gamma}{2} \norm{\nabla f(x^k)}^2 \leq \ExpSub{k}{\Psi^{k}} 
 + \frac{\gamma \varepsilon}{4}.
\end{align*}

Averaging over $k = 0, \dots, K - 1$ and taking the full expectation, we obtain

\begin{align*}
\frac{\Exp{\Psi^{K}}}{K} + \frac{\gamma}{2} \frac{1}{K} \sum\limits_{k = 0}^{K - 1} \Exp{\norm{\nabla f(x^k)}^2} 
  \leq \frac{\Exp{\Psi^{0}}}{K} + \frac{\gamma \varepsilon}{4}.
\end{align*}
Dropping the non-negative term $\Exp{\Psi^{K}}$, dividing both sides of the inequality by $\frac{\gamma}{2}$, we derive
\begin{align*}
\frac{1}{K} \sum\limits_{k = 0}^{K - 1} \Exp{\norm{\nabla f(x^k)}^2} \leq \frac{2 \Exp{\Psi^{0}}}{\gamma K}
  + \frac{\varepsilon}{2}.
\end{align*}
\end{proof}

\begin{theorem}[Equal momentum coefficients]\label{thm:equal_momentum}
Let the assumptions of Theorem~\ref{thm:main_hetero} be satisfied. Assume that $\nu = \mu = \eta \eqdef \min\left\{
    \frac{1}{6} \sqrt{\frac{b n \varepsilon}{\omega (\omega + 1) \sigma^2}}, 
    \frac{b n \varepsilon}{6 \sigma^2},
    \left(\frac{n}{\omega (\omega + 1) \omega_{\textnormal{s}}}\right)^\frac{1}{3}, 1\right\}$, $p_{\textnormal{s}} = \frac{1}{\omega_{\textnormal{s}} + 1}, p = \frac{1}{\omega + 1}$.
Then, for any 
\begin{align}\label{eq:gamma_equal_momentum}
\gamma \leq \left(6 \sqrt{c \left(\omega_{\textnormal{s}} (\omega_{\textnormal{s}} + 1) L_A^2
  + \frac{\omega_{\textnormal{s}}}{n} (\omega_\textnormal{s} + 1) L_B^2
  + \left(\frac{\omega (\omega + 1)}{n} + \frac{1}{\eta^2}\right) L_{\max}^2\right)}\right)^{-1},
\end{align}
where $c$ is the multiplicative factor defined in Theorem~\ref{thm:main_hetero}, the following holds:
\begin{align*}
\frac{1}{K} \sum\limits_{k=0}^{K - 1} \Exp{\norm{\nabla f(x^k)}^2}  \leq \frac{2 \Exp{{\Psi}^0}}{\gamma K}
  + \frac{\varepsilon}{2}.
\end{align*}
\end{theorem}

\begin{proof}
Analogously to Lemma~\ref{lemma:approximate_solution}, for any $\nu \leq \min\left\{
    \frac{1}{6} \sqrt{\frac{b n \varepsilon}{\omega (\omega + 1) \sigma^2}}, 
    \frac{b n \varepsilon}{6 \sigma^2}\right\}$,
we have
\begin{align*}
\frac{36 \nu^2 \omega (\omega + 1) \sigma^2}{n \varepsilon b} + \frac{6 \nu \sigma^2}{n \varepsilon b} \leq 1
\iff \frac{18 \nu^2 \omega \sigma^2}{n p b} + \frac{3 \nu \sigma^2}{n b} \leq \frac{\varepsilon}{2}.
\end{align*}
Therefore, the condition of Theorem~\ref{thm:main_hetero} is satisfied.

Substituting $\mu = \nu = \eta$ into the expression for the step size~\eqref{eq:general_gamma_hetero} from Theorem~\ref{thm:main_hetero} yields
\begin{align*}
&\tilde{C}_A L_A^2 + \tilde{C}_B L_B^2 + \tilde{C}_{\max} L_{\max}^2 \\
&= \left(\frac{\omega_{\textnormal{s}} (p_{\textnormal{s}} + \eta) + 1}{\eta^2}  + \frac{\omega_{\textnormal{s}} p_{\textnormal{s}} + 1}{\eta^2} 
  + \frac{\omega_{\textnormal{s}}}{p_{\textnormal{s}}}\right) L_A^2
  + \left(\frac{\frac{\omega_{\textnormal{s}}}{n} (p_{\textnormal{s}} + \eta) + 1}{\eta^2}  + \frac{\frac{\omega_{\textnormal{s}}}{n} p_{\textnormal{s}} + 1}{\eta^2} 
  + \frac{\omega_{\textnormal{s}}}{n p_{\textnormal{s}}}\right) L_B^2 \\
&\quad + \left(\frac{\eta \omega \omega_\textnormal{s}}{n p} + \frac{\omega (1 + \omega_s p_{\textnormal{s}})}{n p}\right) L_{\max}^2\\
&\leq \left(\omega_{\textnormal{s}} \left(\frac{2 p_{\textnormal{s}}}{\eta^2} + \frac{1}{\eta} + \frac{1}{p_{\textnormal{s}}}\right) + \frac{2}{\eta^2}\right) L_A^2
  + \left(\frac{\omega_{\textnormal{s}}}{n} \left(\frac{2 p_{\textnormal{s}}}{\eta^2} + \frac{1}{\eta} + \frac{1}{p_{\textnormal{s}}}\right) + \frac{2}{\eta^2}\right) L_B^2 \\
&\quad + 2 \left(\frac{\mu \omega \omega_\textnormal{s}}{n p} + \frac{\omega (1 + \omega_s p_{\textnormal{s}})}{n p}\right) L_{\max}^2.
\end{align*}

The inequality $ab \leq a^2 + b^2$ implies that $\frac{1}{\eta} \leq \frac{p_{\textnormal{s}}}{\eta^2} + \frac{1}{p_{\textnormal{s}}}$.
Consequently, 
\begin{align}\label{eq:mkiksmcimsd}
&\tilde{C}_A L_A^2 + \tilde{C}_B L_B^2 + \tilde{C}_{\max} L_{\max}^2 \\
&\leq \omega_{\textnormal{s}} \left(\frac{3 p_{\textnormal{s}}}{\eta^2} +  \frac{2}{p_{\textnormal{s}}} + \frac{2}{\eta^2}\right) L_A^2
  + \frac{\omega_{\textnormal{s}}}{n} \left(\frac{3 p_{\textnormal{s}}}{\eta^2} + \frac{2}{p_{\textnormal{s}}} + \frac{2}{\eta^2}\right) L_B^2 
  + 2 \left(\frac{\mu \omega \omega_\textnormal{s}}{n p} + \frac{\omega (1 + \omega_s p_{\textnormal{s}})}{n p}\right) L_{\max}^2 \nonumber \\
&\leq 3 \left(\left(\frac{\omega_{\textnormal{s}}}{p_{\textnormal{s}}} + \frac{1 + \omega_{\textnormal{s}} p_{\textnormal{s}}}{\eta^2}\right) L_A^2
  + \left(\frac{\omega_{\textnormal{s}}}{n p_{\textnormal{s}}} + \frac{1 + \frac{\omega_{\textnormal{s}}}{n} p_{\textnormal{s}}}{\eta^2}\right) L_B^2
  + \left(\frac{\mu \omega \omega_\textnormal{s}}{n p} + \frac{\omega (1 + \omega_s p_{\textnormal{s}})}{n p}\right) L_{\max}^2\right). \nonumber 
\end{align}

Substituting our choice of probabilities and \eqref{eq:mkiksmcimsd} into \eqref{eq:general_gamma_hetero}, we obtain
\begin{align*}
\gamma &\geq \left(L + \sqrt{3 c \left(\left(\omega_{\textnormal{s}} (\omega_{\textnormal{s}} + 1) + \frac{2}{\eta^2}\right) L_A^2
  + \left(\frac{\omega_{\textnormal{s}}}{n} (\omega_\textnormal{s} + 1) + \frac{1 + \frac{1}{n}}{\eta^2}\right) L_B^2
  + \left(\frac{\eta \omega (\omega + 1) \omega_\textnormal{s}}{n} + \frac{2 \omega (\omega + 1)}{n}\right) L_{\max}^2\right)}\right)^{-1}.
\end{align*}

Since $\frac{2}{\eta^2} L_A^2 + \frac{1 + \frac{1}{n}}{\eta^2} L_B^2 \leq \frac{2}{\eta^2} \left(L_A^2 + L_B^2\right) \leq  \frac{4}{\eta^2} L_{\max}^2$, we get
\begin{align*}
\gamma &\geq \left(L + \sqrt{3 c \left(\omega_{\textnormal{s}} (\omega_{\textnormal{s}} + 1) L_A^2
  + \frac{\omega_{\textnormal{s}}}{n} (\omega_\textnormal{s} + 1) L_B^2
  + \left(\frac{\eta \omega (\omega + 1) \omega_\textnormal{s}}{n} + \frac{2 \omega (\omega + 1)}{n} + \frac{4}{\eta^2}\right) L_{\max}^2\right)}\right)^{-1}.
\end{align*}

Since $L \leq \frac{L_{\max}}{\eta}$, applying the inequality $\sqrt{a} + \sqrt{b} \leq \sqrt{2 (a + b)}$, we obtain 
\begin{align*}
\gamma &\geq \left(\sqrt{6 c \left(\omega_{\textnormal{s}} (\omega_{\textnormal{s}} + 1) L_A^2
  + \frac{\omega_{\textnormal{s}}}{n} (\omega_\textnormal{s} + 1) L_B^2
  + \left(\frac{\eta \omega (\omega + 1) \omega_\textnormal{s}}{n} + \frac{2 \omega (\omega + 1)}{n} + \frac{5}{\eta^2}\right) L_{\max}^2\right)}\right)^{-1}.
\end{align*}

Using our choice of $\eta$, we ensure $\frac{\eta \omega (\omega + 1) \omega_\textnormal{s}}{n} \leq \frac{1}{\eta^2}$. Thus, 
\begin{align*}
\gamma &\geq \left(\sqrt{6 c \left(\omega_{\textnormal{s}} (\omega_{\textnormal{s}} + 1) L_A^2
  + \frac{\omega_{\textnormal{s}}}{n} (\omega_\textnormal{s} + 1) L_B^2
  + \left(\frac{2 \omega (\omega + 1)}{n} + \frac{6}{\eta^2}\right) L_{\max}^2\right)}\right)^{-1} \\
&\geq \left(6 \sqrt{c \left(\omega_{\textnormal{s}} (\omega_{\textnormal{s}} + 1) L_A^2
  + \frac{\omega_{\textnormal{s}}}{n} (\omega_\textnormal{s} + 1) L_B^2
  + \left(\frac{\omega (\omega + 1)}{n} + \frac{1}{\eta^2}\right) L_{\max}^2\right)}\right)^{-1}.
\end{align*}

\end{proof}

\THEOREMHETERITER*

\begin{proof}
Theorem~\ref{thm:equal_momentum} yields that 
$K = \cO\left(\frac{\Exp{\Psi^0}}{\gamma \varepsilon}\right) = \cO\left(\frac{\Delta}{\gamma \varepsilon}
  + \frac{\Exp{\Psi^0 - \Delta}}{\gamma \varepsilon}\right)$. Consider the second term separately.
Since $g^0 = v^0$ and $x^0 = w_i^0 = x_i^0$, we get
\begin{align*}
\Exp{\Psi^0 - \Delta} &= \delta^0 - \Delta + \lambda_A \sqnorm{g^0 - v^0} \\
&\quad + \lambda_B \sqnorm{v^0 - \frac{1}{n} \sum \limits_{i = 1}^n \nabla f_i(x_i^0)}
  + \lambda_C \frac{1}{n} \sum \limits_{i = 1}^n \sqnorm{v_i^{0} - \nabla f_i (x_i^{0})} \\
&\quad + \lambda_D \frac{1}{n} \sum\limits_{i=1}^n \norm{w_i^0 - x_i^0}^2
  + \lambda_E \frac{1}{n} \sum\limits_{i=1}^n \norm{w_i^0 - x^0}^2 \\
&\quad + \lambda_F \norm{\frac{1}{n} \sum\limits_{i=1}^n \left(w_i^0 - x_i^0\right)}^2 
      + \lambda_G \norm{\frac{1}{n} \sum\limits_{i=1}^n w_i^0 - x^0}^2 \\
&\leq \lambda_B \sqnorm{v^0 - \frac{1}{n} \sum \limits_{i = 1}^n \nabla f_i(x_i^0)}
  + \lambda_C \frac{1}{n} \sum \limits_{i = 1}^n \sqnorm{v_i^{0} - \nabla f_i (x_i^{0})}.
\end{align*}
Here we drop the term $\delta^0 - \Delta$, as it is non-positive under Assumption~\ref{ass:lower_bound}. Recall our choice of $\lambda_B$ in \eqref{eq:lambda_choice_B} and $\lambda_C$ in \eqref{eq:lambda_choice_C}.
\begin{align*}
\frac{\Exp{\Psi^0 - \Delta}}{\gamma \varepsilon}
&= \frac{1}{\gamma \varepsilon}\left(\frac{3 \gamma}{2 \nu} \Exp{\sqnorm{v^0 - \frac{1}{n} \sum \limits_{i = 1}^n \nabla f_i(x_i^0)}}
  + \frac{3 \nu \omega \gamma}{n p}  \frac{1}{n} \sum \limits_{i = 1}^n \Exp{\sqnorm{v_i^{0} - \nabla f_i (x_i^{0})}}\right) \\
&= \frac{3}{2 \nu \varepsilon} \Exp{\sqnorm{\frac{1}{n} \sum \limits_{i = 1}^n 
    \frac{1}{b_{\textnormal{init}}} \sum \limits_{b = 1}^{b_{\textnormal{init}}} \nabla f_i(x_i^0; \xi_{i, b}^0) 
      - \frac{1}{n} \sum \limits_{i = 1}^n \nabla f_i(x_i^0)}} \\
&\quad + \frac{3 \nu \omega}{n p \varepsilon}  \frac{1}{n} \sum \limits_{i = 1}^n \Exp{\sqnorm{\frac{1}{b_{\textnormal{init}}} 
  \sum \limits_{b = 1}^{b_{\textnormal{init}}} \nabla f_i(x_i^0; \xi_{i, b}^0) - \nabla f_i (x_i^{0})}} \\
&= \frac{3}{2 \nu \varepsilon} \frac{1}{n^2} \sum \limits_{i = 1}^n  \Exp{\sqnorm{
    \frac{1}{b_{\textnormal{init}}} \sum \limits_{b = 1}^{b_{\textnormal{init}}} \nabla f_i(x_i^0; \xi_{i, b}^0) 
      - \frac{1}{n} \sum \limits_{i = 1}^n \nabla f_i(x_i^0)}} \\
&\quad + \frac{3 \nu \omega}{n p \varepsilon}  \frac{1}{n} \sum \limits_{i = 1}^n \Exp{\sqnorm{\frac{1}{b_{\textnormal{init}}} 
  \sum \limits_{b = 1}^{b_{\textnormal{init}}} \nabla f_i(x_i^0; \xi_{i, b}^0) - \nabla f_i (x_i^{0})}}.
\end{align*}

Using Assumption~\ref{ass:stochastic_variance_bounded_heter}, we get
\begin{align}\label{eq:kdsmfosdhjihgwl}
\frac{\Exp{\Psi^0 - \delta^0}}{\gamma \varepsilon} \leq \frac{3 \sigma^2}{2 \nu n b_{\textnormal{init}} \varepsilon}
  + \frac{3 \nu \omega \sigma^2}{n p b_{\textnormal{init}}\varepsilon}
= \frac{1}{\nu b_{\textnormal{init}}} \left(\frac{3 \nu^2 \omega \sigma^2}{n p \varepsilon}
 + \frac{3 \sigma^2}{2 n \varepsilon}\right).
\end{align}

Substituting our choice of $\nu$ into \eqref{eq:kdsmfosdhjihgwl} yields
\begin{align}\label{eq:kwmdsolnmdsks}
\frac{\Exp{\Psi^0 - \delta^0}}{\gamma \varepsilon} \leq 
\frac{3}{\nu b_{\textnormal{init}}} \left(1 + \frac{\sigma^2}{n \varepsilon}\right).
\end{align}

Plugging in our choice of $b_{\textnormal{init}}$, we obtain
\begin{align*}
\frac{\Exp{\Psi^0 - \delta^0}}{\gamma \varepsilon} \leq 
\frac{3}{\nu \sqrt{\frac{b}{\nu} \left(1 + \frac{\sigma^2}{n \varepsilon}\right)}} \left(1 + \frac{\sigma^2}{n \varepsilon}\right)
= 3 \sqrt{\frac{1}{b \nu} \left(1 + \frac{\sigma^2}{n \varepsilon}\right)}.
\end{align*}

Therefore, 
$K = \cO\left(\frac{\Delta}{\gamma \varepsilon} + \sqrt{\frac{1}{b \nu} \left(1 + \frac{\sigma^2}{n \varepsilon}\right)}\right)$.
\end{proof}

\begin{remark}[Intuition behind $b_{\textnormal{init}}$]\label{remark:B_init_choice}
We select $b_{\textnormal{init}}$ such that the additional iterations incurred by small initial batch size do not outweigh the computational cost of using a large $b_{\textnormal{init}}$. This requirement can be formalized as
\begin{align*}
\underbrace{\frac{\Exp{\Psi^0 - \Delta}}{\gamma \varepsilon}}_{\textnormal{extra iterations}} \times b h
= \cO \left(b_{\textnormal{init}} h\right) 
\iff \frac{\Exp{\Psi^0 - \Delta}}{\gamma \varepsilon} \times b
= \cO \left(b_{\textnormal{init}}\right).
\end{align*}
From \eqref{eq:kwmdsolnmdsks}, we have $\frac{\Exp{\Psi^0 - \Delta}}{\gamma \varepsilon} = \cO\left(\frac{1}{\nu b_{\textnormal{init}}} \left(1 + \frac{\sigma^2}{n \varepsilon}\right)\right)$, 
hence we require
$b_{\textnormal{init}} \geq \frac{b}{\nu b_{\textnormal{init}}} \left(1 + \frac{\sigma^2}{n \varepsilon}\right)$. 
Consequently,
\begin{align}\label{eq:skmkokdmkdsl}
b_{\textnormal{init}} \geq \sqrt{\frac{b}{\nu} \left(1 + \frac{\sigma^2}{n \varepsilon}\right)}.
\end{align}
Thus, this choice of $b_{\textnormal{init}}$ ensures $\frac{\Exp{\Psi^0 - \Delta}}{\gamma \varepsilon} \times b
= \cO \left(b_{\textnormal{init}}\right)$.
\end{remark}

\subsection{Time complexity}

\THMHETERTIME*

\begin{proof}
Note that the time of each iteration is almost the same as in Section~\ref{sec:homo_time_complexity} and consists of sending $2\ell$ coordinates on average by the server,
sending $2 m$ coordinates on average by the workers and computing $b$ gradients.
Consequently, 
\begin{align*}
\Exp{T_\textnormal{time}} = \sum \limits_{k = 0}^{K - 1} \Exp{t_{\textnormal{iter}}^k} + b_{\textnormal{init}} h 
\leq \sum \limits_{k = 0}^{K - 1} 5 t + b_{\textnormal{init}} h = 5 t \times K + b_{\textnormal{init}} h.
\end{align*}
Since at least one gradient 
is computed and compressed by each worker, and the server sends a new point compressed by at least one compressor, $t \geq \max \left\{h, \tau, \kappa\right\}$.

Theorem~\ref{thm:heter} yields 
$K = \cO\left(\frac{\Delta}{\gamma \varepsilon} + \sqrt{\frac{1}{b \nu} \left(1 + \frac{\sigma^2}{n \varepsilon}\right)}\right)$. 
Therefore,
\begin{align*}
\Exp{T_\textnormal{time}} 
= \cO\left(\frac{t \Delta}{\gamma \varepsilon} + t \times \sqrt{\frac{1}{b \nu} \left(1 + \frac{\sigma^2}{n \varepsilon}\right)} +  b_{\textnormal{init}} h\right).
\end{align*}
Since $b = \flr{\frac{t}{h}} \geq \frac{t}{2 h}$ and $t \leq 2 b h$, Remark~\ref{remark:B_init_choice} yields  
$t \times \sqrt{\frac{1}{b \nu} \left(1 + \frac{\sigma^2}{n \varepsilon}\right)} =  \cO \left(b_{\textnormal{init}} h\right)$.
Thus, 
$\Exp{T_\textnormal{time}}  = \cO\left(\frac{t \Delta}{\gamma \varepsilon} +  b_{\textnormal{init}} h\right).$
Consider $b_{\textnormal{init}} h$ separately. Since 
$\frac{1}{\eta} = \Omega \left(1 + \frac{\sigma^2}{n \varepsilon b}\right) 
  = \frac{1}{b} \Omega \left(1 + \frac{\sigma^2}{n \varepsilon}\right)$
we get 
\begin{align*}
b_{\textnormal{init}} h = h \sqrt{\frac{b}{\eta} \left(1 + \frac{\sigma^2}{n \varepsilon}\right)} 
= \cO \left(\frac{b h}{\eta}\right) = \cO \left(\frac{t}{\eta}\right).
\end{align*}
Substituting $\gamma$ defined in \eqref{eq:gamma_equal_momentum} into the time complexity, we obtain

\begin{align*}
\Exp{T_\textnormal{time}}  
&= \cO \left(\frac{t \Delta}{\varepsilon} \left(\sqrt{\omega_{\textnormal{s}} (\omega_{\textnormal{s}} + 1) L_A^2
  + \frac{\omega_{\textnormal{s}}}{n} (\omega_\textnormal{s} + 1) L_B^2
  + \left(\frac{\omega (\omega + 1)}{n} + \frac{1}{\eta^2}\right) L_{\max}^2}\right) + \frac{t}{\eta}\right) \\
&= \cO \left(\frac{t \Delta}{\varepsilon} \left(\sqrt{\omega_{\textnormal{s}} (\omega_{\textnormal{s}} + 1)} L_A
  + \sqrt{\frac{\omega_{\textnormal{s}}}{n} (\omega_\textnormal{s} + 1)} L_B
+ \left(\sqrt{\frac{\omega (\omega + 1)}{n}}
  + \frac{1}{\eta}\right) L_{\max}\right) + \frac{t}{\eta}\right).
\end{align*}

Recall that the compressors' parameters are
$\omega = \frac{d}{m} - 1 \leq \frac{d}{m}$ and $\omega_\textnormal{s} = \frac{d}{\ell} - 1 \leq \frac{d}{\ell}$. 
Therefore, our choice of the $b, m$ and $\ell$ results in
\begin{align*}
\Exp{T_\textnormal{time}}
&=\cO \left(\frac{t \Delta}{\varepsilon} \left(\frac{d  L_A}{\ell} + \frac{d L_B}{\ell \sqrt{n}} + \left(\frac{d}{m \sqrt{n}}  + \frac{1}{\eta}\right) L_{\max}\right)
  + \frac{t}{\eta}\right) \nonumber \\
&= \cO \left(\frac{t \Delta}{\varepsilon} \left(\frac{d  L_A}{\ell} + \frac{d L_B}{\ell \sqrt{n}} + \frac{d L_{\max}}{m \sqrt{n}}\right)
  + \frac{t}{\eta} \left(\frac{\Delta L_{\max}}{\varepsilon} + 1\right)\right). \nonumber 
\end{align*}

By Lemma~\ref{lemma:bounded_initial_gradient}, we have $\sqnorm{\nabla f(x_0)} \leq 2 L \Delta$. 
Assuming that $x_0$ is not already an $\varepsilon$-stationary point (i.e., $\sqnorm{\nabla f(x_0)} \geq \varepsilon$) and noting that $L \leq L_{\max}$, we obtain $\varepsilon \leq 2 L_{\max} \Delta$, which implies $\frac{\Delta L_{\max}}{\varepsilon} \geq \frac{1}{2}$. 
Consequently, $1 \leq \frac{2 \Delta L_{\max}}{\varepsilon}$, and therefore
$\frac{\Delta L_{\max}}{\varepsilon} + 1 \leq \frac{3 \Delta L_{\max}}{\varepsilon}$.
Then, using the inequalities $b = \flr{\frac{t}{h}} \geq \frac{t}{2h}$, 
$m = \flr{\frac{t}{\tau}} \geq \frac{t}{2\tau}$ and $\ell = \flr{\frac{t}{\kappa}} \geq \frac{t}{2\kappa}$,
we get 

\begin{align}
\Exp{T_\textnormal{time}} 
&= \cO \left(\frac{\Delta}{\varepsilon} \left(d \kappa L_A + \frac{d \kappa L_B}{\sqrt{n}} 
  + \frac{d \tau L_{\max}}{\sqrt{n}}\right)
  + \frac{t}{\eta} \times \frac{\Delta L_{\max}}{\varepsilon}\right). \label{eq:oiwioqiohh}
\end{align}

Consider $\frac{t}{\eta}$ separately
\begin{align*}
\frac{t}{\eta} &= t \times  \cO \left(\sqrt{\frac{\omega (\omega + 1) \sigma^2}{n \varepsilon b}} +  \frac{\sigma^2}{n \varepsilon b}
  + \left(\frac{\omega (\omega + 1) \omega_{\textnormal{s}}}{n}\right)^\frac{1}{3} + 1\right) \\
&=t \times \cO \left(\frac{d}{m}\sqrt{\frac{\sigma^2}{n \varepsilon b}} +  \frac{\sigma^2}{n \varepsilon b}
  + \left(\frac{d^3}{m^2 \ell n}\right)^\frac{1}{3} + 1\right) \\
&= \cO \left(d \tau \sqrt{\frac{\sigma^2 h}{n \varepsilon t}} +  \frac{\sigma^2 h}{n \varepsilon}
  + \left(\frac{d^3 \tau^2 \kappa}{n}\right)^\frac{1}{3} + t\right).
\end{align*}

Next, we balance the terms $ d \tau \sqrt{\frac{\sigma^2 h}{n \varepsilon t}}$ and $t$:
\begin{align*}
d \tau\sqrt{\frac{h \sigma^2}{n \varepsilon t}} = t 
\iff  t^\frac{3}{2} = d \tau\sqrt{\frac{h \sigma^2}{n \varepsilon}}
\iff  t = \sqrt[3]{\frac{d^2 \tau^2 h \sigma^2}{n \varepsilon}}.
\end{align*}

Thus, choosing $t = \max\left\{h, \tau, \kappa, \sqrt[3]{\frac{d^2 \tau^2 h \sigma^2}{n \varepsilon}}\right\}$ we ensure 
\begin{align}\label{eq:iosdiewhogw}
\frac{t}{\eta} = \cO \left(\frac{\sigma^2 h}{n \varepsilon}
  + \left(\frac{d^3 \tau^2 \kappa}{n}\right)^\frac{1}{3} + \max\left\{h, \tau, \kappa, \sqrt[3]{\frac{d^2 \tau^2 h \sigma^2}{n \varepsilon}}\right\}\right).
\end{align}
Substituting \eqref{eq:iosdiewhogw} into \eqref{eq:oiwioqiohh} we derive

\begin{align*}
\Exp{T_\textnormal{time}} = \frac{\Delta}{\varepsilon} \times \cO
\left(d \kappa L_A + \frac{d \kappa L_B}{\sqrt{n}} + \left(\frac{d \tau}{\sqrt{n}} + \frac{\sigma^2 h}{n \varepsilon}
  + \left(\frac{d^3 \tau^2 \kappa}{n}\right)^\frac{1}{3} + \max\left\{h, \tau, \kappa, \sqrt[3]{\frac{d^2 \tau^2 h \sigma^2}{n \varepsilon}}\right\}\right)L_{\max}\right).
\end{align*}
\end{proof}

\clearpage
\newpage
\section{Numerical Estimation of $L_A$ from Assumption~\ref{ass:AB_assumption}}
\label{sec:emp_l_a}
\begin{figure}[h]
\centering
\includegraphics[width=\textwidth]{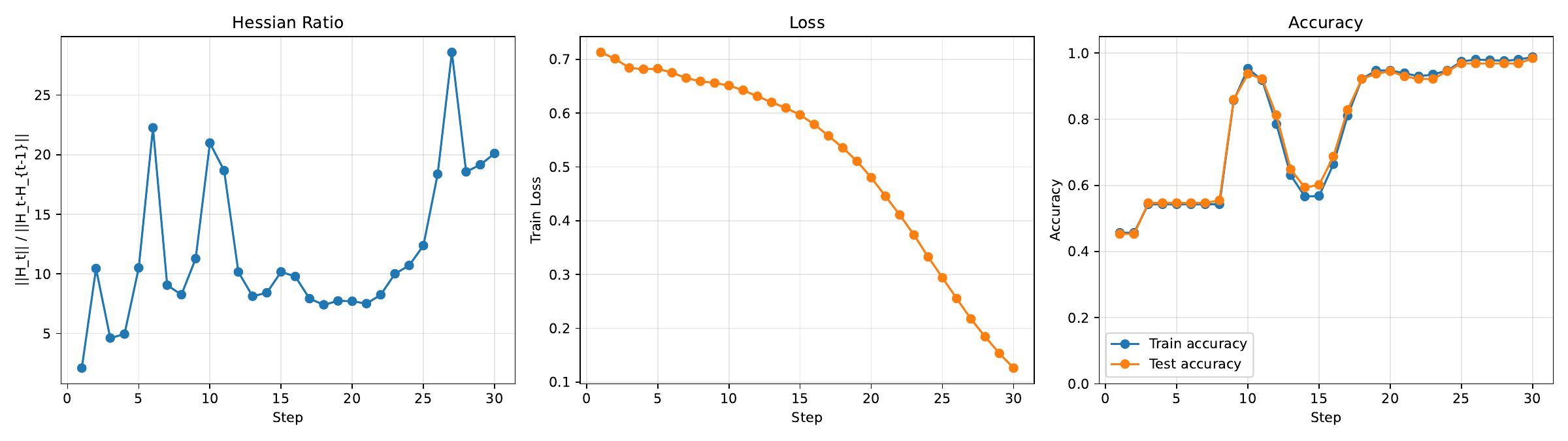}
\caption{Hessian ratio, loss, and accuracy during training for step size $= 0.01.$}
\label{fig:hessian_trajectory}
\end{figure}
\begin{figure}[h]
\centering
\includegraphics[width=\textwidth]{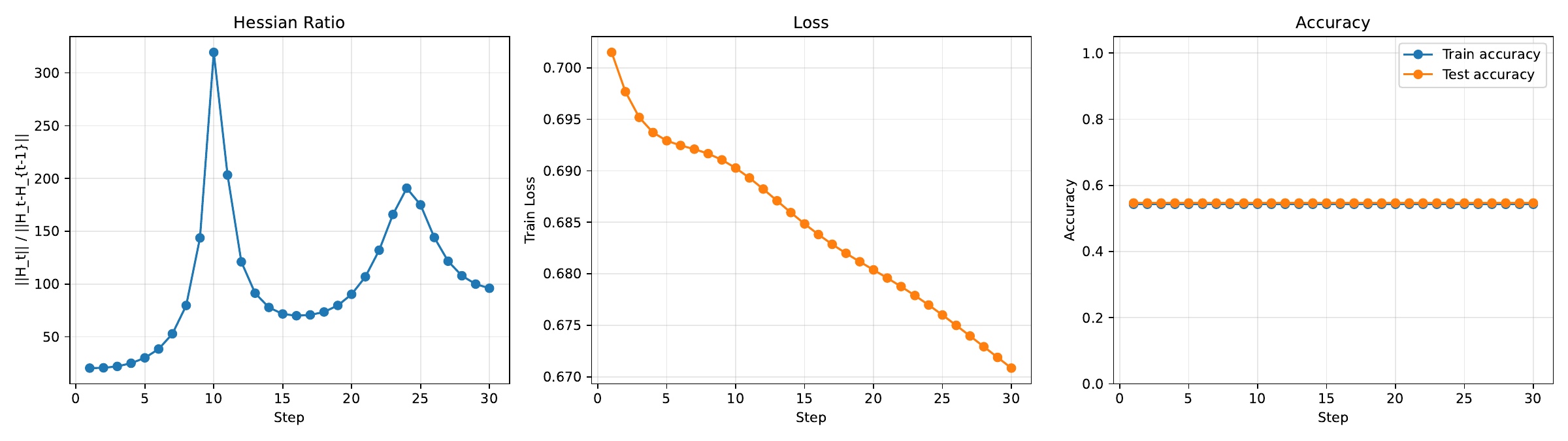}
\caption{Hessian ratio, loss, and accuracy during training for step size $= 0.001.$}
\label{fig:hessian_trajectory_lr_0_001}
\end{figure}

The parameter $L_A$ plays a crucial role in the time complexity of the methods. Due to Propositions~\ref{prop:quadratic_inequality} and \ref{theorem:a_b_hess}, we know that it is bounded and is small when the Hessian ``does not change too much''. In this section, we provide empirical evidence that $L_A$ is small. The code was prepared in Python~3 and executed on a machine with 52~CPUs (Intel(R) Xeon(R) Gold~6278C @~2.60GHz).

We consider a setup where we train a small convolutional neural network on a two-class classification problem and measure how much the Hessian changes along the optimization trajectory. We use the \emph{MNIST} dataset \citep{lecun2010mnist} restricted to two classes. In our experiments, we take the classes $0$ and $1$, resize all images to $8 \times 8$, use $512$ random training samples, and use $128$ random test samples for estimating the Hessian. The considered neural network consists of one convolutional layer with $4$ channels, the \emph{softplus} activation, one hidden linear layer of width $8$, another \emph{softplus} activation, and a final linear layer with two outputs. We use the cross-entropy loss and train the model with Adam \citep{kingma2014adam} for $30$ iterations.

For each pair of consecutive checkpoints $x^{k-1}$ and $x^k$, we compute the Hessian of the loss on the test set and report
\[
\frac{\|\nabla^2 f(x^k)\|}
{\|\nabla^2 f(x^k)-\nabla^2 f(x^{k-1})\|} \approx \frac{L}{L_A},
\]
which empirically measures the ratio between local $L$ and $L_A$ due to Proposition~\ref{theorem:a_b_hess} with $z_1 = x^k$, $z_2 = x^{k-1}$, and $n = 2.$ Large values of this ratio indicate that the Hessian changes slowly relative to its norm, which supports the regime where $L_A$ is small. In practice, we estimate $L_A$ locally. However, estimating $\norm{\nabla^2 f(z_1) - \frac{1}{n} \sum_{j=1}^n \nabla^2 f(z_j)}$ is infeasible, even locally, for all $z_1, \dots, z_n.$ Thus, in this section, we assume that the Hessians at the iterates $\{x^k\}_{k \geq 0}$ reflect the correct local geometry.

In Figure~\ref{fig:hessian_trajectory}, we plot the Hessian norm, loss, and accuracy during training for step size $= 0.01.$ We observe that the ratio equals $10 \sim 25.$ This provides empirical evidence that the Hessian does not vary significantly during training, and therefore the corresponding local value of $L_A$ can be much smaller than the worst-case smoothness constant $L,$ up to $25$ times.

We also repeat the same experiment for step sizes $= 0.001$ and $= 0.1.$ With step size $= 0.1,$ the method does not converge. In Figure~\ref{fig:hessian_trajectory_lr_0_001}, we repeat the experiment for step size $= 0.001,$ and observe that the ratio equals $50 \sim 300,$ meaning that $L_A$ is smaller than $L$ locally, which is a significant improvement. The reason for this improvement is that the step size is smaller, and therefore the distance between the iterates is smaller. This indicates that our new methods might be especially effective at the last phase of training, when the learning rate is small.

We also consider the experiment from Figure~\ref{fig:hessian_trajectory} and estimate the ratio using iterates $x^k$ and $x^{k-\delta}$ with $\delta \in \{2, 10\}$ to see how the Hessian distance changes between non-consecutive iterates. In Figures~\ref{fig:hessian_trajectory_lr_0_01_delta_2} and \ref{fig:hessian_trajectory_lr_0_01_delta_10}, we observe that the ratio decreases compared to Figure~\ref{fig:hessian_trajectory}. This is expected since the distance between the iterates is larger, and therefore the Hessian can change more.

\begin{figure}[h]
  \centering
  \includegraphics[width=\textwidth]{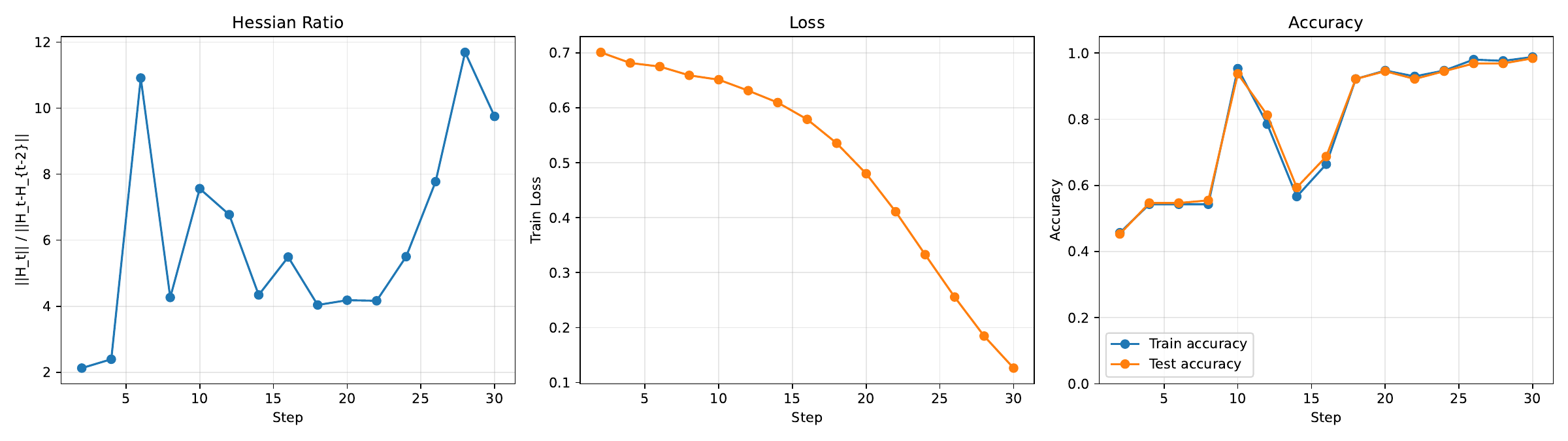}
  \caption{Hessian ratio, loss, and accuracy during training for step size $= 0.01$ and $\delta = 2$.}
  \label{fig:hessian_trajectory_lr_0_01_delta_2}
\end{figure}

\begin{figure}[h]
  \centering
  \includegraphics[width=\textwidth]{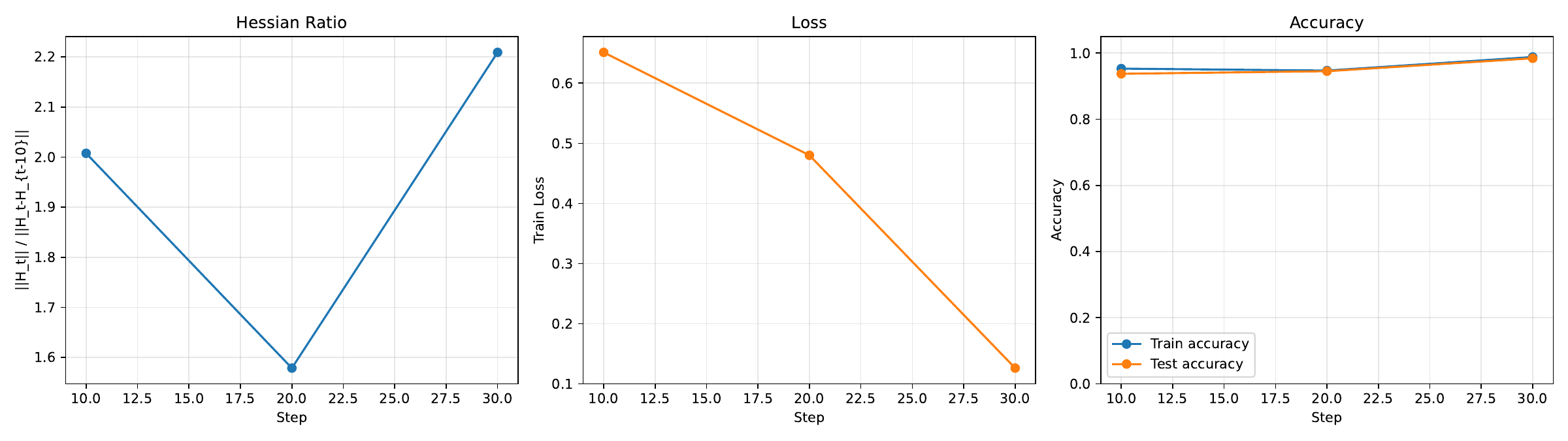}
  \caption{Hessian ratio, loss, and accuracy during training for step size $= 0.01$ and $\delta = 10$.}
  \label{fig:hessian_trajectory_lr_0_01_delta_10}
\end{figure}

\section{Experiments} \label{sec:experiments}
The experiments were prepared in Python. The distributed environment was emulated on a machine with 52 CPUs (Intel(R) Xeon(R) Gold 6278C @ 2.60GHz).

\subsection{Homogeneous Block-Regularized Quadratic}
We start our experiments by evaluating the algorithms on a synthetic homogeneous quadratic optimization problem with controlled conditioning. In this setting, we can control different parameters of the problem and make interpretable conclusions. The objective function is defined as
\begin{equation}\label{eq:homo_quad_A}
    f(x) = \frac{1}{2} x^\top \mA x, \quad 
    A = \begin{pmatrix}
        \mI_{d/2} & 0 \\
        0 & \lambda \mI_{d/2}
    \end{pmatrix},
\end{equation}
where $d = 300$ is the problem dimension and $\lambda = 0.01$ is a parameter that controls the condition number $\frac{1}{\lambda}$. This construction yields a problem with two blocks of coordinates having distinct curvature:
the first $d/2$ coordinates have unit Hessian eigenvalues, while the remaining $d/2$ coordinates 
are scaled by $\lambda$. The unique minimizer is $x^* = 0$, which allows for precise tracking 
of suboptimality $f(x^k) - f^*$.

Stochastic gradients are obtained by adding controllable Gaussian noise:
$\nabla f(x; \xi) = \nabla f(x) + \zeta$, where $\zeta \sim \mathcal{N}(0, \sigma^2)$.

We compare the proposed methods against Synchronous SGD, which is optimal in the centralized distributed optimization with the large noise regime \citep{tyurin2025proving}.
In this setup, each worker computes a single stochastic gradient per iteration. Both proposed methods employ Rand$K$ compression for server-worker and worker-server communication, where we tune $K$ from the set $\{1, 3, 10, 30, 50, 100, 200, 300\}$. For all three methods, we tune the step size $\gamma$ over the range $\{2^{-10}, 2^{-9}, \ldots, 2^{3}\}$ and report the best performance. In this experiment, we also tune the parameter $\eta$ from \ref{eq:mthree} from the set $\{0.1, 0.2, \dots, 1\}$.

We consider a different number of workers $n \in \{50, 100, 300\}$ under three noise levels: $\sigma = 0.001$ (Figure~\ref{fig:noise_low_homo_quad}), $\sigma = 0.01$ (Figure~\ref{fig:noise_high_homo_quad}), and $\sigma = 0.1$ (Figure~\ref{fig:noise_very_high_homo_quad}).
We observe that the performance of \ref{eq:inkheart} and \ref{eq:mthree} improves as the number of workers increases, which is expected due to the theory.
This is especially notable under low computational costs ($h = 0$ and $h = 0.1$), when communication costs dominate over computational costs, particularly under small noise ($\sigma = 0.001$). In the high noise level $\sigma = 0.1$ (Figure~\ref{fig:noise_very_high_homo_quad}), as expected, the gap between the methods decreases.

\begin{figure}[htp]
    \centering
    \captionsetup[subfigure]{labelformat=empty, font=scriptsize}
    \setlength{\tabcolsep}{3pt}
    
    \begin{subfigure}[b]{0.32\textwidth}
        \centering\includegraphics[width=\textwidth]{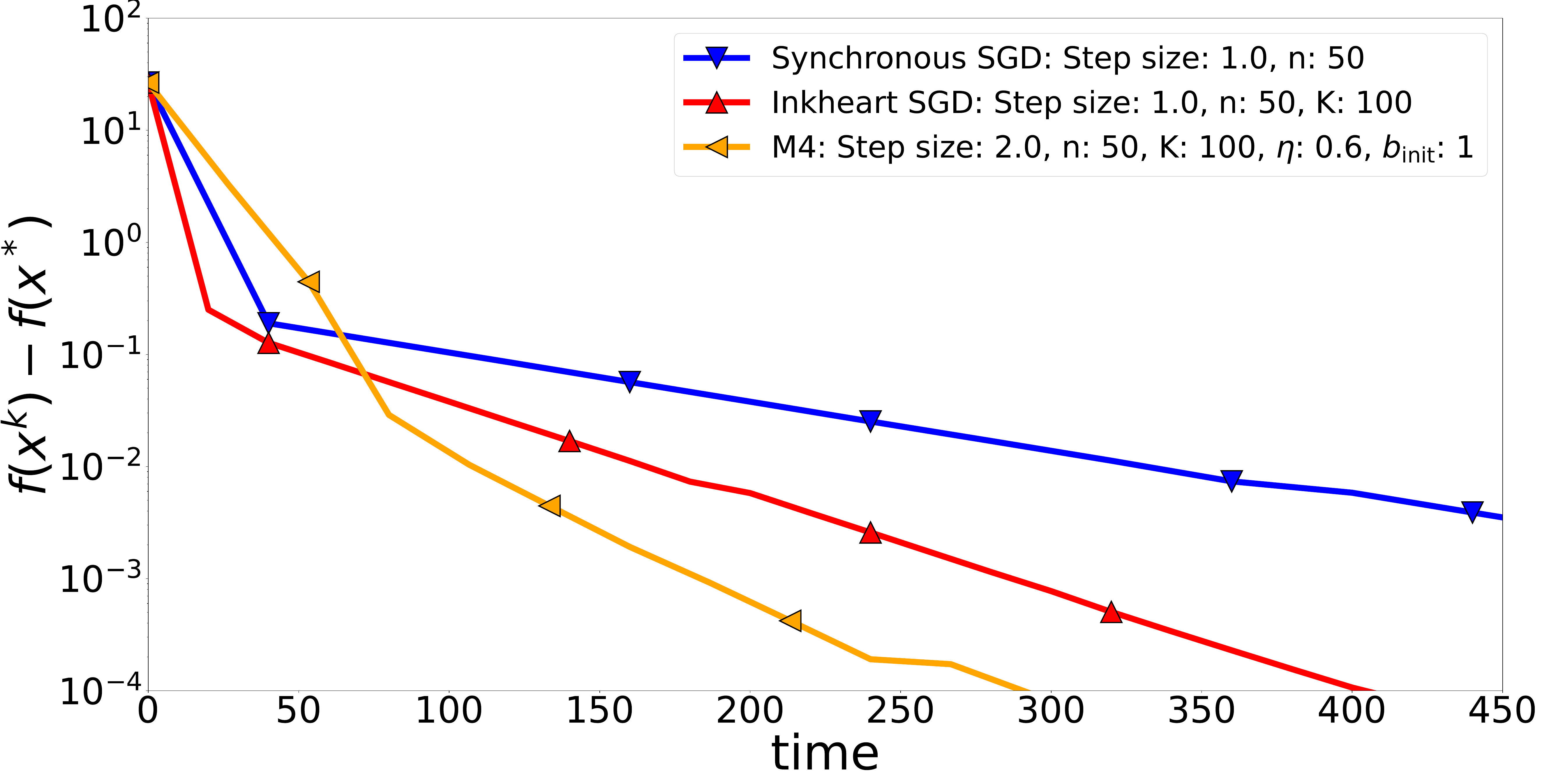}
        \caption{$h=0$}
    \end{subfigure}\hfill
    \begin{subfigure}[b]{0.32\textwidth}
        \centering\includegraphics[width=\textwidth]{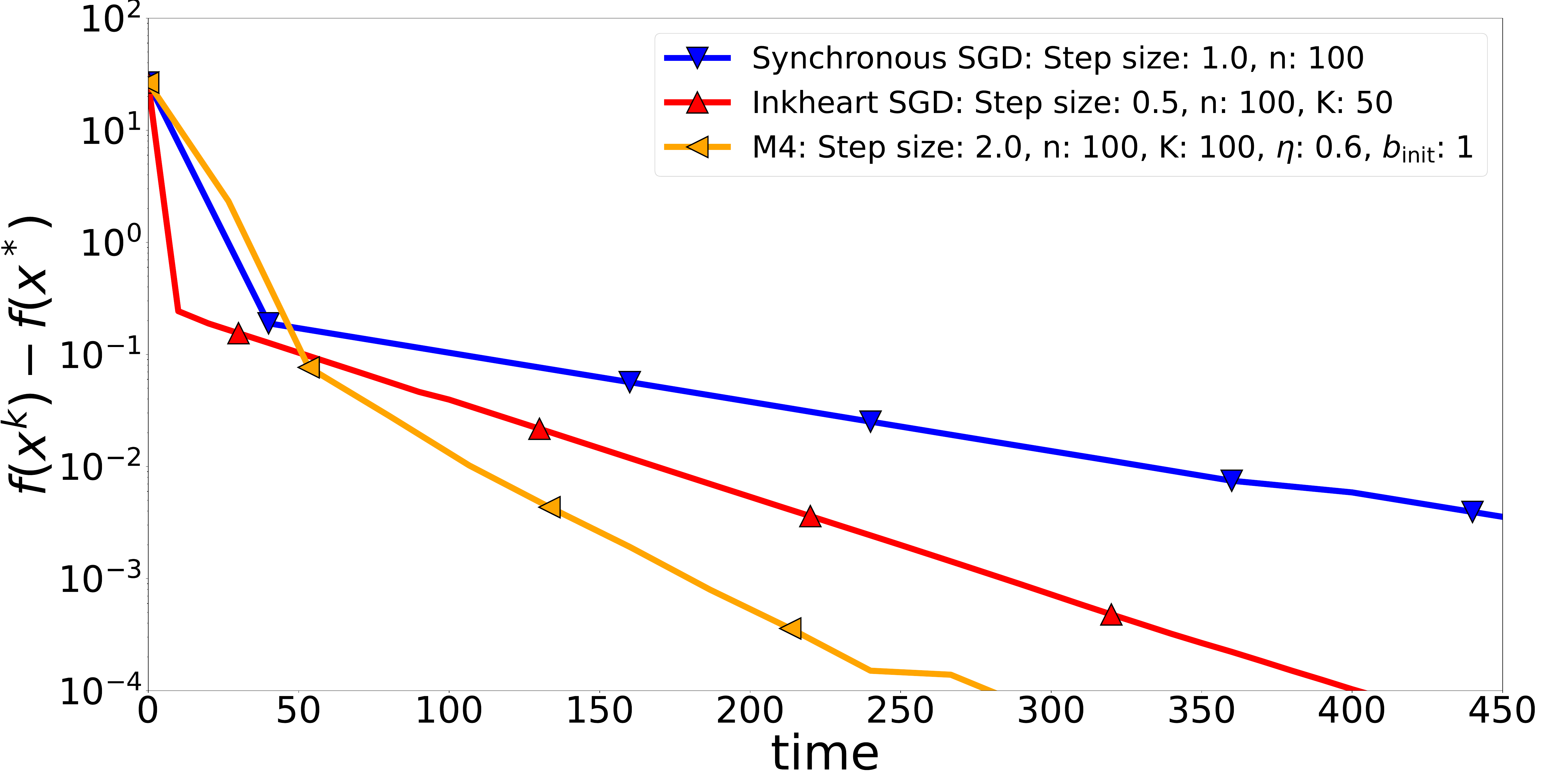}
        \caption{$h=0$}
    \end{subfigure}\hfill
    \begin{subfigure}[b]{0.32\textwidth}
        \centering\includegraphics[width=\textwidth]{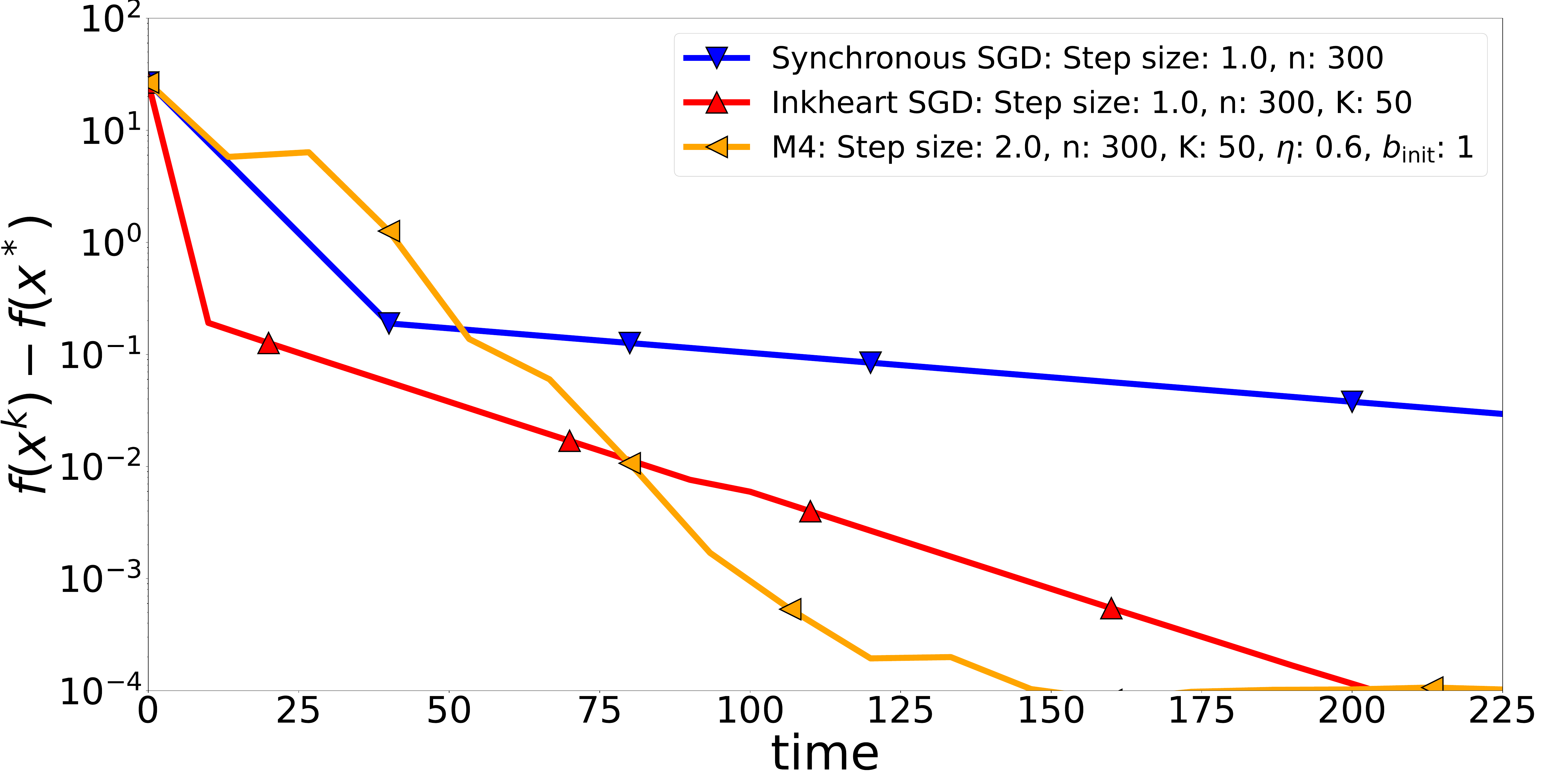}
        \caption{$h=0$}
    \end{subfigure}
    
    \vspace{0.25cm}
    
    \begin{subfigure}[b]{0.32\textwidth}
        \centering\includegraphics[width=\textwidth]{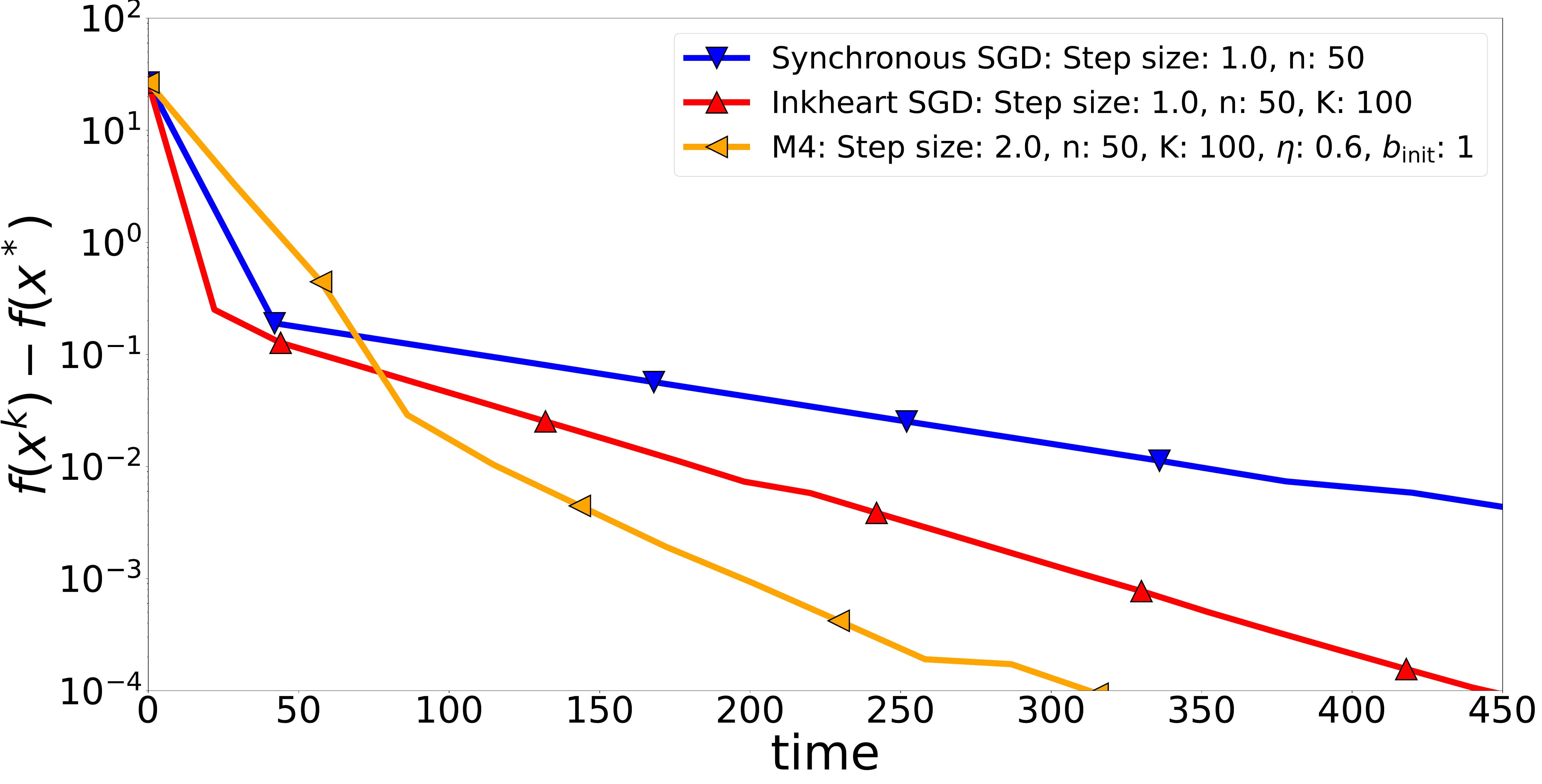}
        \caption{$h=0.1$}
    \end{subfigure}\hfill
    \begin{subfigure}[b]{0.32\textwidth}
        \centering\includegraphics[width=\textwidth]{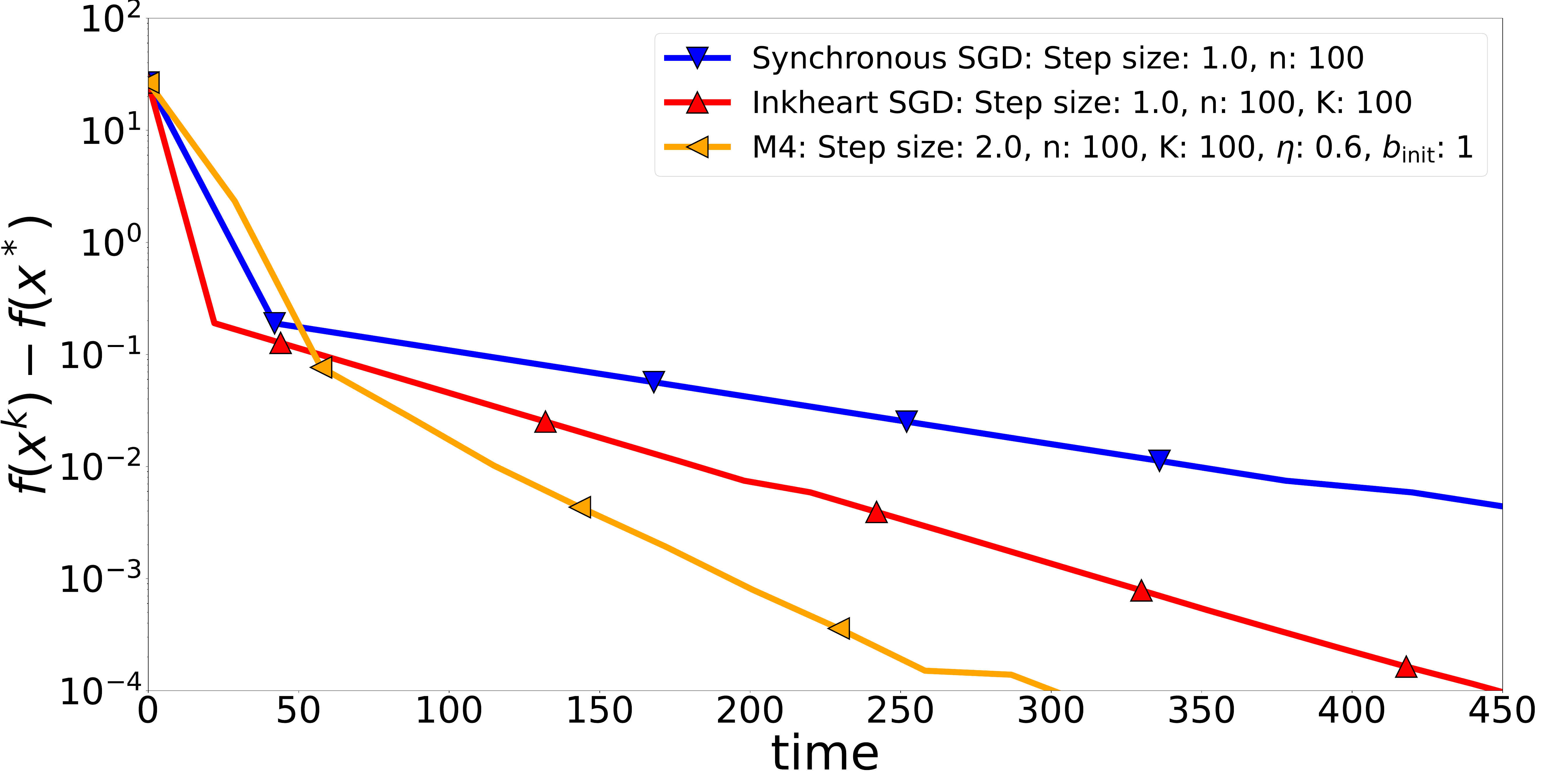}
        \caption{$h=0.1$}
    \end{subfigure}\hfill
    \begin{subfigure}[b]{0.32\textwidth}
        \centering\includegraphics[width=\textwidth]{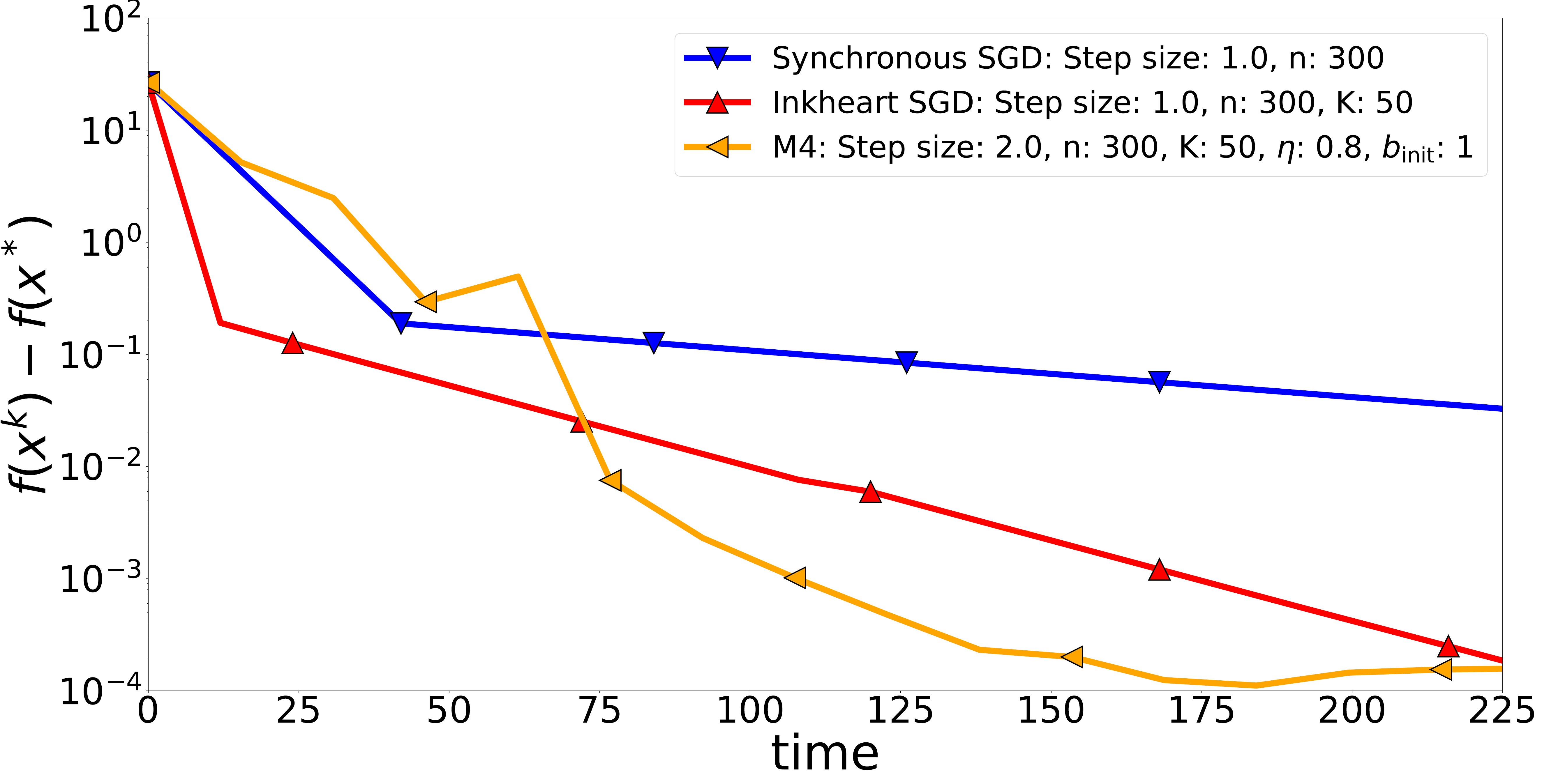}
        \caption{$h=0.1$}
    \end{subfigure}
    
    \vspace{0.25cm}
    
    \begin{subfigure}[b]{0.32\textwidth}
        \centering\includegraphics[width=\textwidth]{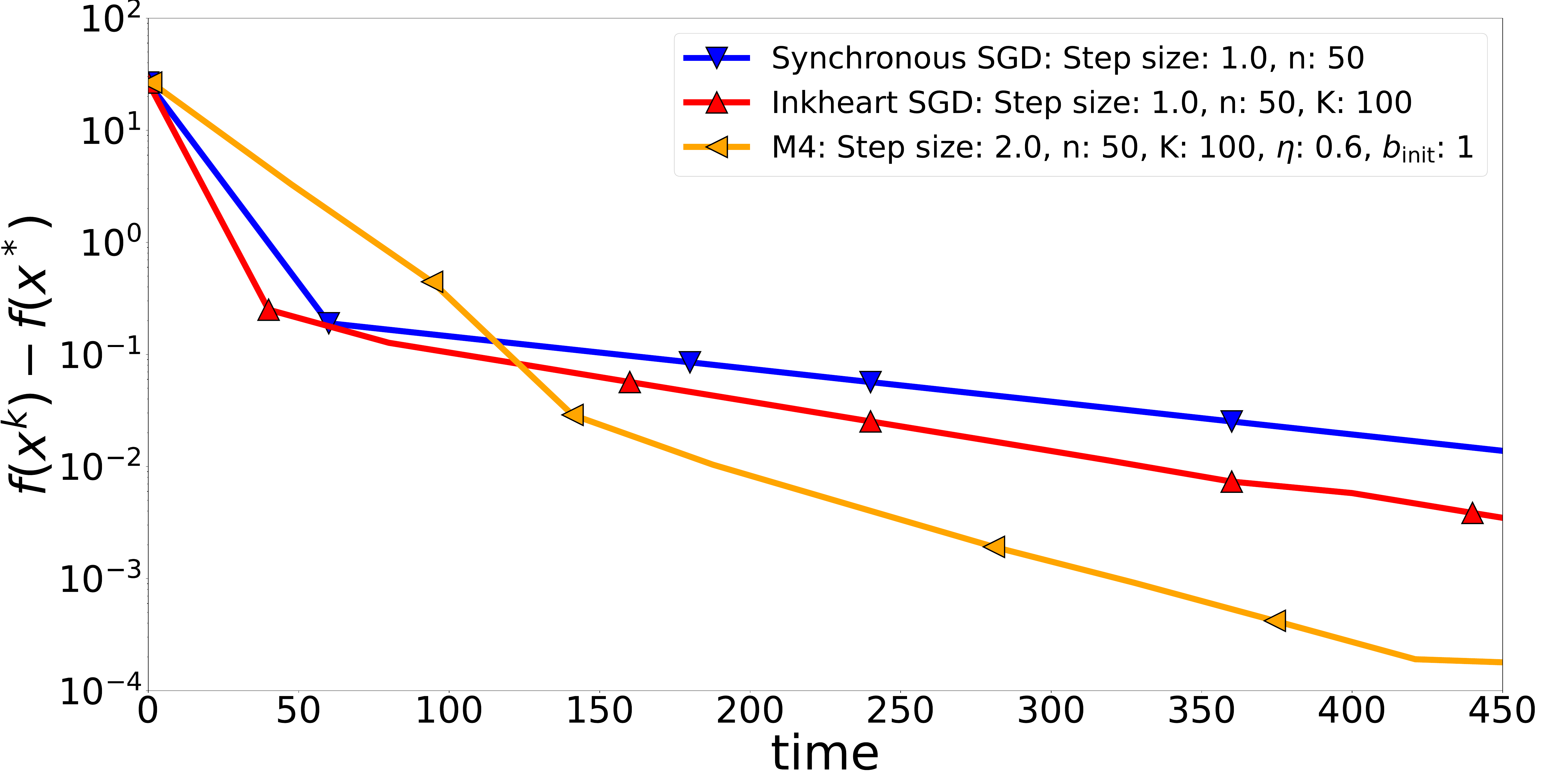}
        \caption{$h=1.0$}
    \end{subfigure}\hfill
    \begin{subfigure}[b]{0.32\textwidth}
        \centering\includegraphics[width=\textwidth]{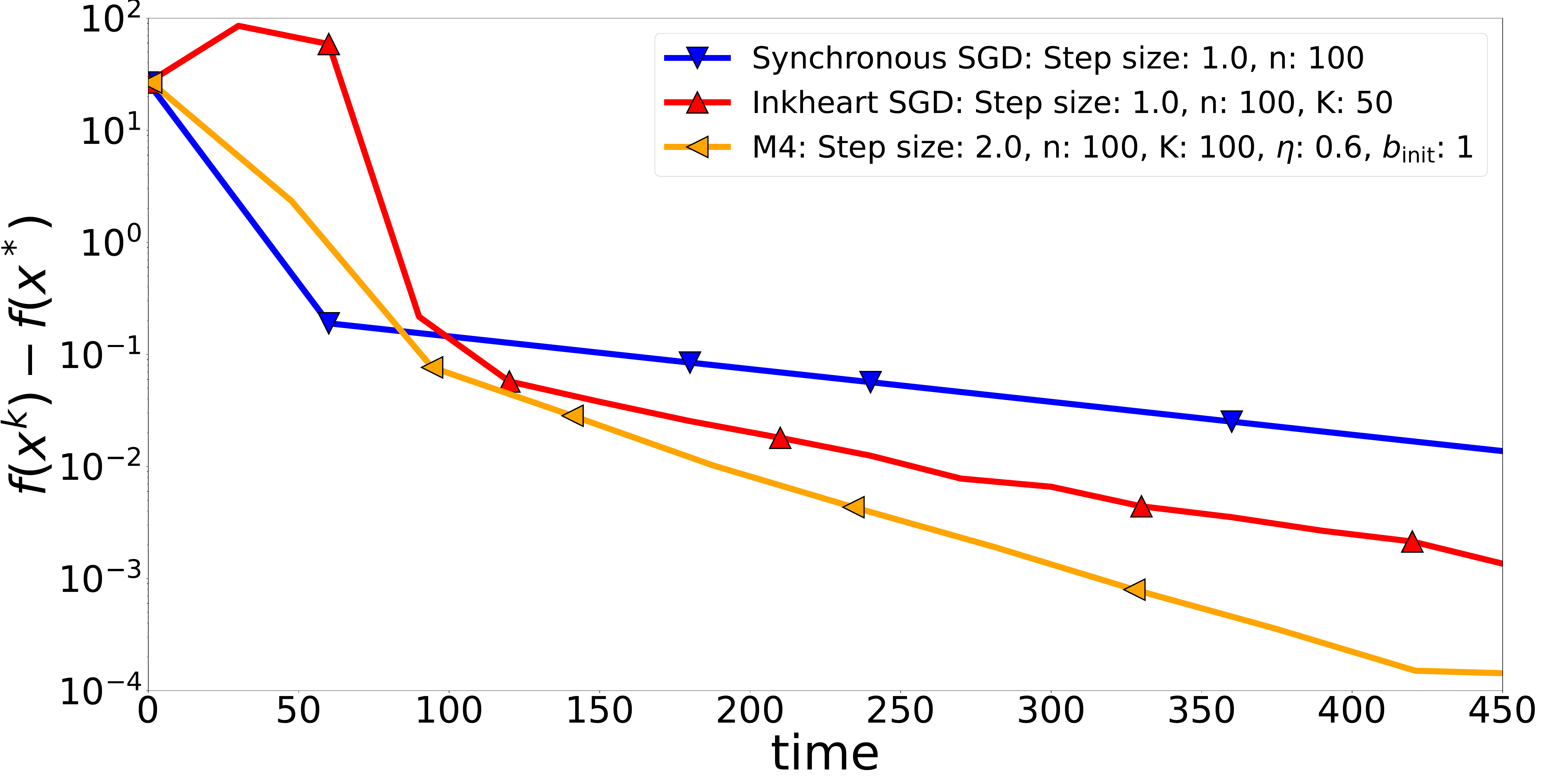}
        \caption{$h=1.0$}
    \end{subfigure}\hfill
    \begin{subfigure}[b]{0.32\textwidth}
        \centering\includegraphics[width=\textwidth]{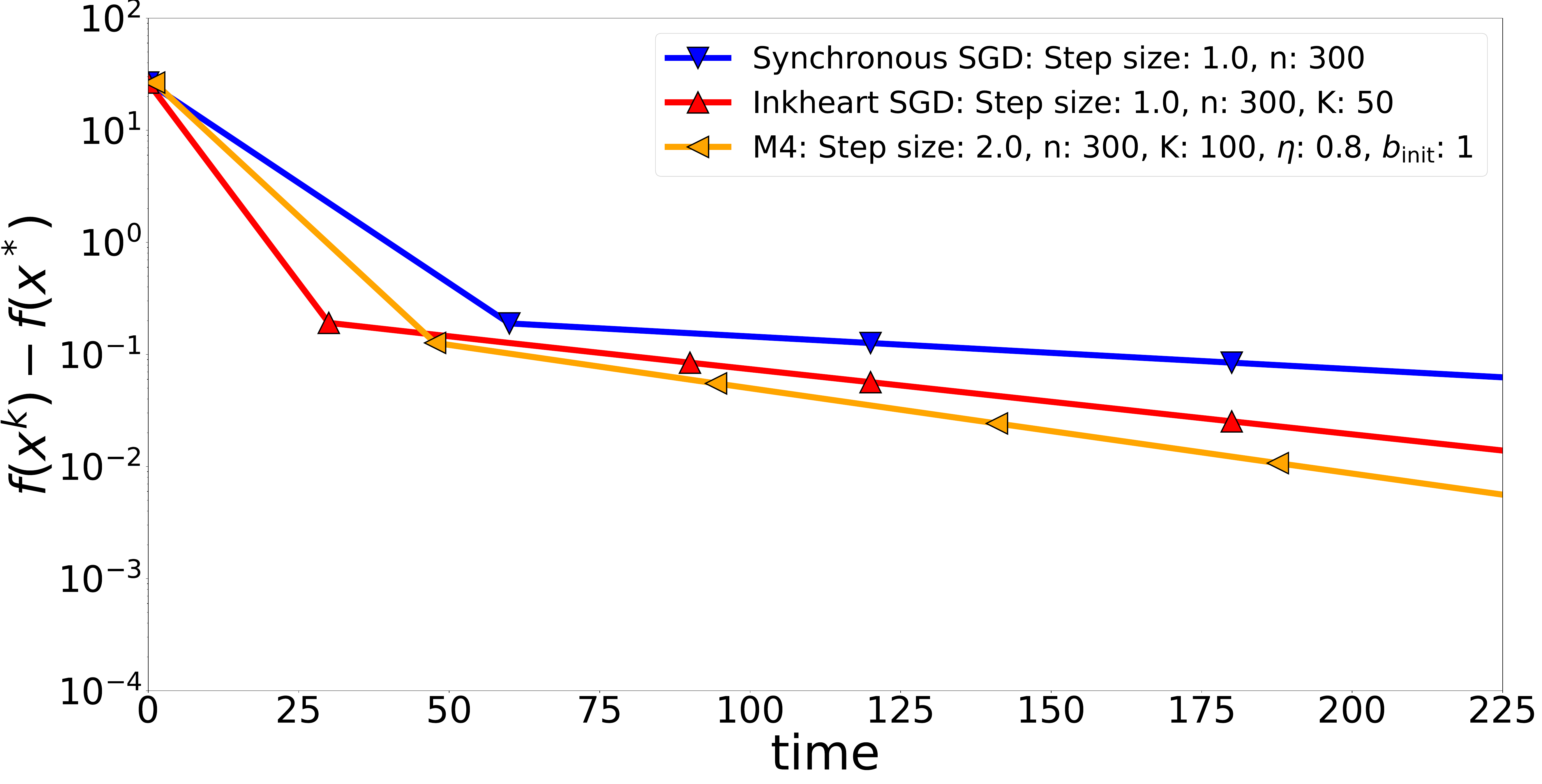}
        \caption{$h=1.0$}
    \end{subfigure}
    
    \caption{Convergence under low gradient noise ($\sigma = 0.001$). 
    Fixed parameters: $d=300$, $\kappa = \nicefrac{1}{d}$, $\tau = \nicefrac{1}{d}$. 
    Rows vary the gradient computation time $h$; columns correspond to the number of workers $n \in \{50, 100, 300\}$.}
    \label{fig:noise_low_homo_quad}
\end{figure}

\begin{figure}[htp]
    \centering
    \captionsetup[subfigure]{labelformat=empty, font=scriptsize}
    \setlength{\tabcolsep}{3pt}
    
    \begin{subfigure}[b]{0.32\textwidth}
        \centering\includegraphics[width=\textwidth]{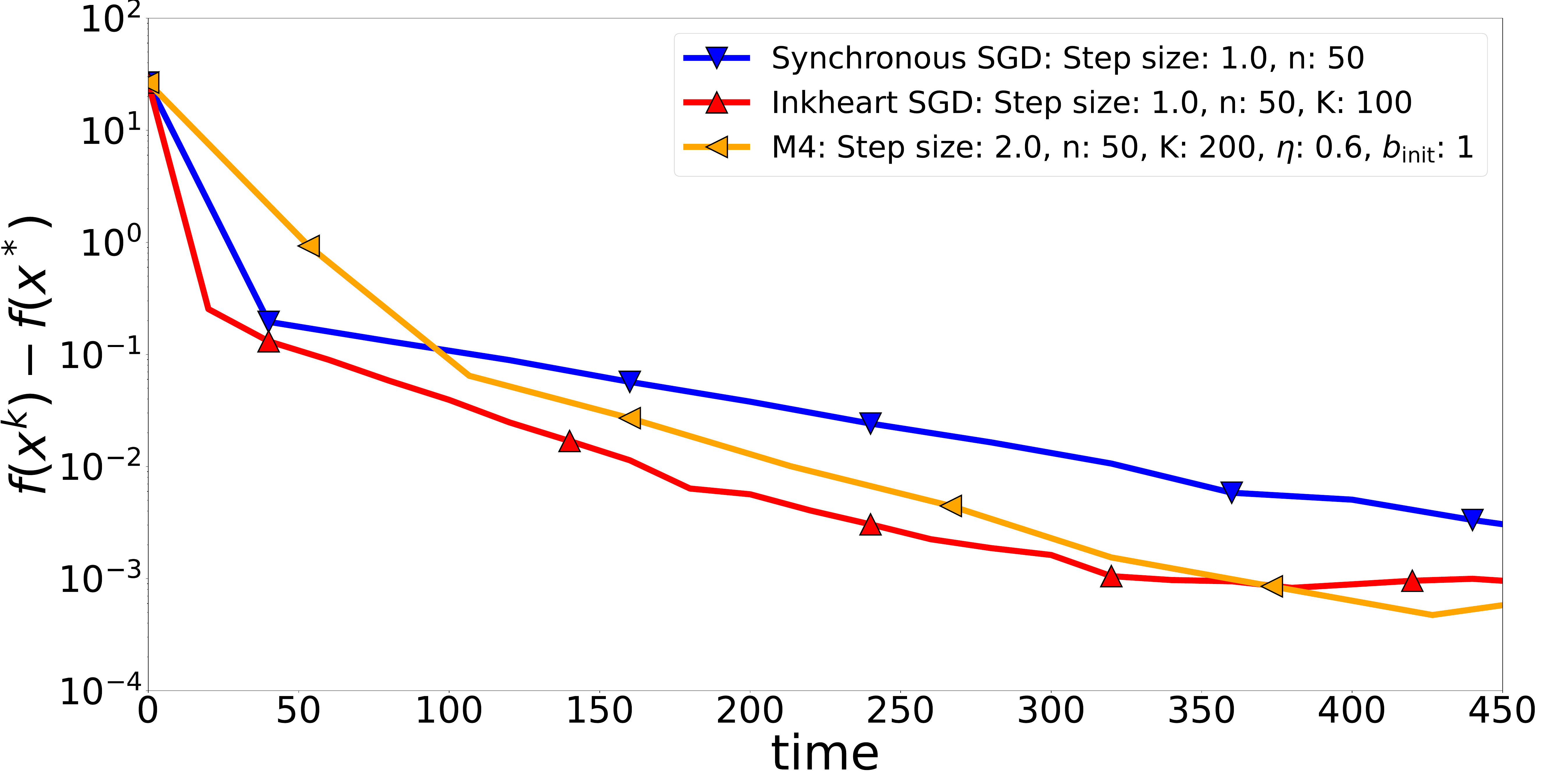}
        \caption{$h=0$}
    \end{subfigure}\hfill
    \begin{subfigure}[b]{0.32\textwidth}
        \centering\includegraphics[width=\textwidth]{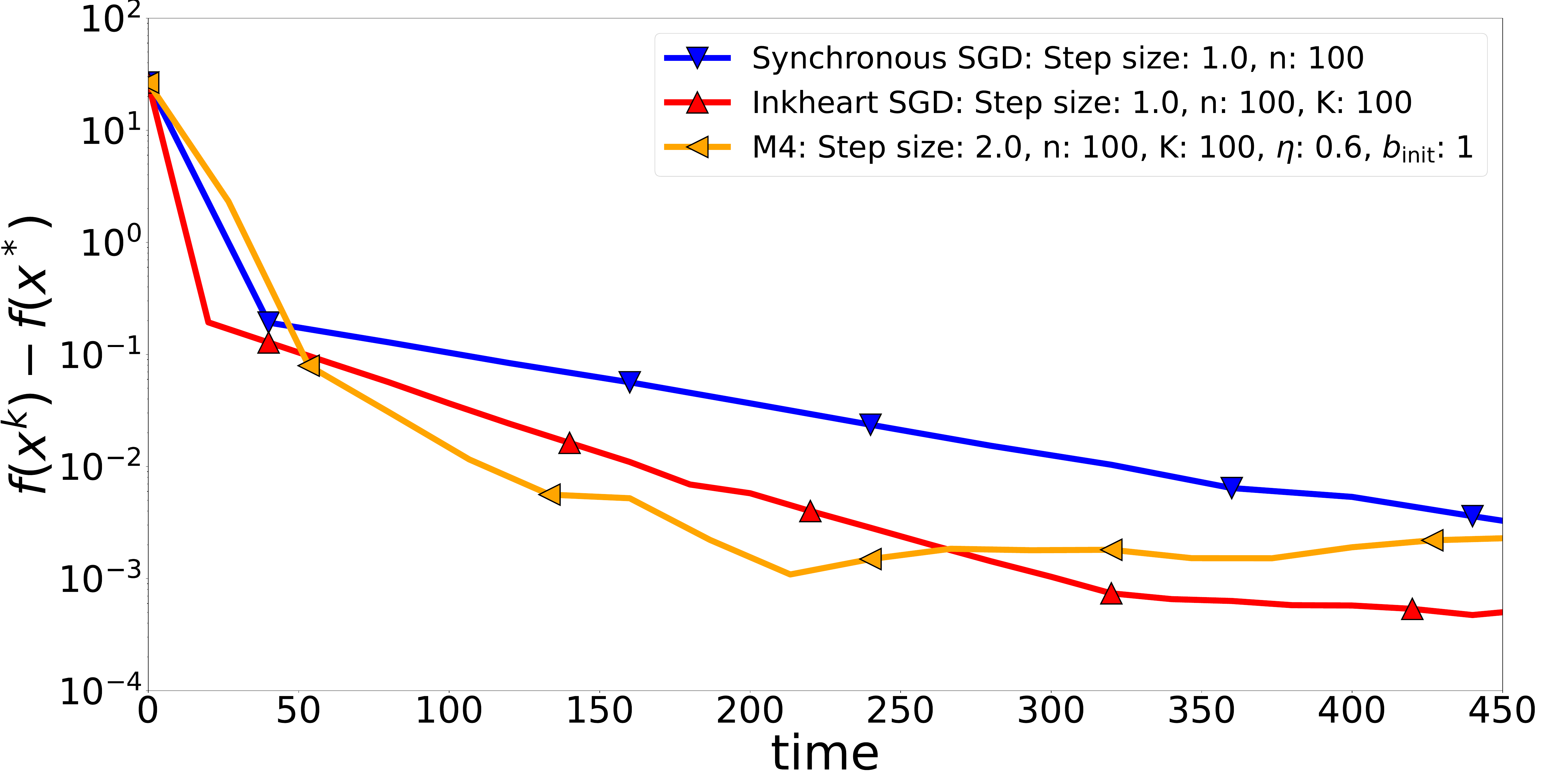}
        \caption{$h=0$}
    \end{subfigure}\hfill
    \begin{subfigure}[b]{0.32\textwidth}
        \centering\includegraphics[width=\textwidth]{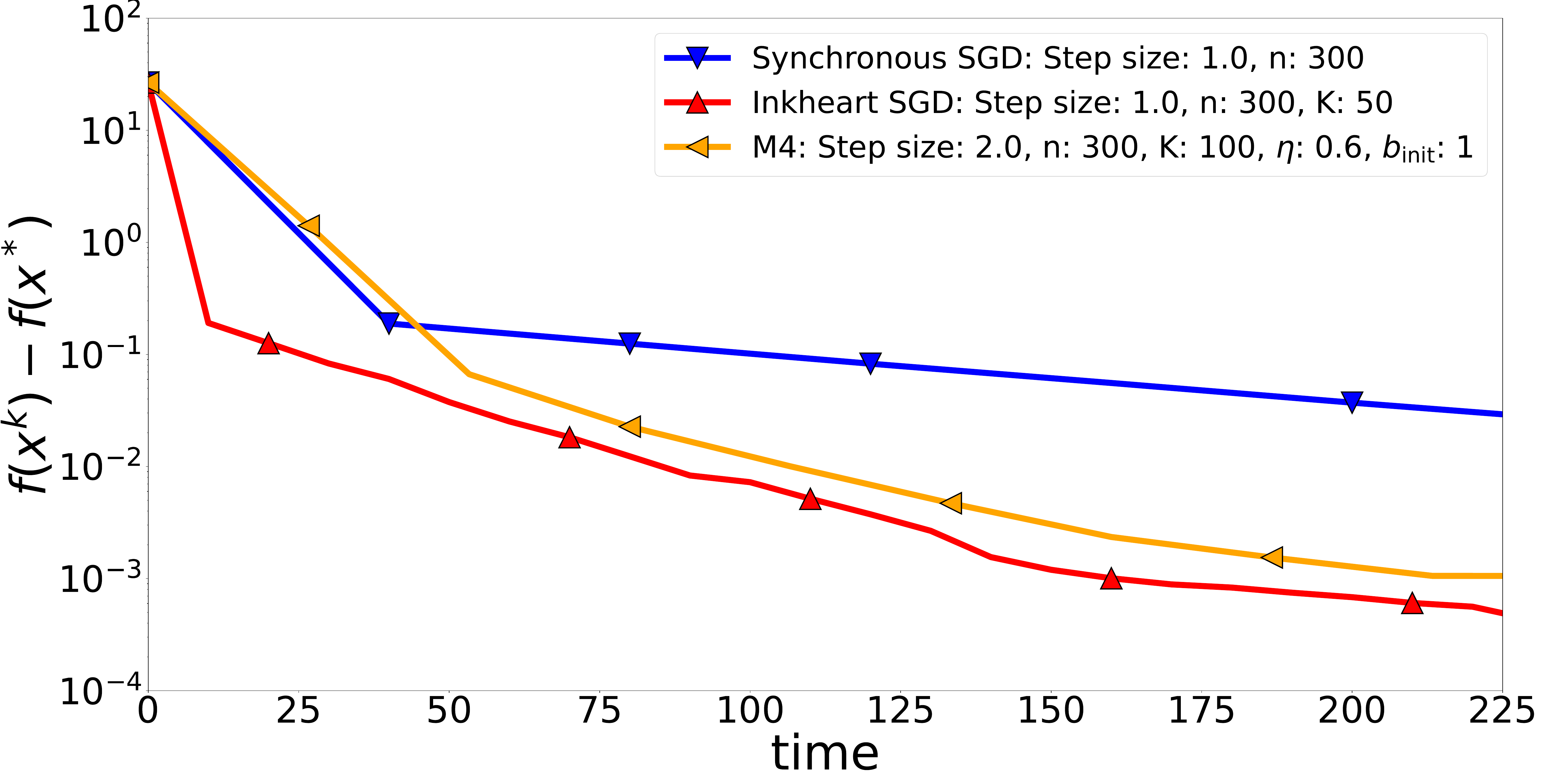}
        \caption{$h=0$}
    \end{subfigure}
    
    \vspace{0.25cm}
    
    \begin{subfigure}[b]{0.32\textwidth}
        \centering\includegraphics[width=\textwidth]{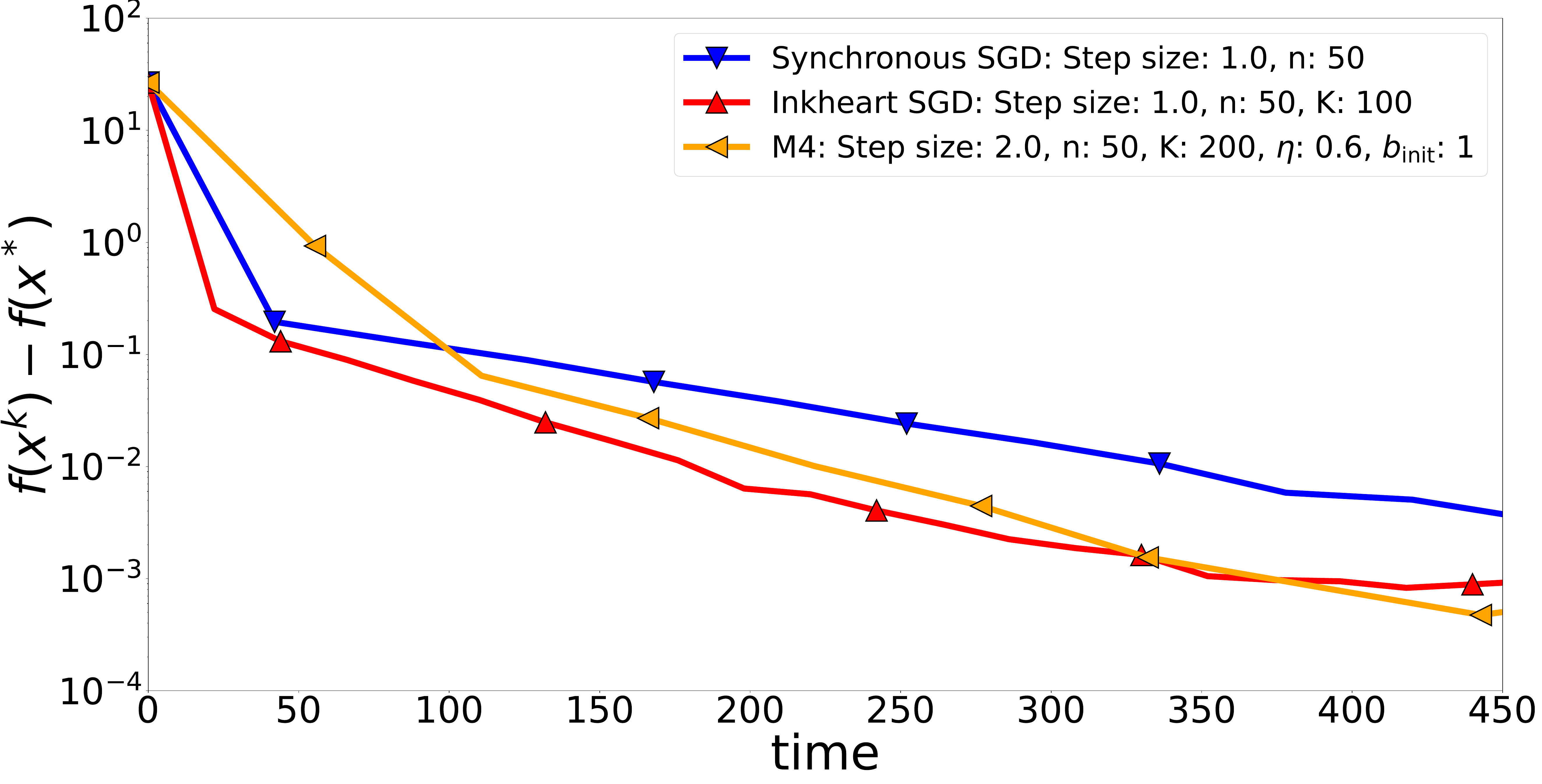}
        \caption{$h=0.1$}
    \end{subfigure}\hfill
    \begin{subfigure}[b]{0.32\textwidth}
        \centering\includegraphics[width=\textwidth]{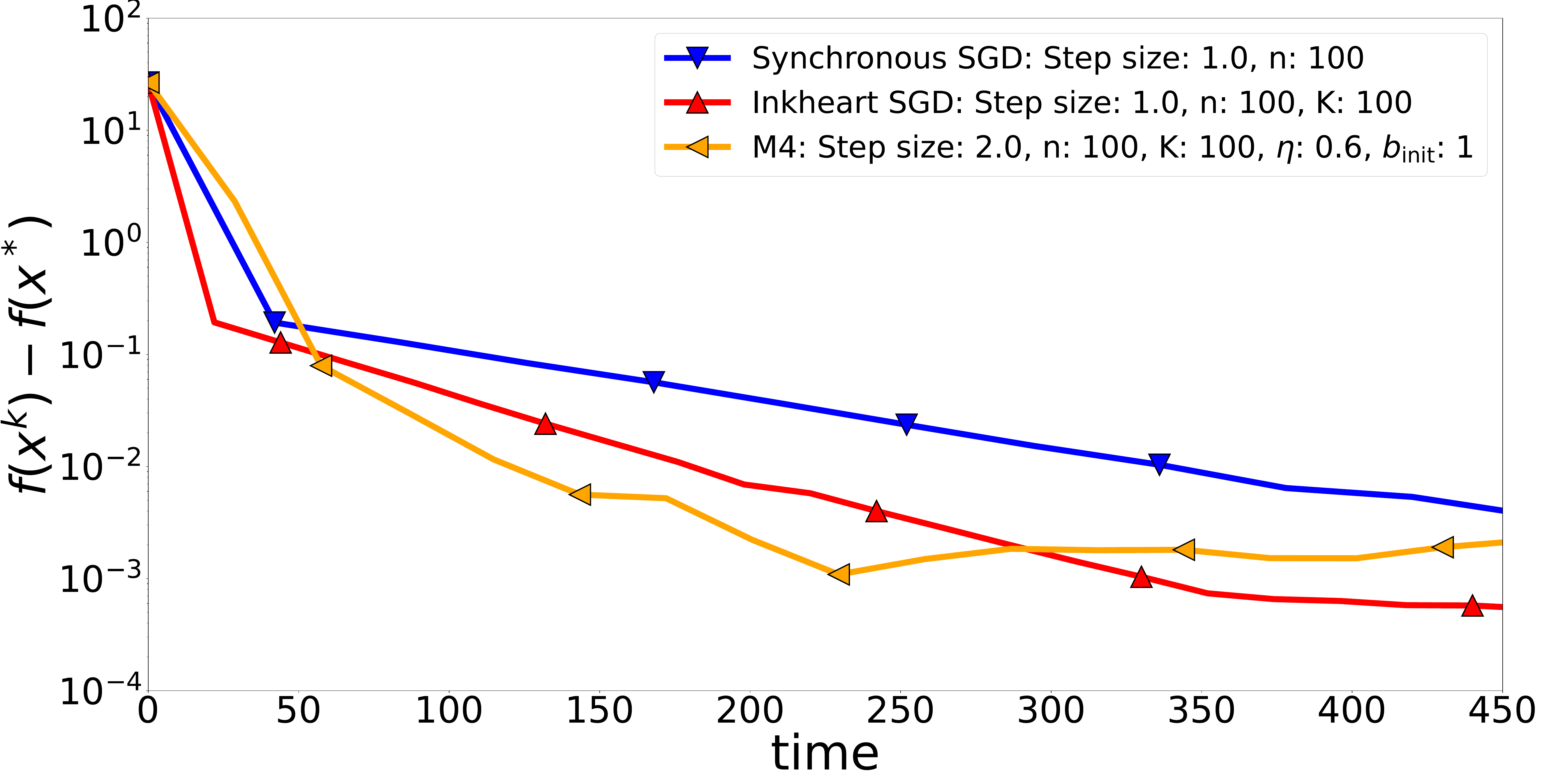}
        \caption{$h=0.1$}
    \end{subfigure}\hfill
    \begin{subfigure}[b]{0.32\textwidth}
        \centering\includegraphics[width=\textwidth]{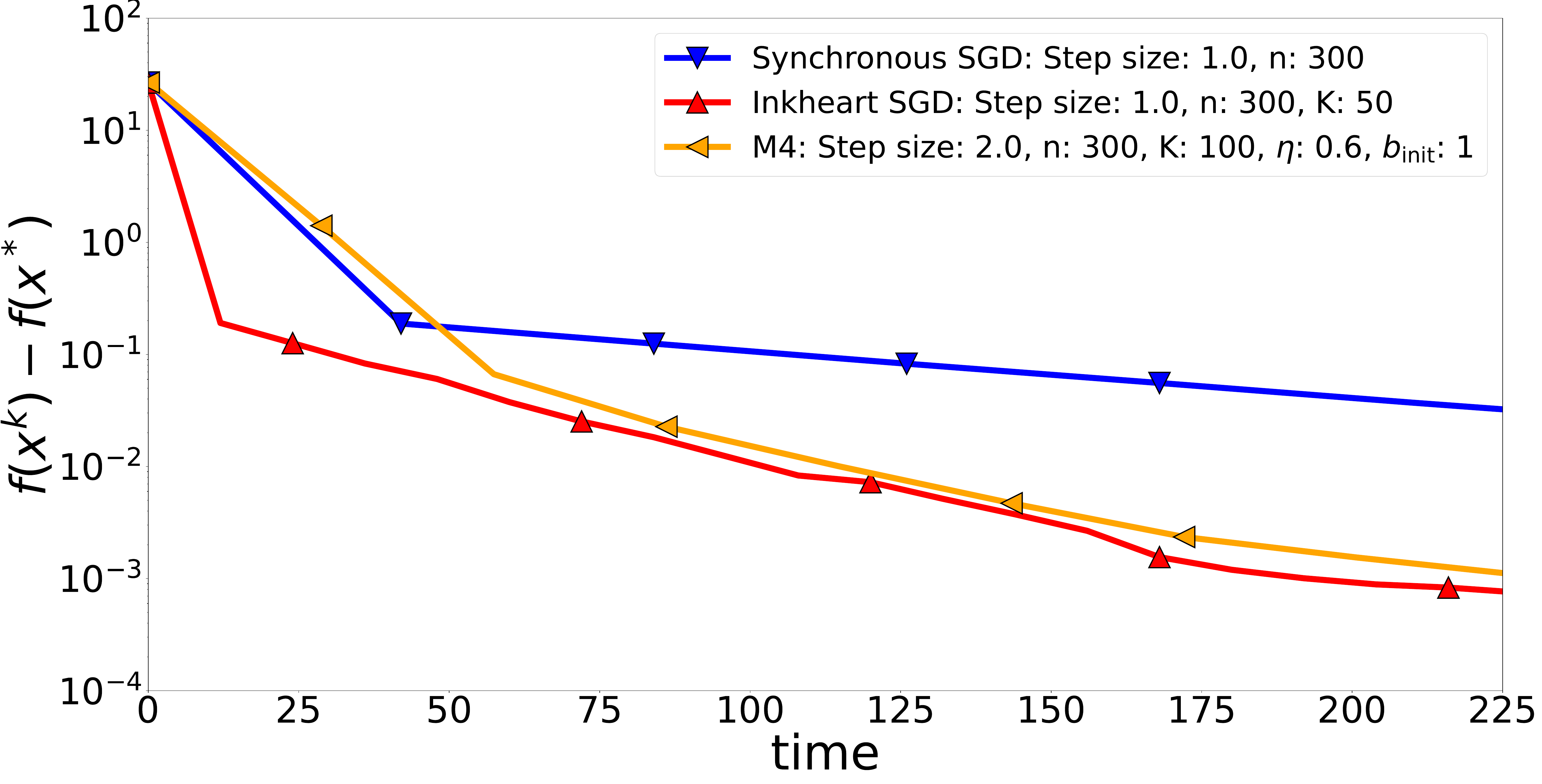}
        \caption{$h=0.1$}
    \end{subfigure}
    
    \vspace{0.25cm}
    
    \begin{subfigure}[b]{0.32\textwidth}
        \centering\includegraphics[width=\textwidth]{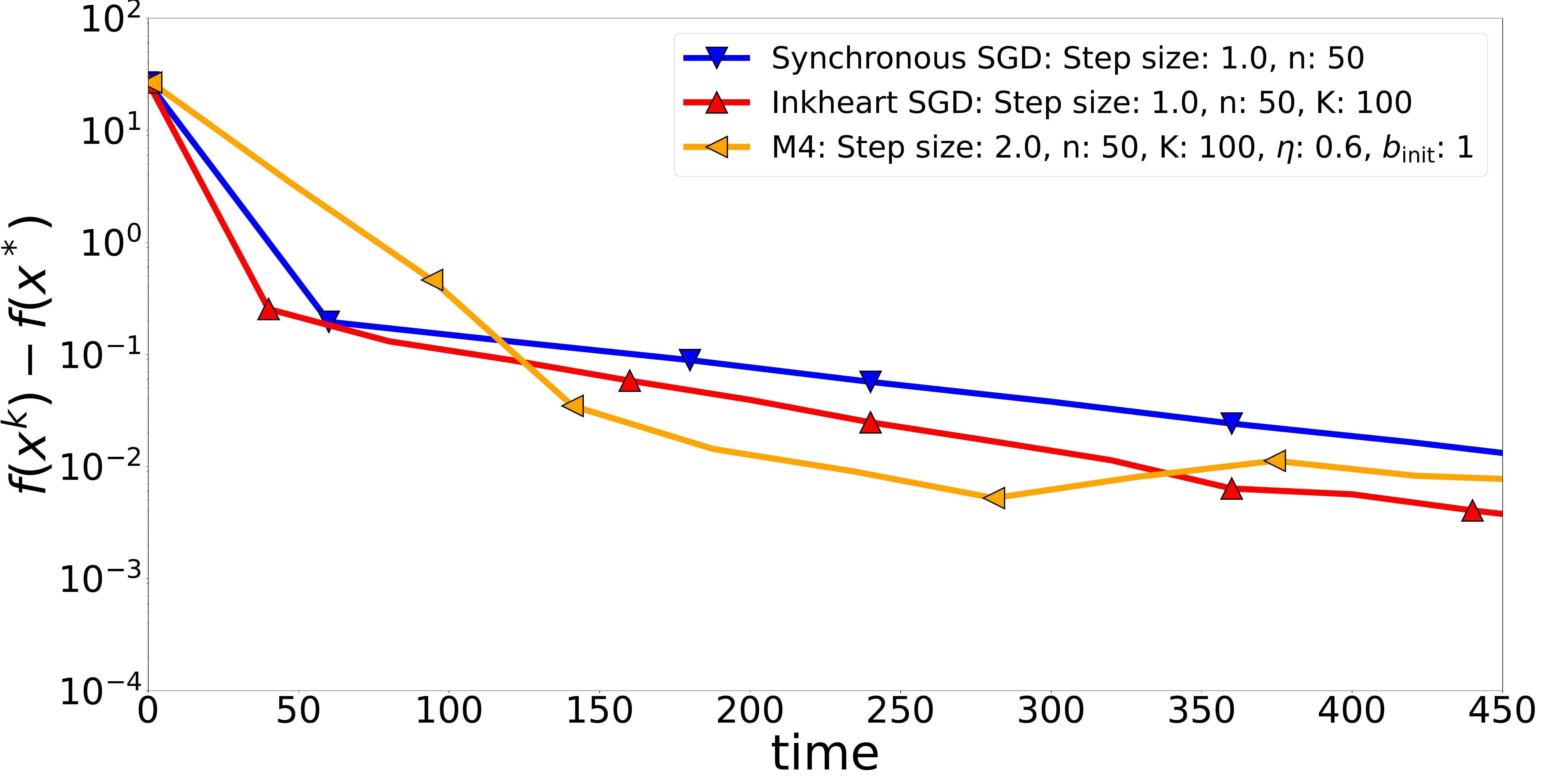}
        \caption{$h=1.0$}
    \end{subfigure}\hfill
    \begin{subfigure}[b]{0.32\textwidth}
        \centering\includegraphics[width=\textwidth]{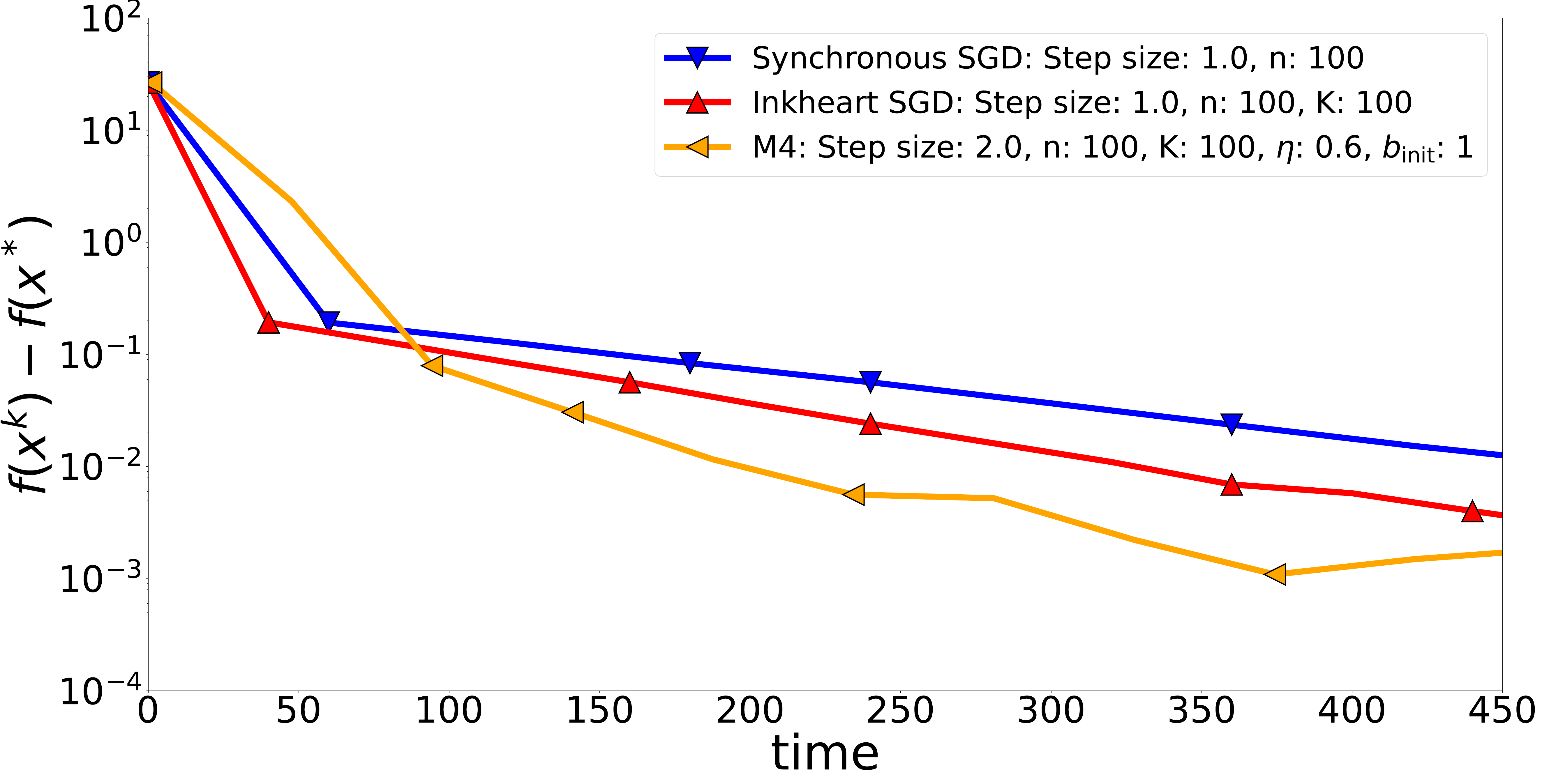}
        \caption{$h=1.0$}
    \end{subfigure}\hfill
    \begin{subfigure}[b]{0.32\textwidth}
        \centering\includegraphics[width=\textwidth]{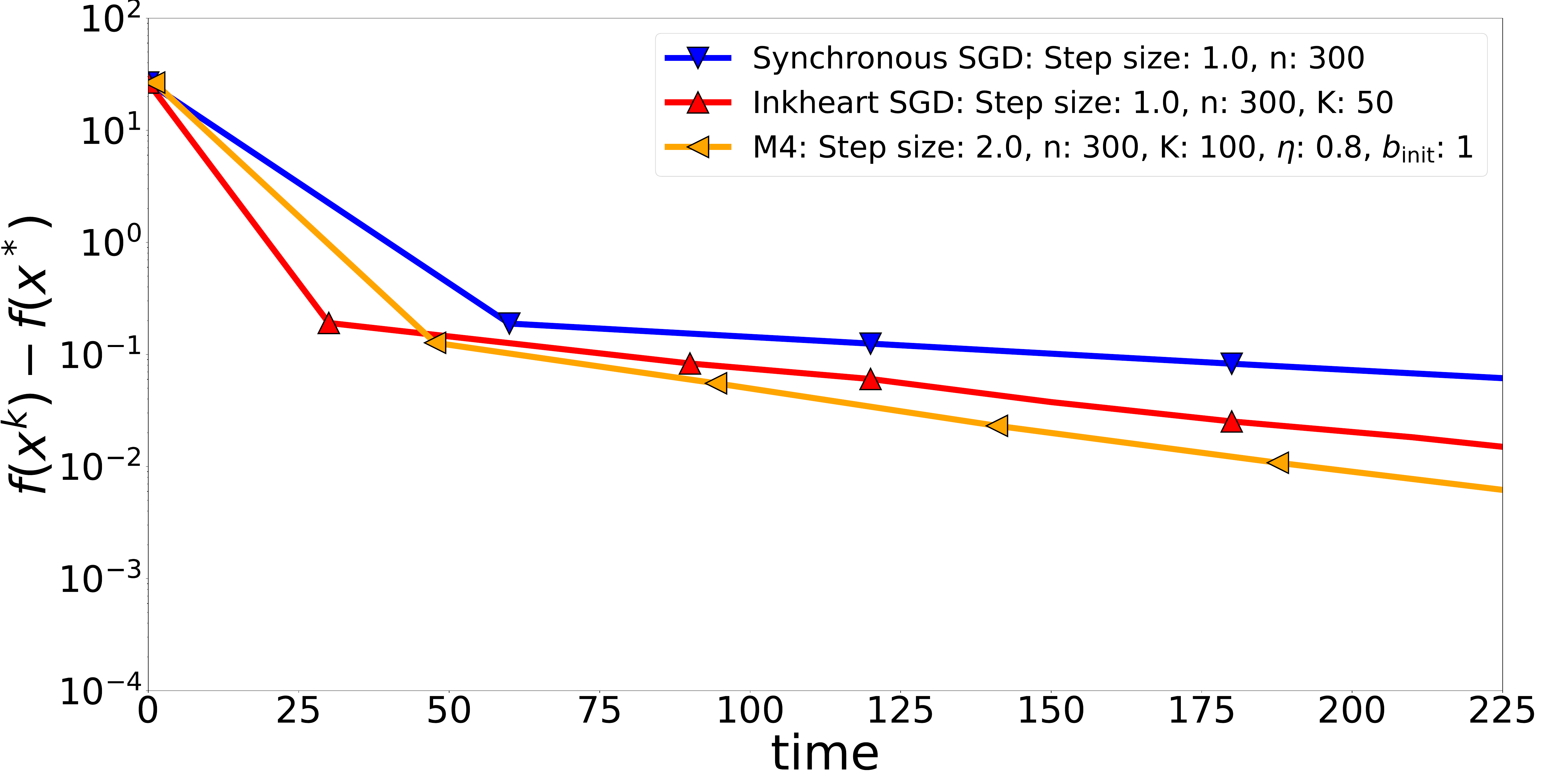}
        \caption{$h=1.0$}
    \end{subfigure}
    
    \caption{Convergence under medium gradient noise ($\sigma = 0.01$). 
    Fixed parameters: $d=300$, $\kappa = \nicefrac{1}{d}$, $\tau = \nicefrac{1}{d}$. 
    Rows vary the gradient computation time $h$; columns correspond to the number of workers $n \in \{50, 100, 300\}$.}
    \label{fig:noise_high_homo_quad}
\end{figure}

\begin{figure}[htp]
    \centering
    \captionsetup[subfigure]{labelformat=empty, font=scriptsize}
    \setlength{\tabcolsep}{3pt}
    
    \begin{subfigure}[b]{0.32\textwidth}
        \centering\includegraphics[width=\textwidth]{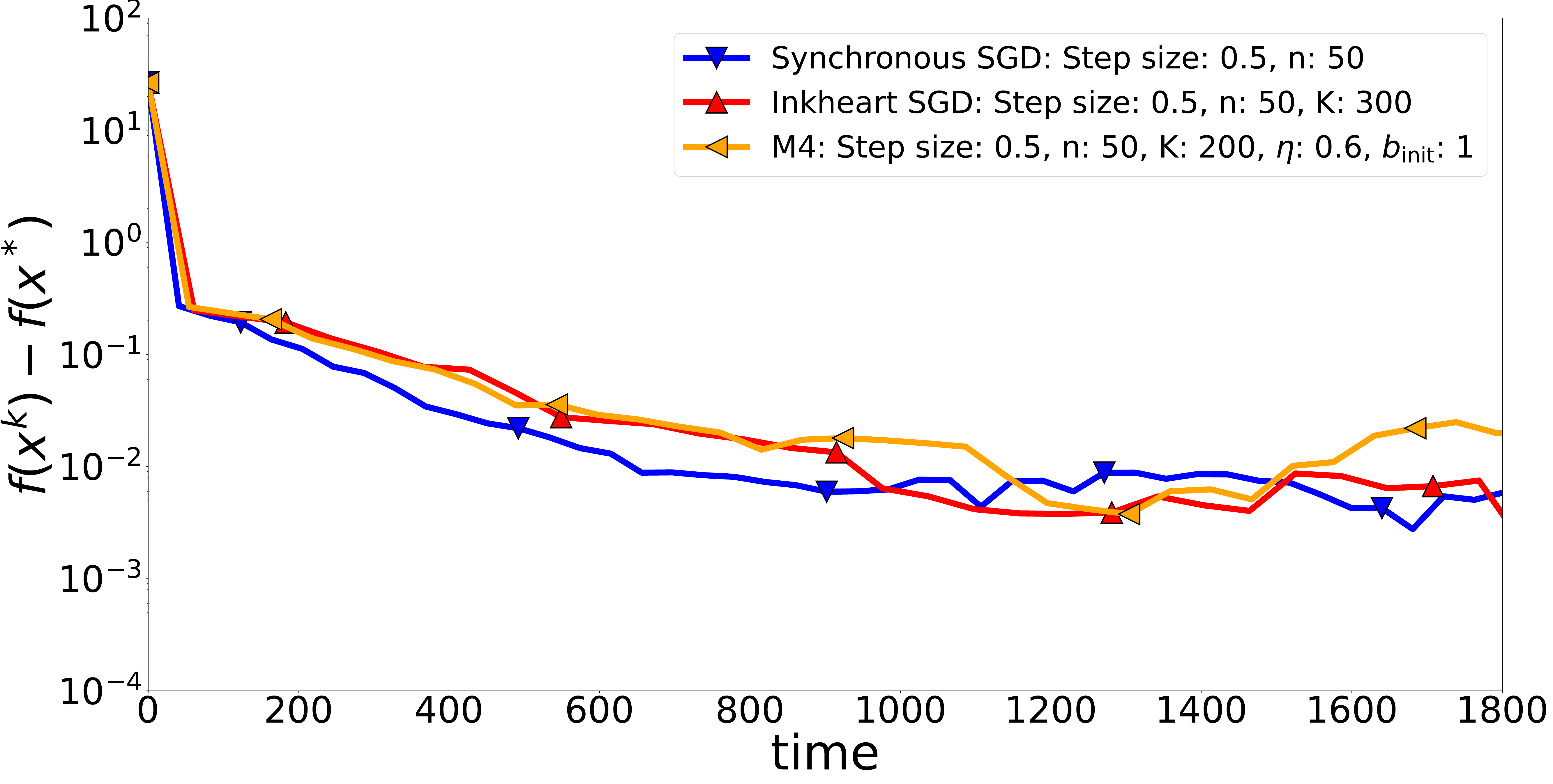}
        \caption{$h=0.05$}
    \end{subfigure}\hfill
    \begin{subfigure}[b]{0.32\textwidth}
        \centering\includegraphics[width=\textwidth]{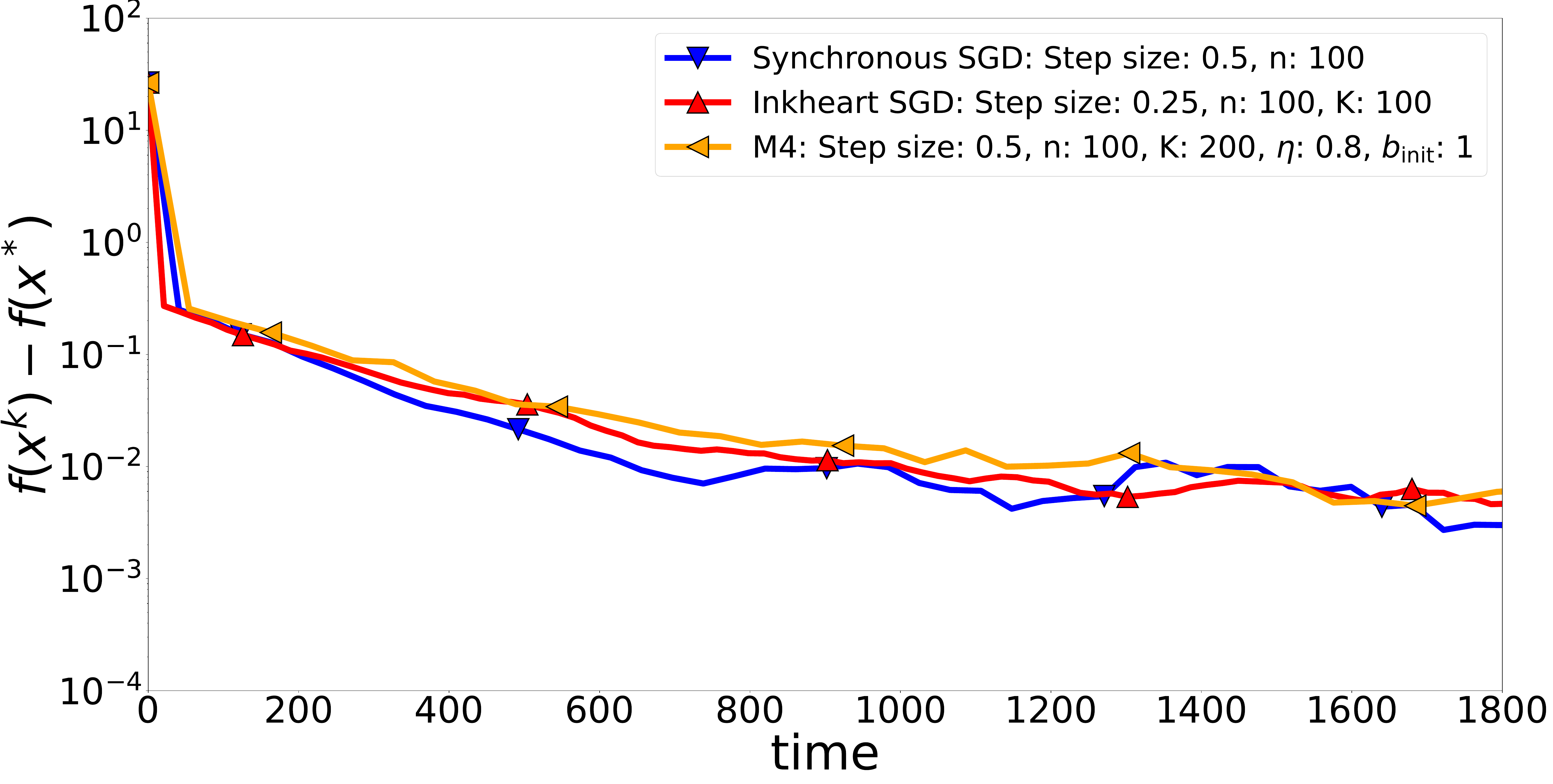}
        \caption{$h=0.05$}
    \end{subfigure}\hfill
    \begin{subfigure}[b]{0.32\textwidth}
        \centering\includegraphics[width=\textwidth]{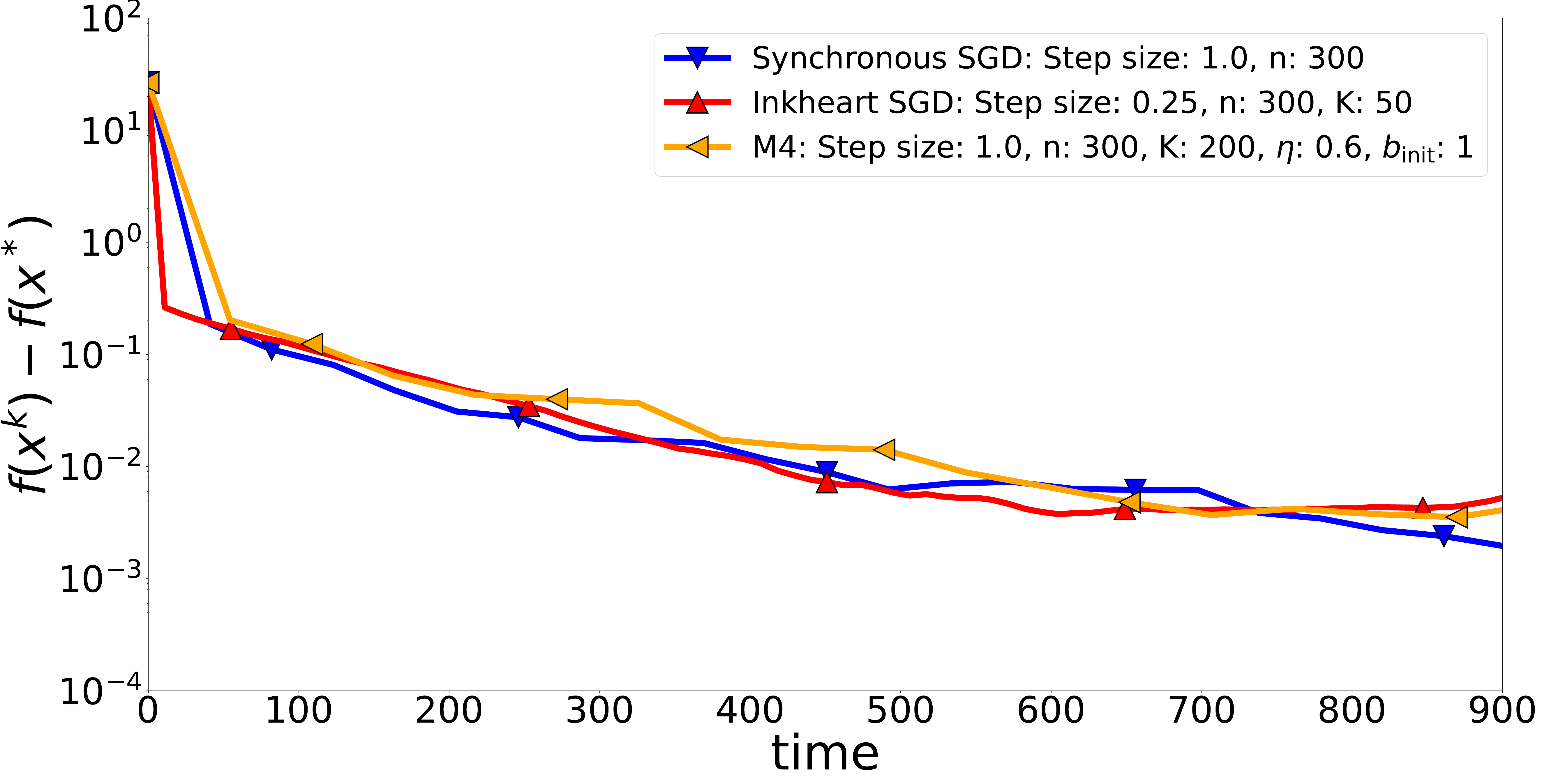}
        \caption{$h=0.05$}
    \end{subfigure}
    
    \vspace{0.25cm}
    
    \begin{subfigure}[b]{0.32\textwidth}
        \centering\includegraphics[width=\textwidth]{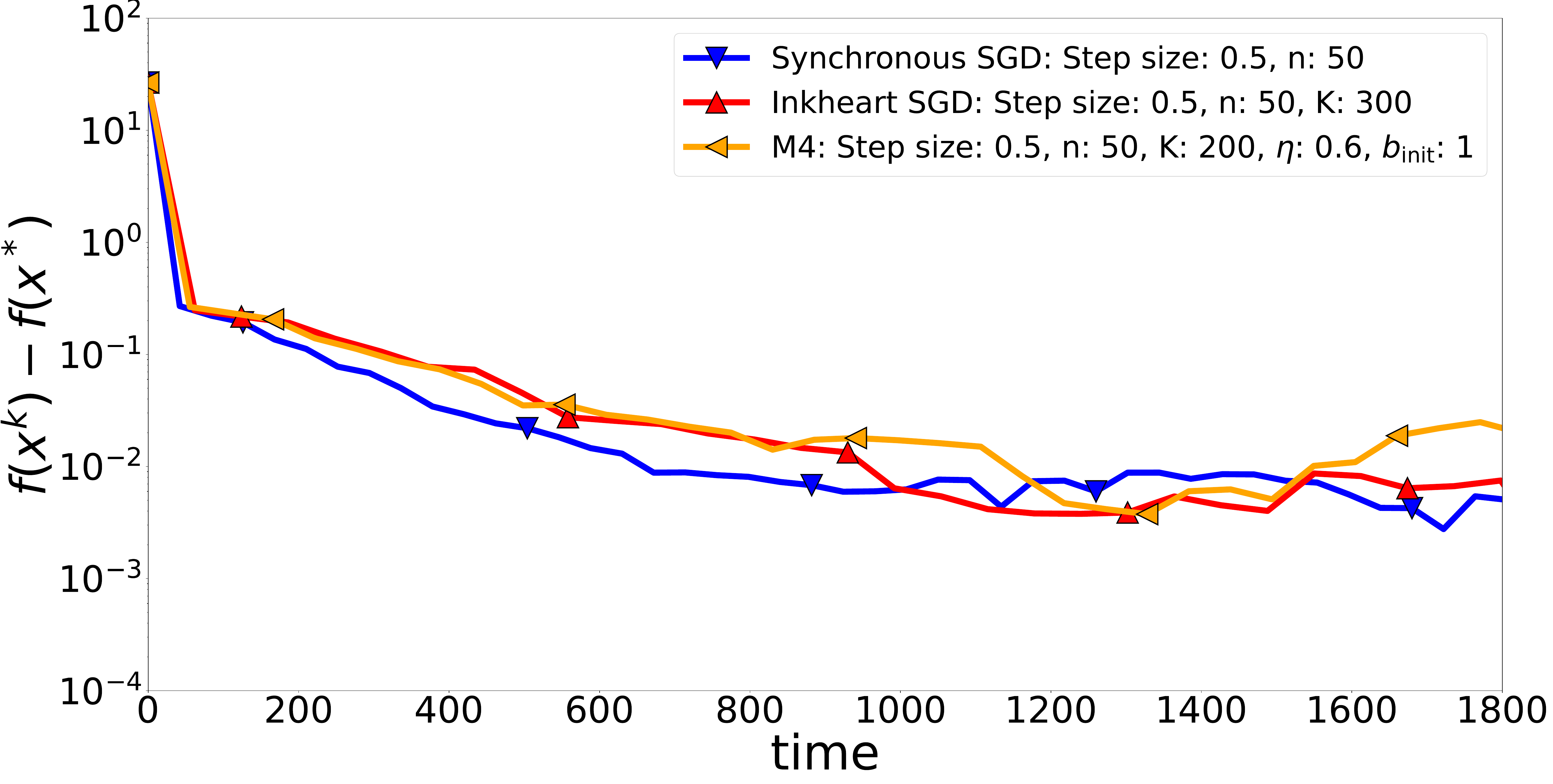}
        \caption{$h=0.1$}
    \end{subfigure}\hfill
    \begin{subfigure}[b]{0.32\textwidth}
        \centering\includegraphics[width=\textwidth]{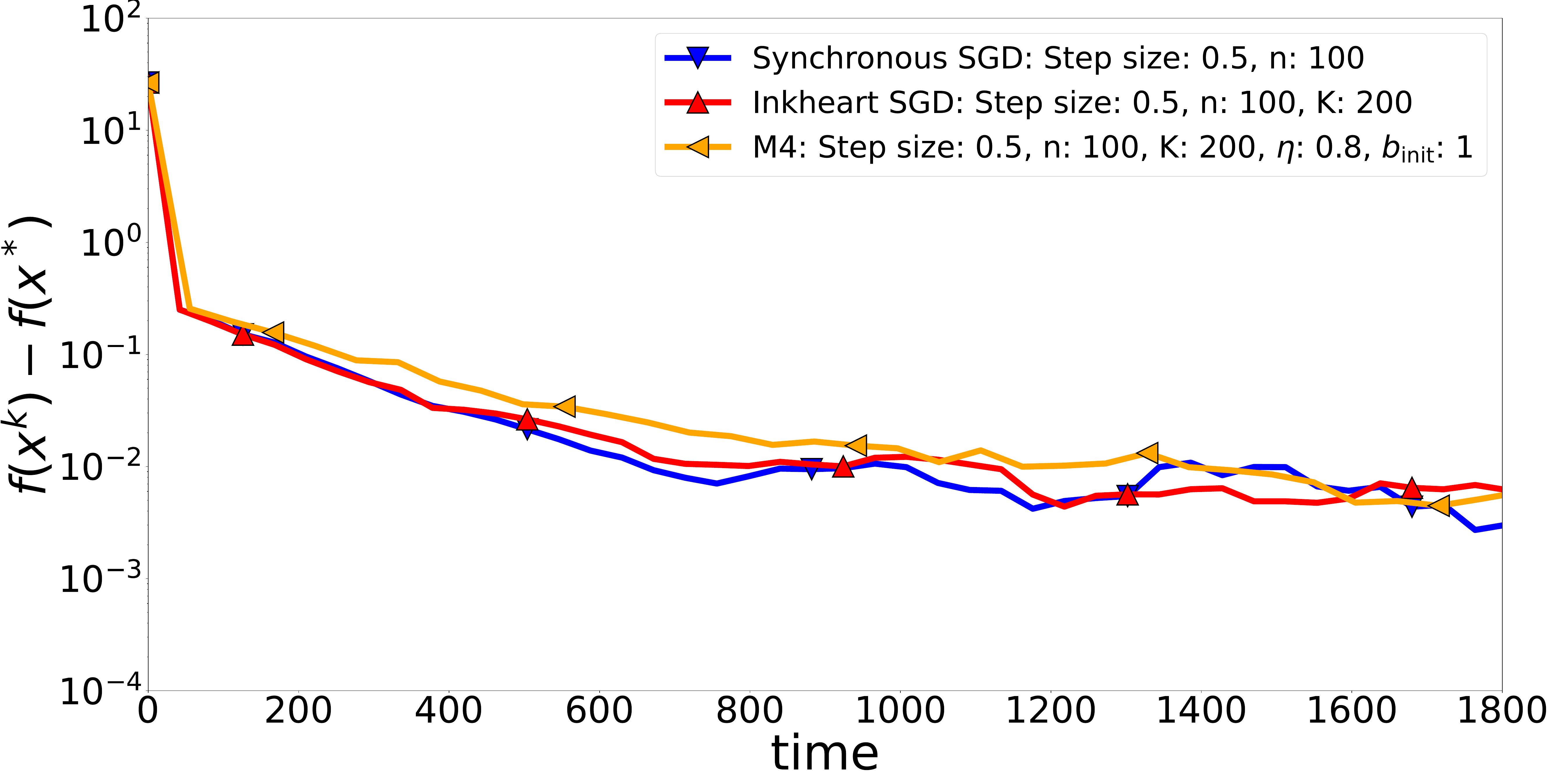}
        \caption{$h=0.1$}
    \end{subfigure}\hfill
    \begin{subfigure}[b]{0.32\textwidth}
        \centering\includegraphics[width=\textwidth]{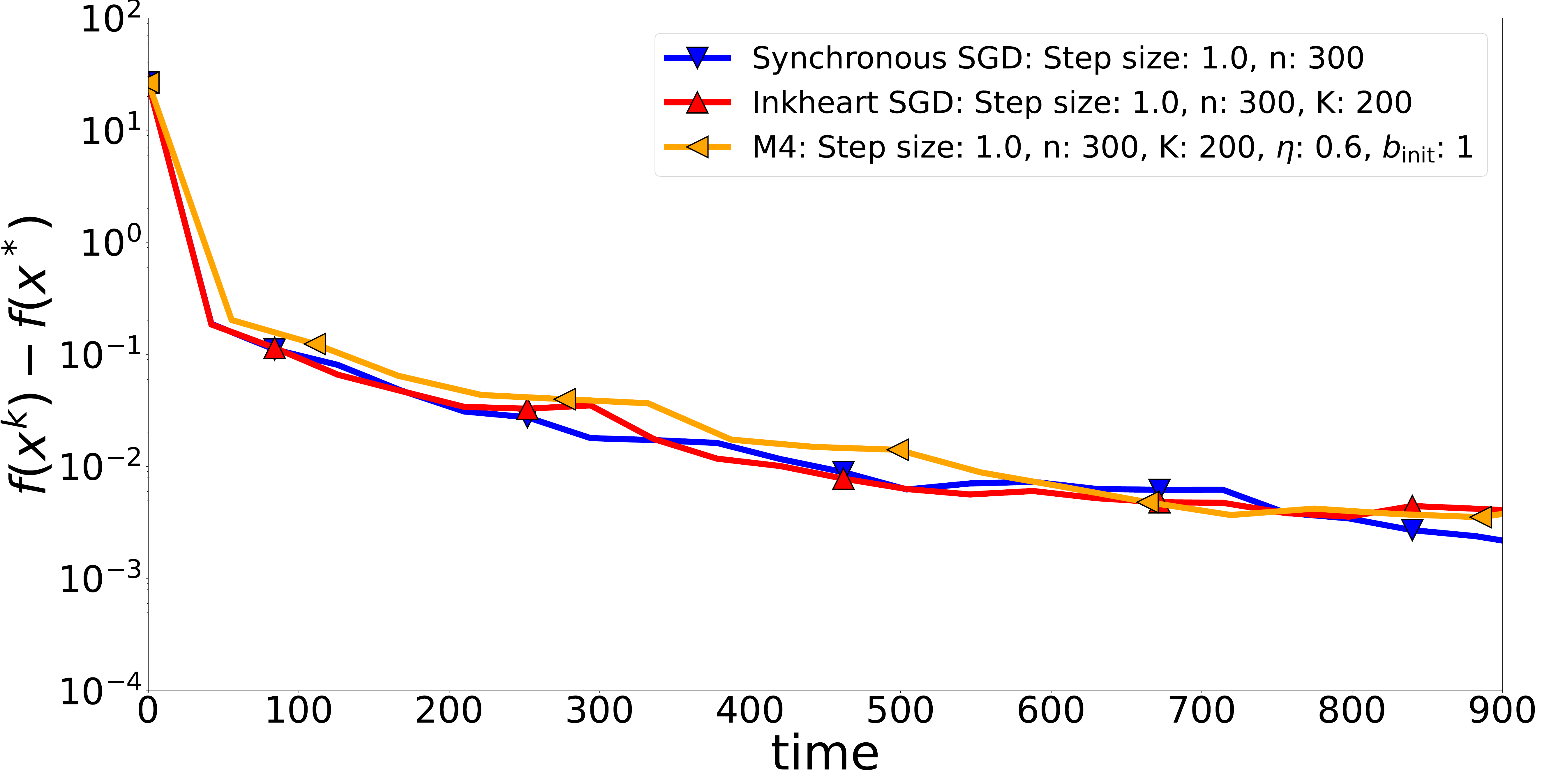}
        \caption{$h=0.1$}
    \end{subfigure}
    
    \vspace{0.25cm}
    
    \begin{subfigure}[b]{0.32\textwidth}
        \centering\includegraphics[width=\textwidth]{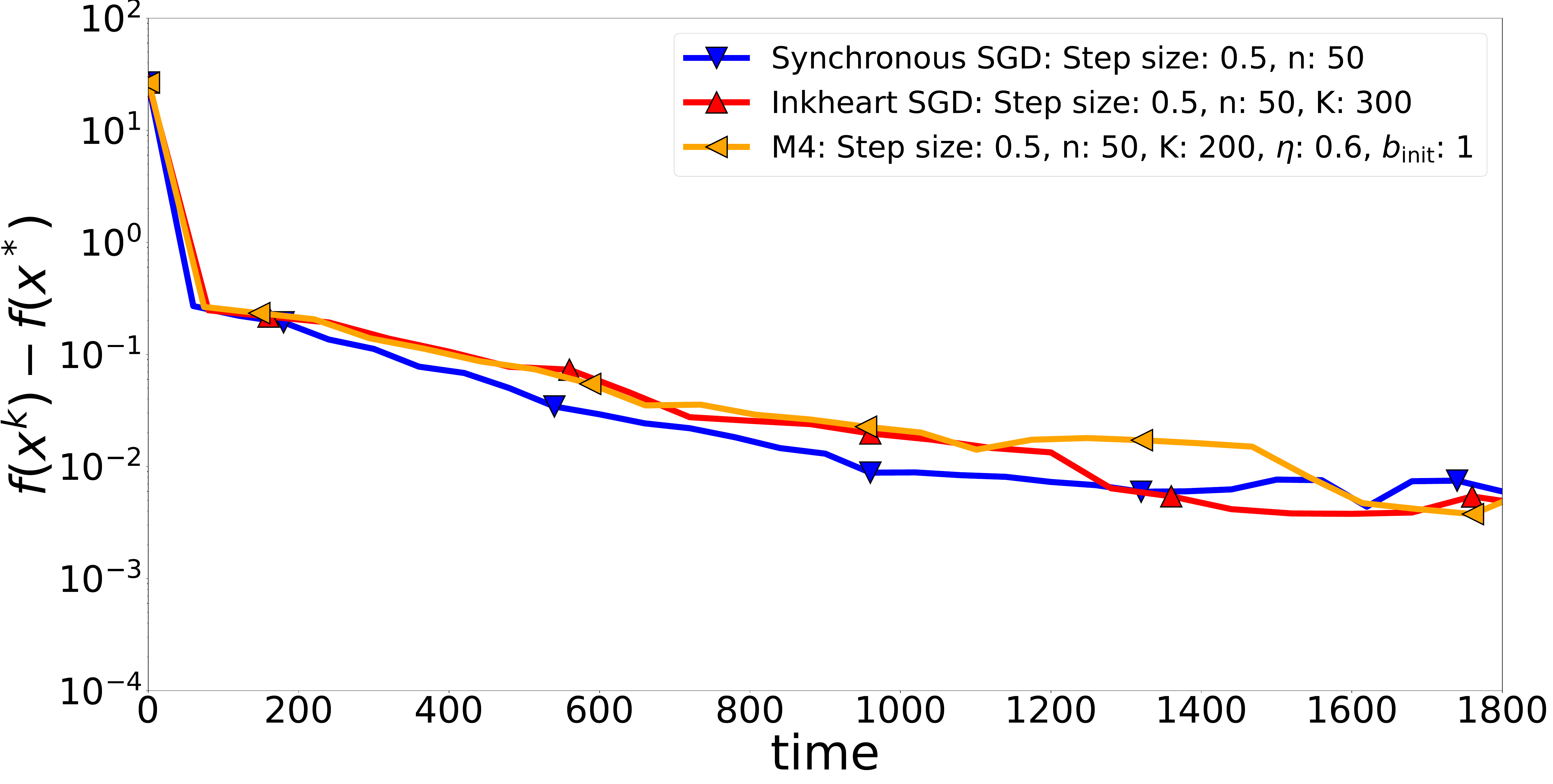}
        \caption{$h=1.0$}
    \end{subfigure}\hfill
    \begin{subfigure}[b]{0.32\textwidth}
        \centering\includegraphics[width=\textwidth]{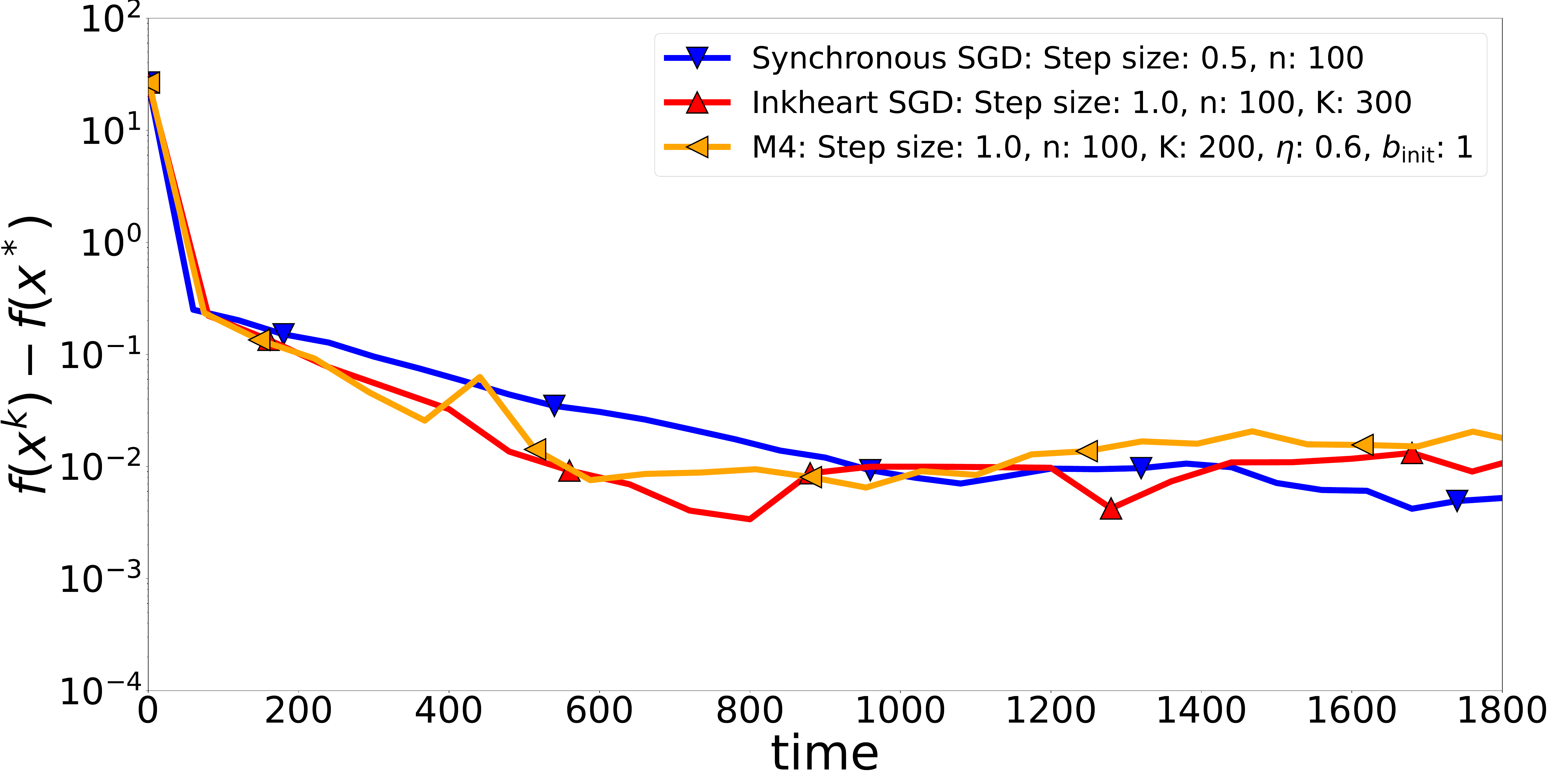}
        \caption{$h=1.0$}
    \end{subfigure}\hfill
    \begin{subfigure}[b]{0.32\textwidth}
        \centering\includegraphics[width=\textwidth]{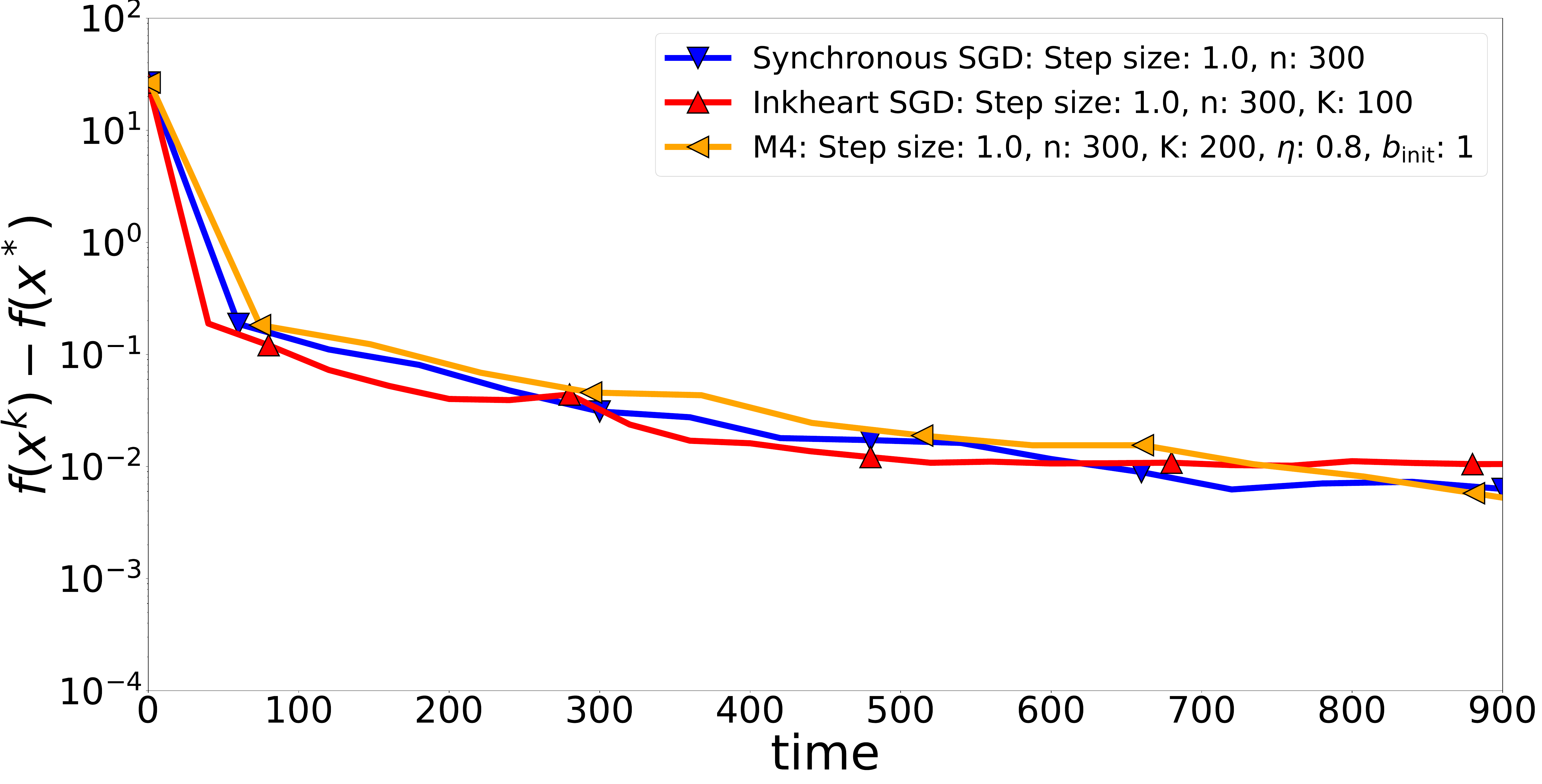}
        \caption{$h=1.0$}
    \end{subfigure}
    
    \caption{Convergence under high gradient noise ($\sigma = 0.1$). 
    Fixed parameters: $d=300$, $\kappa = \nicefrac{1}{d}$, $\tau = \nicefrac{1}{d}$. 
    Rows vary the gradient computation time $h$; columns correspond to the number of workers $n \in \{50, 100, 300\}$.}
    \label{fig:noise_very_high_homo_quad}
\end{figure}

\clearpage
\newpage
\subsection{Heterogeneous Quadratic Task}
Here we consider a heterogeneous quadratic problem. The objective function is almost the same as in \eqref{eq:homo_quad_A}, but we additionally scale $\mA$
to achieve heterogeneity.
\begin{equation}\label{eq:hetero_quad}
    f_i(x) = \frac{1}{2} x^\top \left(\xi_i \mA\right) x, \quad 
    \textnormal{where } \xi_i \sim \mathcal{N}(1, \ell^2) \textnormal{ truncated to } [0.1, 2].
\end{equation}
We consider this problem with $\ell = 0.3$ (Figures~\ref{fig:quad_lip_small_noise_small}, \ref{fig:lip_small})
and $\ell = 0.5$ (Figures~\ref{fig:quad_lip_large_noise_small}, \ref{fig:lip_large}). 
As in the previous section, we vary the noise in the stochastic gradients ($\sigma = 0.01, \sigma = 0.001$).
We tune hyperparameters $K, \eta,$ and $\gamma$ using the same grid as before.
We observe that the methods exhibit behavior consistent with the previous experiments.

\begin{figure}[htp]
    \centering
    \captionsetup[subfigure]{labelformat=empty, font=scriptsize}
    \setlength{\tabcolsep}{3pt}
    
    \begin{subfigure}[b]{0.32\textwidth}
        \centering\includegraphics[width=\textwidth]{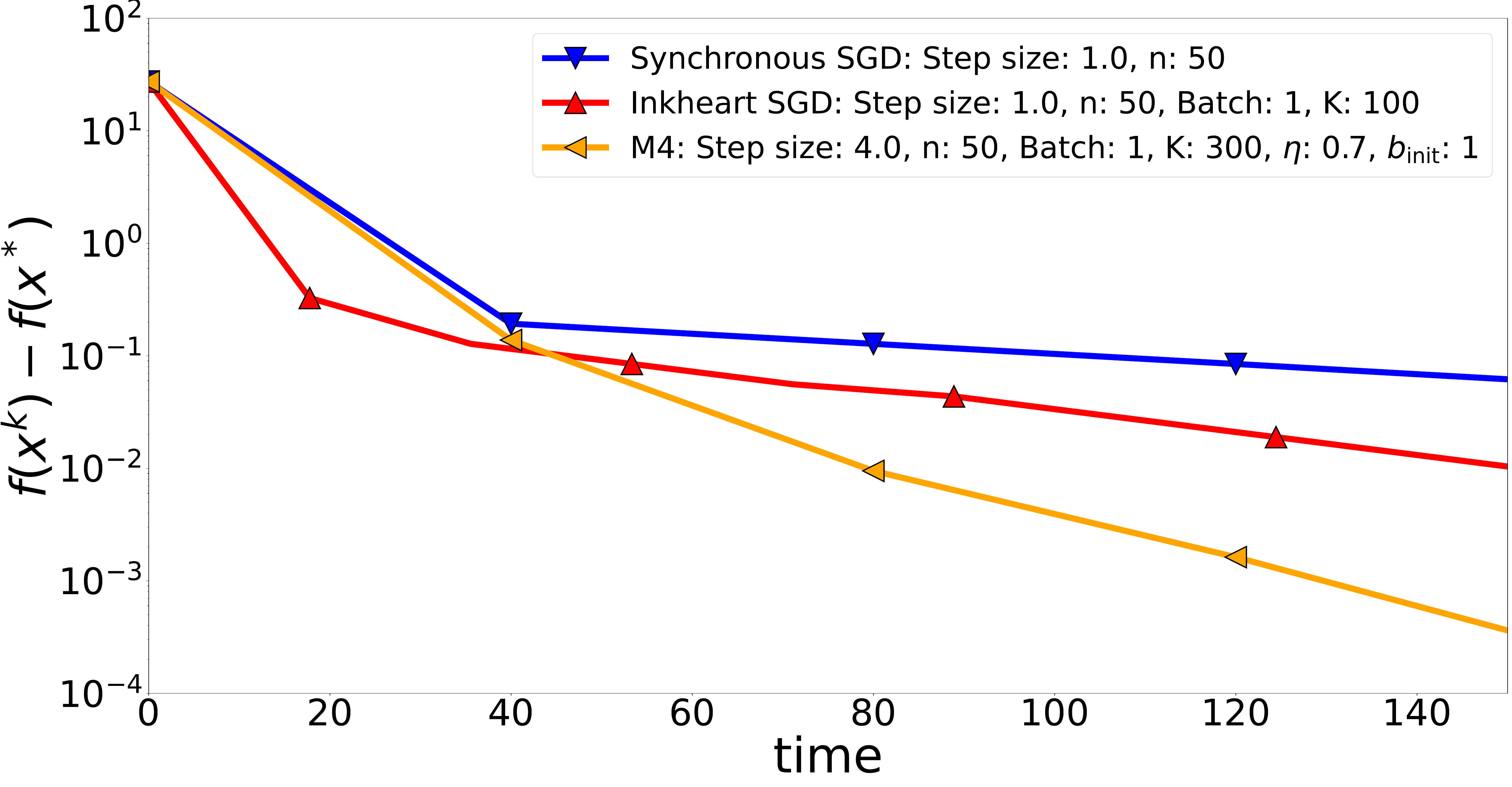}
        \caption{$h=0$}
    \end{subfigure}\hfill
    \begin{subfigure}[b]{0.32\textwidth}
        \centering\includegraphics[width=\textwidth]{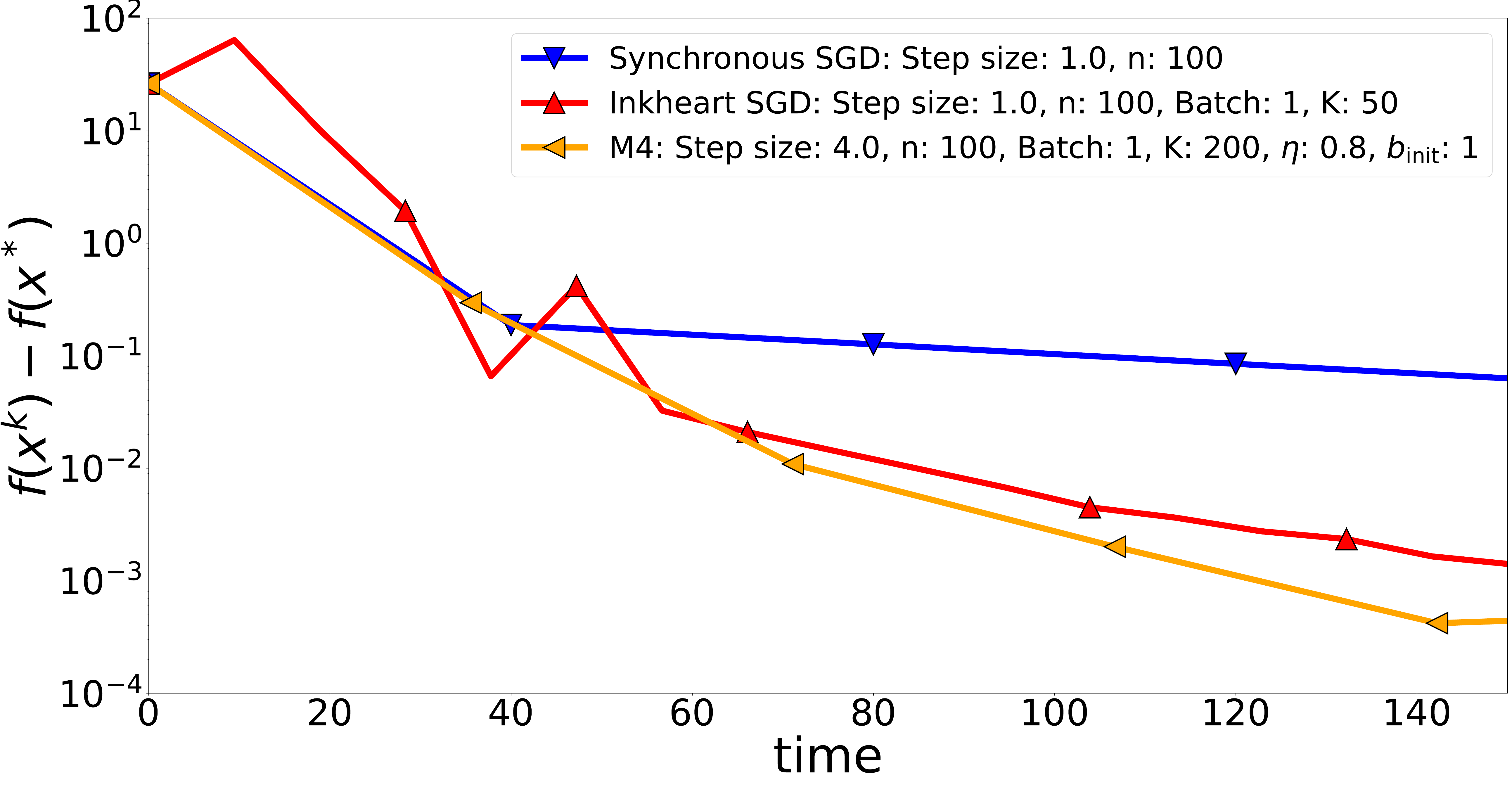}
        \caption{$h=0$}
    \end{subfigure}\hfill
    \begin{subfigure}[b]{0.32\textwidth}
        \centering\includegraphics[width=\textwidth]{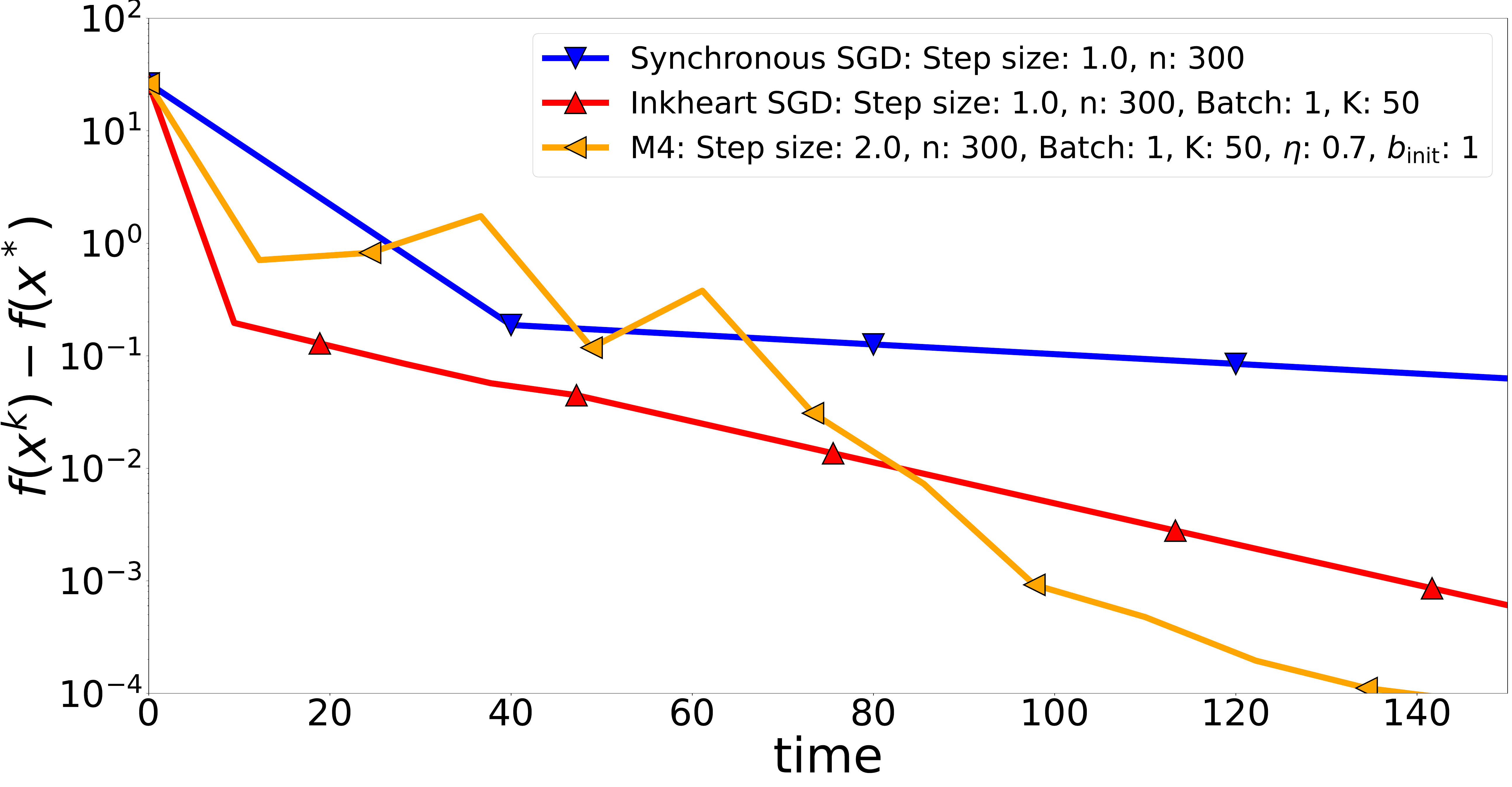}
        \caption{$h=0$}
    \end{subfigure}
    
    \vspace{0.25cm}
    
    \begin{subfigure}[b]{0.32\textwidth}
        \centering\includegraphics[width=\textwidth]{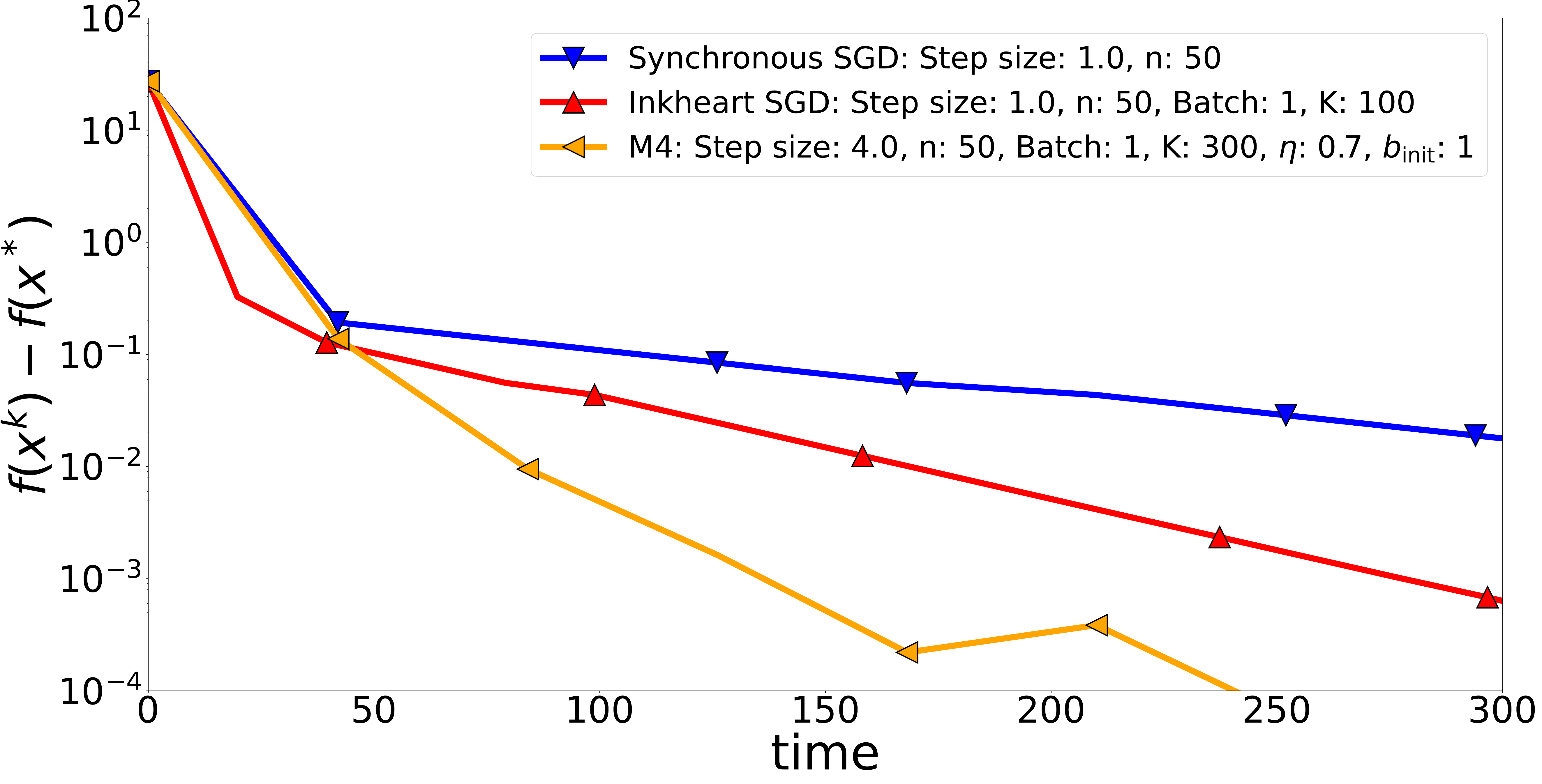}
        \caption{$h=0.1$}
    \end{subfigure}\hfill
    \begin{subfigure}[b]{0.32\textwidth}
        \centering\includegraphics[width=\textwidth]{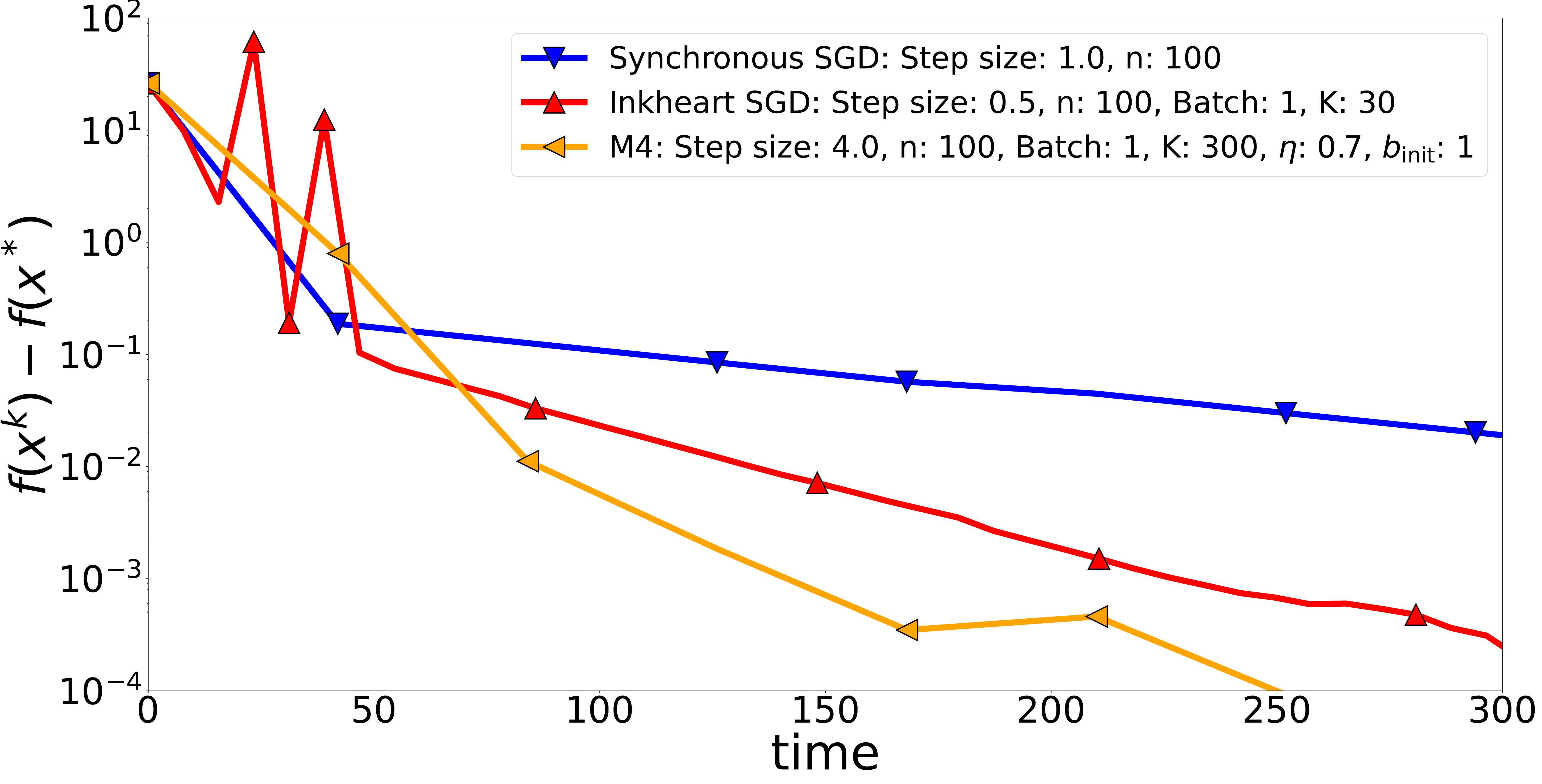}
        \caption{$h=0.1$}
    \end{subfigure}\hfill
    \begin{subfigure}[b]{0.32\textwidth}
        \centering\includegraphics[width=\textwidth]{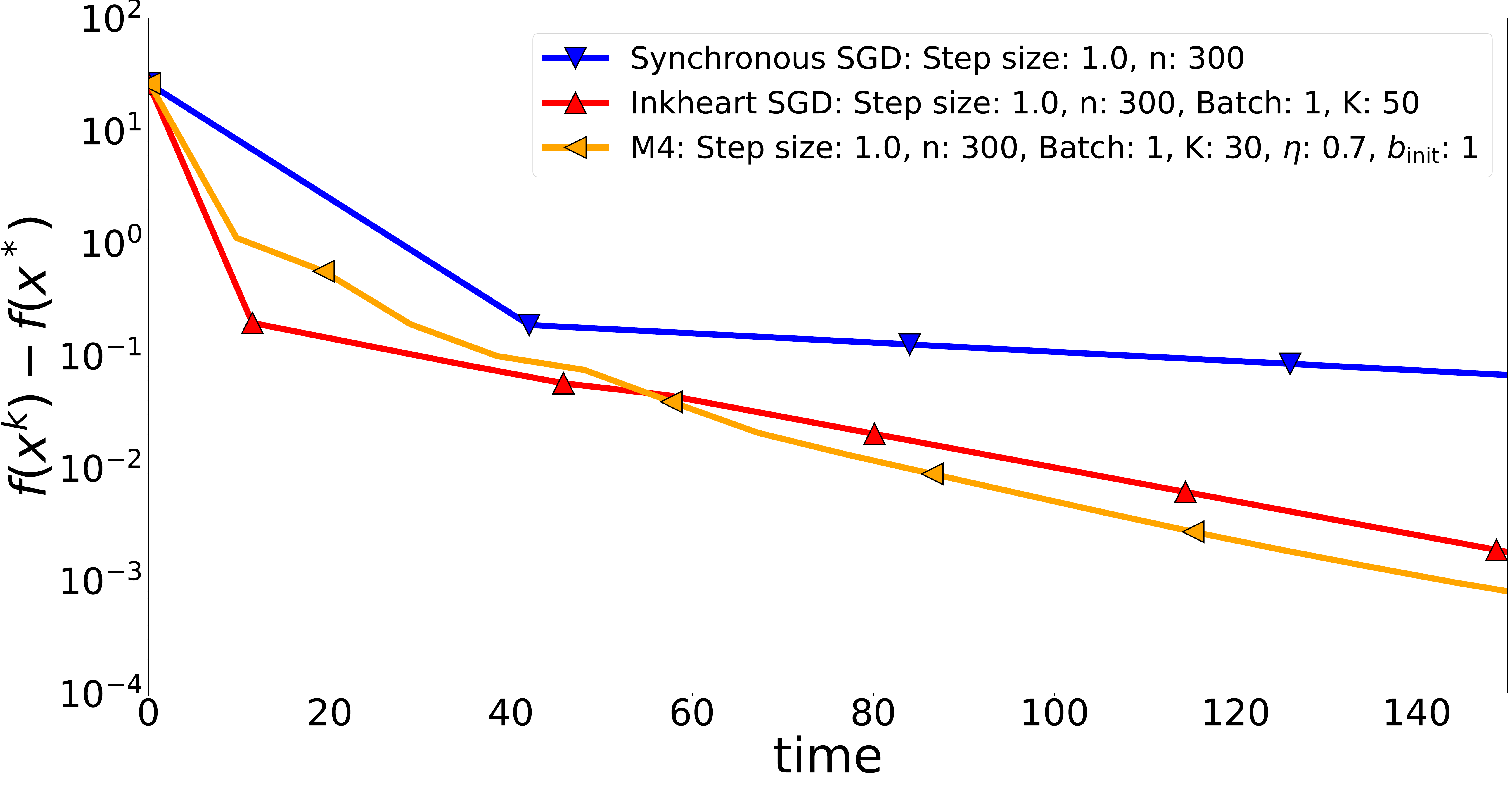}
        \caption{$h=0.1$}
    \end{subfigure}
    
    \vspace{0.25cm}
    
    \begin{subfigure}[b]{0.32\textwidth}
        \centering\includegraphics[width=\textwidth]{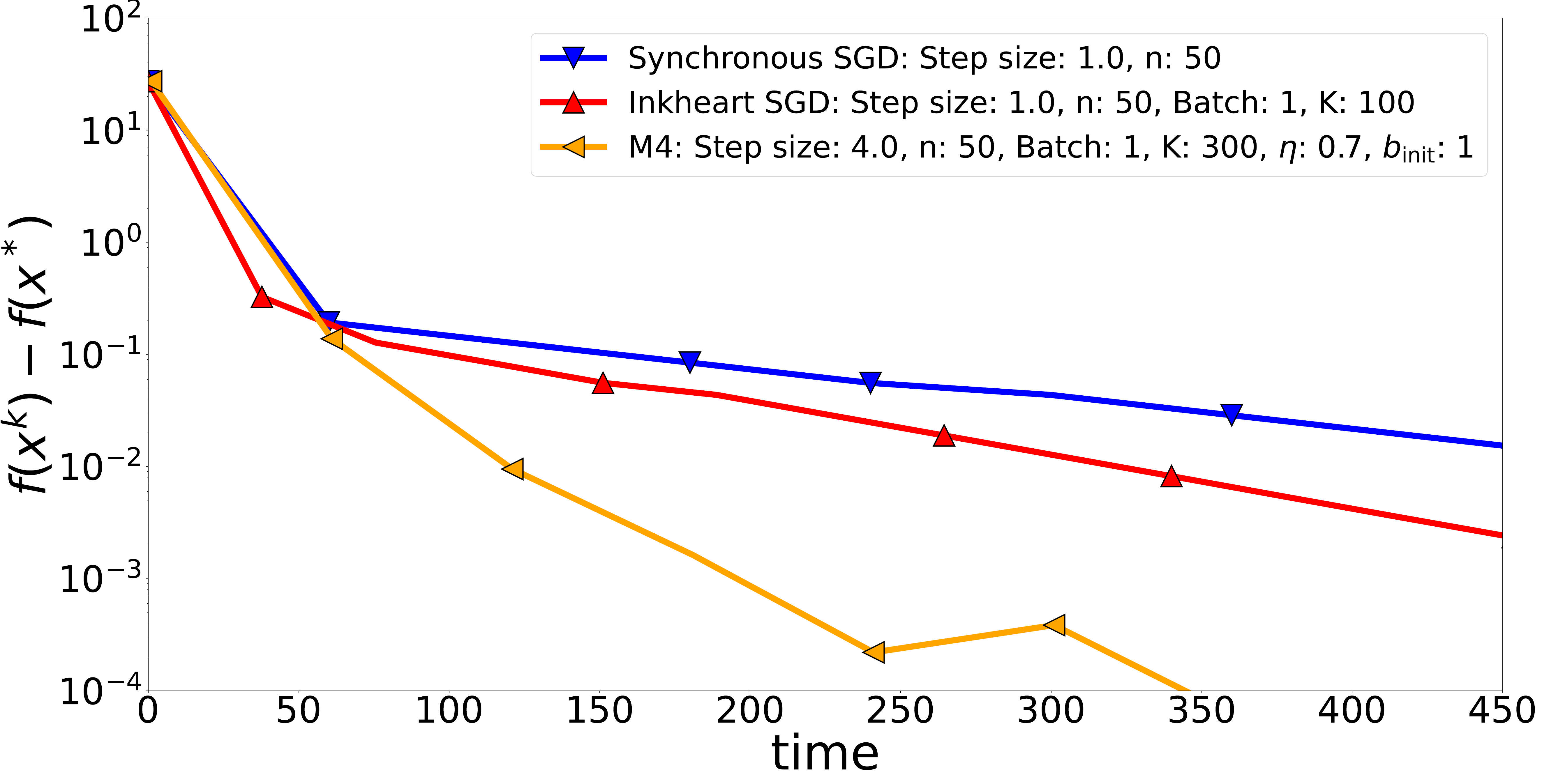}
        \caption{$h=1.0$}
    \end{subfigure}\hfill
    \begin{subfigure}[b]{0.32\textwidth}
        \centering\includegraphics[width=\textwidth]{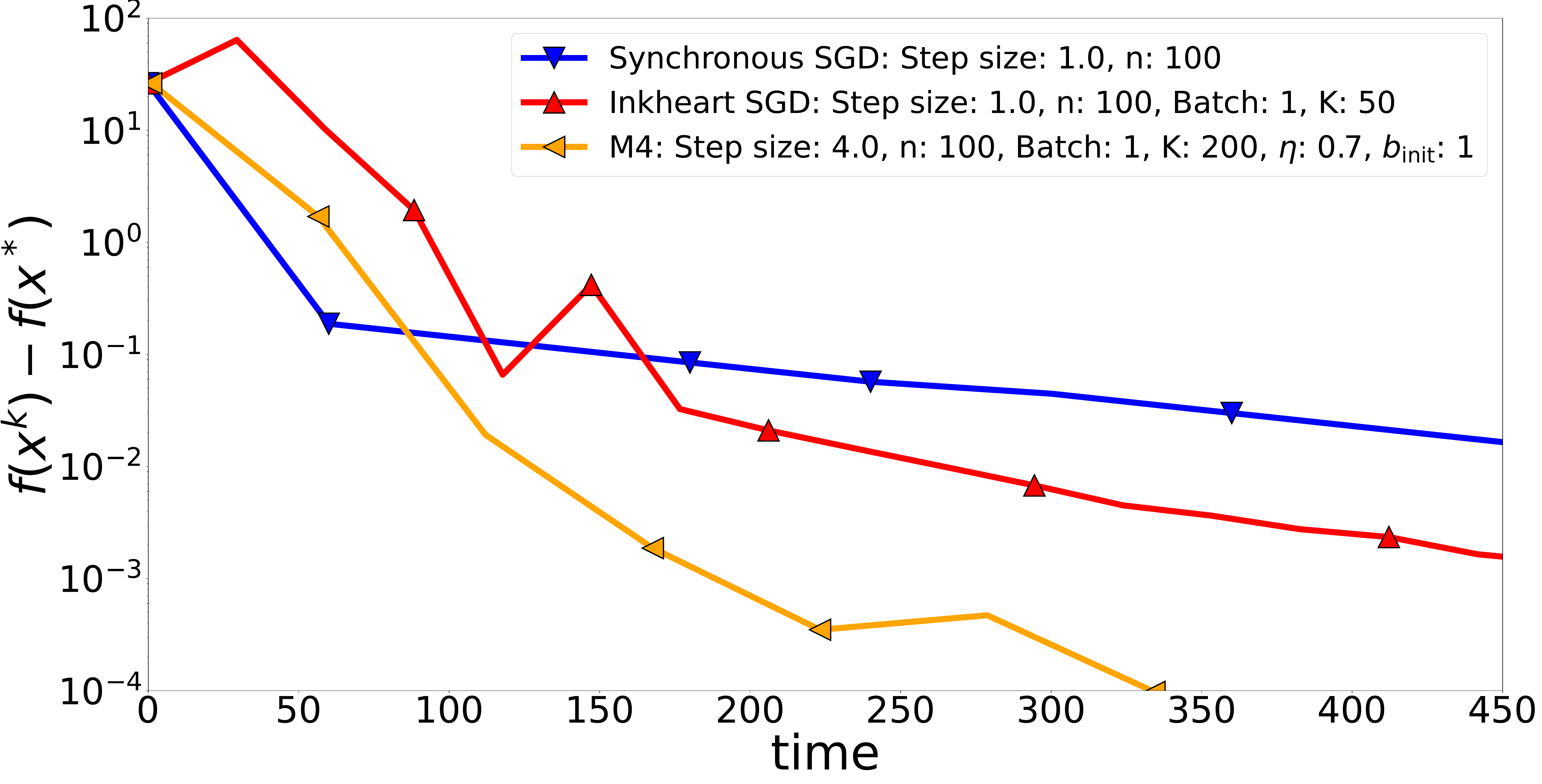}
        \caption{$h=1.0$}
    \end{subfigure}\hfill
    \begin{subfigure}[b]{0.32\textwidth}
        \centering\includegraphics[width=\textwidth]{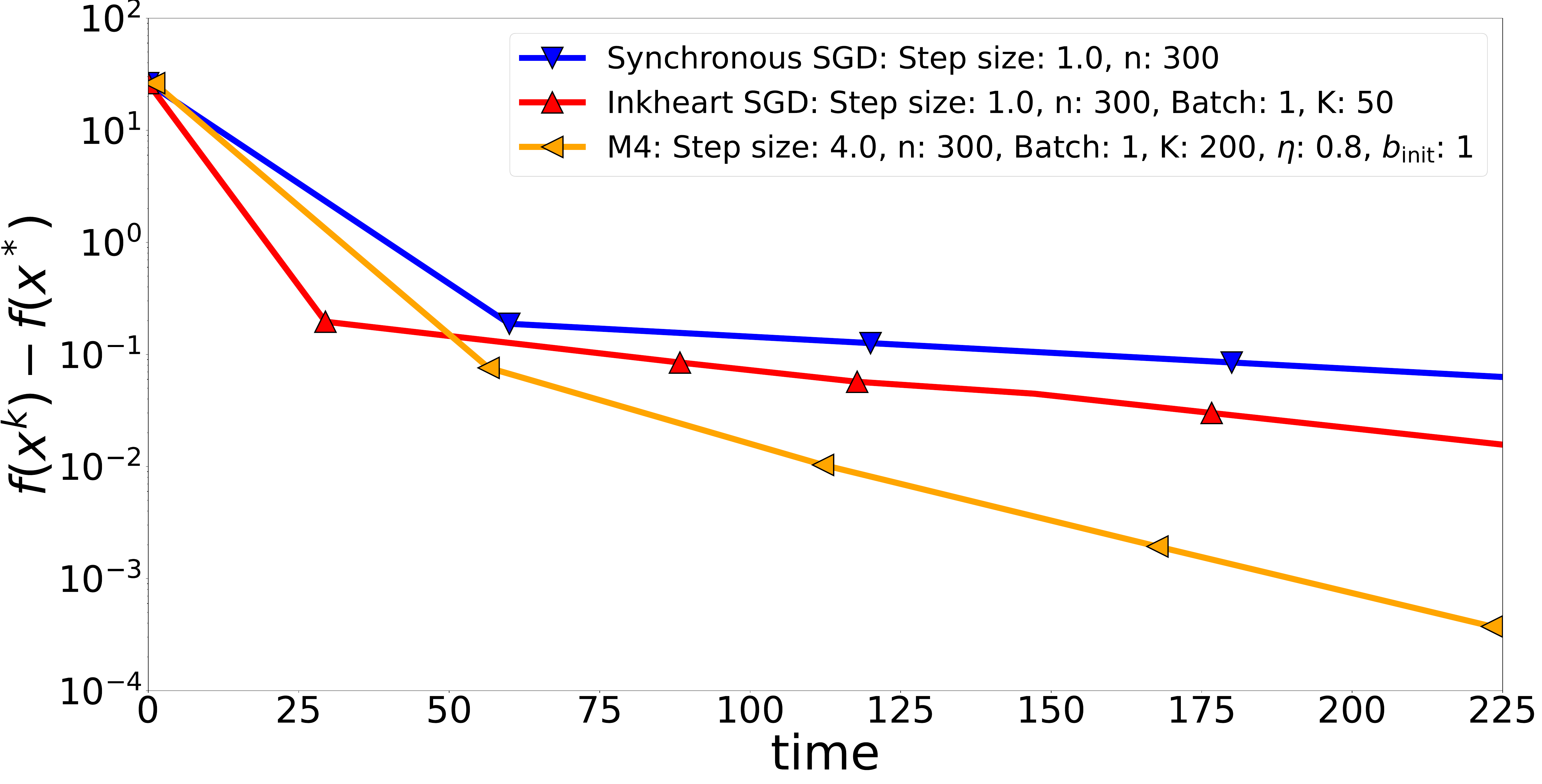}
        \caption{$h=1.0$}
    \end{subfigure}
    
    \caption{Convergence for medium heterogeneity ($\ell = 0.3$) and low noise $\sigma = 0.001$.
    Fixed parameters: $d=300$, $\kappa = \nicefrac{1}{d}$, $\tau = \nicefrac{1}{d}$. 
    Rows vary the gradient computation time $h$; columns correspond to the number of workers $n \in \{50, 100, 300\}$.}
    \label{fig:quad_lip_small_noise_small}
\end{figure}

\begin{figure}[htp]
    \centering
    \captionsetup[subfigure]{labelformat=empty, font=scriptsize}
    \setlength{\tabcolsep}{3pt}
    
    \begin{subfigure}[b]{0.32\textwidth}
        \centering\includegraphics[width=\textwidth]{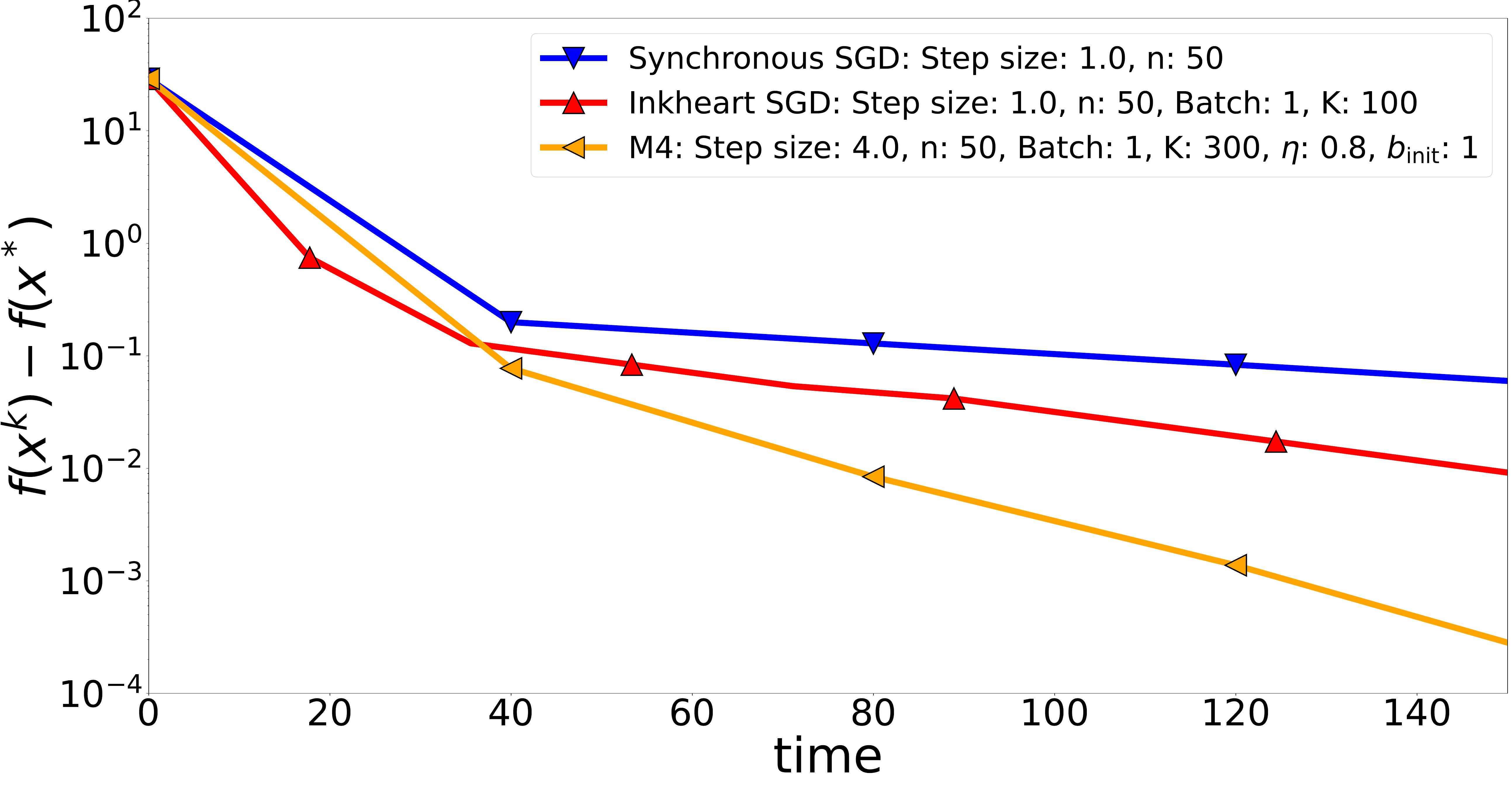}
        \caption{$h=0$}
    \end{subfigure}\hfill
    \begin{subfigure}[b]{0.32\textwidth}
        \centering\includegraphics[width=\textwidth]{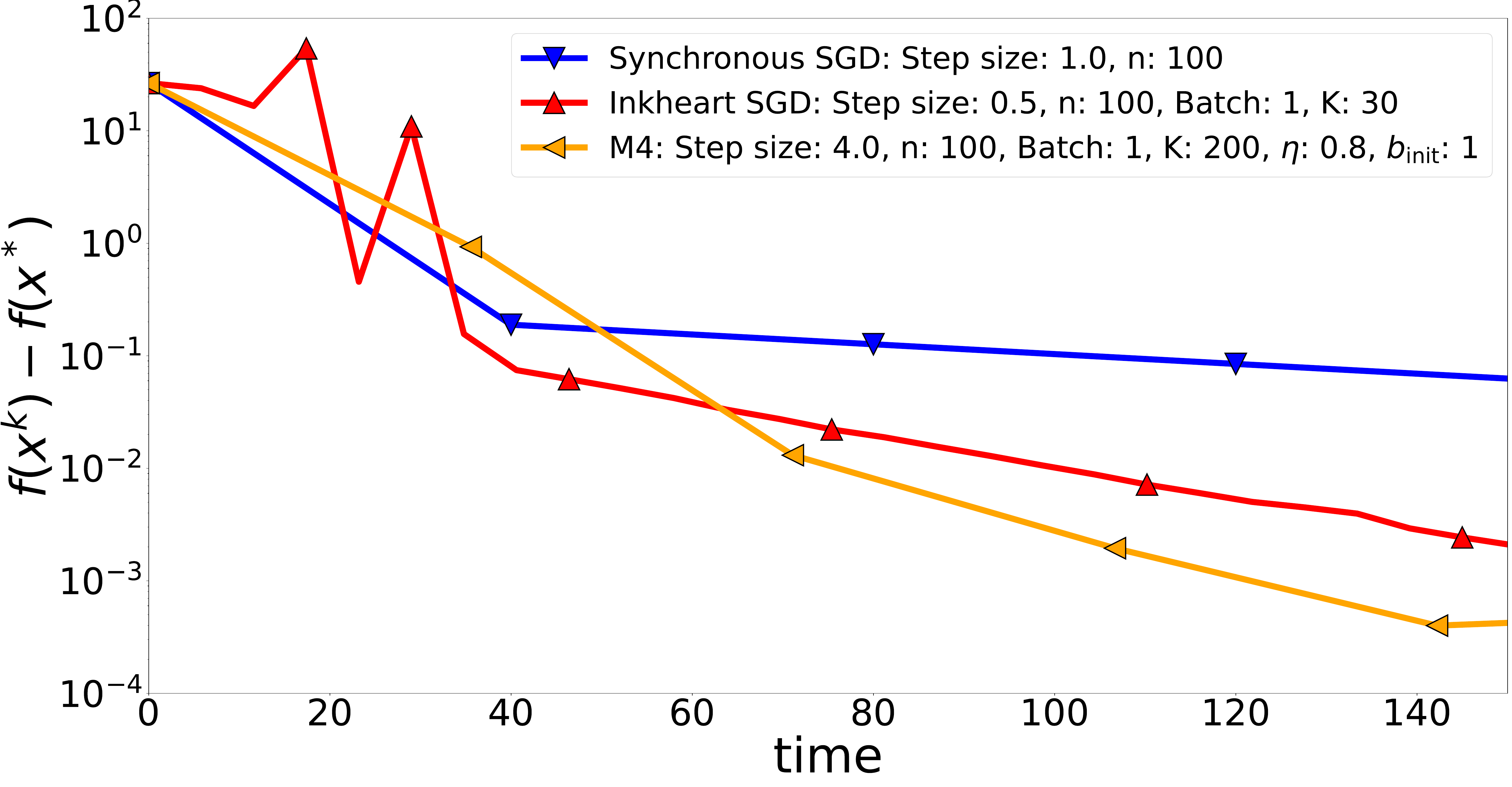}
        \caption{$h=0$}
    \end{subfigure}\hfill
    \begin{subfigure}[b]{0.32\textwidth}
        \centering\includegraphics[width=\textwidth]{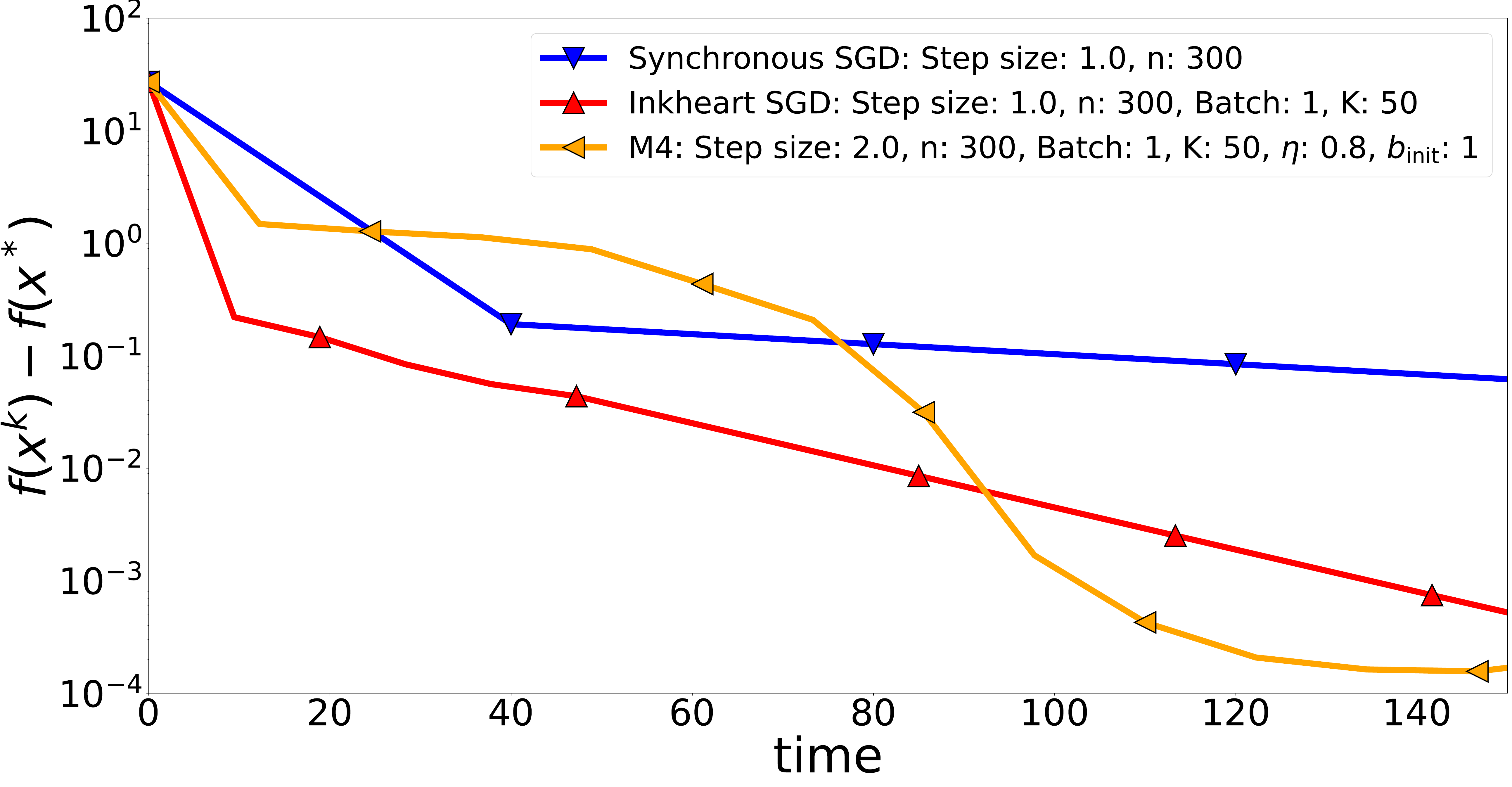}
        \caption{$h=0$}
    \end{subfigure}
    
    \vspace{0.25cm}
    
    \begin{subfigure}[b]{0.32\textwidth}
        \centering\includegraphics[width=\textwidth]{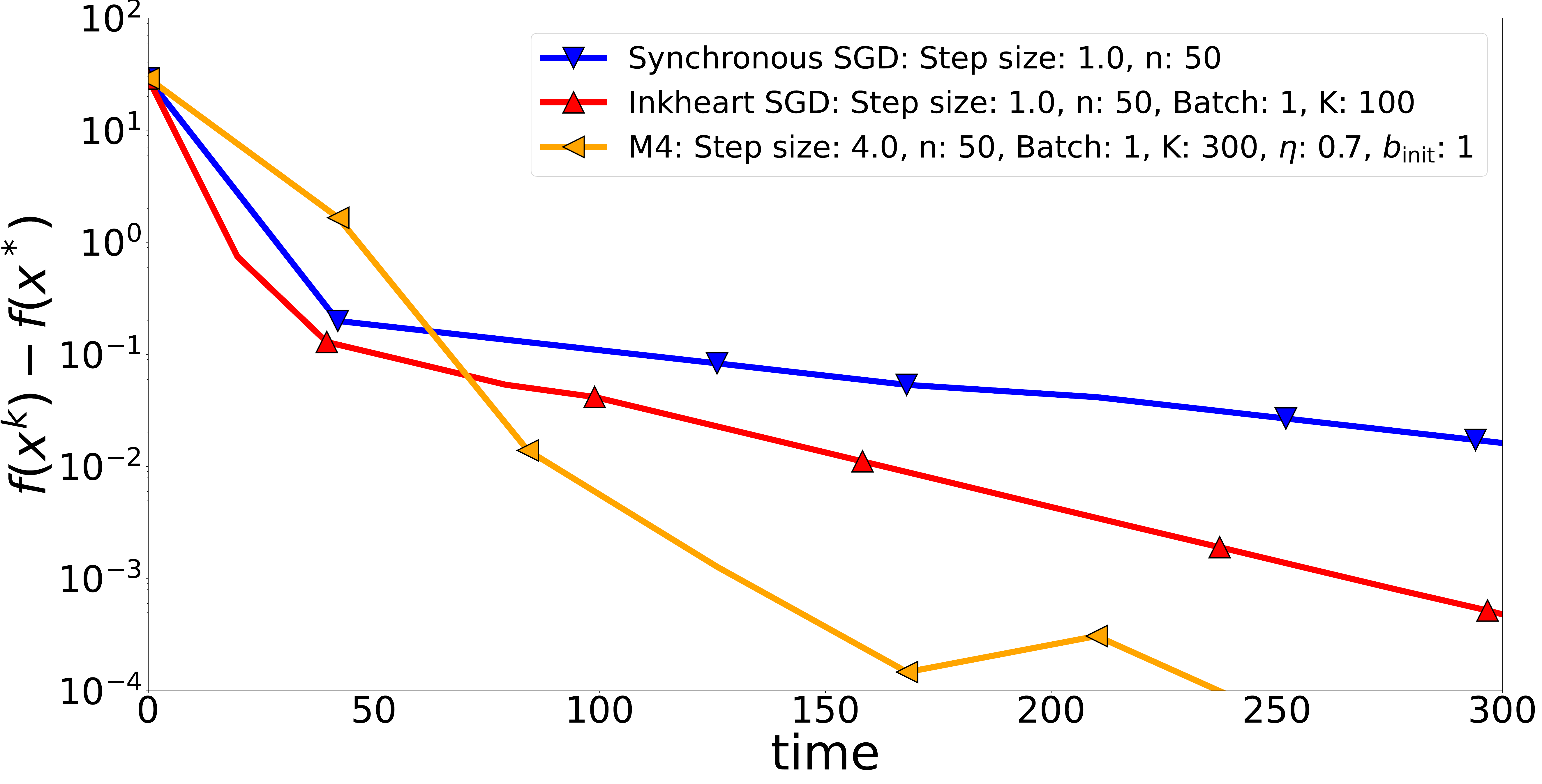}
        \caption{$h=0.1$}
    \end{subfigure}\hfill
    \begin{subfigure}[b]{0.32\textwidth}
        \centering\includegraphics[width=\textwidth]{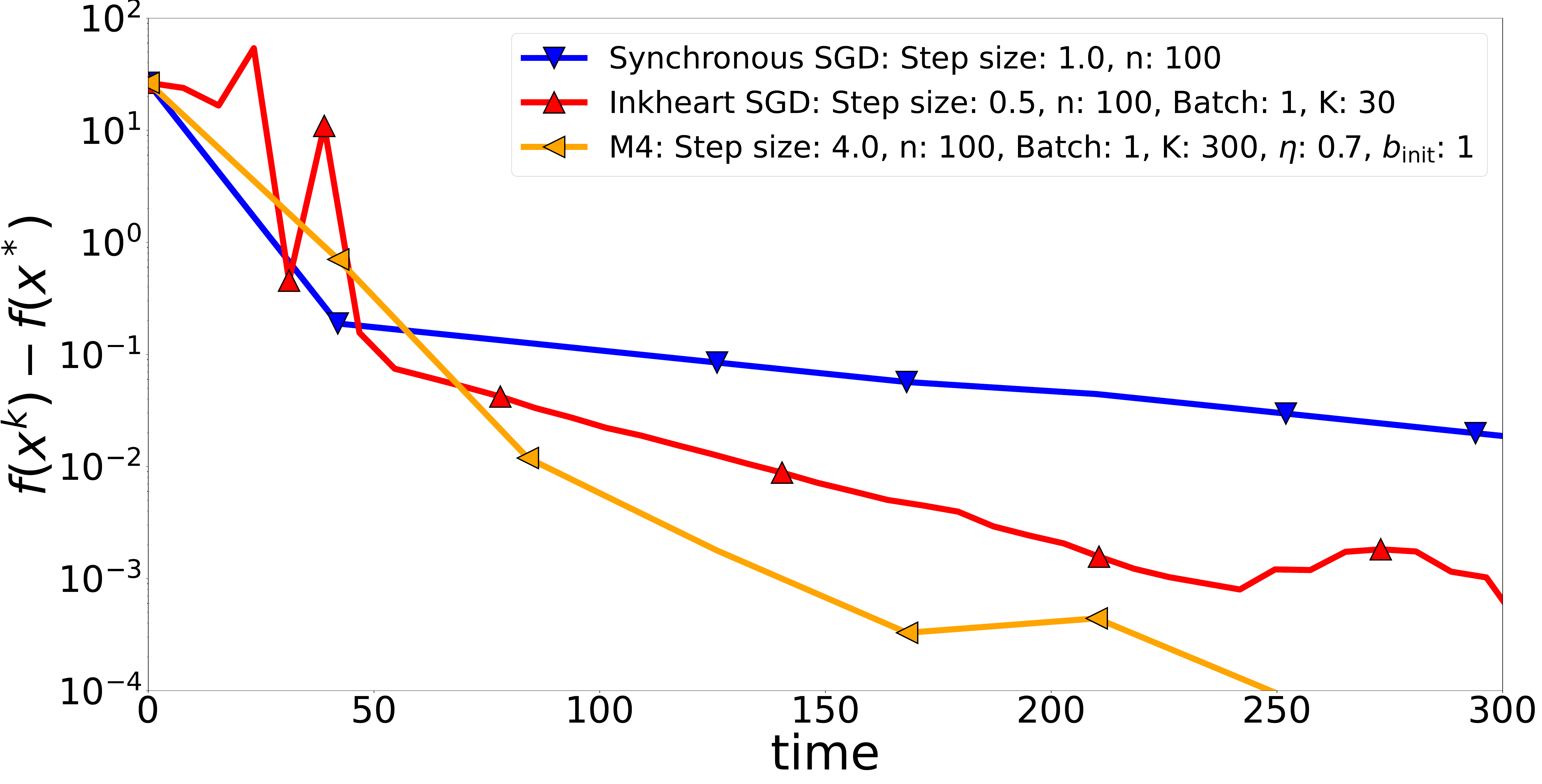}
        \caption{$h=0.1$}
    \end{subfigure}\hfill
    \begin{subfigure}[b]{0.32\textwidth}
        \centering\includegraphics[width=\textwidth]{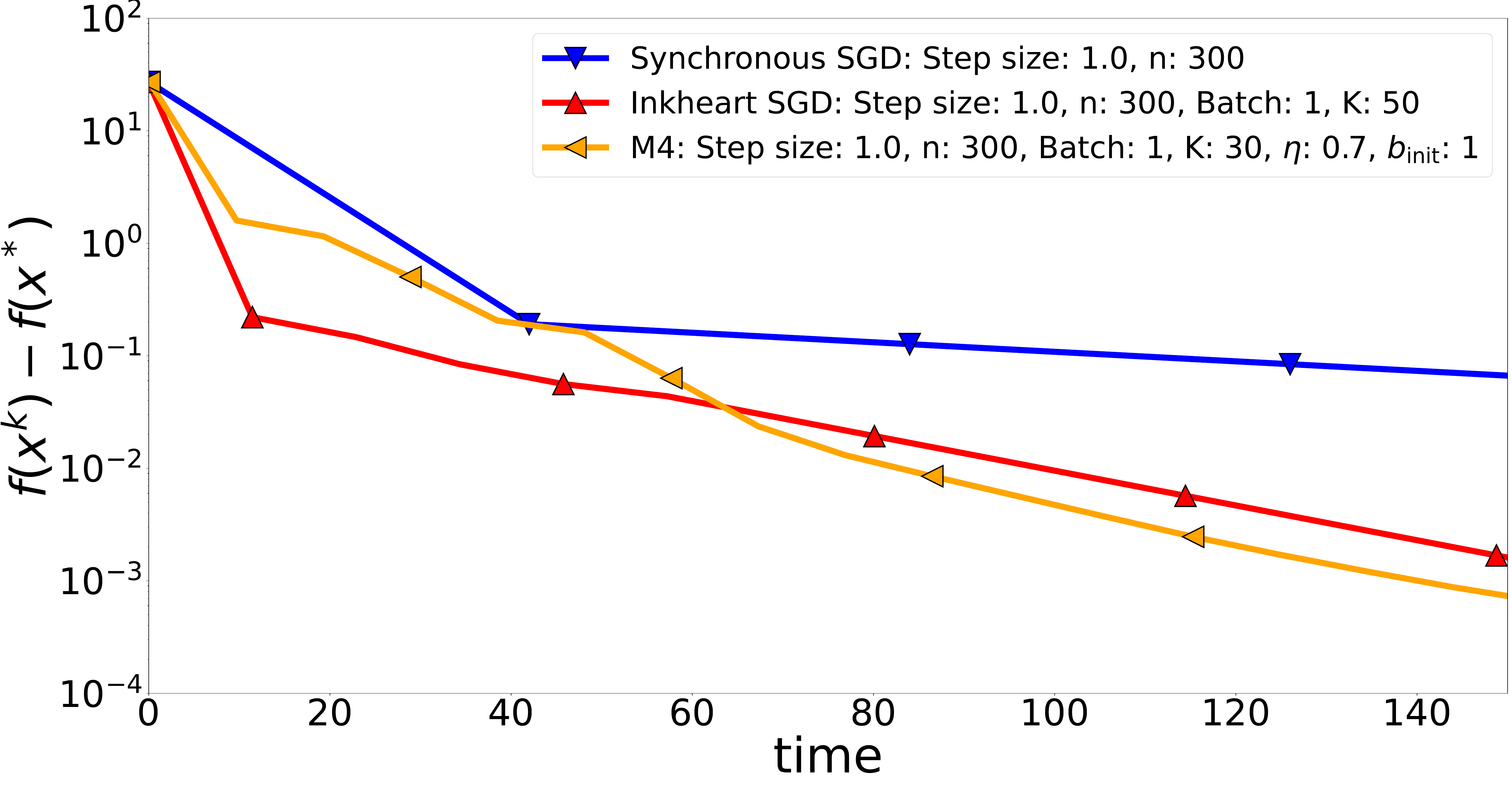}
        \caption{$h=0.1$}
    \end{subfigure}
    
    \vspace{0.25cm}
    
    \begin{subfigure}[b]{0.32\textwidth}
        \centering\includegraphics[width=\textwidth]{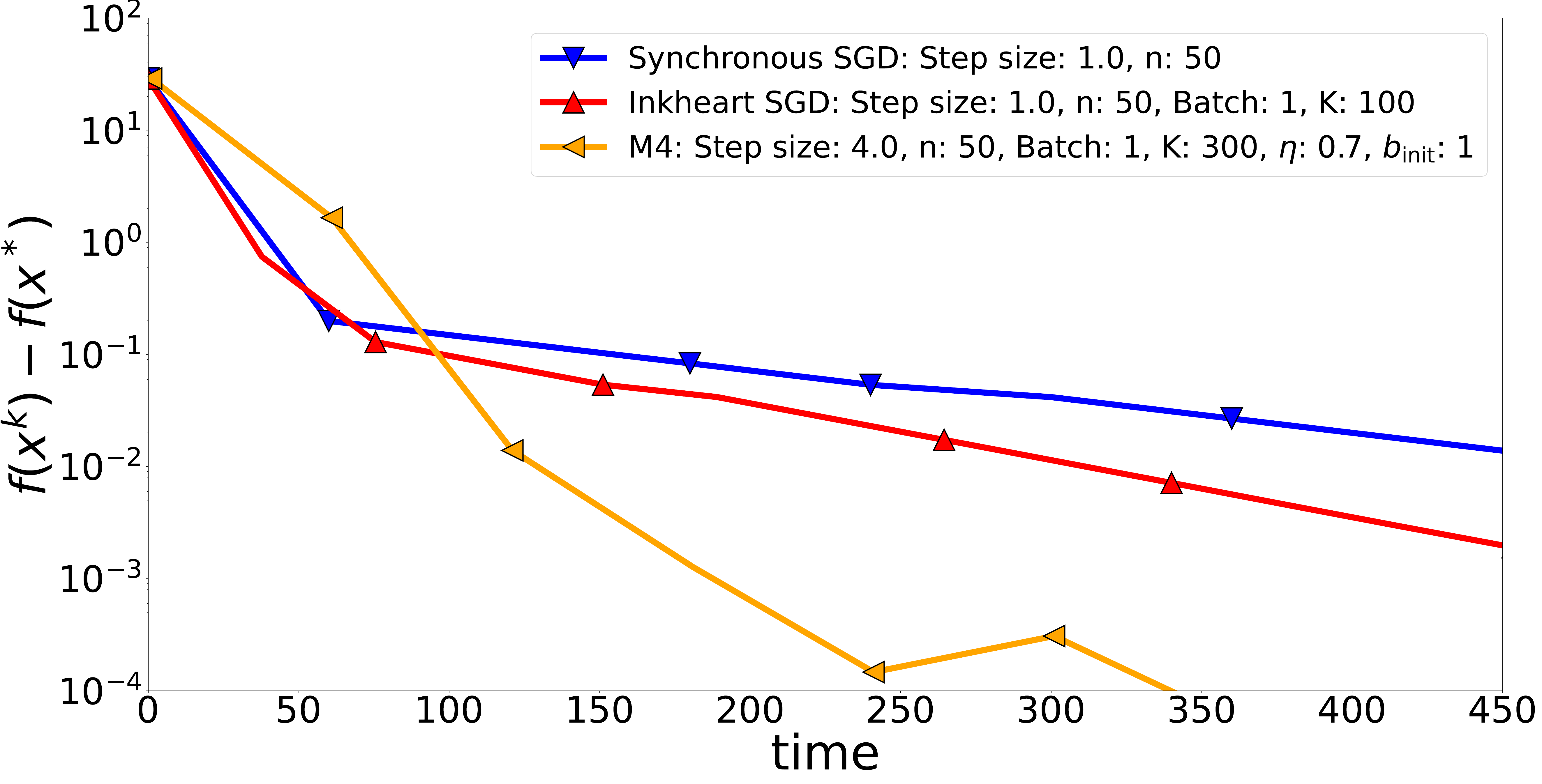}
        \caption{$h=1.0$}
    \end{subfigure}\hfill
    \begin{subfigure}[b]{0.32\textwidth}
        \centering\includegraphics[width=\textwidth]{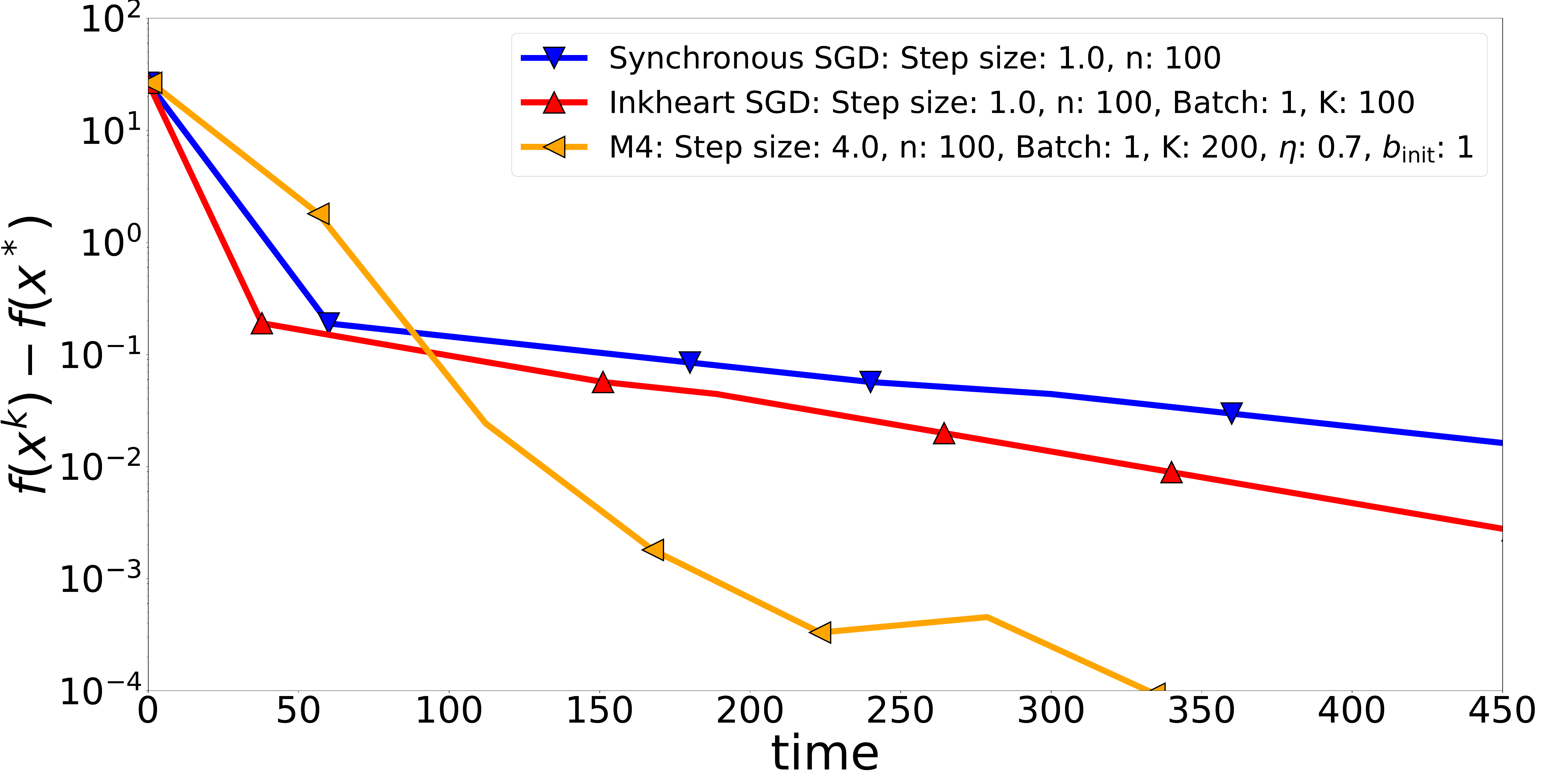}
        \caption{$h=1.0$}
    \end{subfigure}\hfill
    \begin{subfigure}[b]{0.32\textwidth}
        \centering\includegraphics[width=\textwidth]{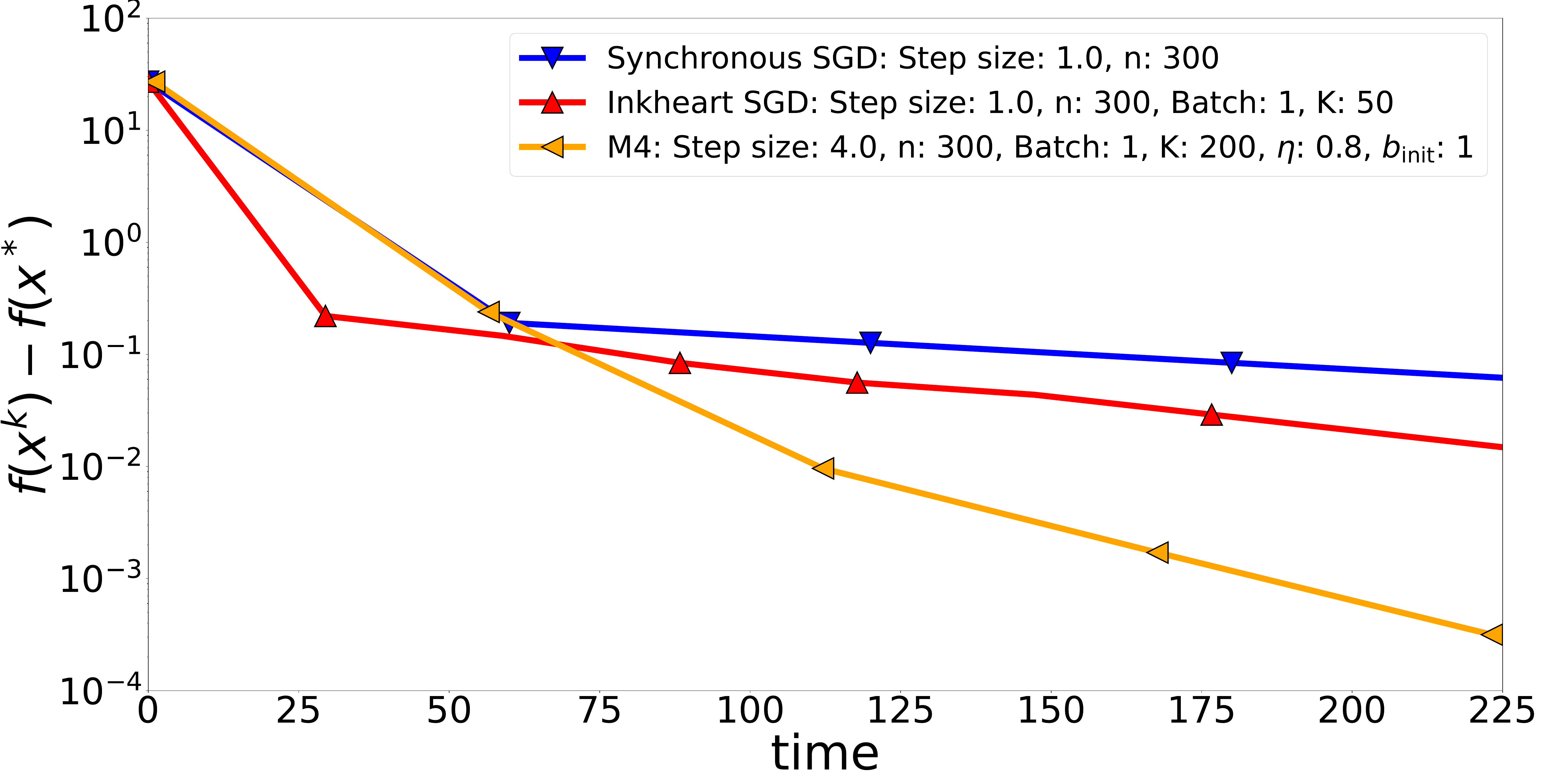}
        \caption{$h=1.0$}
    \end{subfigure}
    
    \caption{Convergence for high heterogeneity ($\ell = 0.5$) and low noise $\sigma = 0.001$.
    Fixed parameters: $d=300$, $\kappa = \nicefrac{1}{d}$, $\tau = \nicefrac{1}{d}$. 
    Rows vary the gradient computation time $h$; columns correspond to the number of workers $n \in \{50, 100, 300\}$.}
    \label{fig:quad_lip_large_noise_small}
\end{figure}

\begin{figure}[htp]
    \centering
    \captionsetup[subfigure]{labelformat=empty, font=scriptsize}
    \setlength{\tabcolsep}{3pt}
    
    \begin{subfigure}[b]{0.32\textwidth}
        \centering\includegraphics[width=\textwidth]{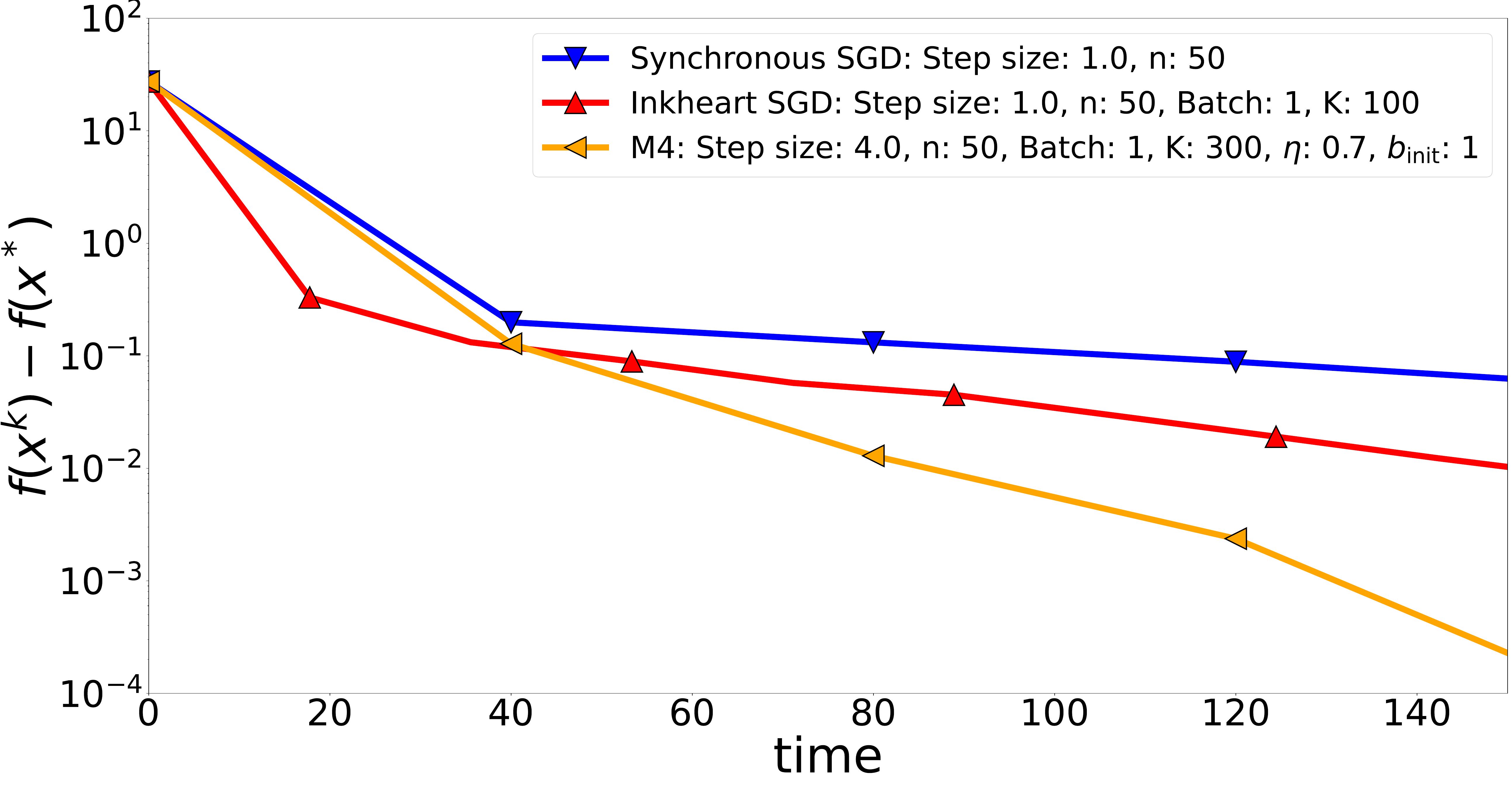}
        \caption{$h=0$}
    \end{subfigure}\hfill
    \begin{subfigure}[b]{0.32\textwidth}
        \centering\includegraphics[width=\textwidth]{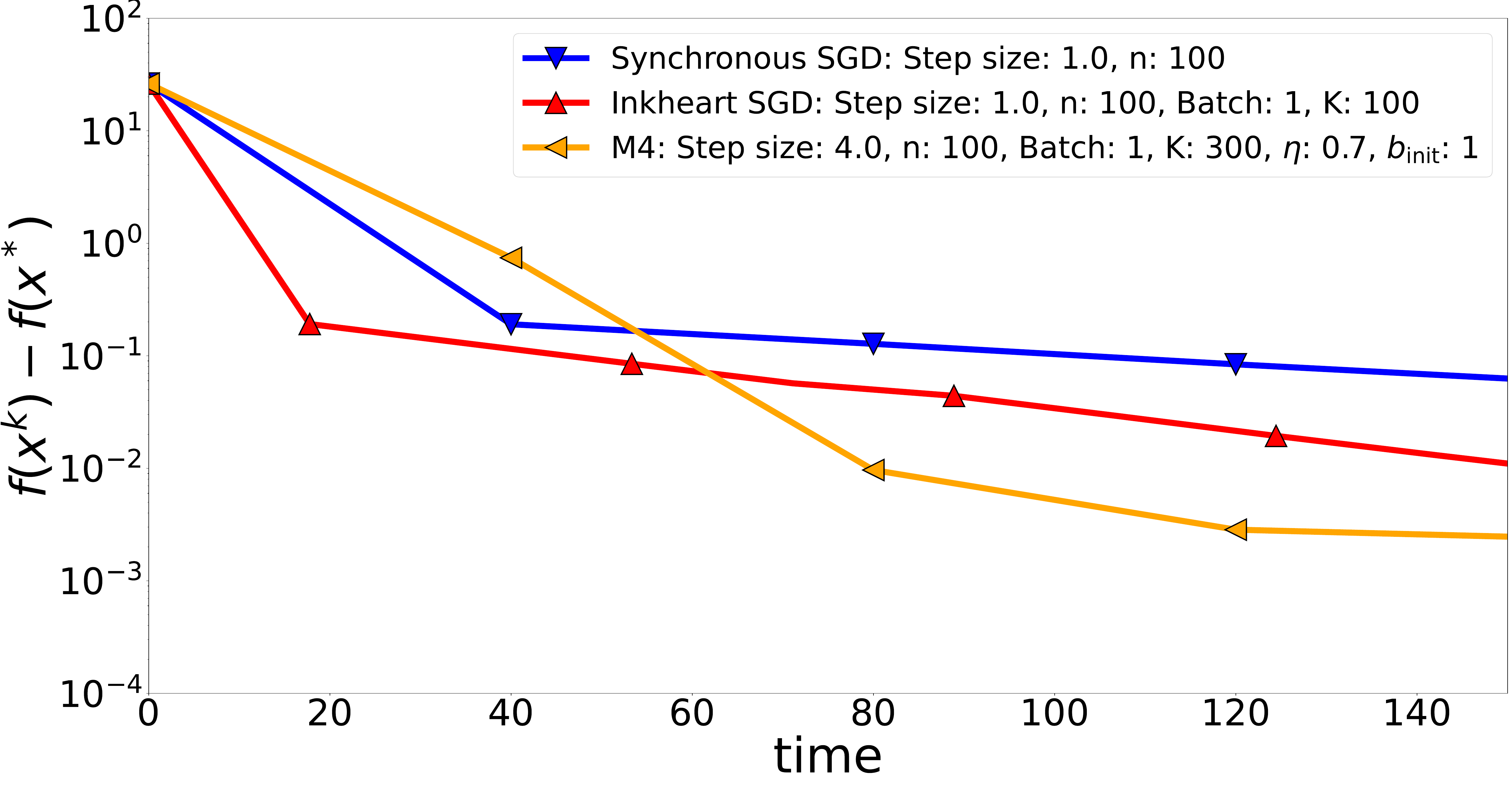}
        \caption{$h=0$}
    \end{subfigure}\hfill
    \begin{subfigure}[b]{0.32\textwidth}
        \centering\includegraphics[width=\textwidth]{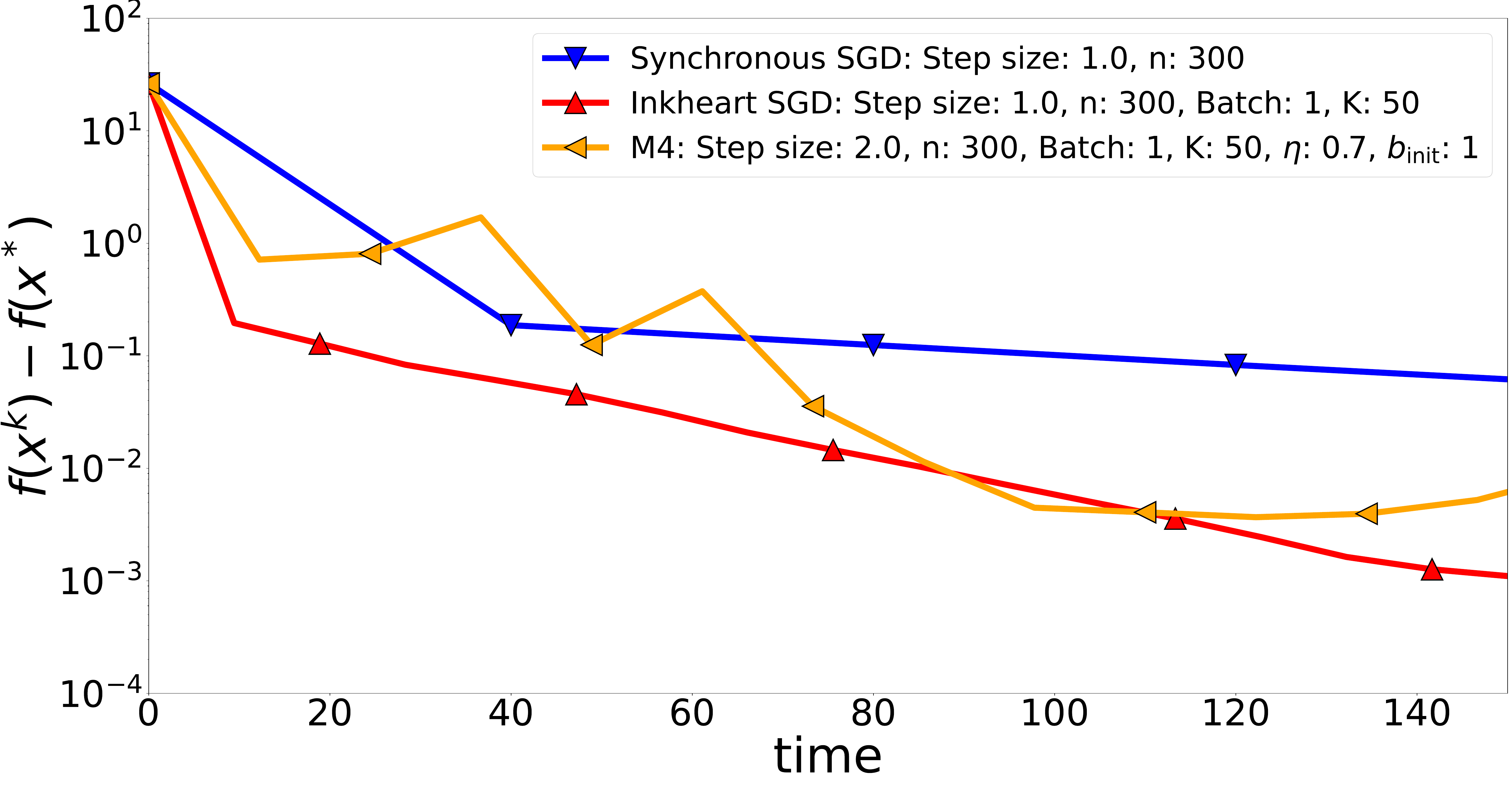}
        \caption{$h=0$}
    \end{subfigure}
    
    \vspace{0.25cm}
    
    \begin{subfigure}[b]{0.32\textwidth}
        \centering\includegraphics[width=\textwidth]{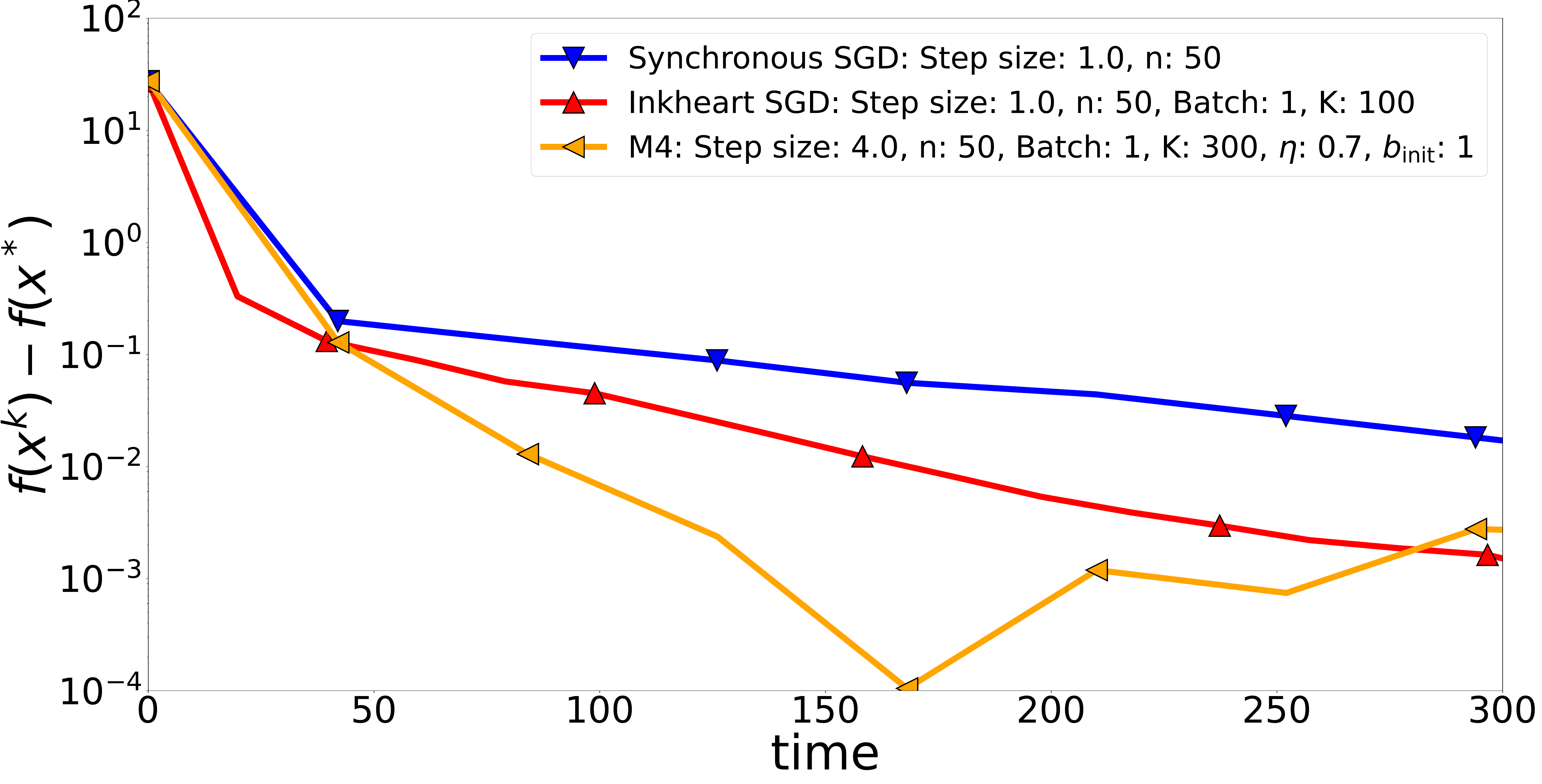}
        \caption{$h=0.1$}
    \end{subfigure}\hfill
    \begin{subfigure}[b]{0.32\textwidth}
        \centering\includegraphics[width=\textwidth]{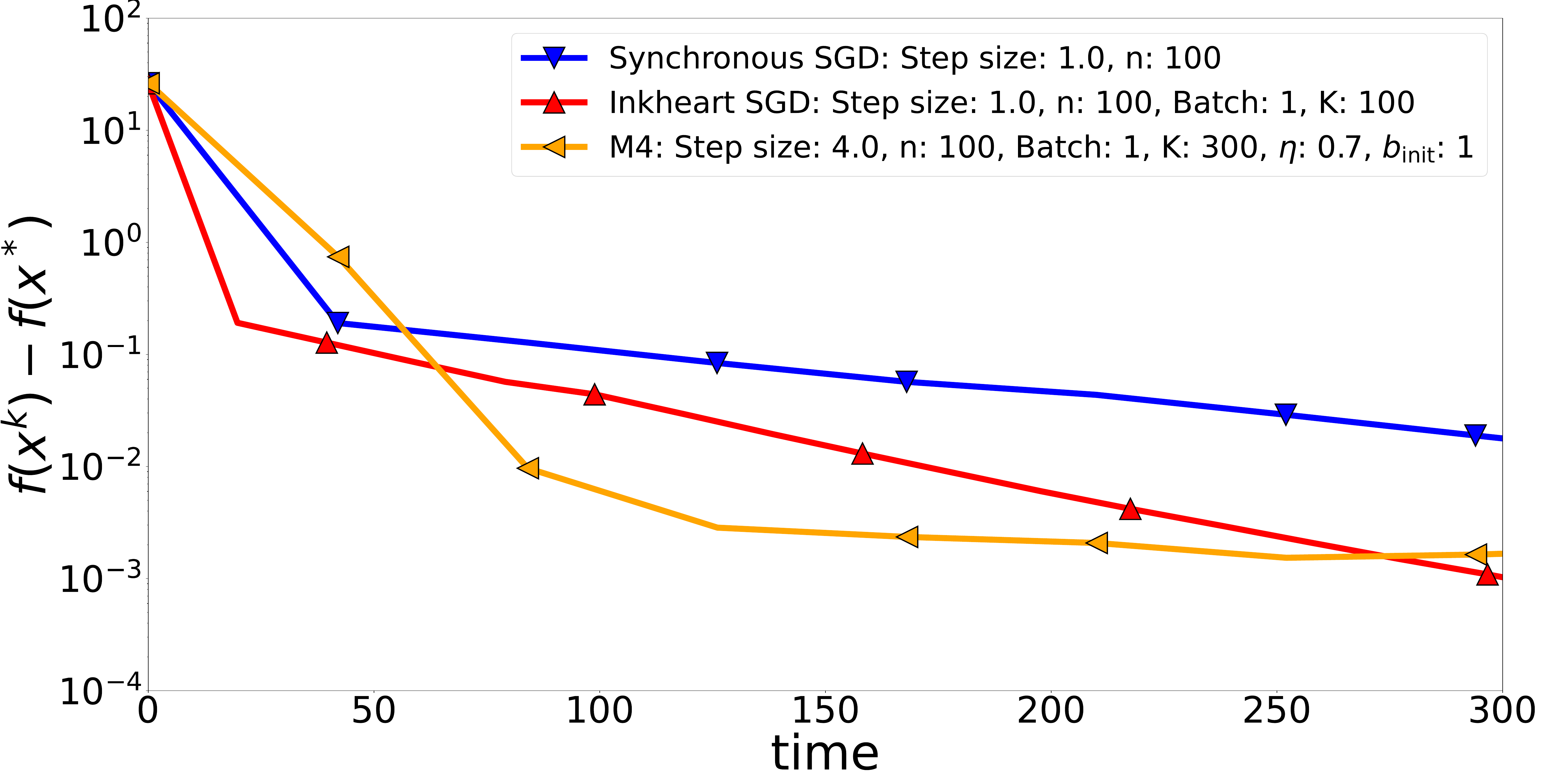}
        \caption{$h=0.1$}
    \end{subfigure}\hfill
    \begin{subfigure}[b]{0.32\textwidth}
        \centering\includegraphics[width=\textwidth]{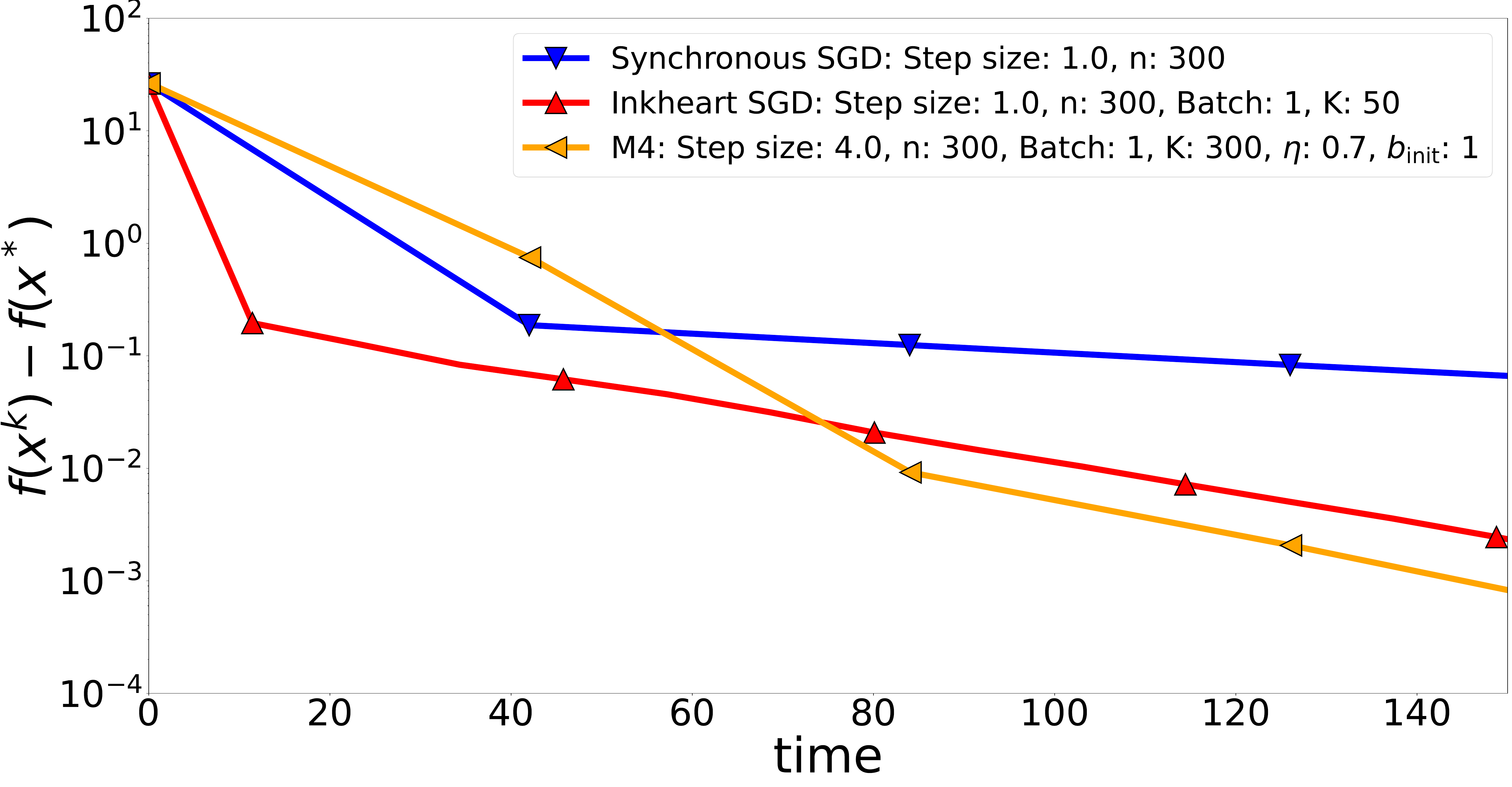}
        \caption{$h=0.1$}
    \end{subfigure}
    
    \vspace{0.25cm}
    
    \begin{subfigure}[b]{0.32\textwidth}
        \centering\includegraphics[width=\textwidth]{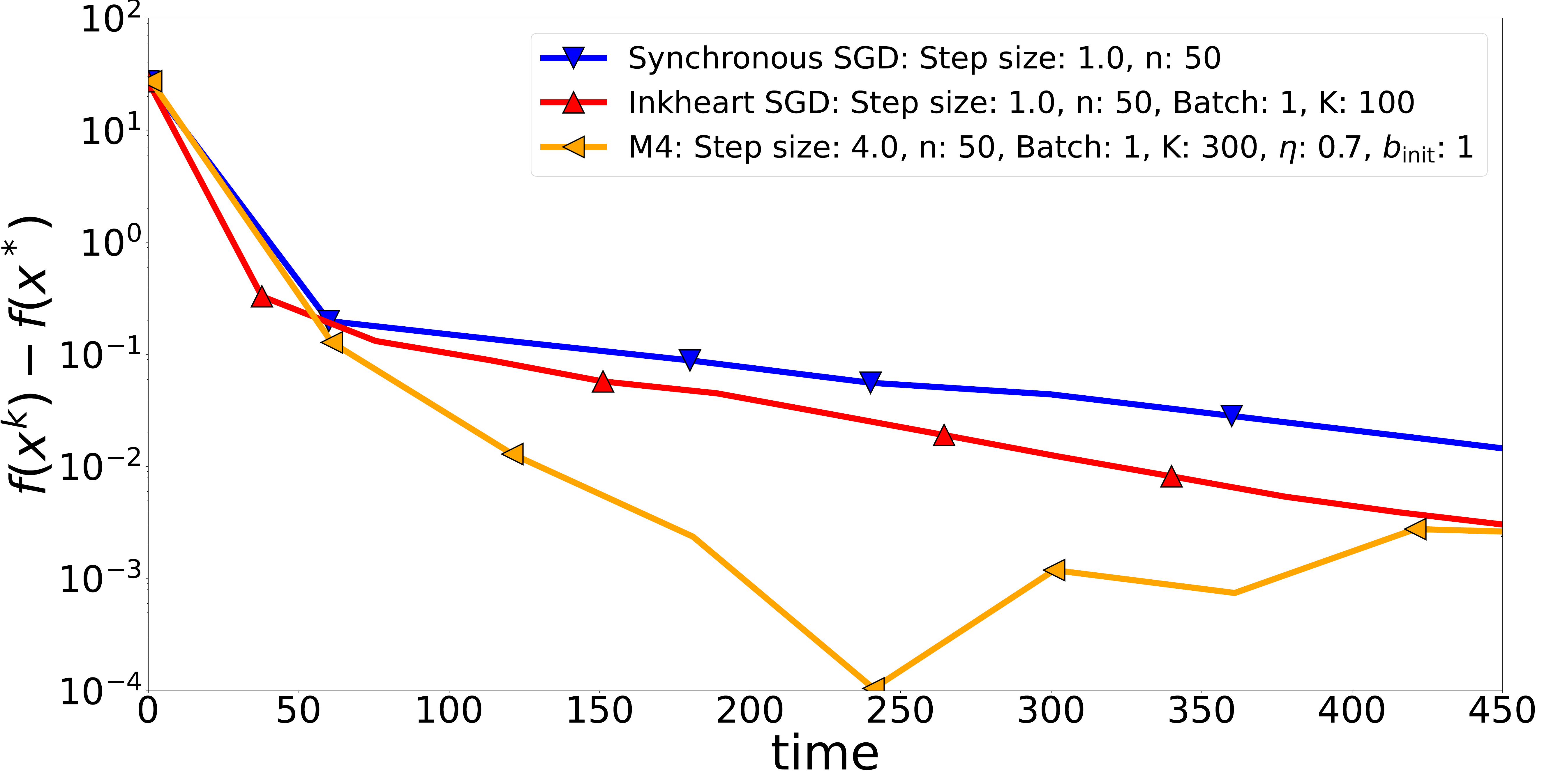}
        \caption{$h=1.0$}
    \end{subfigure}\hfill
    \begin{subfigure}[b]{0.32\textwidth}
        \centering\includegraphics[width=\textwidth]{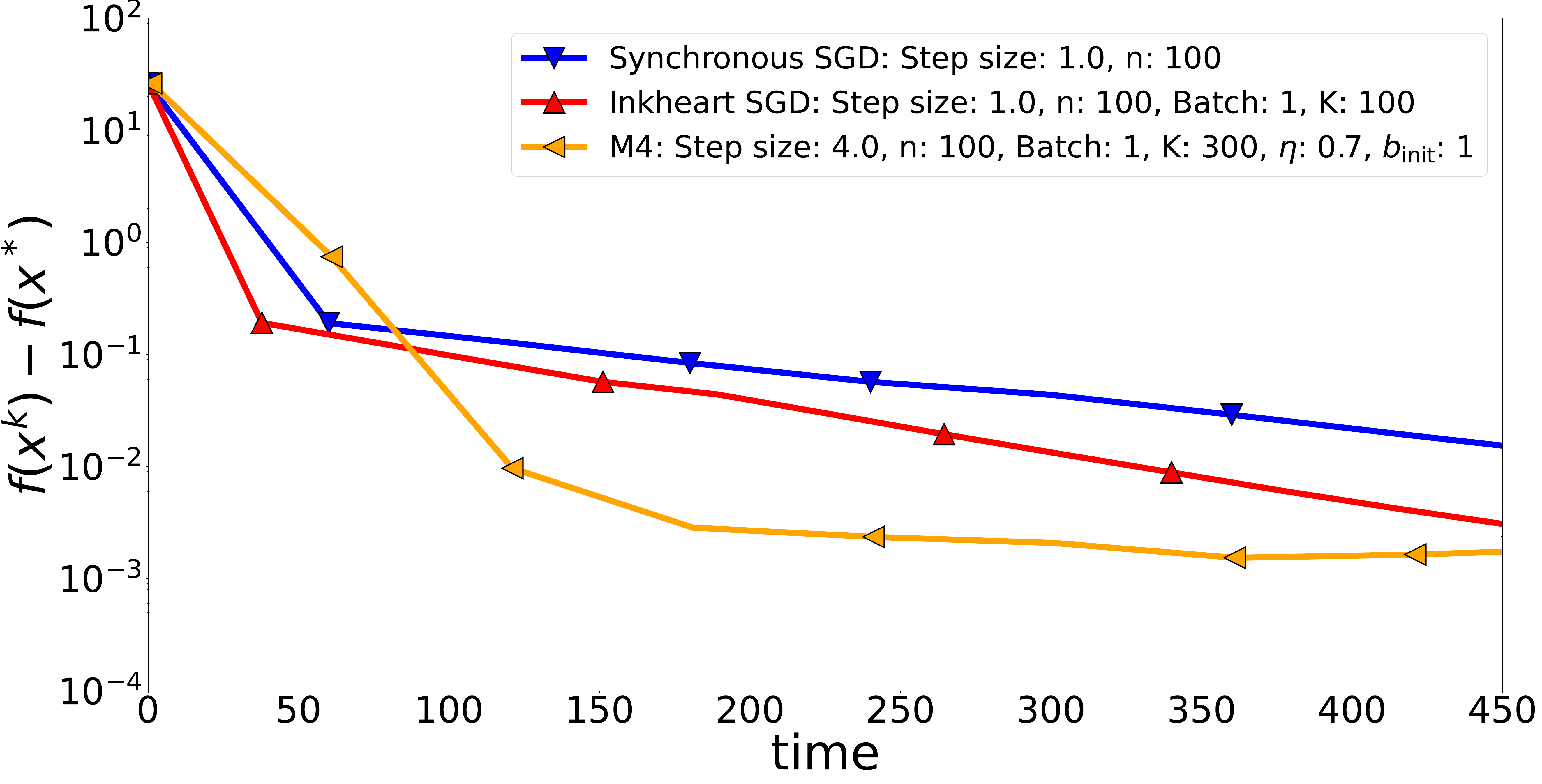}
        \caption{$h=1.0$}
    \end{subfigure}\hfill
    \begin{subfigure}[b]{0.32\textwidth}
        \centering\includegraphics[width=\textwidth]{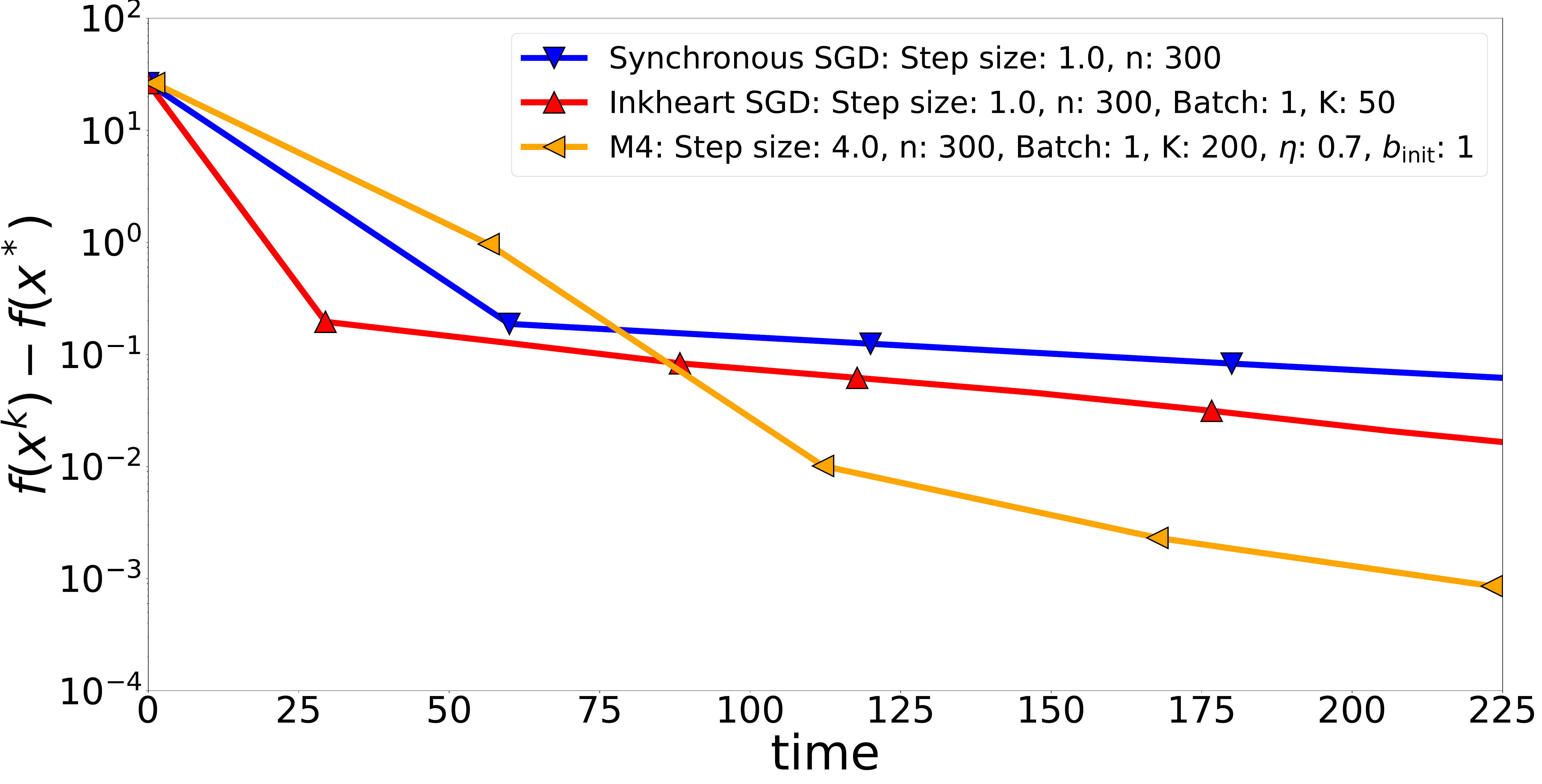}
        \caption{$h=1.0$}
    \end{subfigure}
    
    \caption{Convergence for medium heterogeneity ($\ell = 0.3$) and medium noise $\sigma = 0.01$.
    Fixed parameters: $d=300$, $\kappa = \nicefrac{1}{d}$, $\tau = \nicefrac{1}{d}$. 
    Rows vary the gradient computation time $h$; columns correspond to the number of workers $n \in \{50, 100, 300\}$.}
    \label{fig:lip_small}
\end{figure}

\begin{figure}[htp]
    \centering
    \captionsetup[subfigure]{labelformat=empty, font=scriptsize}
    \setlength{\tabcolsep}{3pt}
    
    \begin{subfigure}[b]{0.32\textwidth}
        \centering\includegraphics[width=\textwidth]{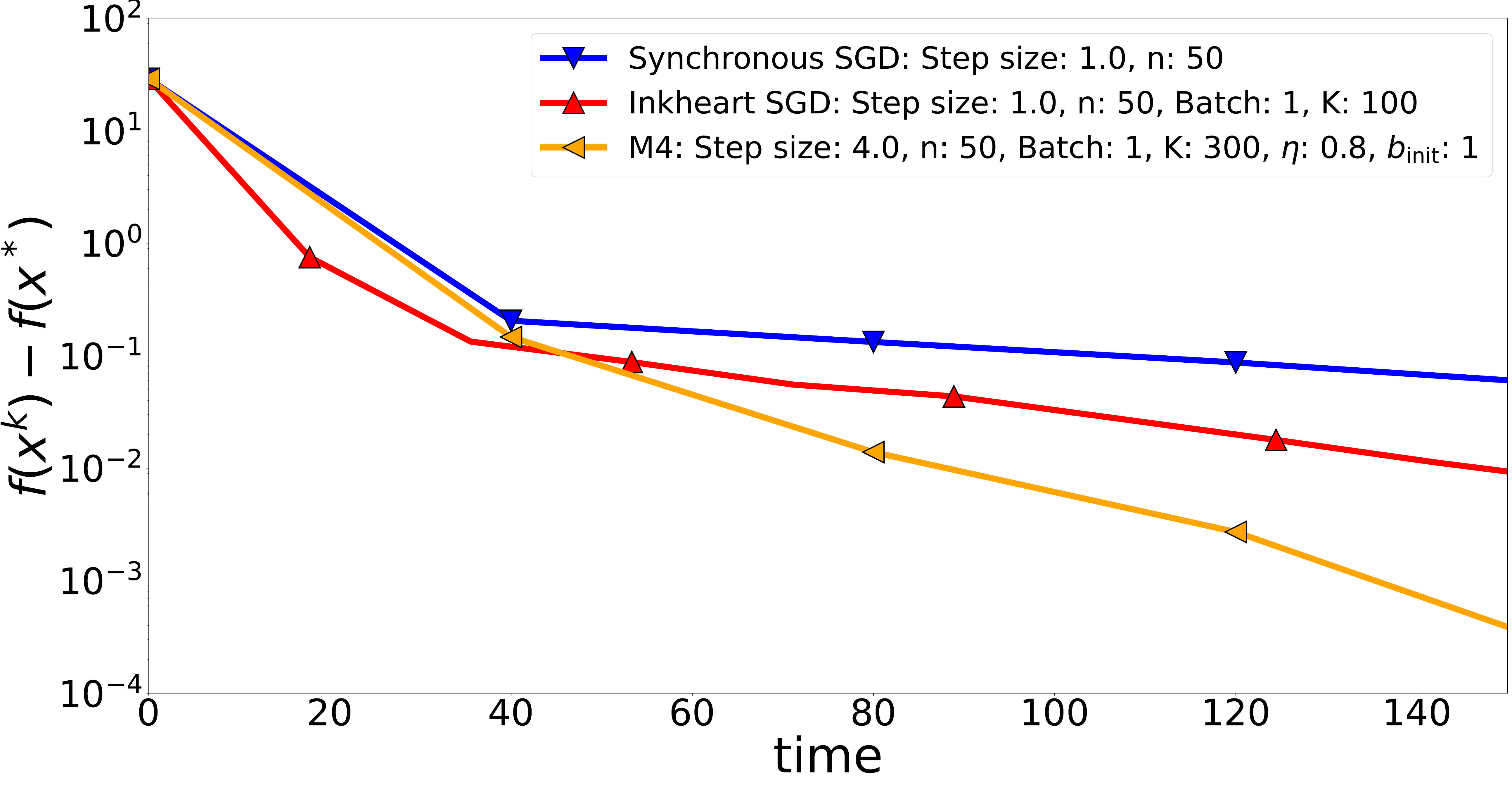}
        \caption{$h=0$}
    \end{subfigure}\hfill
    \begin{subfigure}[b]{0.32\textwidth}
        \centering\includegraphics[width=\textwidth]{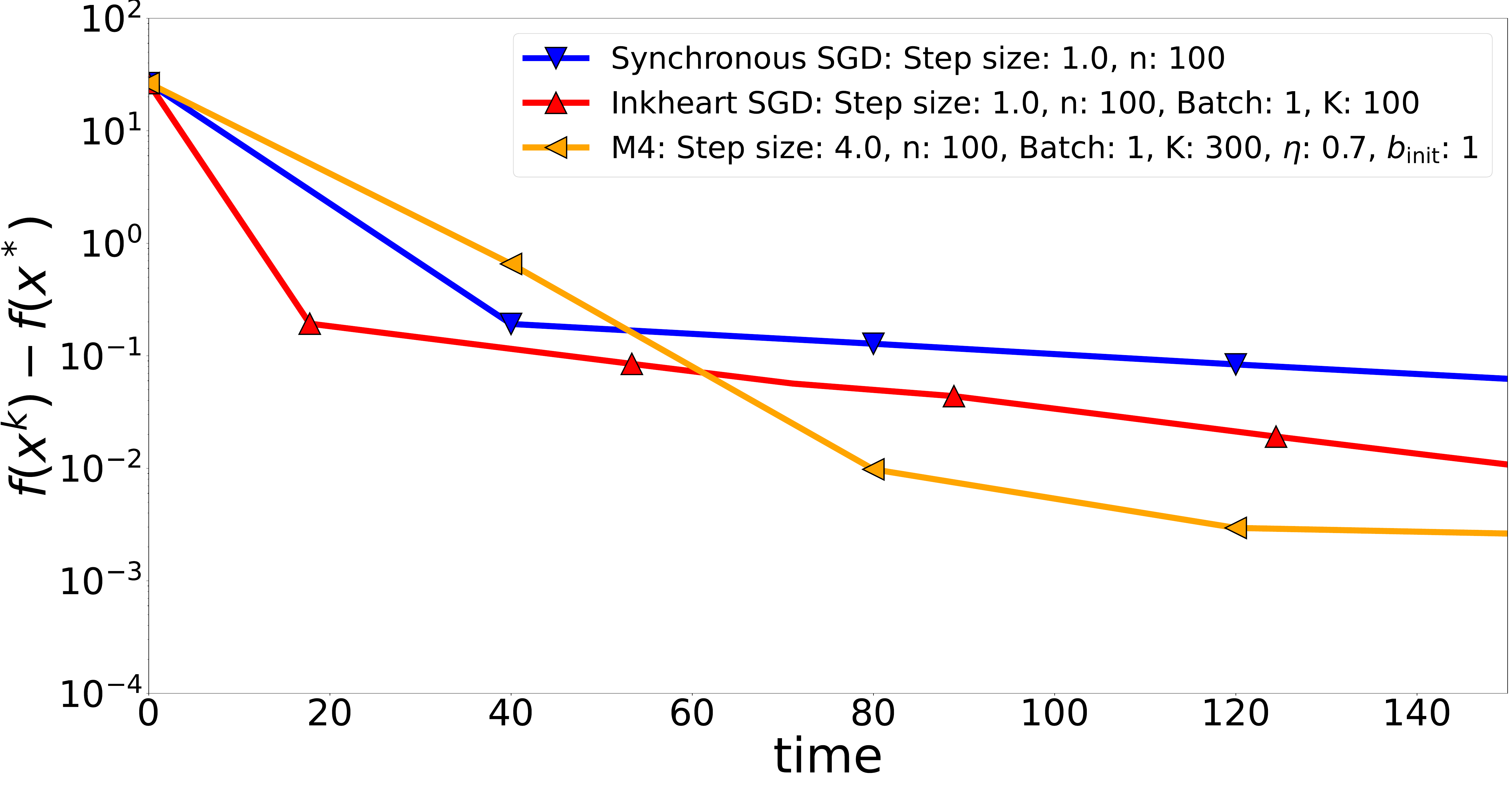}
        \caption{$h=0$}
    \end{subfigure}\hfill
    \begin{subfigure}[b]{0.32\textwidth}
        \centering\includegraphics[width=\textwidth]{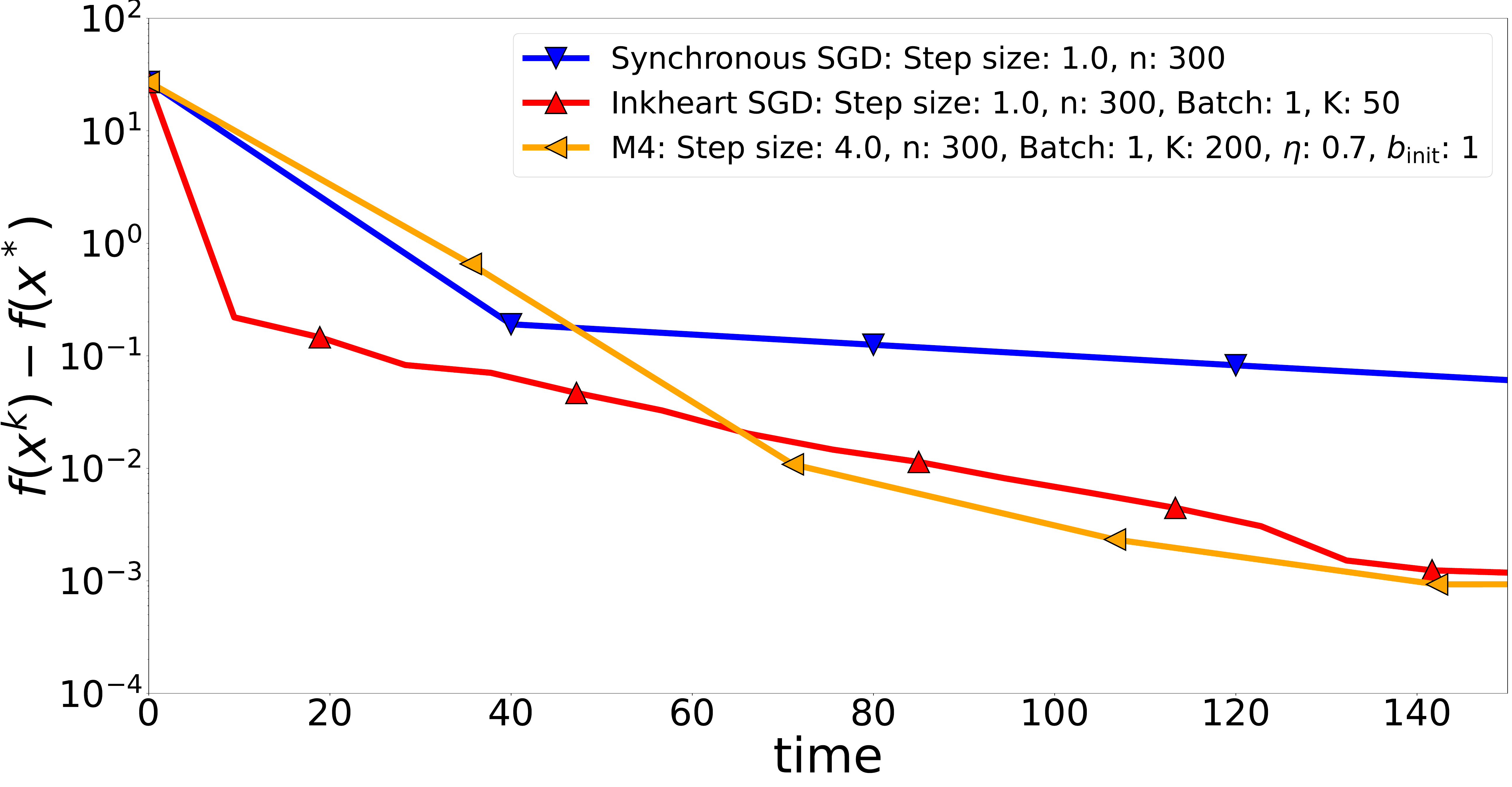}
        \caption{$h=0$}
    \end{subfigure}
    
    \vspace{0.25cm}
    
    \begin{subfigure}[b]{0.32\textwidth}
        \centering\includegraphics[width=\textwidth]{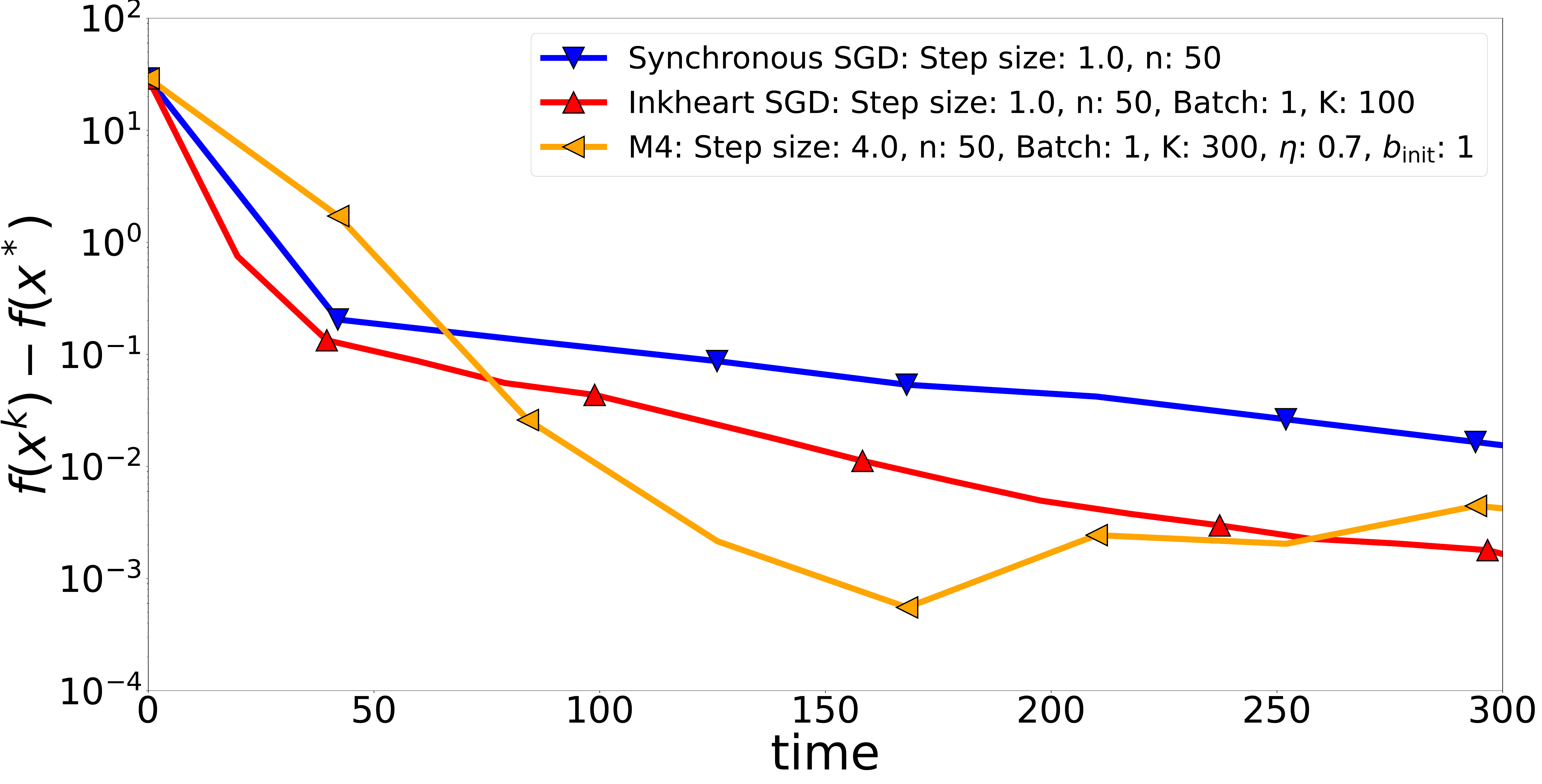}
        \caption{$h=0.1$}
    \end{subfigure}\hfill
    \begin{subfigure}[b]{0.32\textwidth}
        \centering\includegraphics[width=\textwidth]{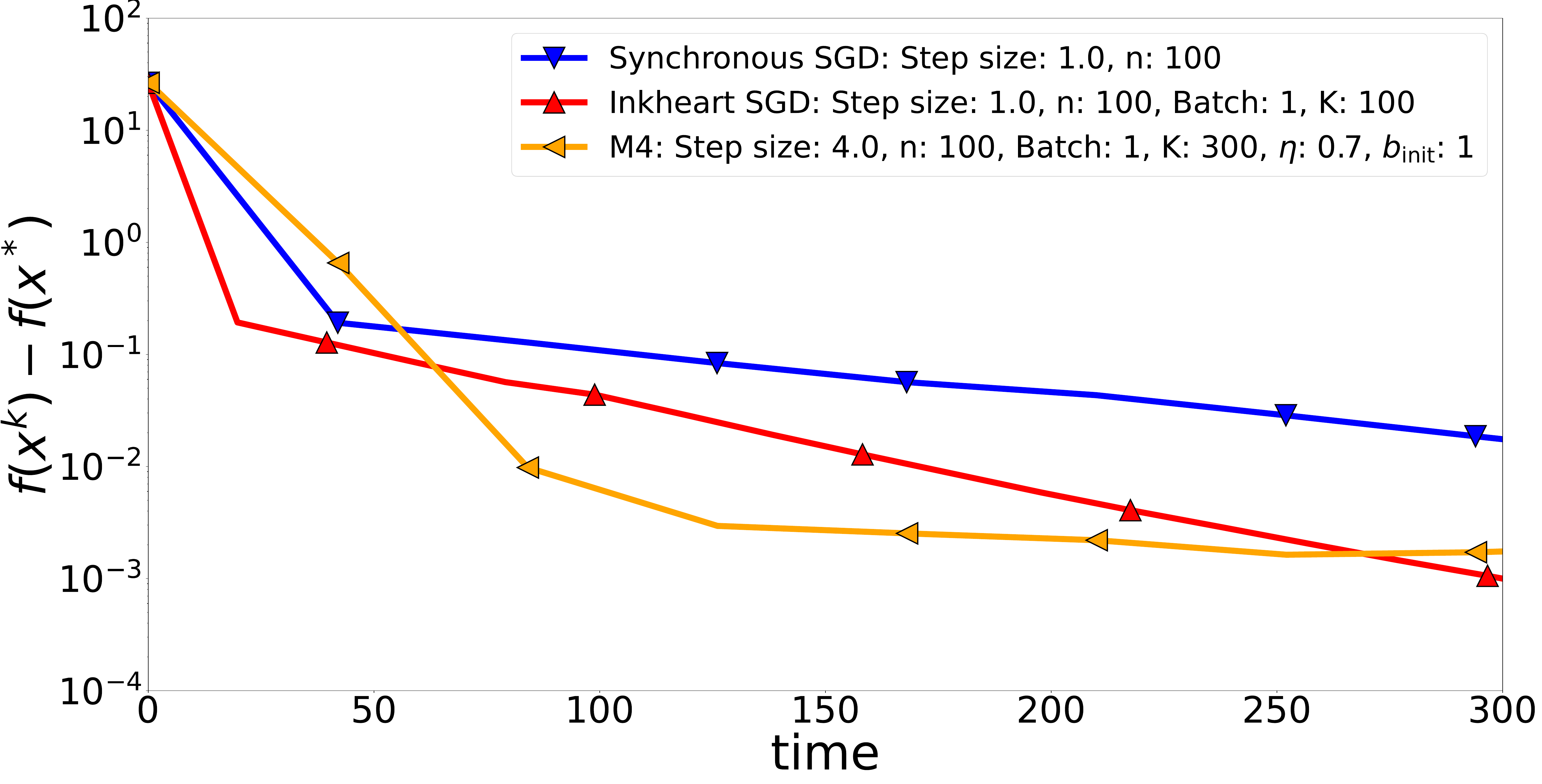}
        \caption{$h=0.1$}
    \end{subfigure}\hfill
    \begin{subfigure}[b]{0.32\textwidth}
        \centering\includegraphics[width=\textwidth]{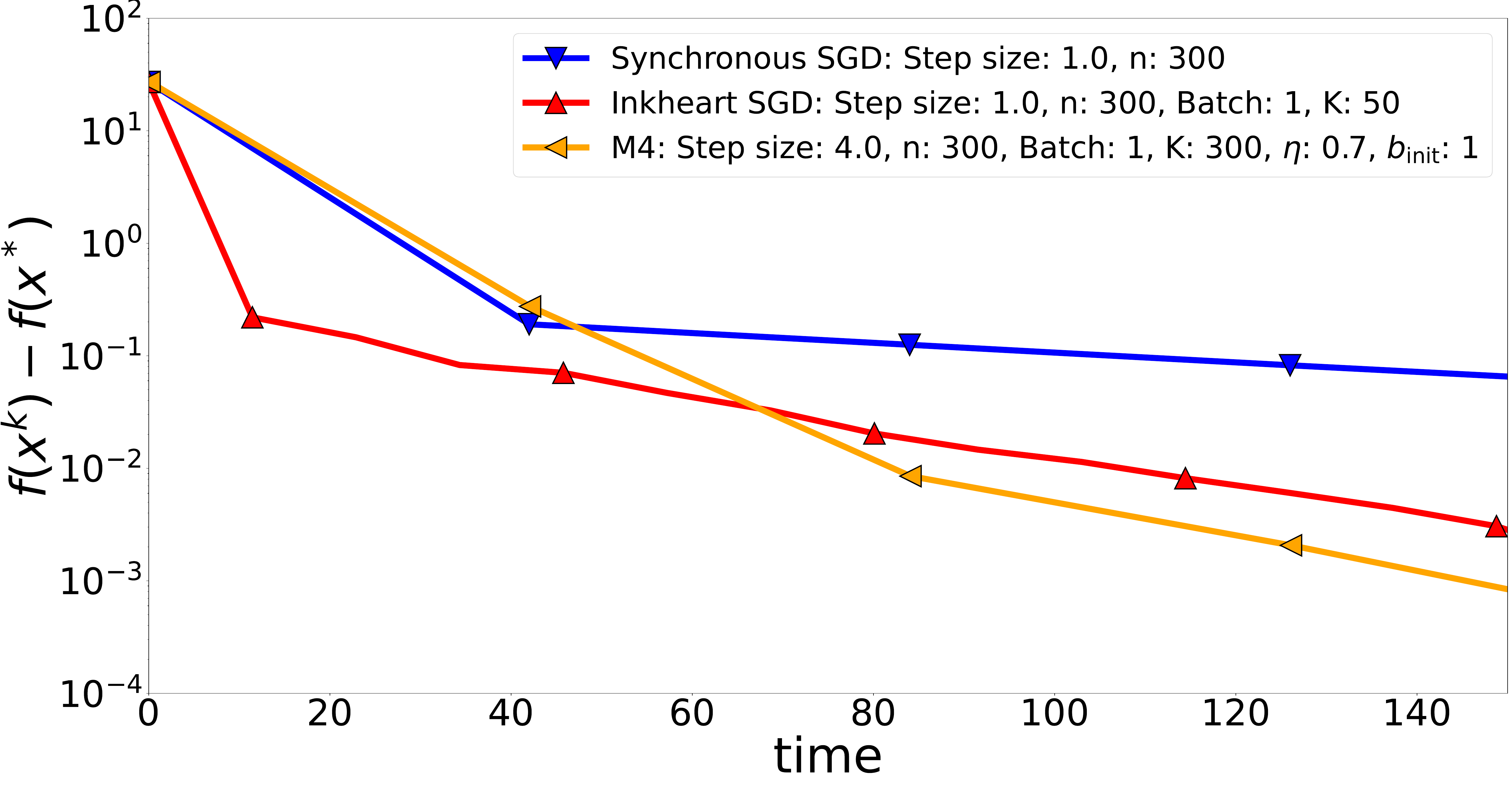}
        \caption{$h=0.1$}
    \end{subfigure}
    
    \vspace{0.25cm}
    
    \begin{subfigure}[b]{0.32\textwidth}
        \centering\includegraphics[width=\textwidth]{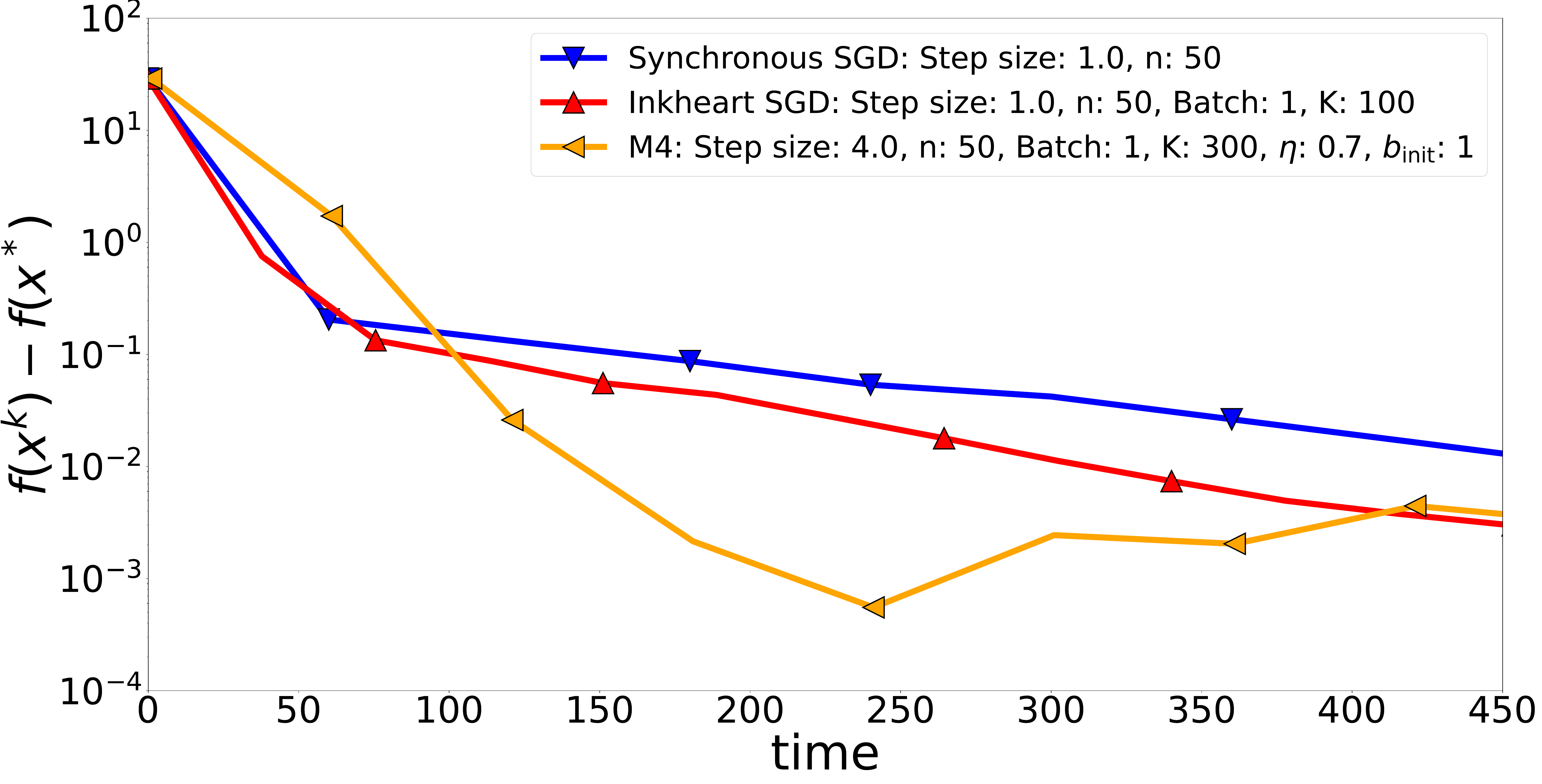}
        \caption{$h=1.0$}
    \end{subfigure}\hfill
    \begin{subfigure}[b]{0.32\textwidth}
        \centering\includegraphics[width=\textwidth]{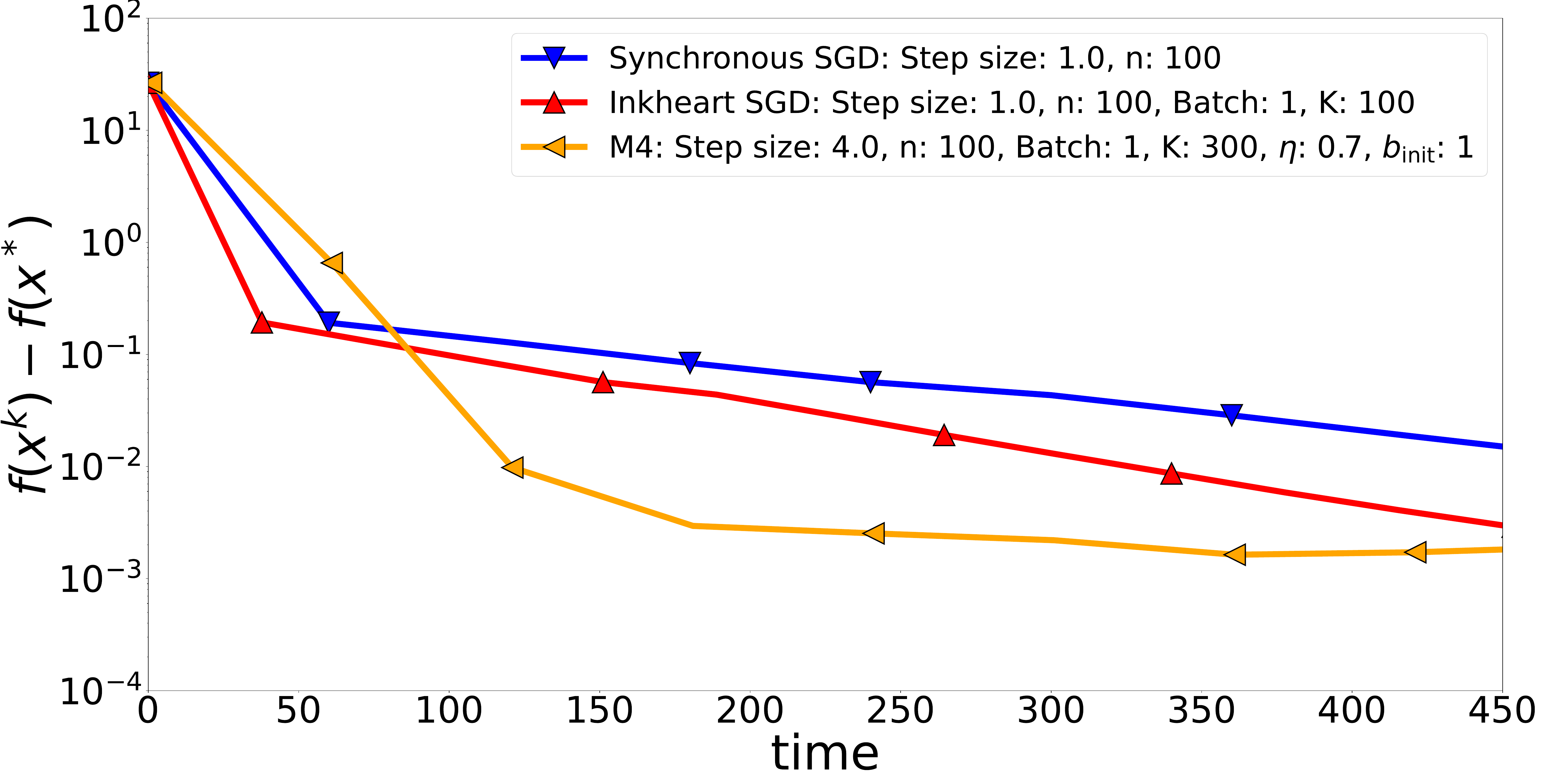}
        \caption{$h=1.0$}
    \end{subfigure}\hfill
    \begin{subfigure}[b]{0.32\textwidth}
        \centering\includegraphics[width=\textwidth]{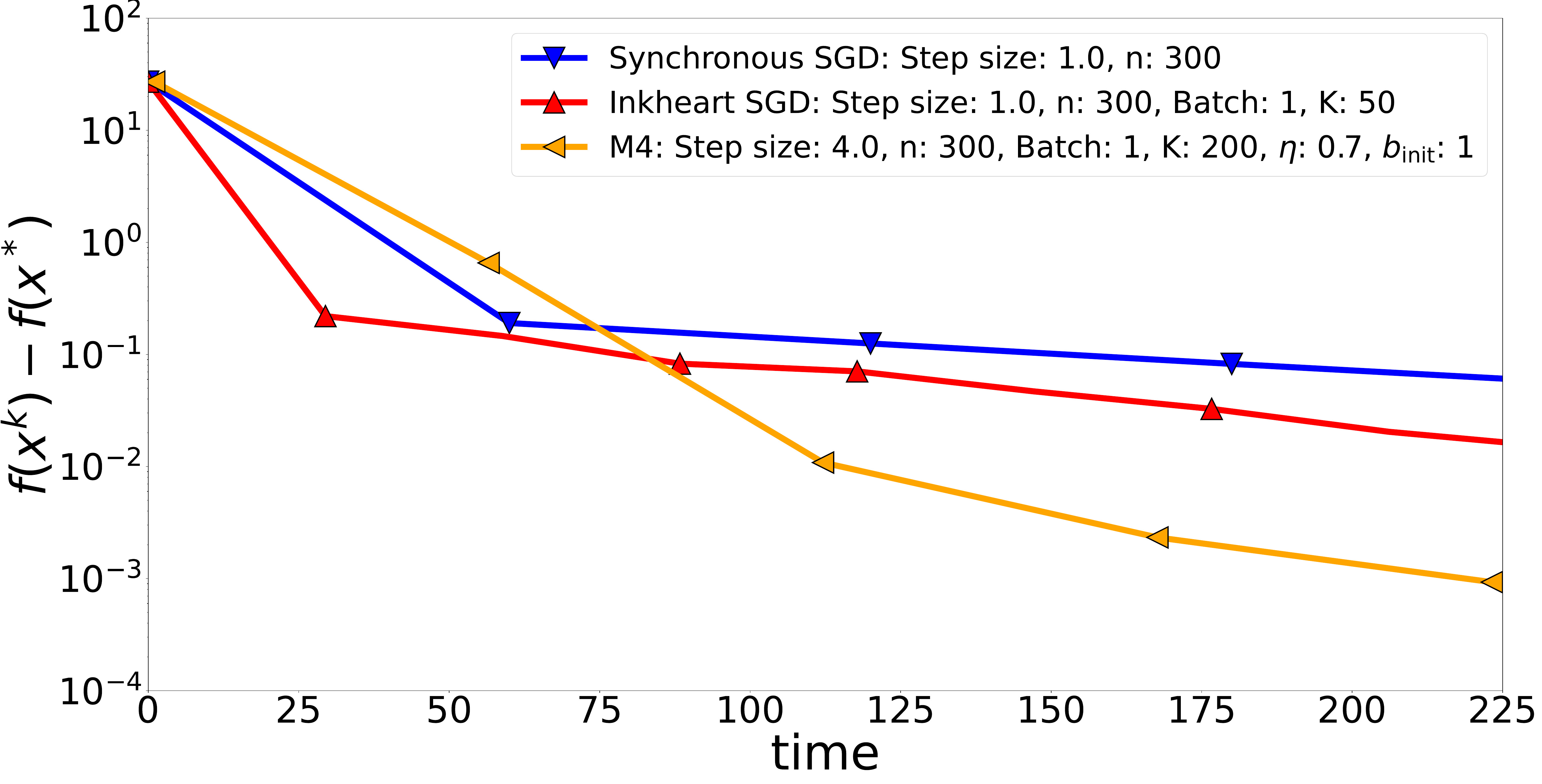}
        \caption{$h=1.0$}
    \end{subfigure}
    
    \caption{Convergence for high heterogeneity ($\ell = 0.5$) and medium noise $\sigma = 0.01$.
    Fixed parameters: $d=300$, $\kappa = \nicefrac{1}{d}$, $\tau = \nicefrac{1}{d}$. 
    Rows vary the gradient computation time $h$; columns correspond to the number of workers $n \in \{50, 100, 300\}$.}
    \label{fig:lip_large}
\end{figure}

\clearpage
\newpage
\subsection{Homogeneous Small-Scale Machine Learning Task}
In this section, we train a two-layer neural network (NN) with the architecture 
$\text{Linear}(\text{input\_dim}, 32) \to \text{ReLU} \to \text{Linear}(32, \text{num\_classes})$, 
optimized with the logistic loss on the \emph{MNIST} dataset \citep{lecun2010mnist}.
This setup allows us to compare methods on tasks where workers compute stochastic gradients via uniform sampling, 
with each worker having access to the full dataset.
The total number of network parameters is $d=25\,450$.

We tune the step size $\gamma$ over the same range as in the previous sections. We tune the momentum parameter in \ref{eq:mthree} over $\{0.1, 0.2, \dots, 1\}$
and take the initial batch size $b_\textnormal{init} = 1$.
The optimal compression parameter $K$ and the batch size for \ref{eq:mthree} and \ref{eq:inkheart} 
are selected from the sets $\{100, 1000, 2500, 5000, 7500, \dots, 25\,000\}$ and $\{1, 4, 8, 16\}$, respectively. 

We vary the number of workers $n \in \{10, 100\}$ and the per-sample computation time $h \in \{0, 0.1, 1.0\}$.
Figure~\ref{fig:mnist_homo} presents the training loss and accuracy curves.
When per-sample computation is cheap, larger batches and aggressive compression enable \ref{eq:mthree} and \ref{eq:inkheart} to outperform Synchronous SGD. 
Under high computation cost, lighter compression and smaller batches yield better performance.

\begin{figure}[htp]
    \centering
    \small
    \begin{subfigure}[b]{0.32\linewidth}
        \includegraphics[width=\linewidth]{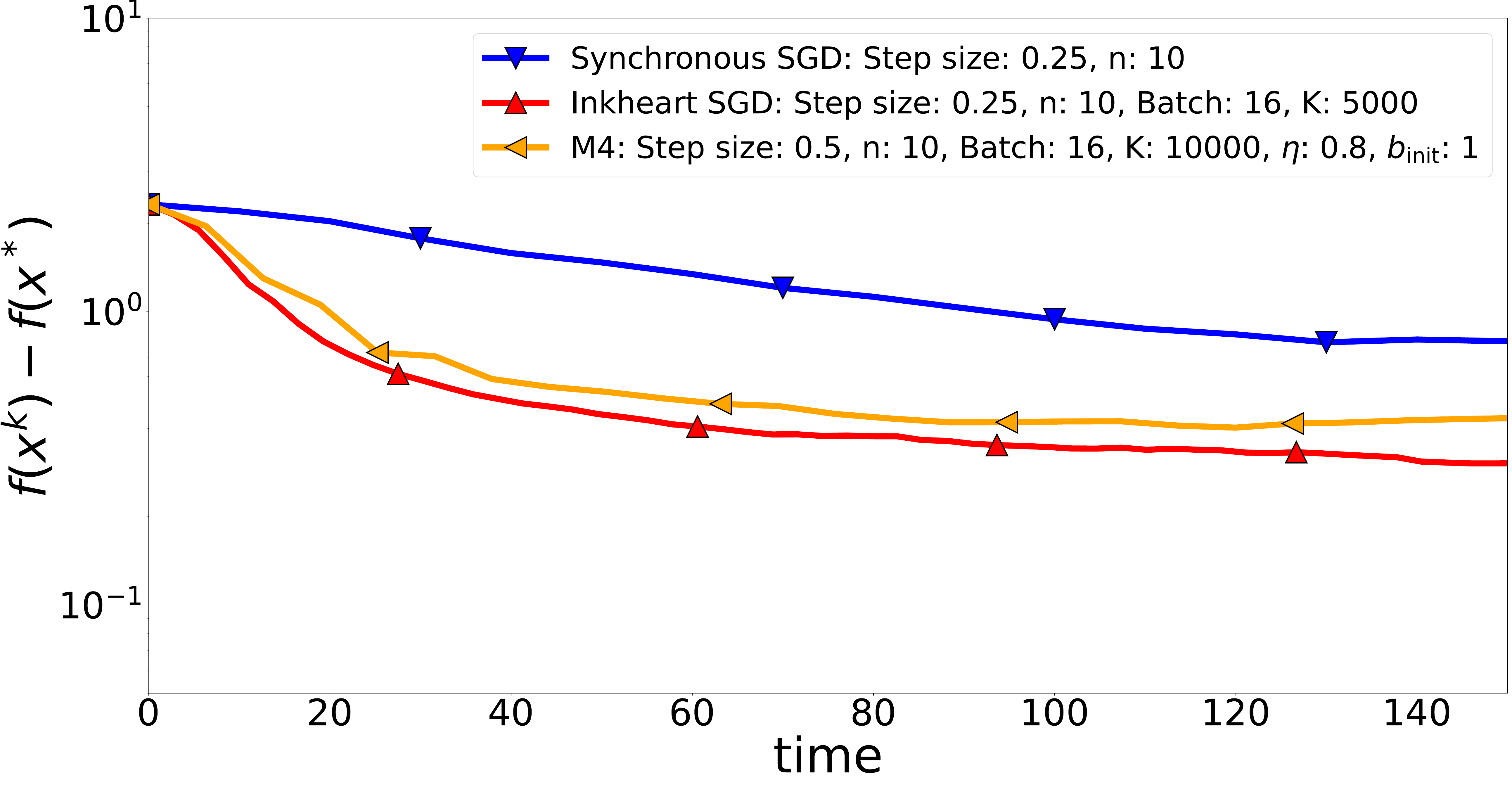}
    \end{subfigure}
    \hfill
    \begin{subfigure}[b]{0.32\linewidth}
        \includegraphics[width=\linewidth]{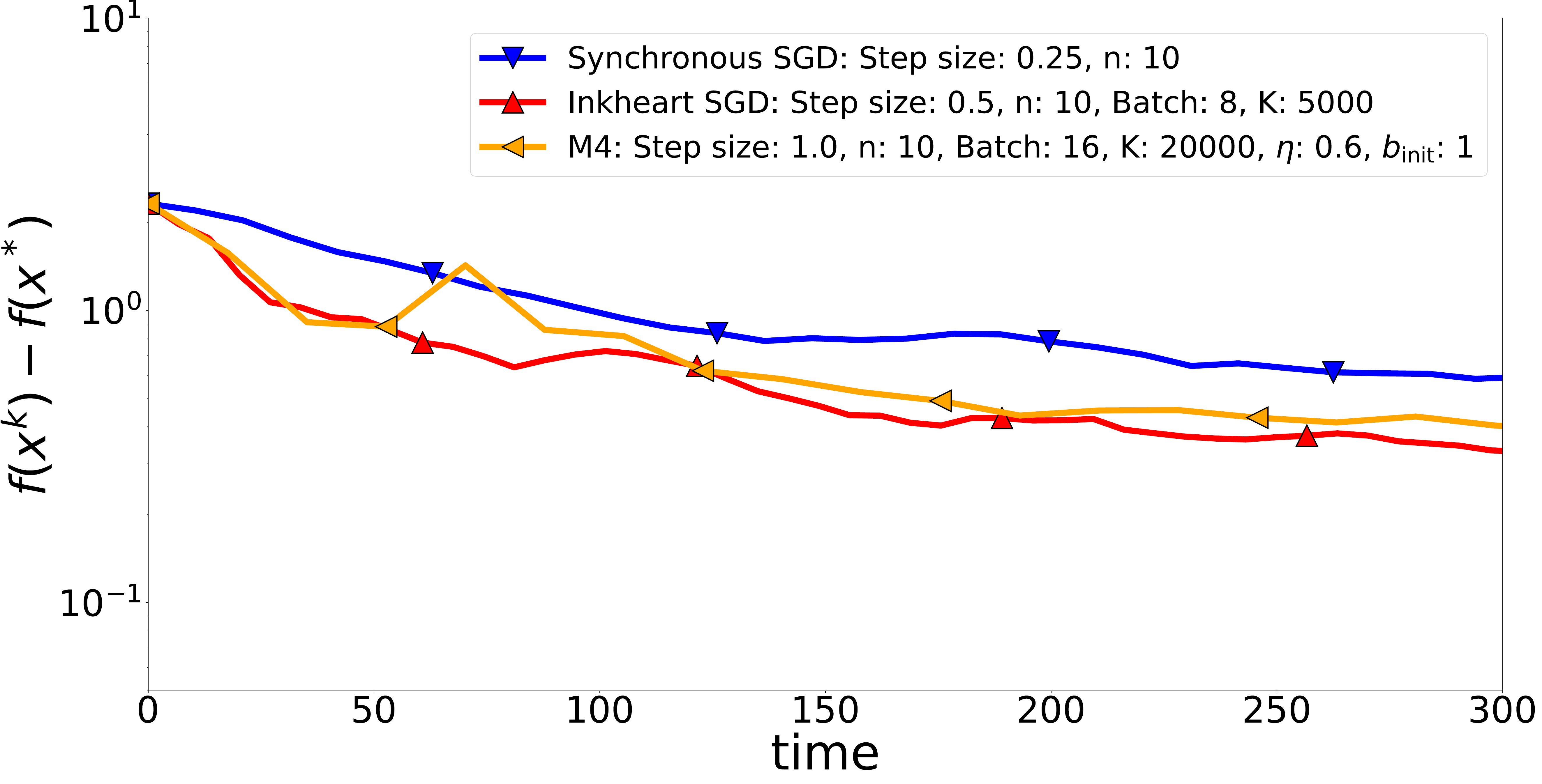}
    \end{subfigure}
    \hfill
    \begin{subfigure}[b]{0.32\linewidth}
        \includegraphics[width=\linewidth]{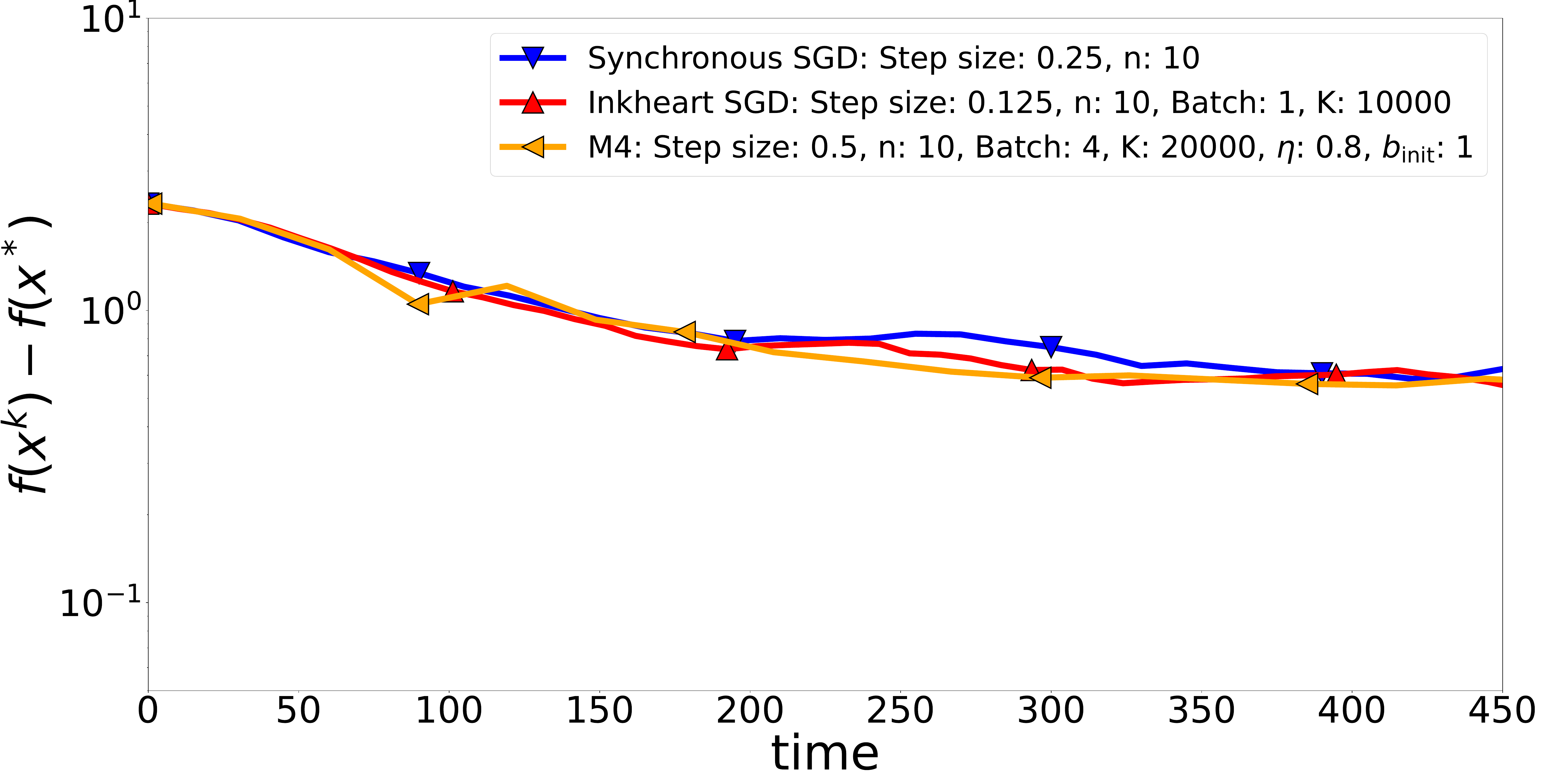}
    \end{subfigure}
    
    \begin{subfigure}[b]{0.32\linewidth}
        \includegraphics[width=\linewidth]{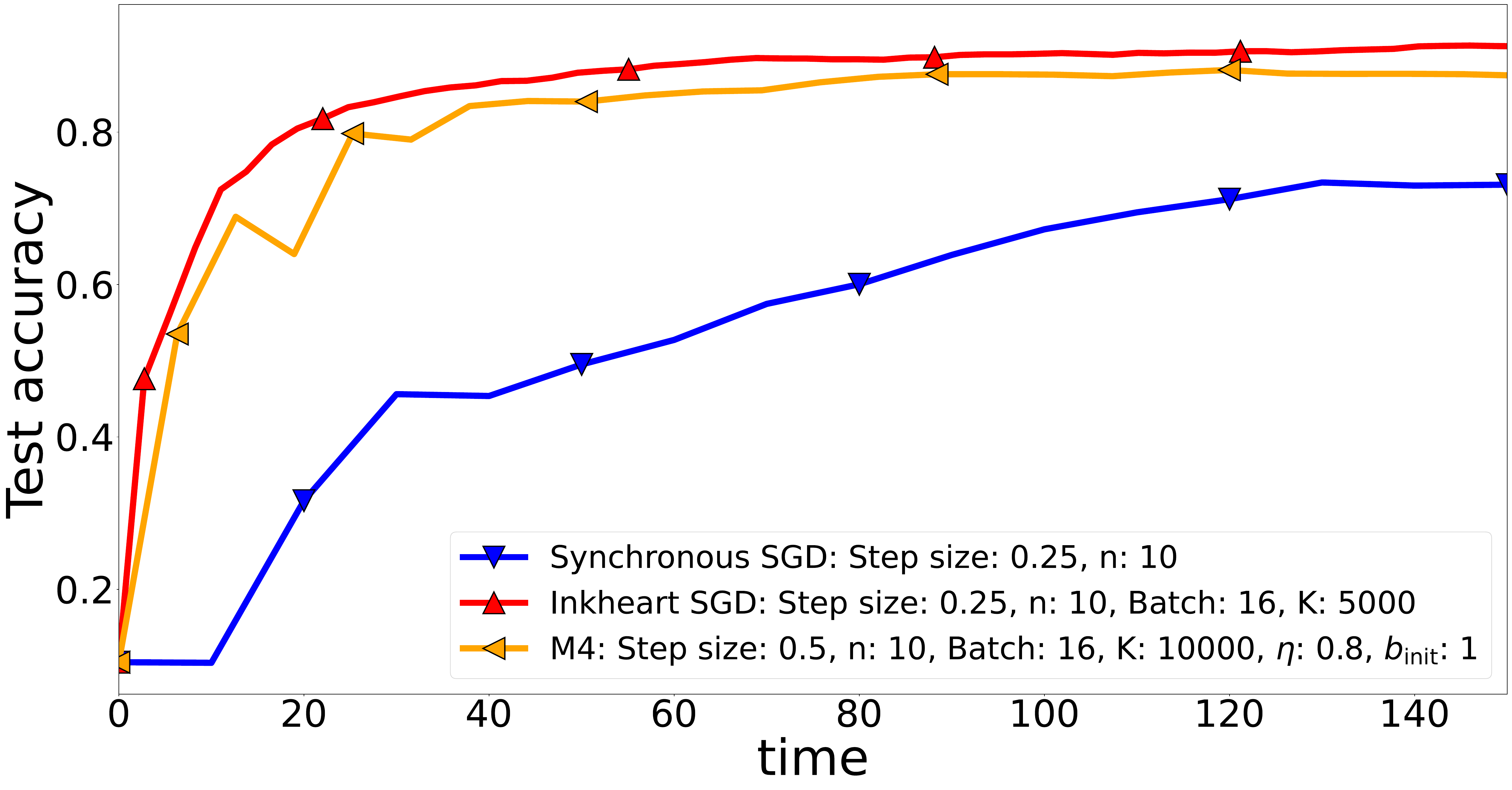}
    \end{subfigure}
    \hfill
    \begin{subfigure}[b]{0.32\linewidth}
        \includegraphics[width=\linewidth]{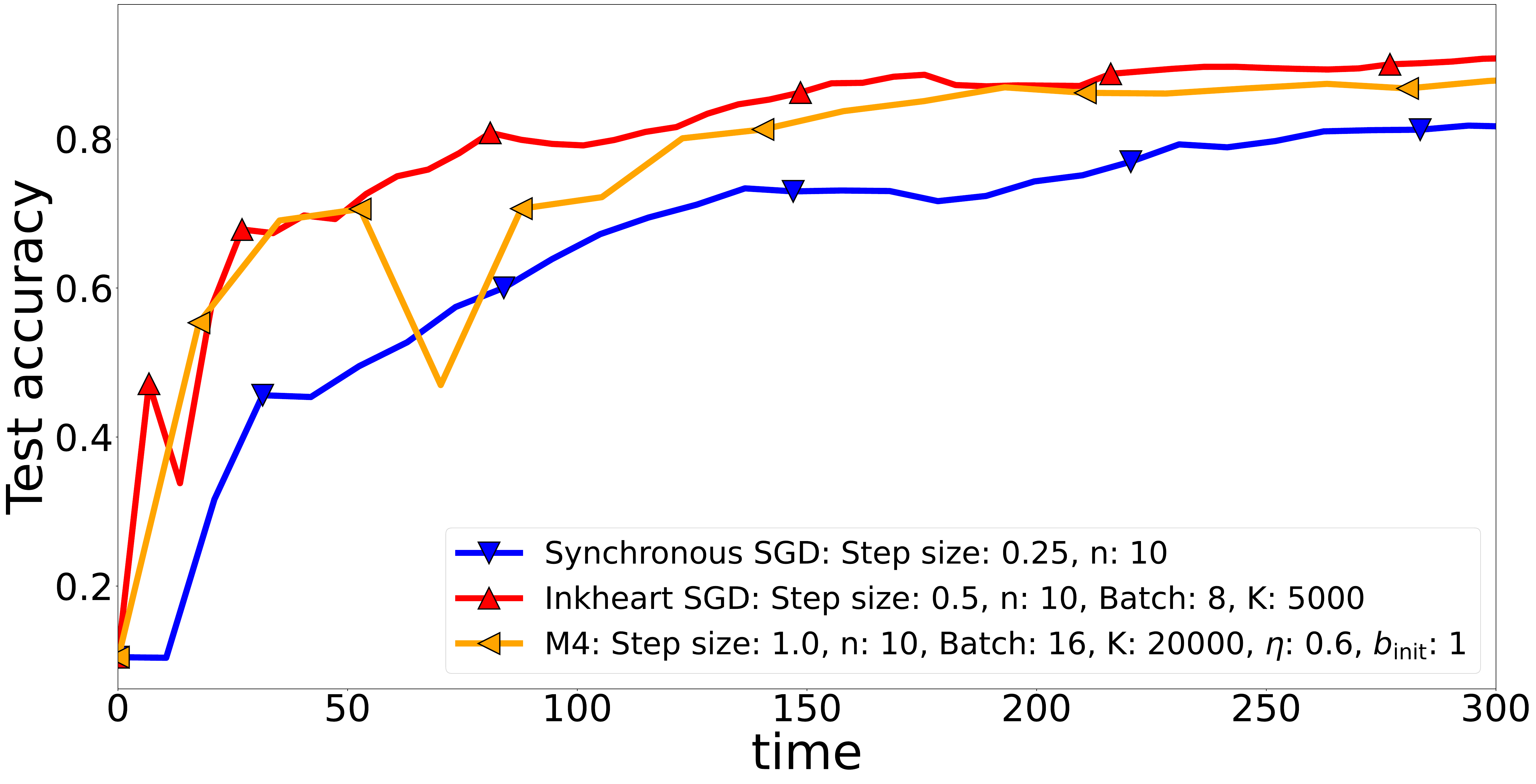}
    \end{subfigure}
    \hfill
    \begin{subfigure}[b]{0.32\linewidth}
        \includegraphics[width=\linewidth]{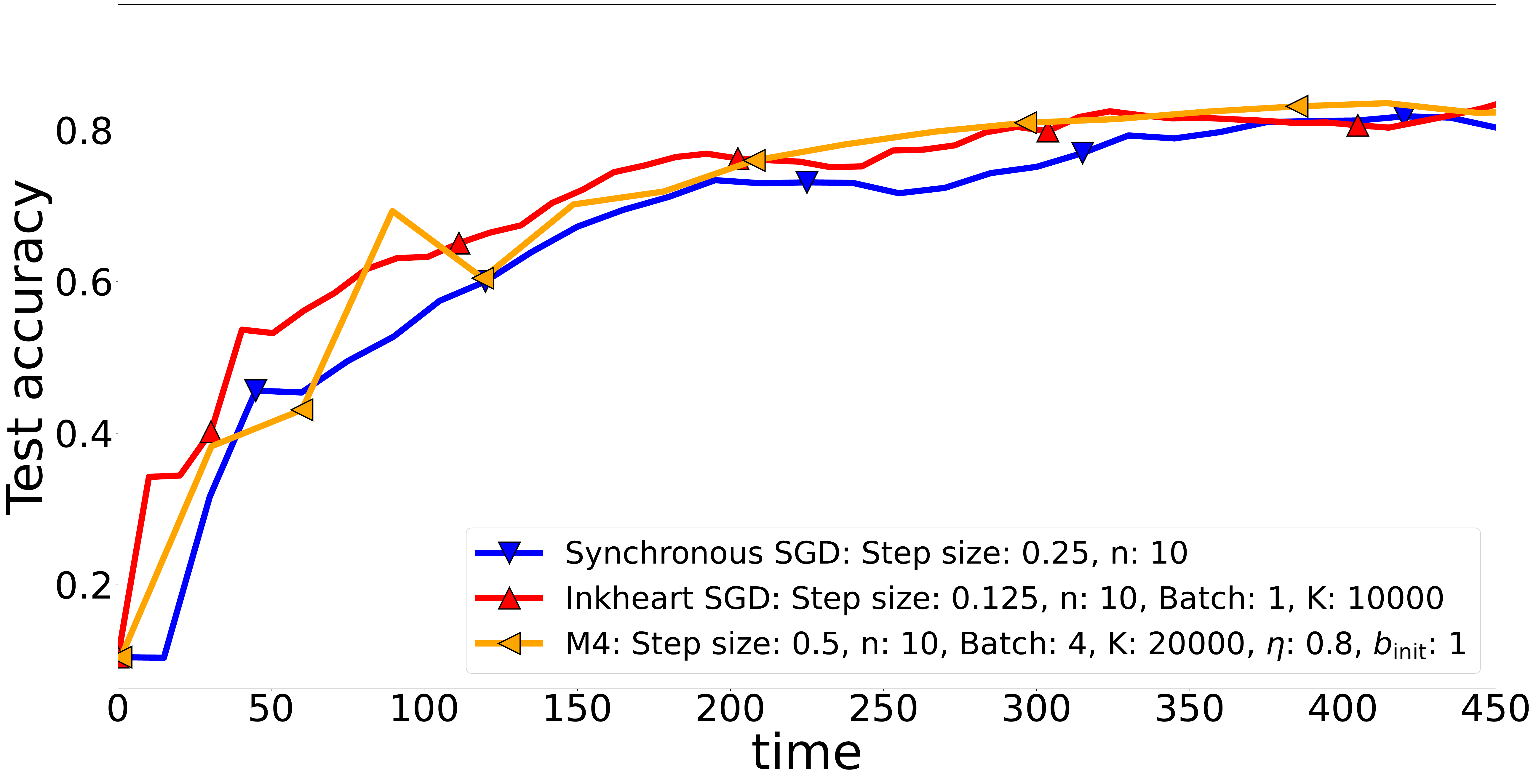}
    \end{subfigure}
    
    \begin{subfigure}[b]{0.32\linewidth}
        \includegraphics[width=\linewidth]{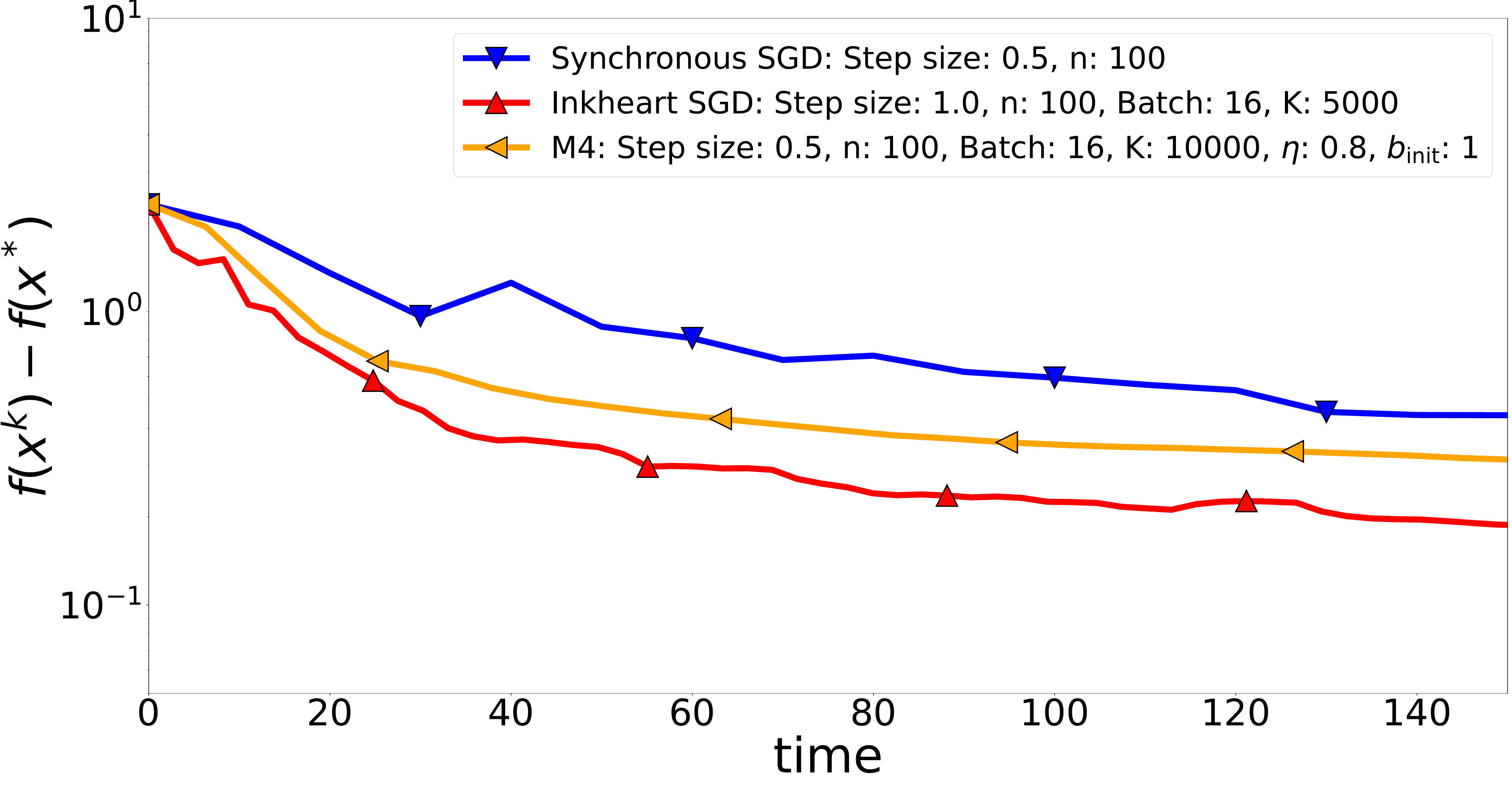}
    \end{subfigure}
    \hfill
    \begin{subfigure}[b]{0.32\linewidth}
        \includegraphics[width=\linewidth]{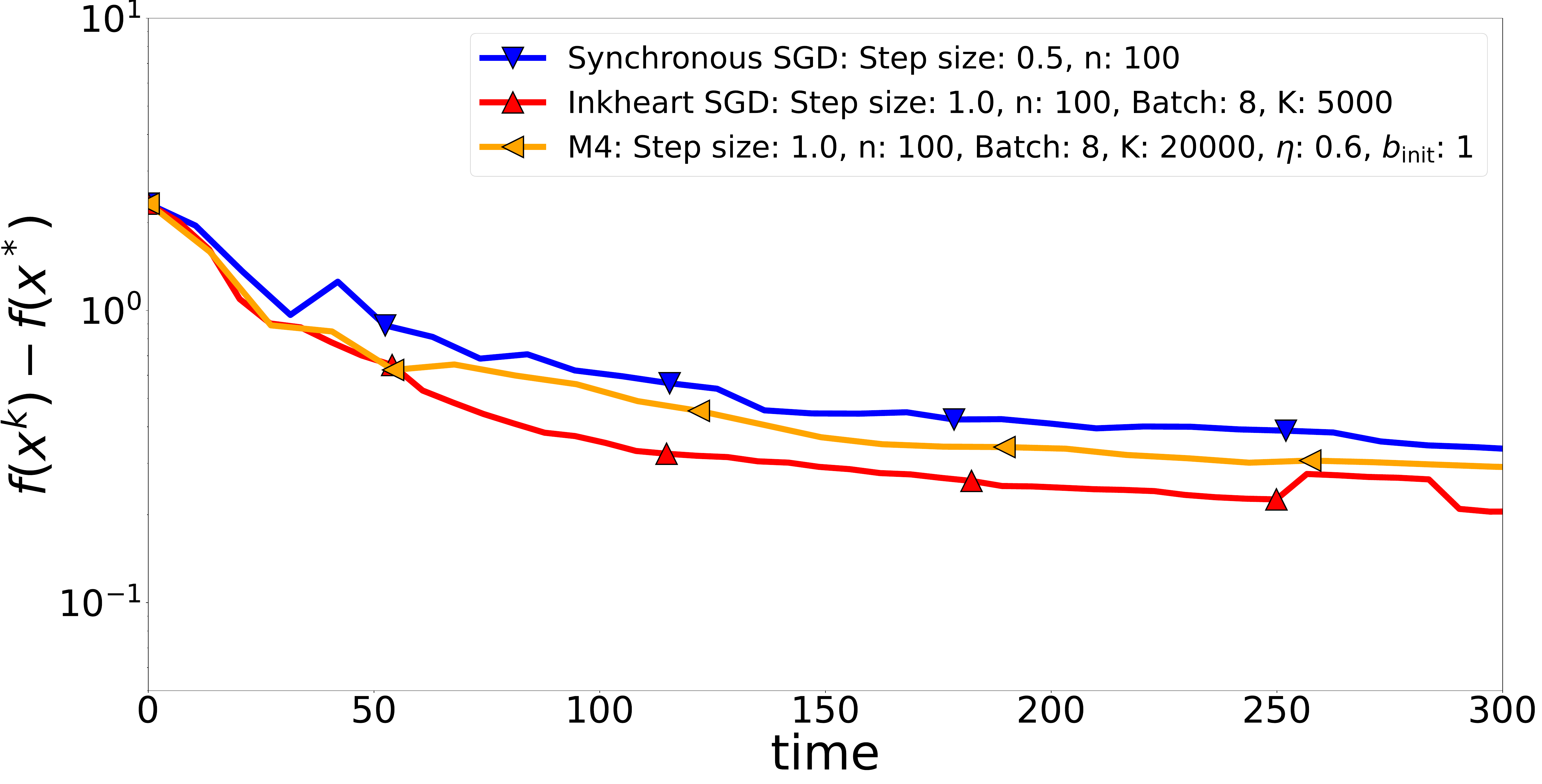}
    \end{subfigure}
    \hfill
    \begin{subfigure}[b]{0.32\linewidth}
        \includegraphics[width=\linewidth]{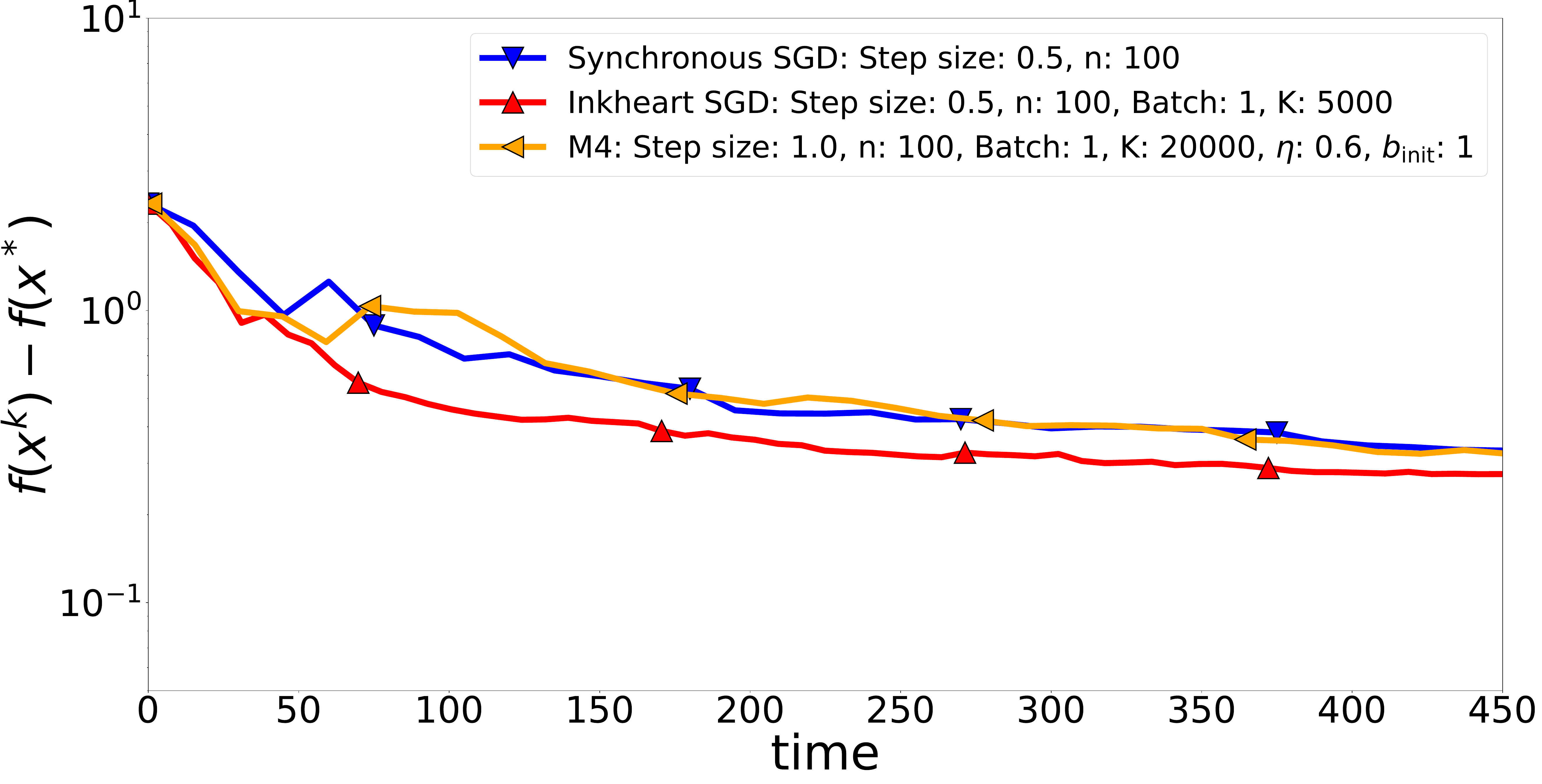}
    \end{subfigure}
    
    \begin{subfigure}[b]{0.32\linewidth}
        \includegraphics[width=\linewidth]{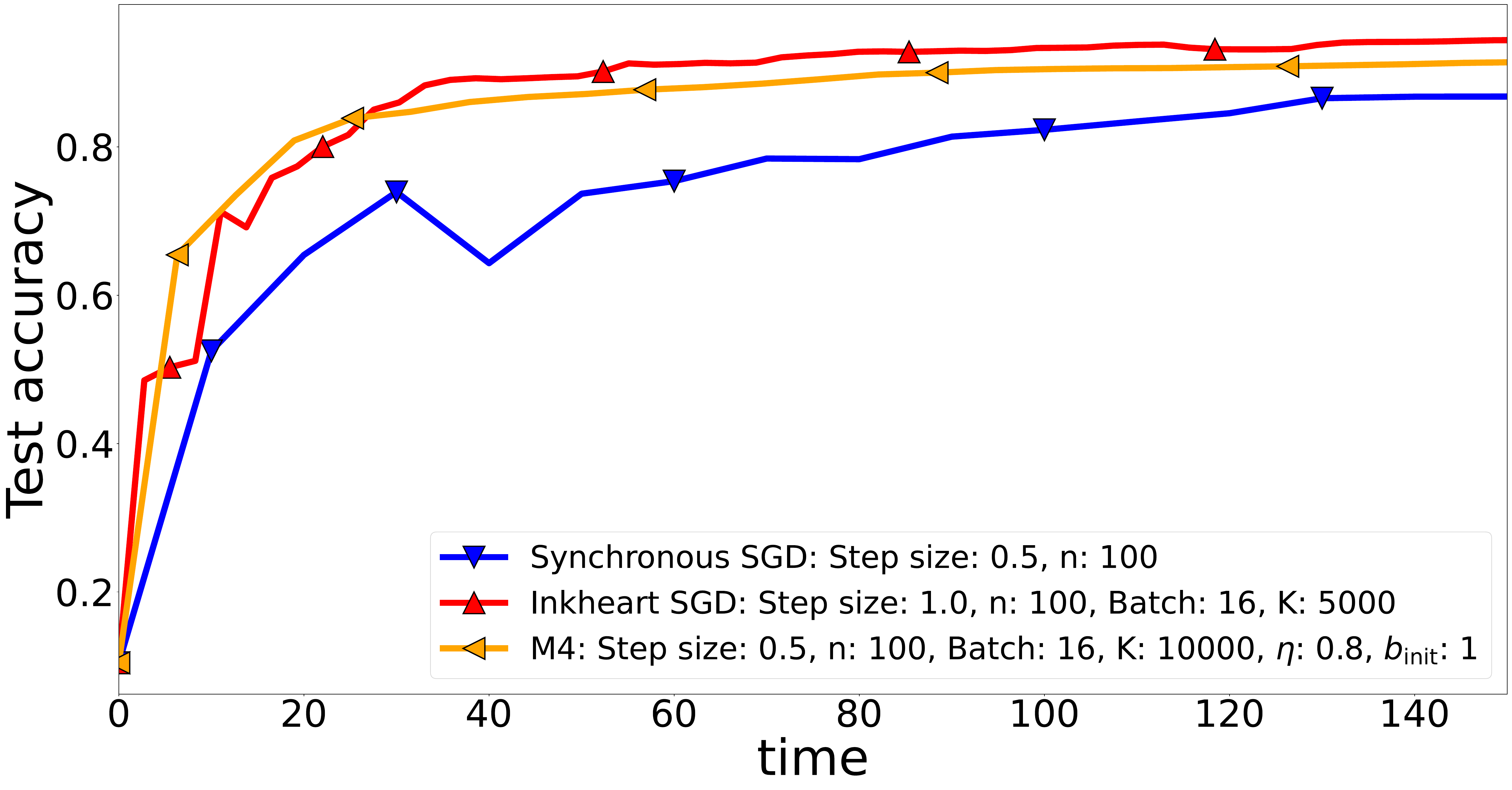}
    \end{subfigure}
    \hfill
    \begin{subfigure}[b]{0.32\linewidth}
        \includegraphics[width=\linewidth]{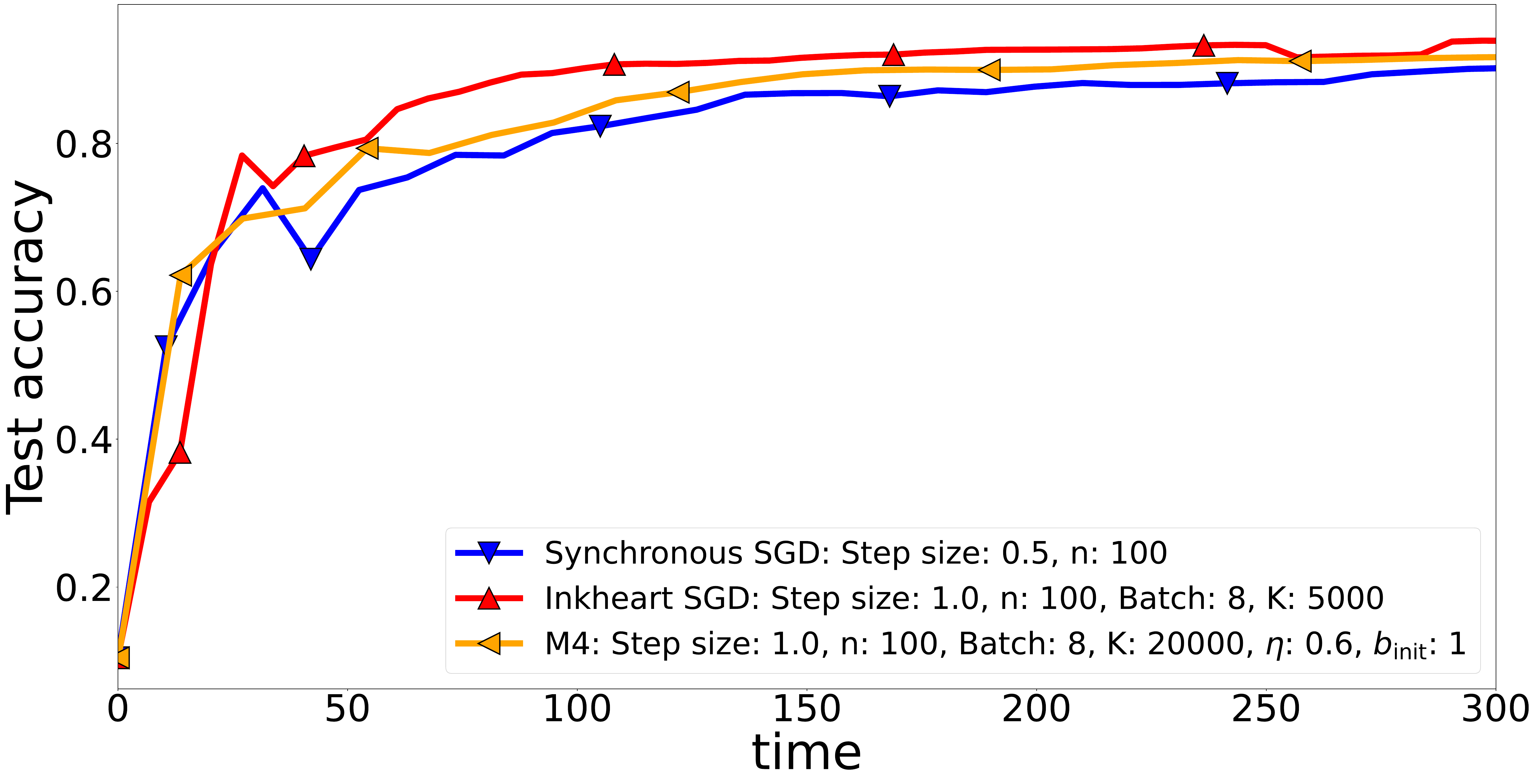}
    \end{subfigure}
    \hfill
    \begin{subfigure}[b]{0.32\linewidth}
        \includegraphics[width=\linewidth]{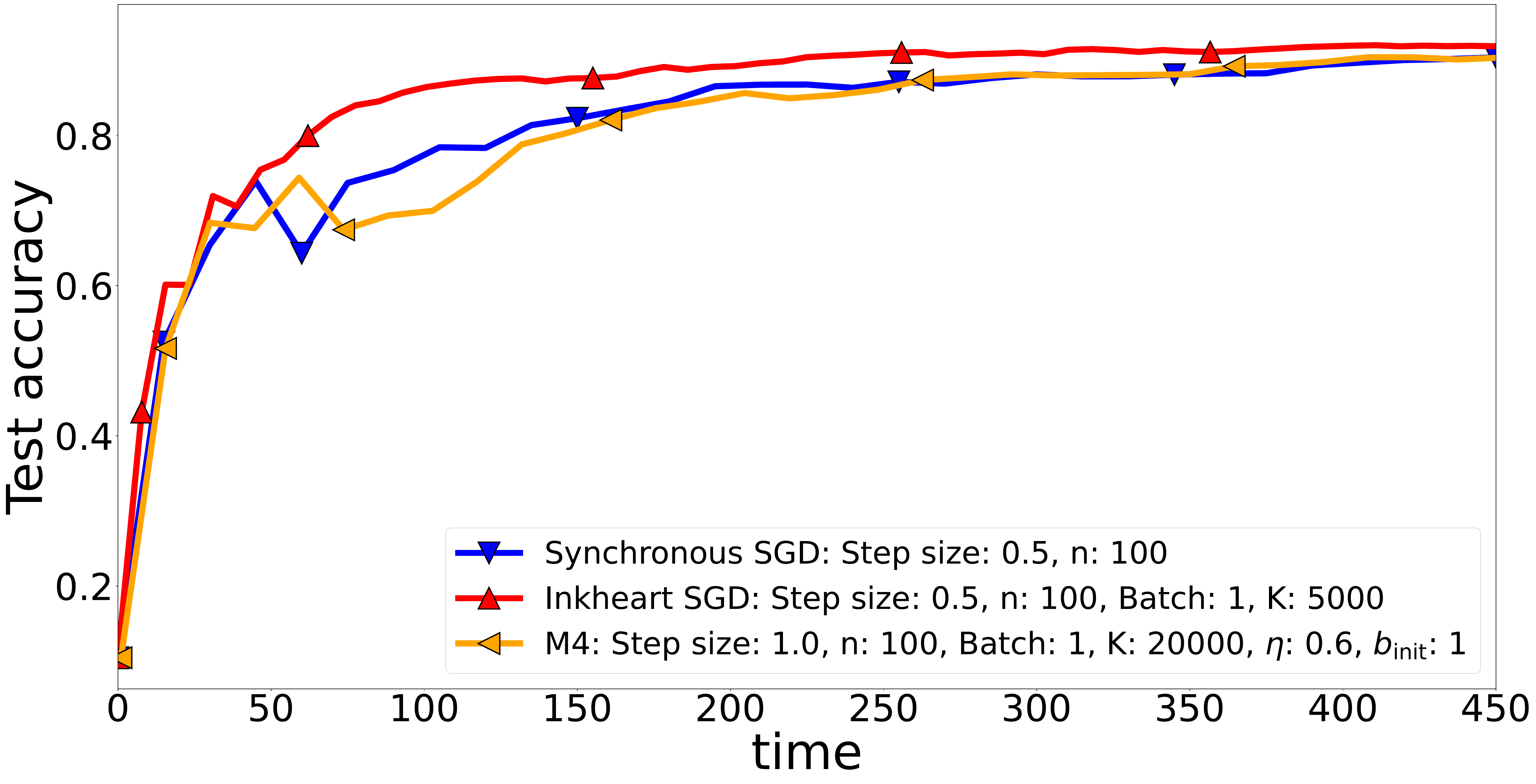}
    \end{subfigure}
    
    \vspace{0.3em}
    \begin{minipage}{0.32\linewidth}
        \centering\footnotesize $h = 0$
    \end{minipage}
    \hfill
    \begin{minipage}{0.32\linewidth}
        \centering\footnotesize $h = 0.1$
    \end{minipage}
    \hfill
    \begin{minipage}{0.32\linewidth}
        \centering\footnotesize $h = 1.0$
    \end{minipage}
    
    \caption{Training loss and accuracy on homogeneous MNIST. 
    Columns correspond to per-sample computation time $h \in \{0, 0.1, 1.0\}$. 
    Rows (top to bottom): $n=10$ (loss), $n=10$ (accuracy), $n=100$ (loss), $n=100$ (accuracy). 
    Fixed parameters: $d=25\,450$, $\kappa = \nicefrac{1}{d}$, $\tau = \nicefrac{1}{d}$.}
    \label{fig:mnist_homo} 
\end{figure}

\subsection{Heterogeneous Small-Scale Machine Learning Task}
In this section we conduct the same experiment except that \emph{MNIST} is randomly split between workers.
The results are presented in Figure~\ref{fig:mnist_k1_results}. We observe that the performance of \ref{eq:mthree} and \ref{eq:inkheart} improves as the number of workers increases. We can see that \ref{eq:inkheart} and \ref{eq:mthree} converge much faster in different computation regimes (different values of $h$) and for different numbers of workers $n,$ supporting our theoretical results.

\begin{figure}[htp]
    \centering
    \captionsetup[subfigure]{labelformat=empty, font=scriptsize}
    \setlength{\tabcolsep}{3pt}
    
    \begin{subfigure}[b]{0.32\textwidth}
        \centering\includegraphics[width=\textwidth]{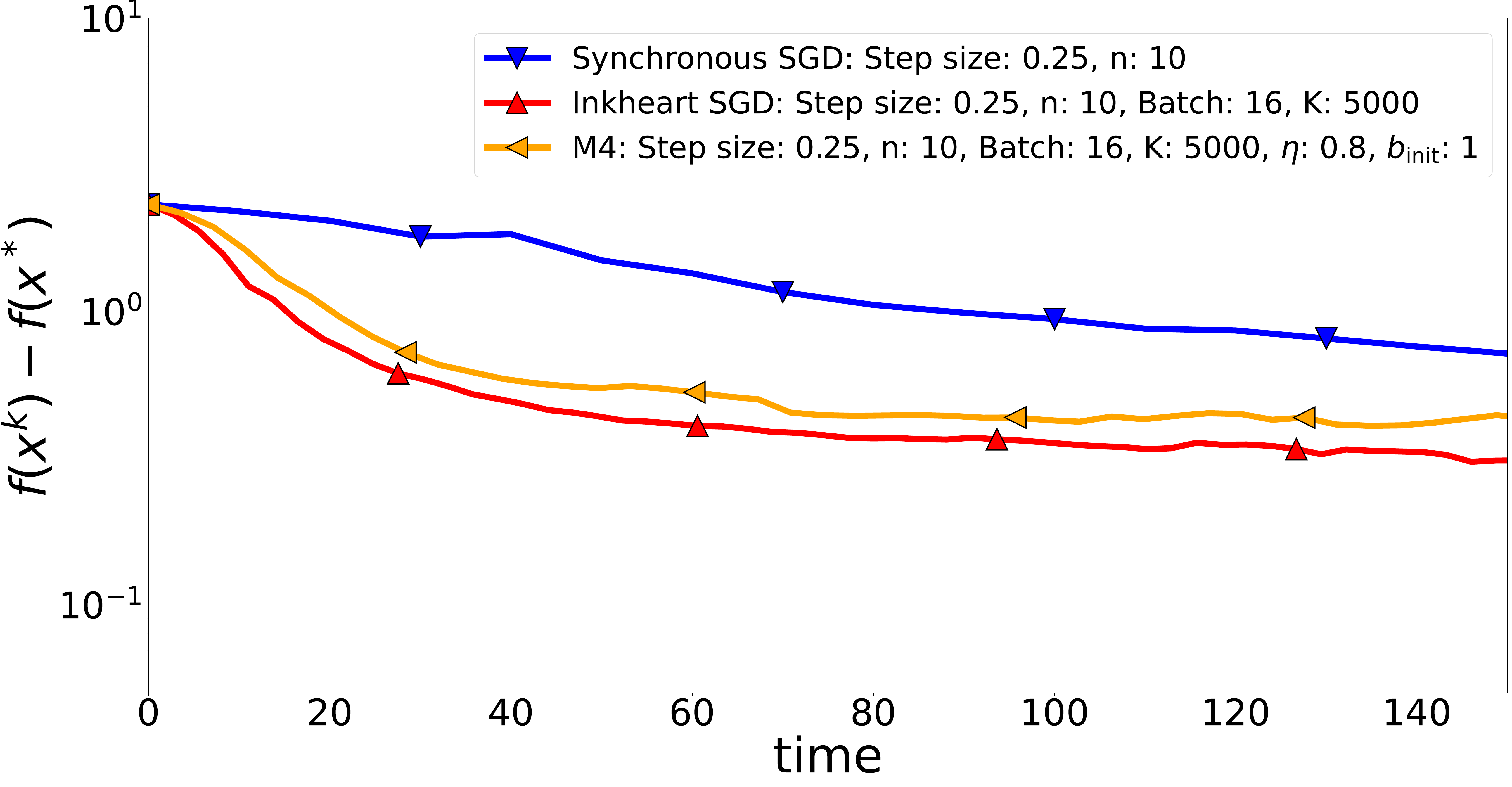}
    \end{subfigure}\hfill
    \begin{subfigure}[b]{0.32\textwidth}
        \centering\includegraphics[width=\textwidth]{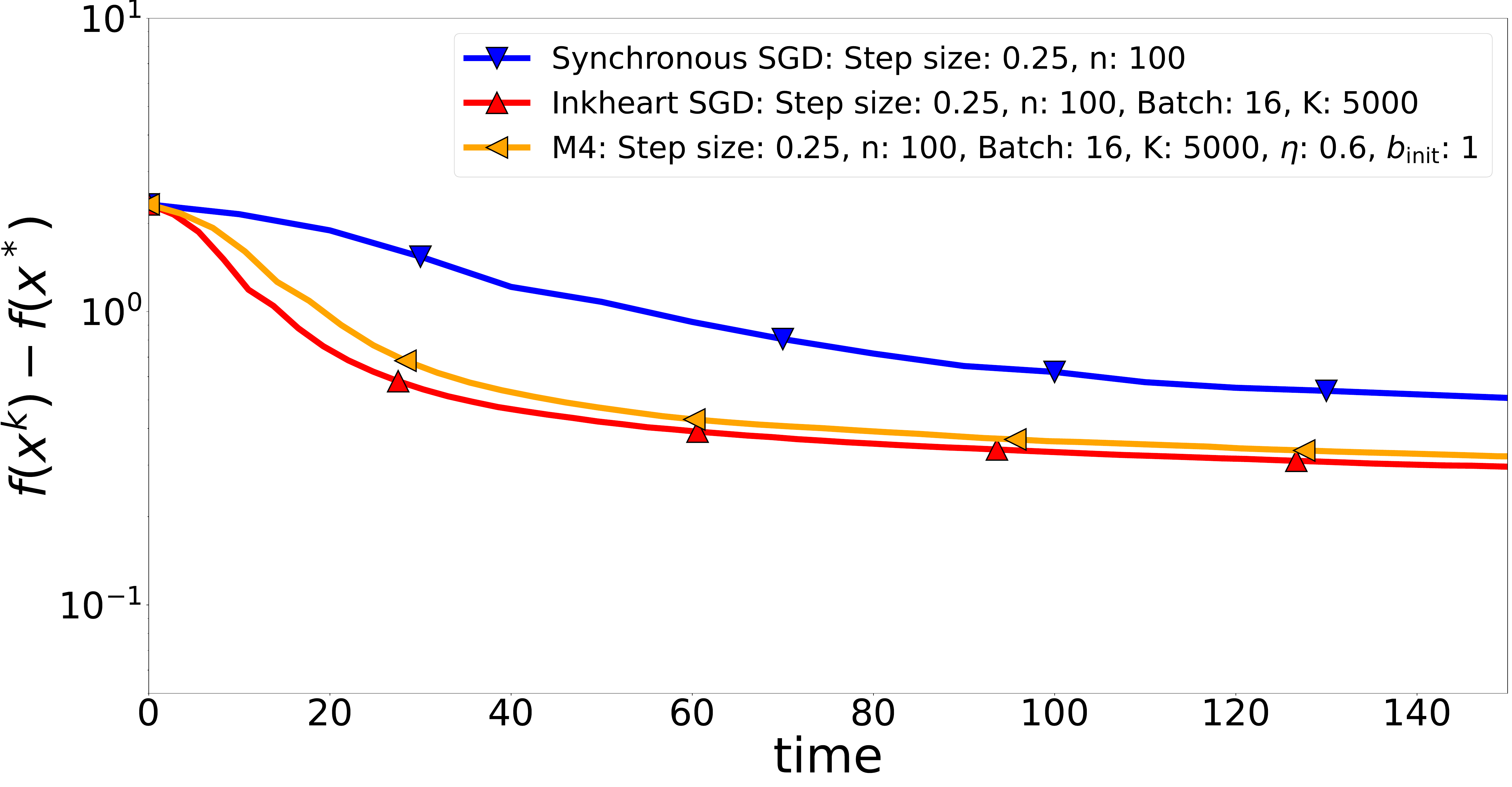}
    \end{subfigure}\hfill
    \begin{subfigure}[b]{0.32\textwidth}
        \centering\includegraphics[width=\textwidth]{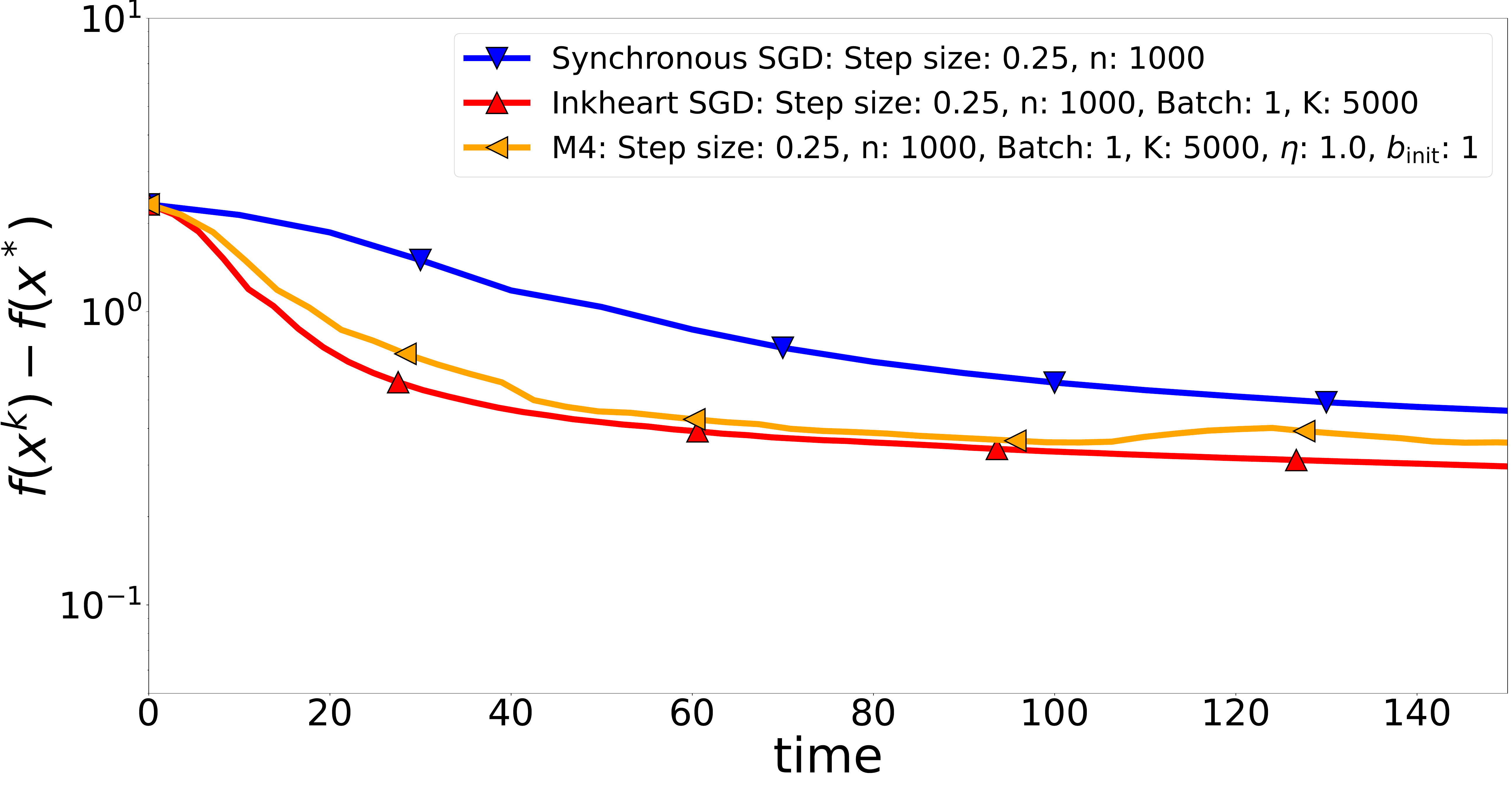}
    \end{subfigure}
    
    \vspace{0.15cm}
    
    \begin{subfigure}[b]{0.32\textwidth}
        \centering\includegraphics[width=\textwidth]{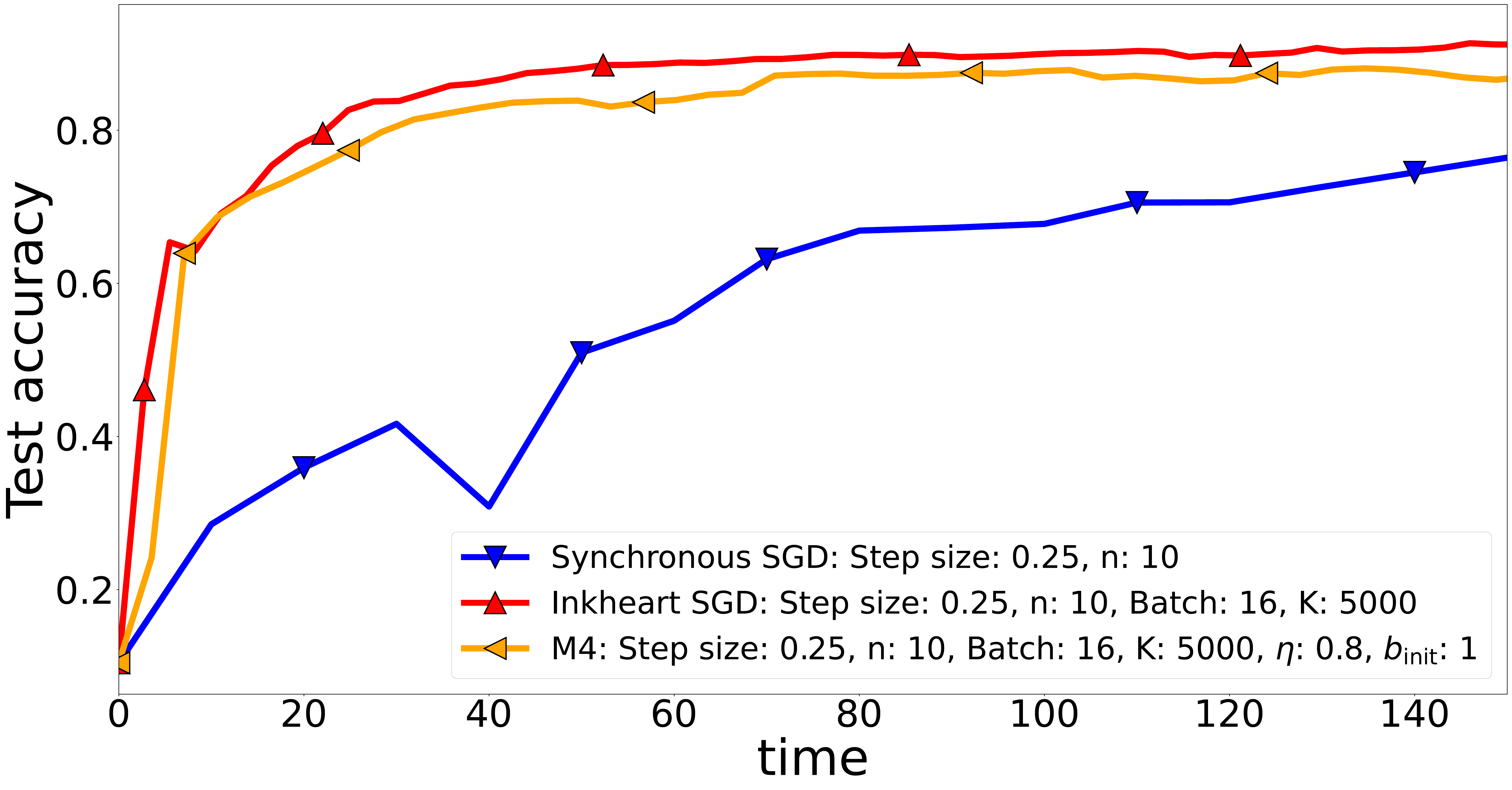}
        \caption{$h=0$}
    \end{subfigure}\hfill
    \begin{subfigure}[b]{0.32\textwidth}
        \centering\includegraphics[width=\textwidth]{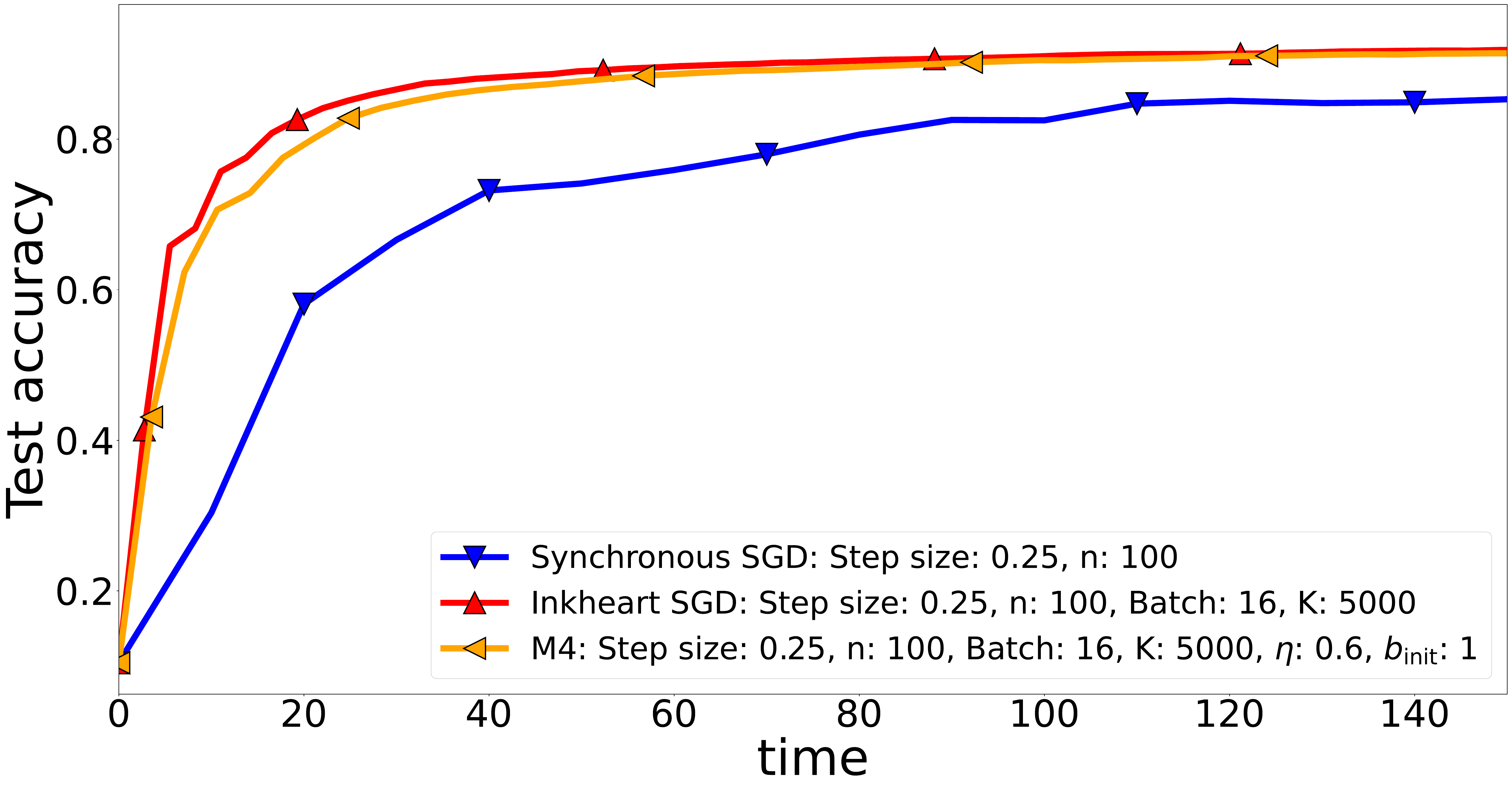}
        \caption{$h=0$}
    \end{subfigure}\hfill
    \begin{subfigure}[b]{0.32\textwidth}
        \centering\includegraphics[width=\textwidth]{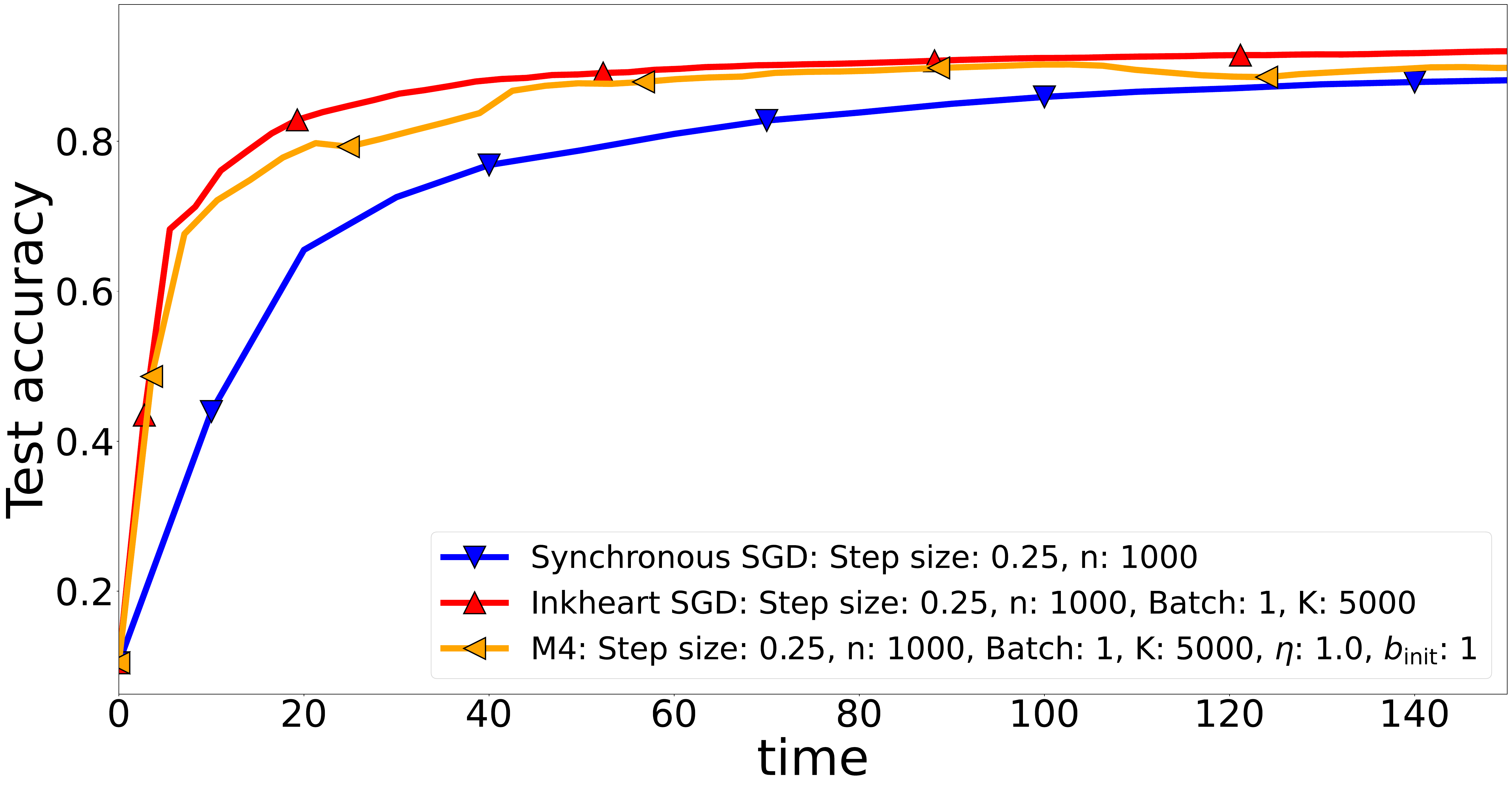}
        \caption{$h=0$}
    \end{subfigure}
    
    \vspace{0.15cm}
    
    \begin{subfigure}[b]{0.32\textwidth}
        \centering\includegraphics[width=\textwidth]{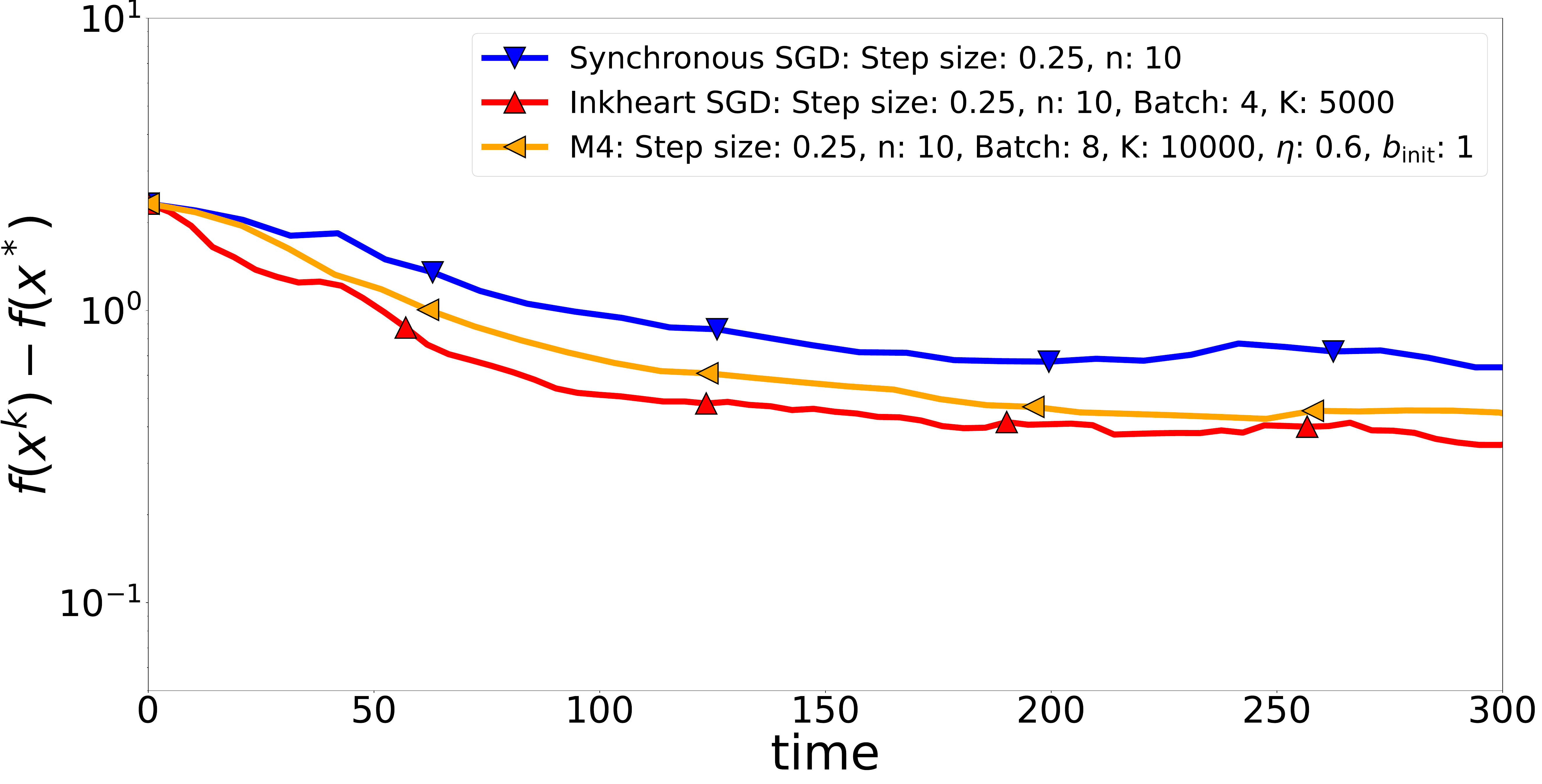}
    \end{subfigure}\hfill
    \begin{subfigure}[b]{0.32\textwidth}
        \centering\includegraphics[width=\textwidth]{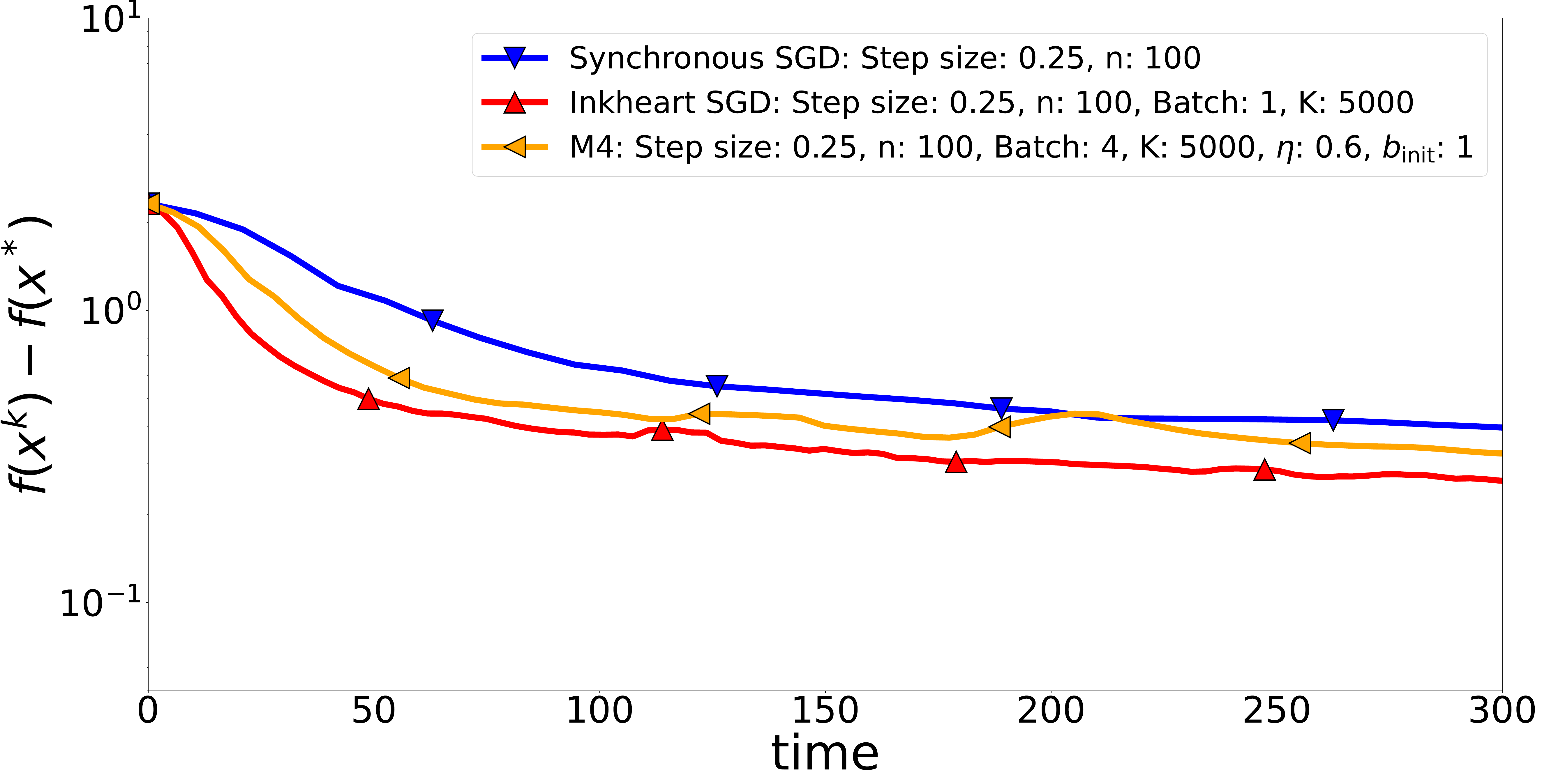}
    \end{subfigure}\hfill
    \begin{subfigure}[b]{0.32\textwidth}
        \centering\includegraphics[width=\textwidth]{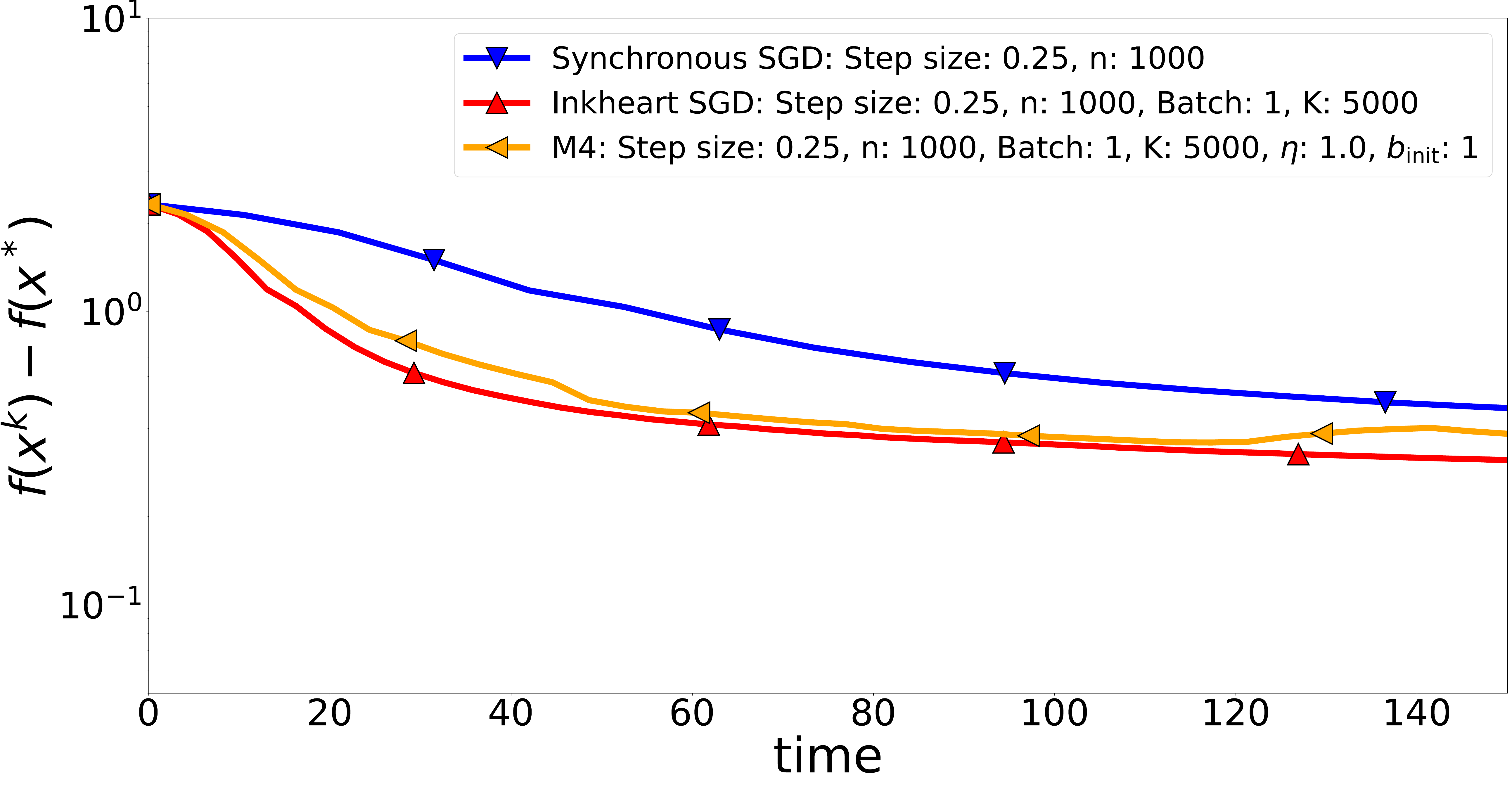}
    \end{subfigure}
    
    \vspace{0.15cm}
    
    \begin{subfigure}[b]{0.32\textwidth}
        \centering\includegraphics[width=\textwidth]{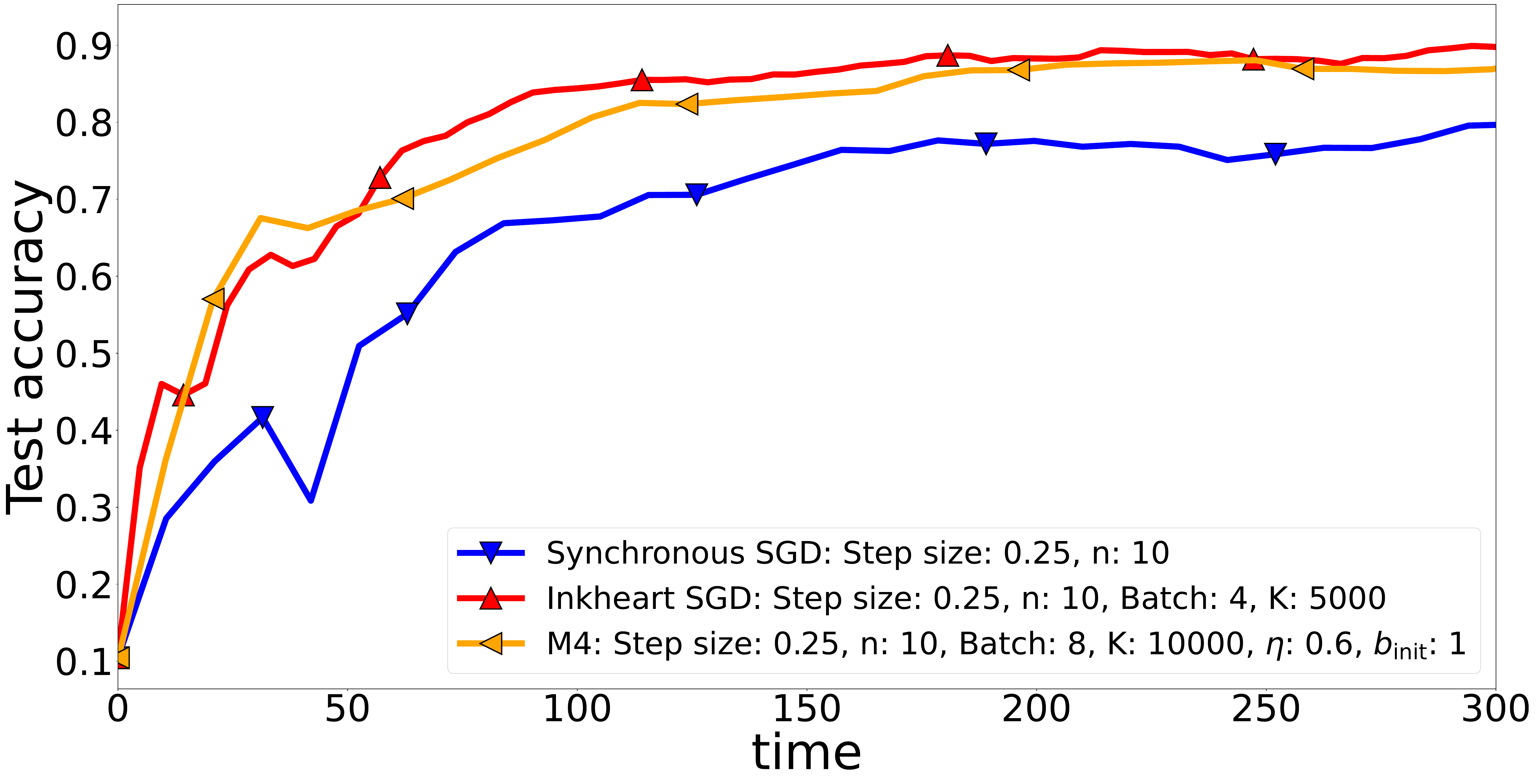}
        \caption{$h=0.1$}
    \end{subfigure}\hfill
    \begin{subfigure}[b]{0.32\textwidth}
        \centering\includegraphics[width=\textwidth]{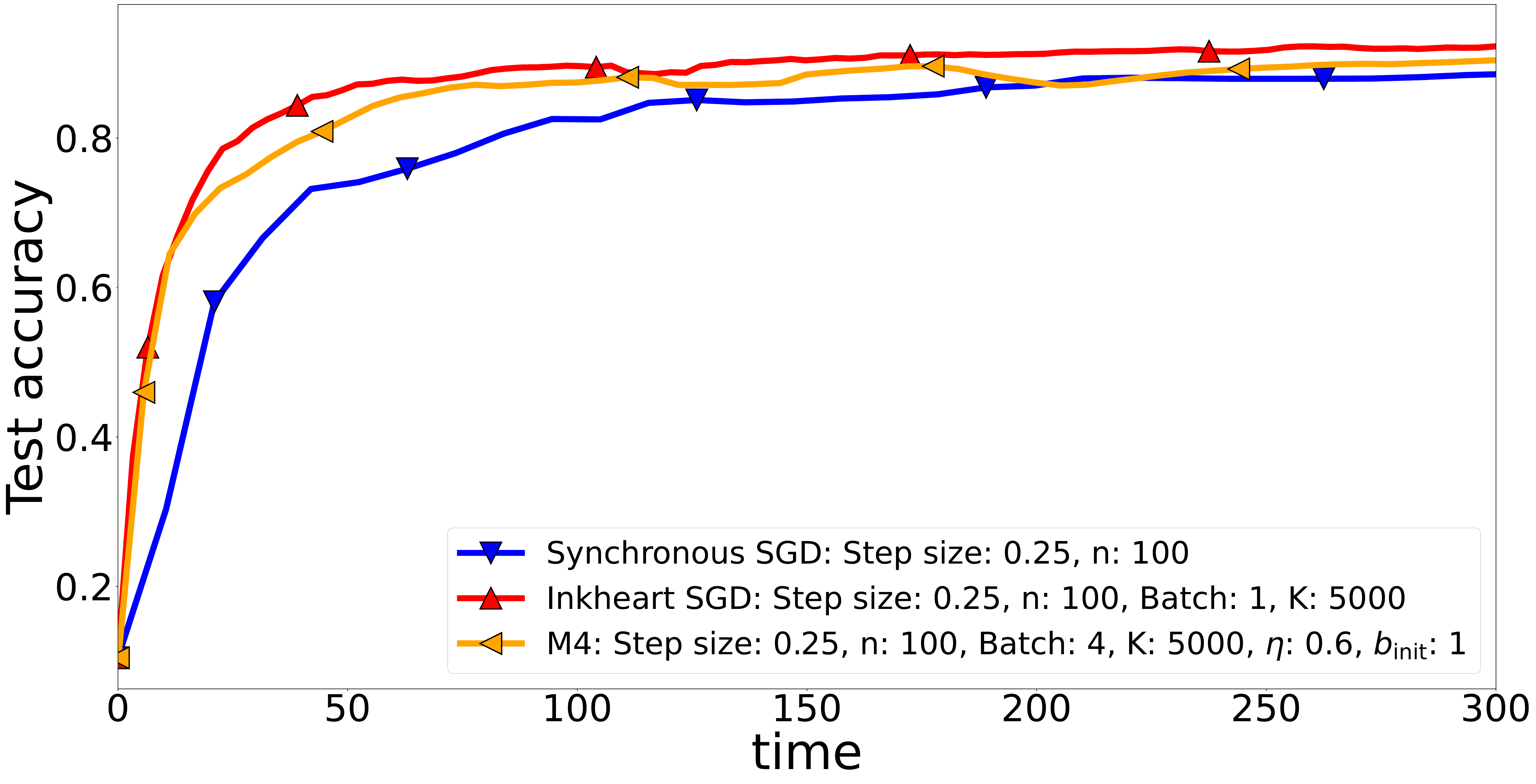}
        \caption{$h=0.1$}
    \end{subfigure}\hfill
    \begin{subfigure}[b]{0.32\textwidth}
        \centering\includegraphics[width=\textwidth]{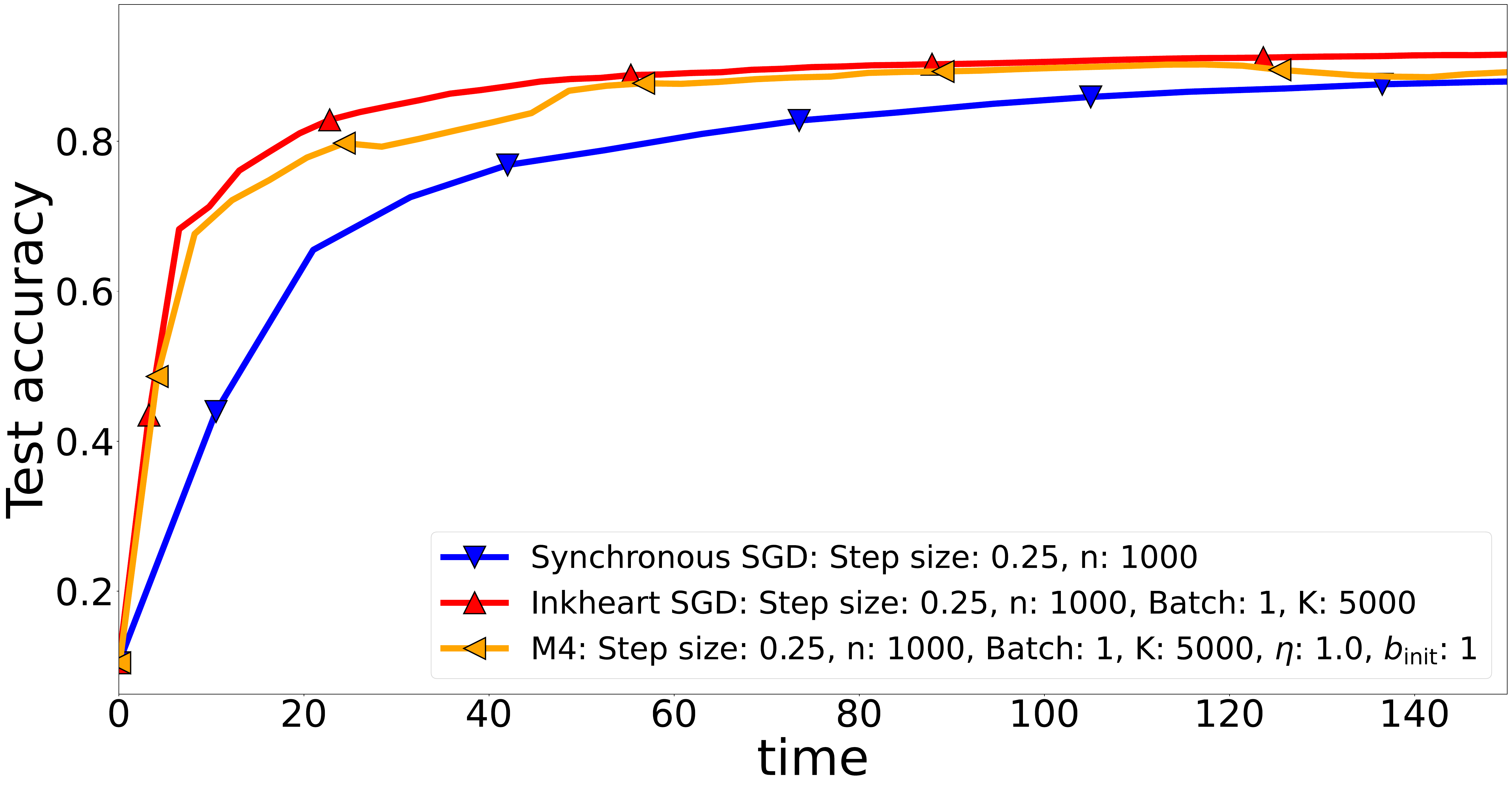}
        \caption{$h=0.1$}
    \end{subfigure}
    
    \vspace{0.15cm}
    
    \begin{subfigure}[b]{0.32\textwidth}
        \centering\includegraphics[width=\textwidth]{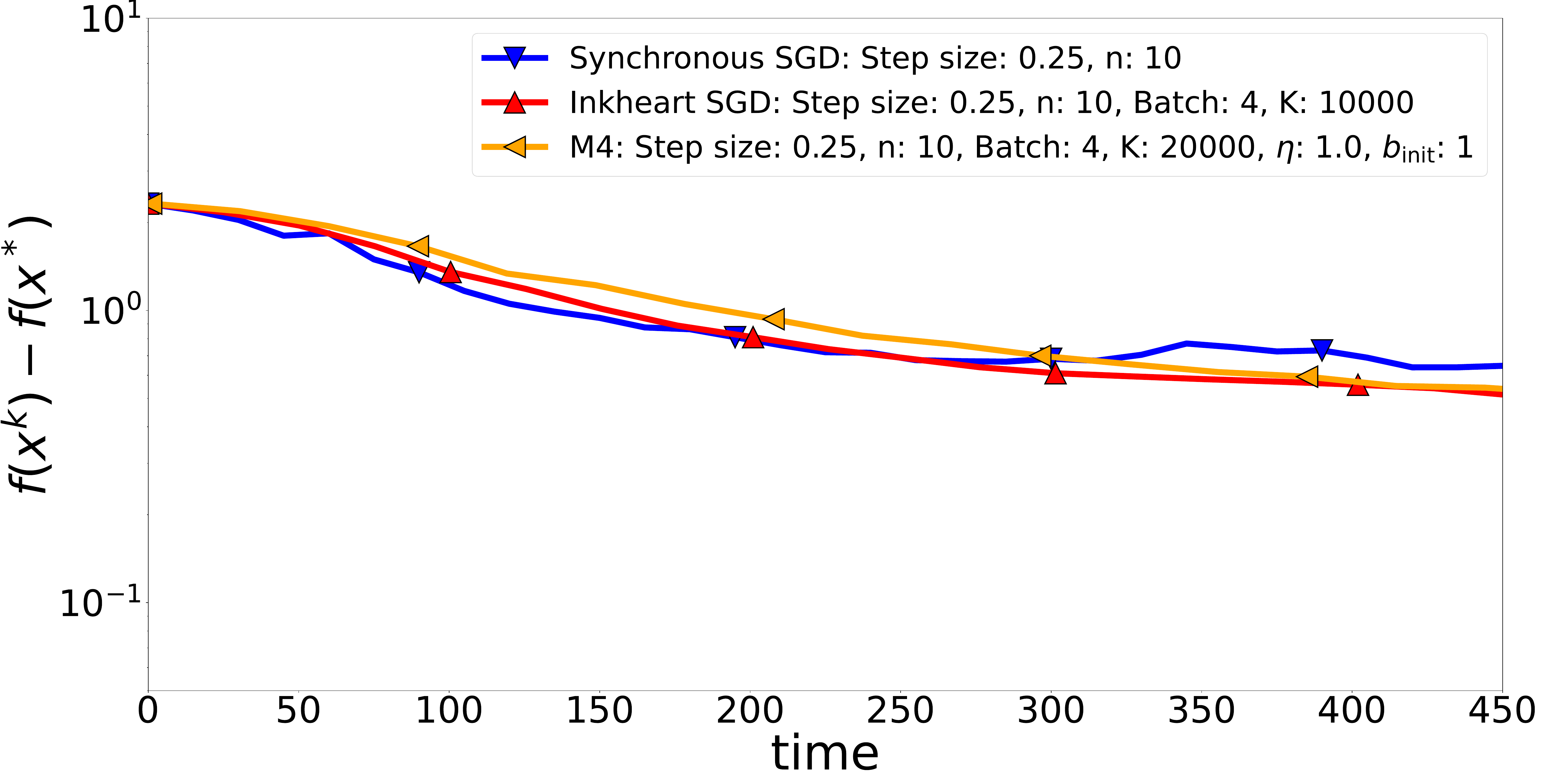}
    \end{subfigure}\hfill
    \begin{subfigure}[b]{0.32\textwidth}
        \centering\includegraphics[width=\textwidth]{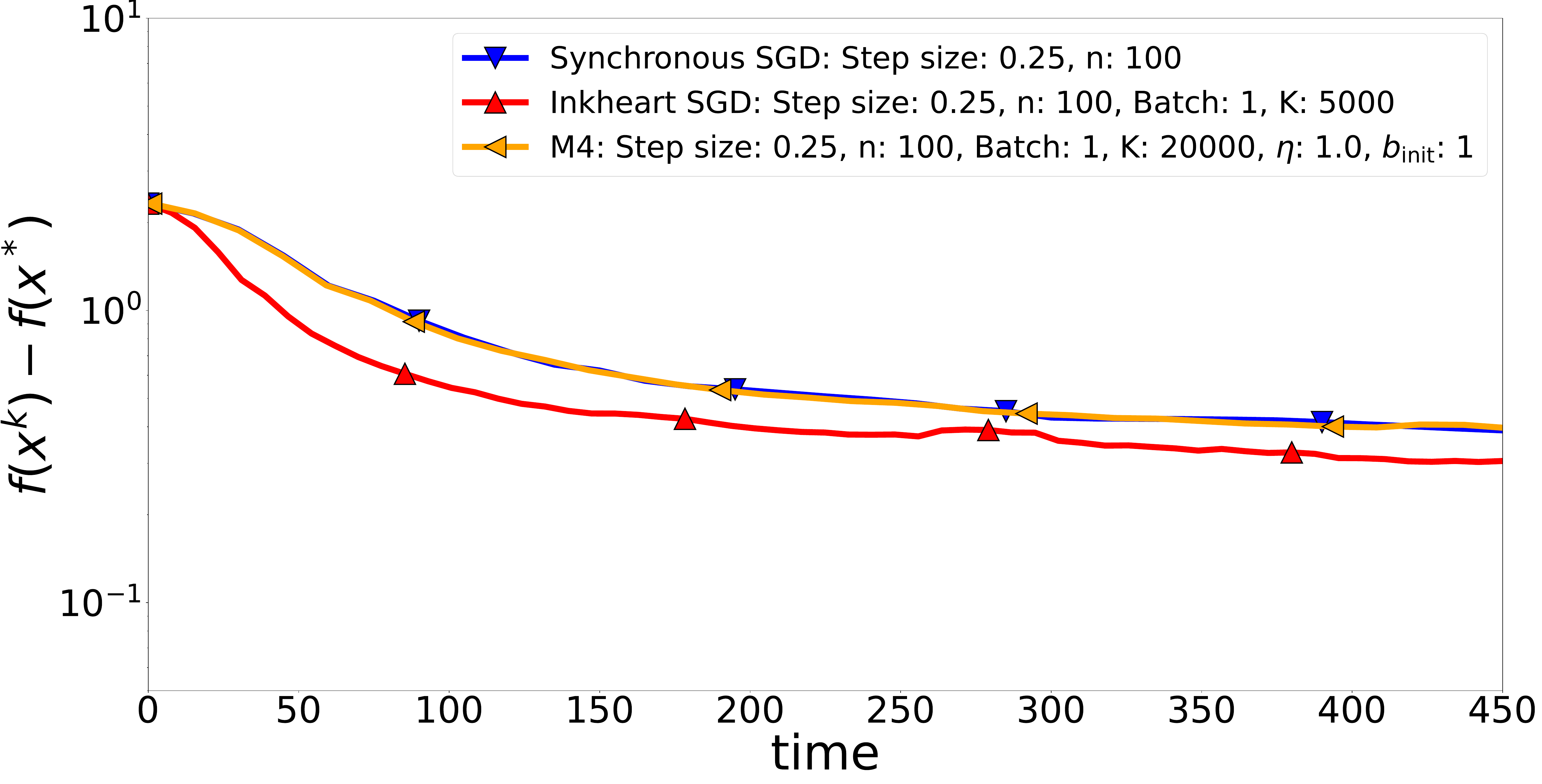}
    \end{subfigure}\hfill
    \begin{subfigure}[b]{0.32\textwidth}
        \centering\includegraphics[width=\textwidth]{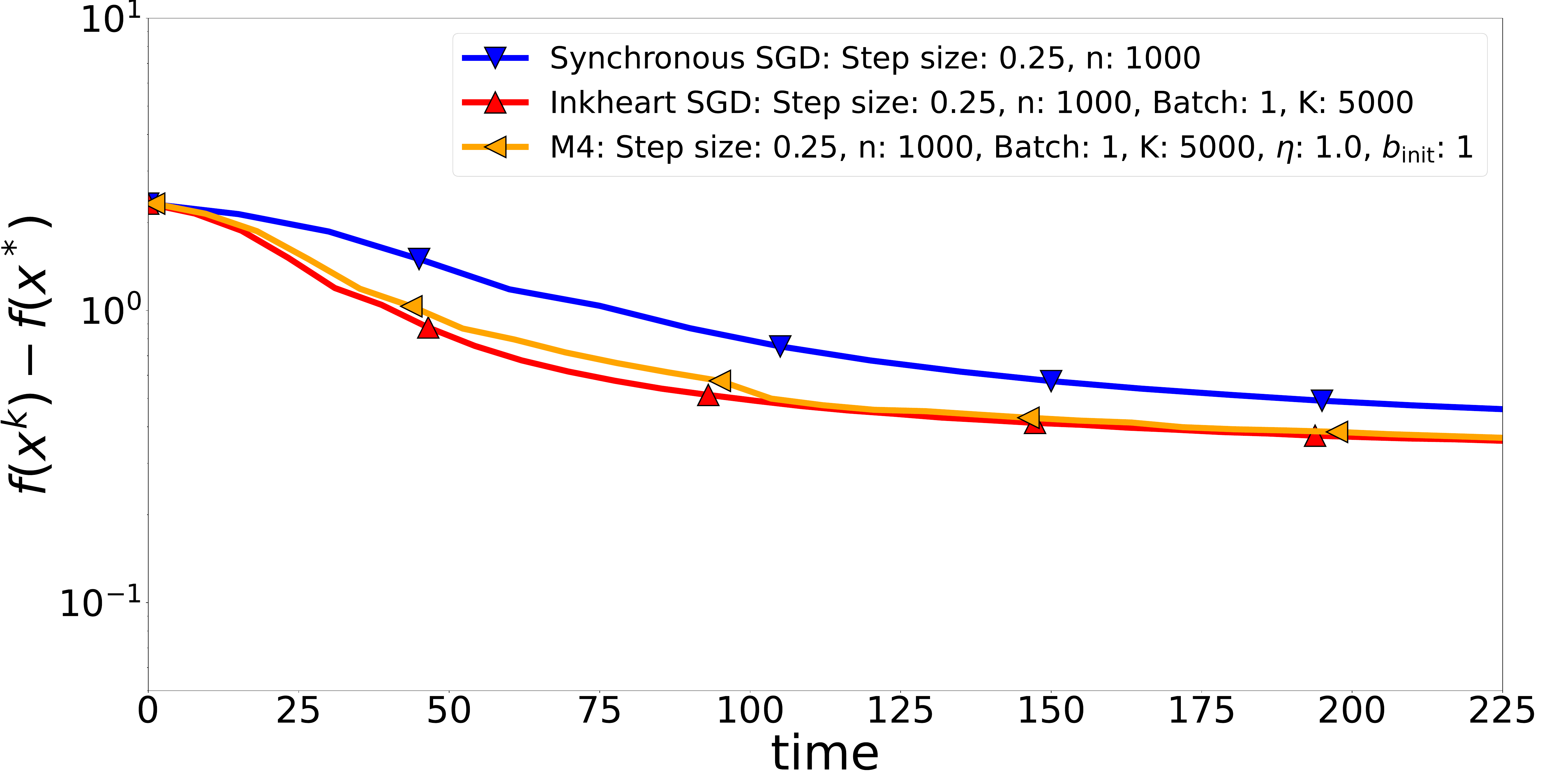}
    \end{subfigure}
    
    \vspace{0.15cm}
    
    \begin{subfigure}[b]{0.32\textwidth}
        \centering\includegraphics[width=\textwidth]{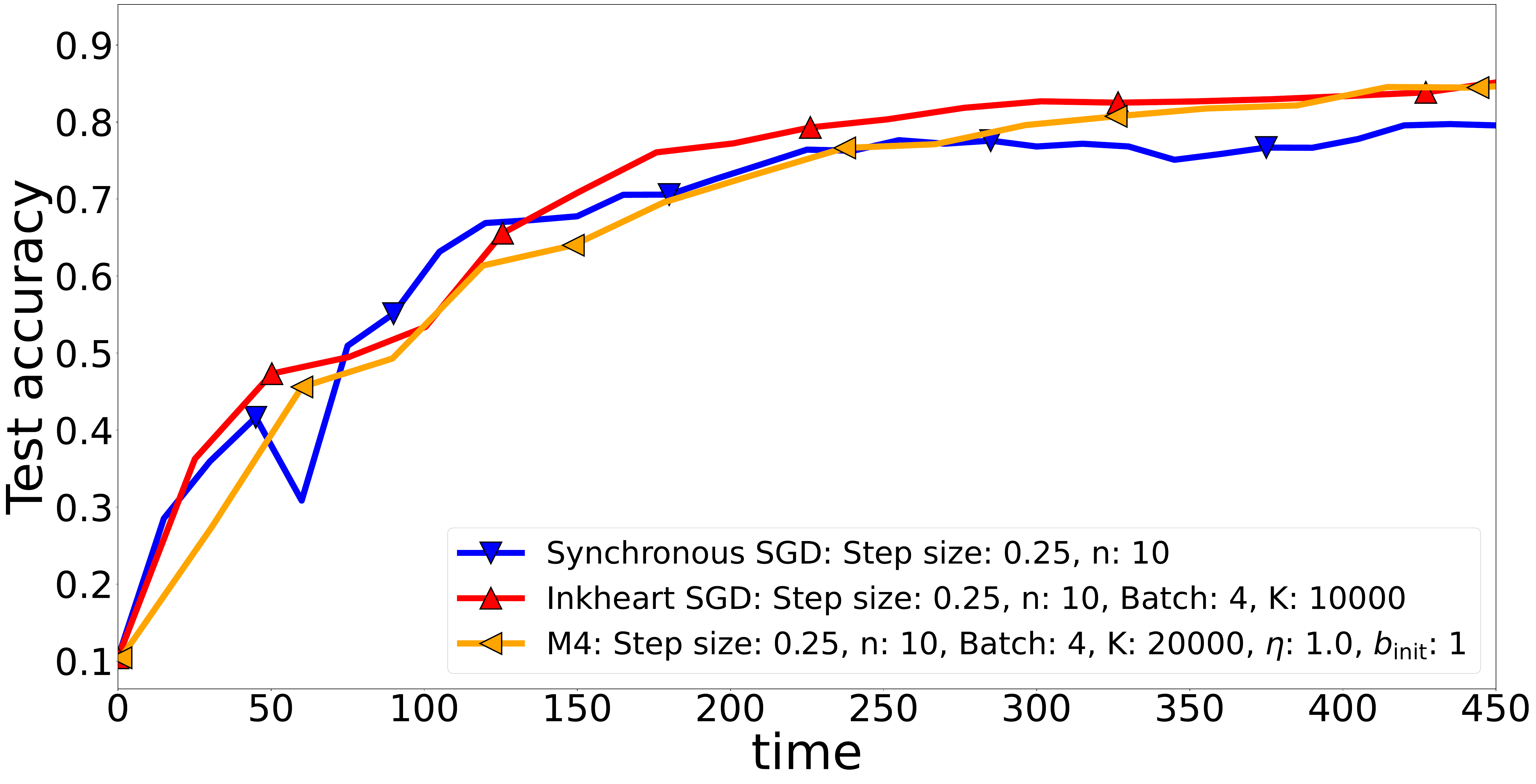}
        \caption{$h=1.0$}
    \end{subfigure}\hfill
    \begin{subfigure}[b]{0.32\textwidth}
        \centering\includegraphics[width=\textwidth]{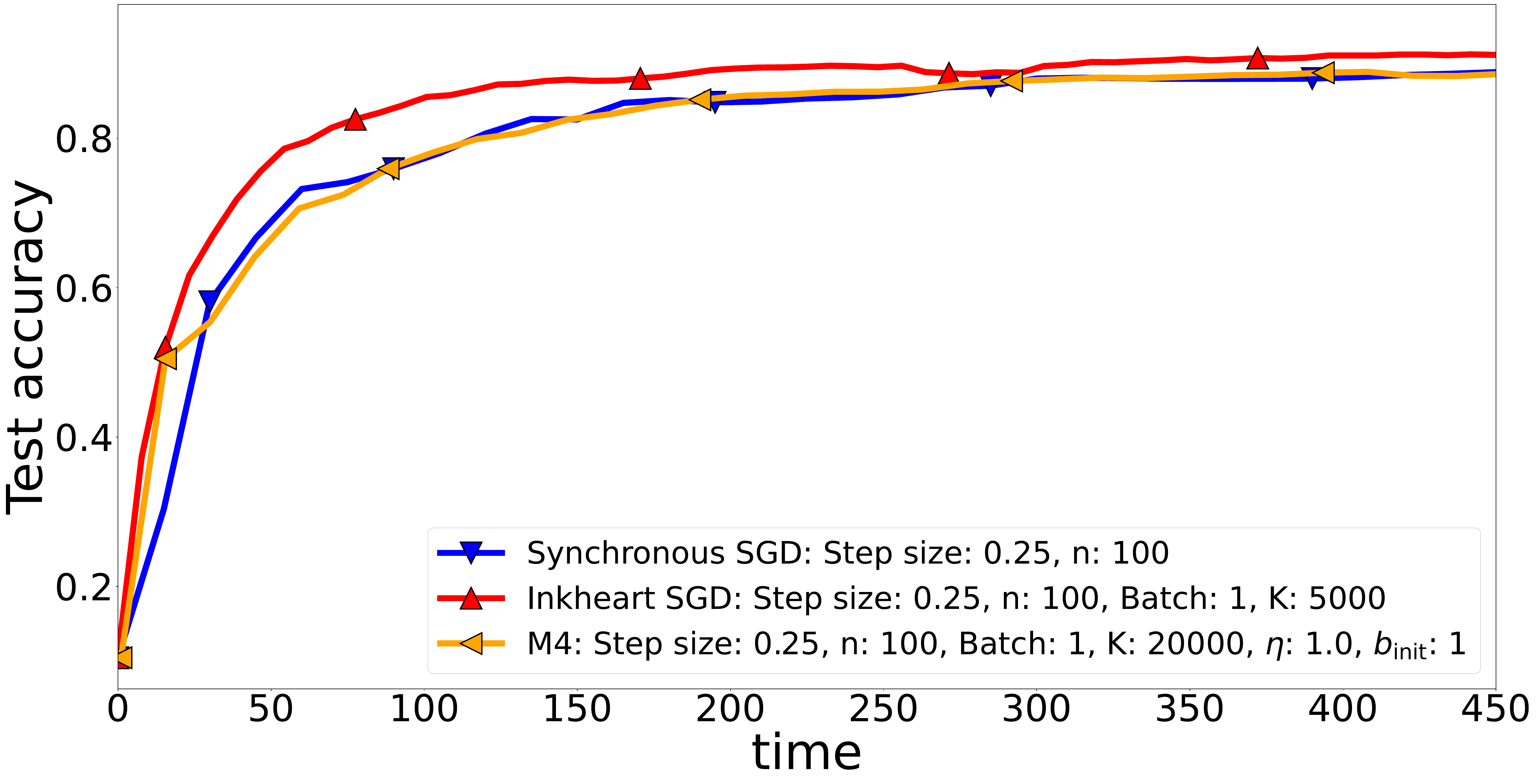}
        \caption{$h=1.0$}
    \end{subfigure}\hfill
    \begin{subfigure}[b]{0.32\textwidth}
        \centering\includegraphics[width=\textwidth]{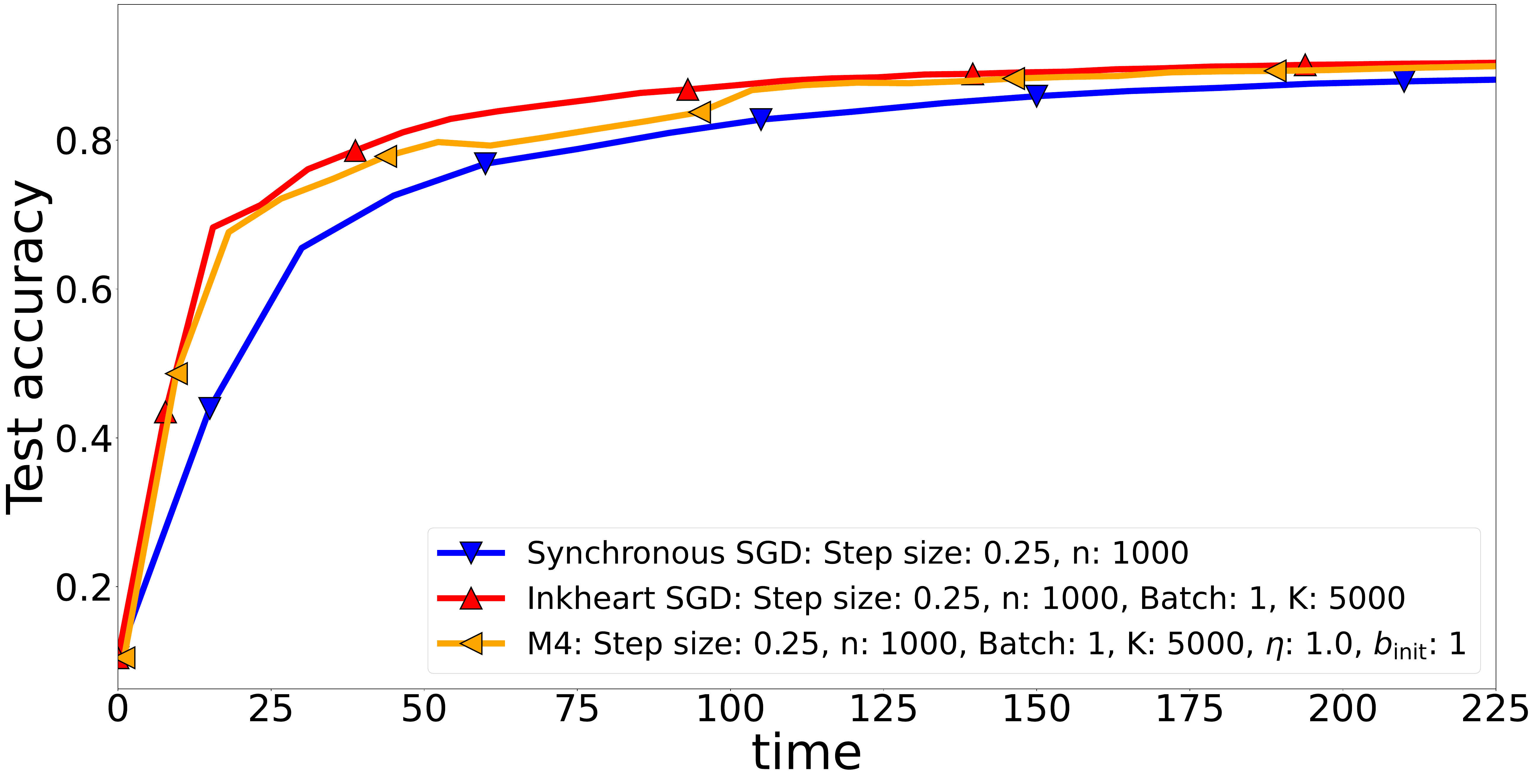}
        \caption{$h=1.0$}
    \end{subfigure}
    
    \caption{Training loss and accuracy on MNIST. 
Fixed parameters: $d=25\,450$, $\kappa = \nicefrac{1}{d}$, $\tau = \nicefrac{1}{d}$. 
Columns vary the number of workers $n \in \{10, 100, 1000\}$; 
rows alternate between loss and accuracy for the per-sample computation time $h \in \{0, 0.1, 1.0\}$.}
    \label{fig:mnist_k1_results}
\end{figure}

\subsection{Homogeneous CIFAR-10 experiments with ResNet-18}
\label{sec:cifar10_resnet}
Additionally, we evaluate the methods on CIFAR-10~\citep{krizhevsky2010cifar} using ResNet-18~\citep{He_2016_CVPR}. The resulting model has $d=11{,}173{,}962$ trainable parameters.
Each worker has access to the full training dataset.

We use $32$ workers, a global batch size of $128$, and hence a local batch size of $4$ per worker.
All methods are trained for $200$ epochs with the same constant step size of $0.005$ and weight decay
$5 \cdot 10^{-4}$. For Inkheart SGD, we use the compression ratio $K/d = 0.2$.
We set $d \kappa = d \tau = 1$, where $d$ is the number
of model parameters, and let $h$ denote the computation time of a stochastic gradient for one sample.
Thus, each distributed worker spends $4h$ on computation per iteration, while one-node SGD processes
all $128$ samples sequentially and spends $128h$.

\begin{figure}[t]
\centering

\begin{minipage}[t]{0.48\linewidth}
    \centering
    \includegraphics[width=\linewidth]
    {\detokenize{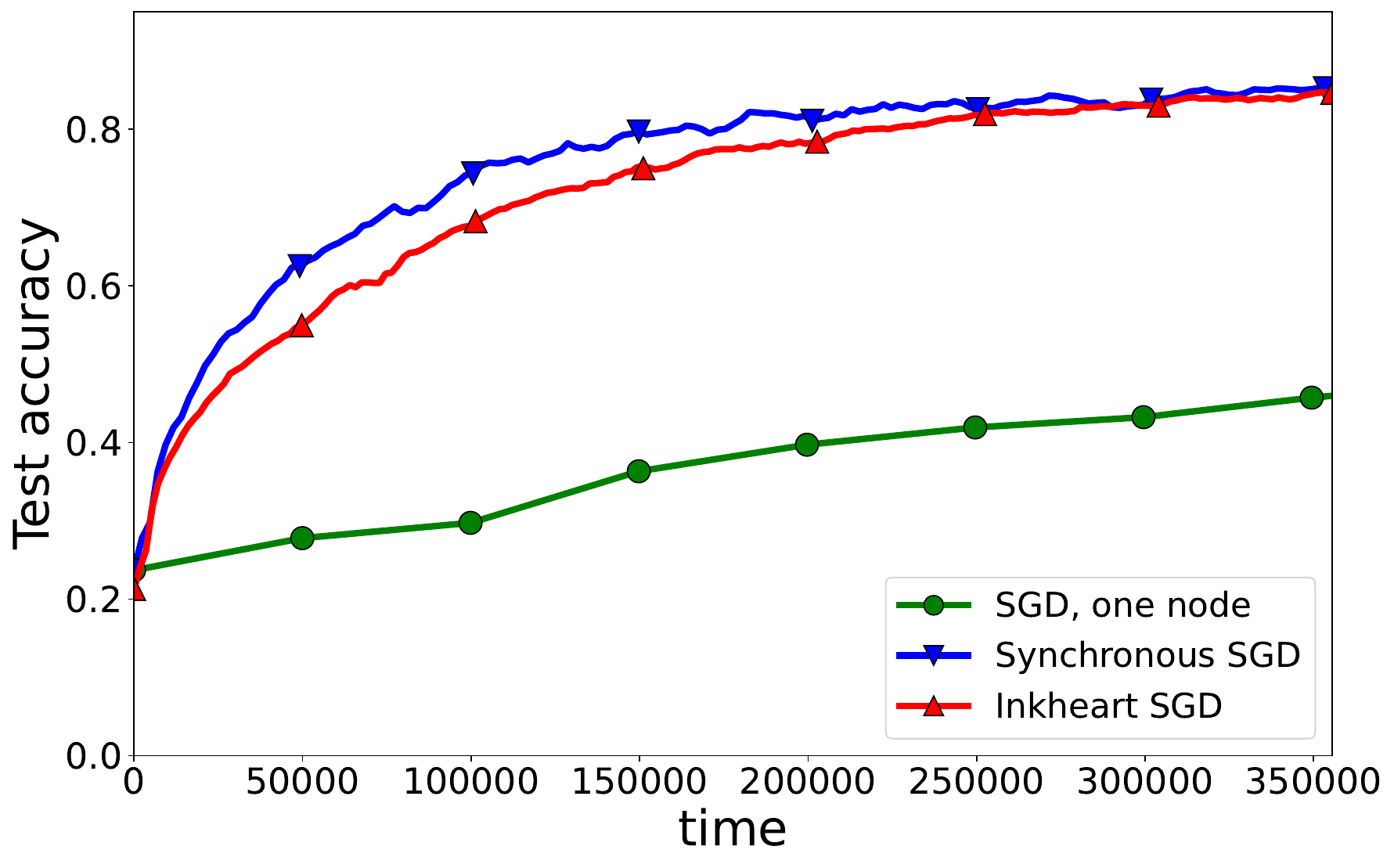}}

    \small $h=1$
\end{minipage}
\hfill
\begin{minipage}[t]{0.48\linewidth}
    \centering
    \includegraphics[width=\linewidth]
    {\detokenize{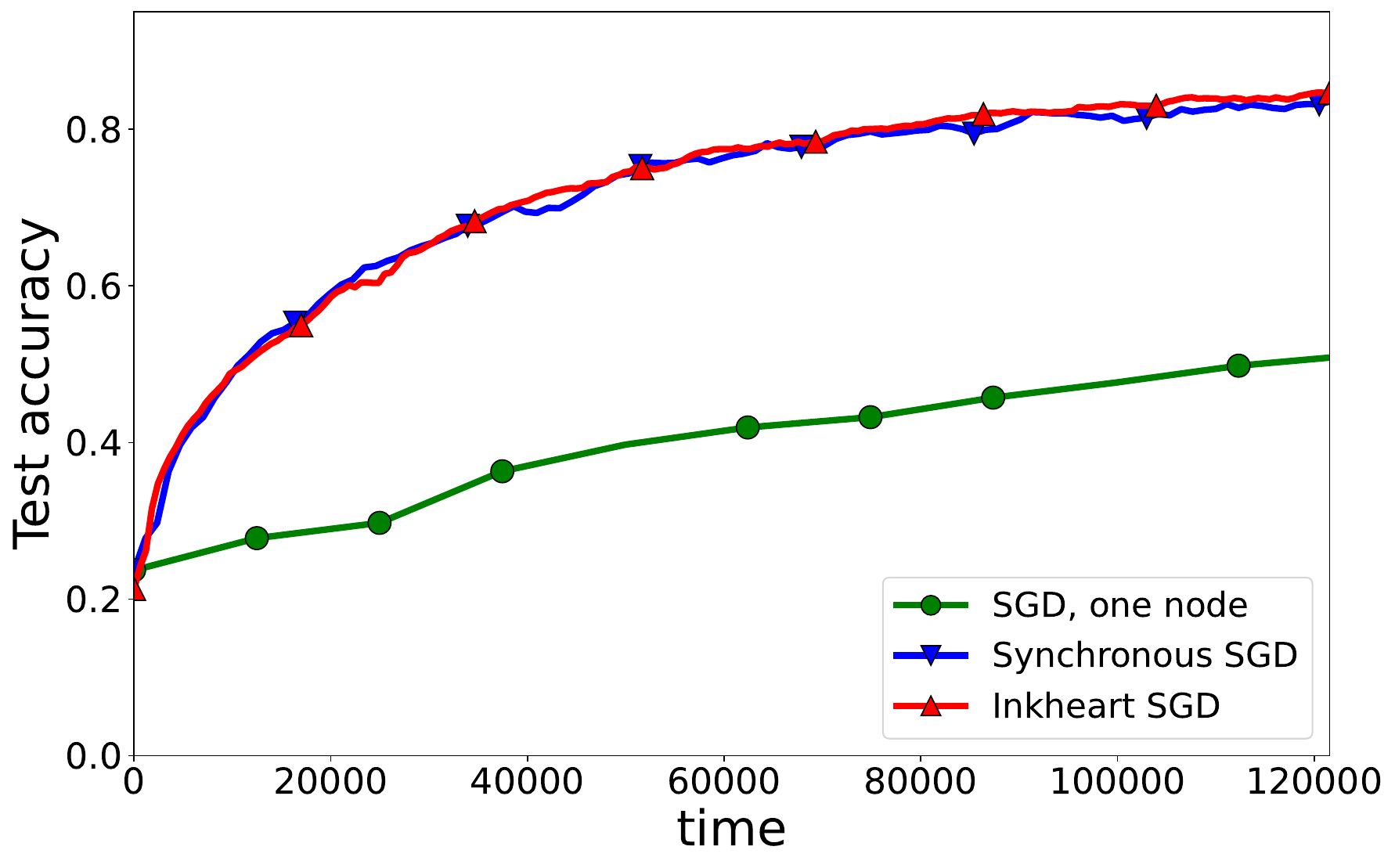}}

    \small $h=1/4$
\end{minipage}

\vspace{0.5em}

\begin{minipage}[t]{0.48\linewidth}
    \centering
    \includegraphics[width=\linewidth]
    {\detokenize{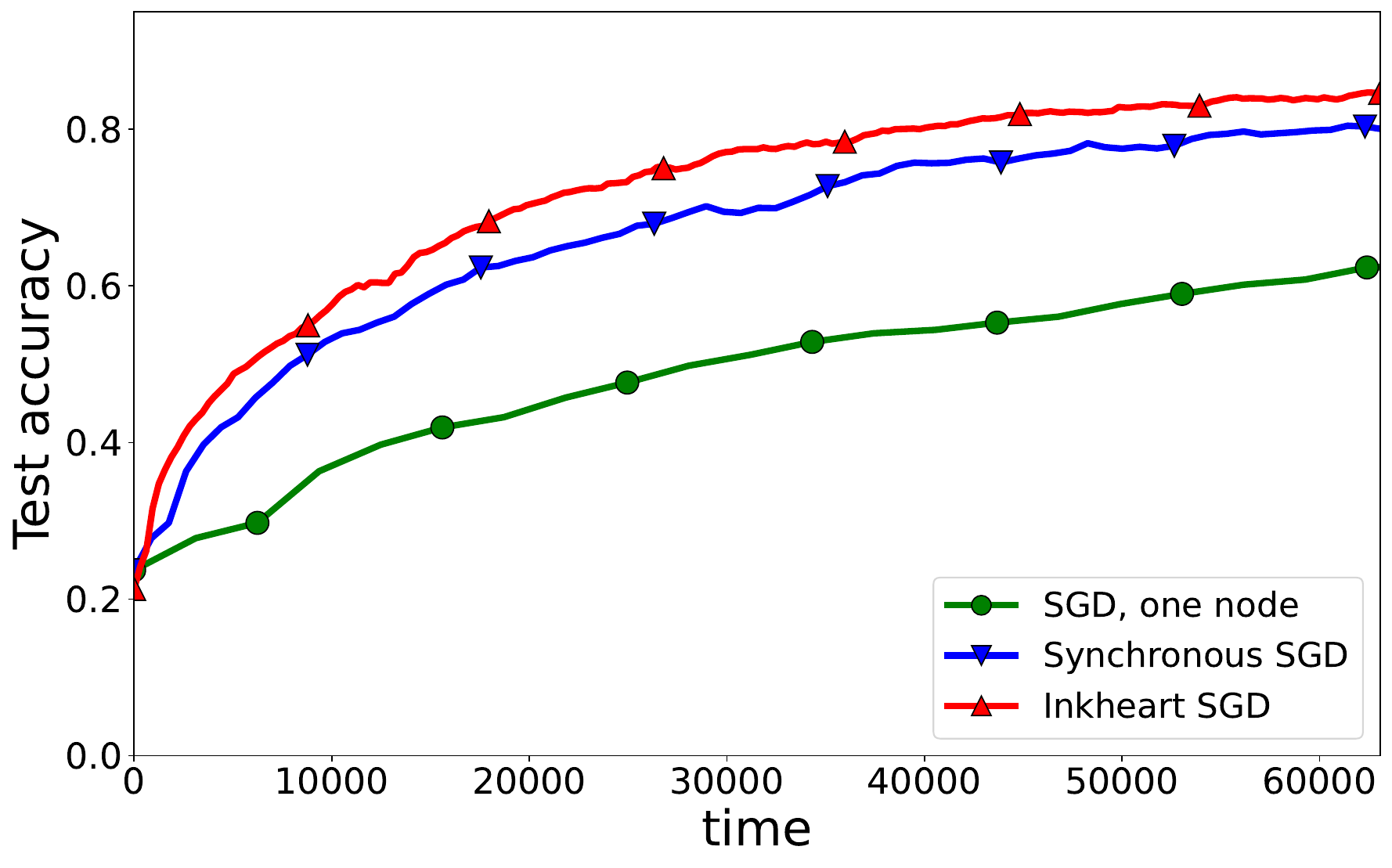}}

    \small $h=1/16$
\end{minipage}
\hfill
\begin{minipage}[t]{0.48\linewidth}
    \centering
    \includegraphics[width=\linewidth]
    {\detokenize{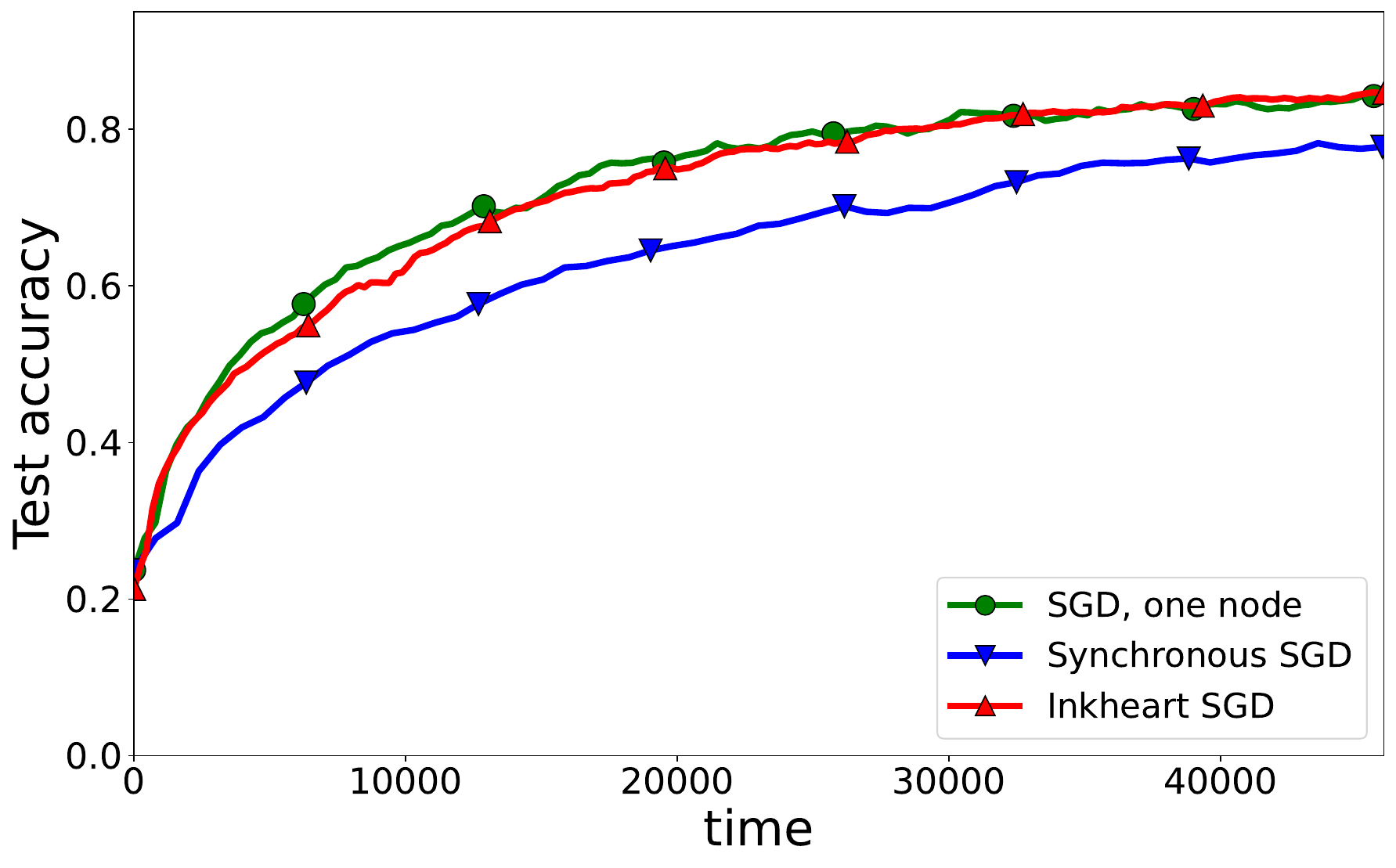}}

    \small $h=1/128$
\end{minipage}

\caption{
Test accuracy versus simulated execution time for ResNet-18 trained on CIFAR-10.
A batch of size $128$ is split among the workers. Inkheart SGD and Synchronous SGD use $32$ workers,
and all $32$ workers have access to the full training dataset, with $4$ samples per worker.
One-node SGD uses a single worker.
For Inkheart SGD, $K/d=0.2$, where $K$ is the compression parameter of Rand$K$.
All methods use a constant step size of $0.005$.
The time required to compute a single-sample stochastic gradient is $h$ seconds, and the communication
times satisfy $d\kappa=d\tau=1$.
}
\label{fig:cifar_resnet_time}
\end{figure}

Figure~\ref{fig:cifar_resnet_time} compares the performance under different computation costs.
When gradient computation dominates, as in the $h=1$ regime, Synchronous SGD performs
best, since compression provides little benefit in this setting. When computation and communication costs become
comparable, Inkheart SGD becomes preferable due to bidirectional compression. Finally, when communication
strongly dominates computation, the benefit of distributing the computation becomes smaller, and one-node SGD
becomes competitive since it avoids communication completely. Thus, the three regimes exhibit the expected
behavior: Synchronous SGD is preferable in computation-dominated settings, Inkheart SGD is preferable in
intermediate regimes, and one-node SGD can be preferable when computation is negligible.

\newpage
\section*{NeurIPS Paper Checklist}

\begin{enumerate}

\item {\bf Claims}
    \item[] Question: Do the main claims made in the abstract and introduction accurately reflect the paper's contributions and scope?
    \item[] Answer: \answerYes{} %
    \item[] Justification: Sections~\ref{sec:homog}, \ref{sec:heter}, and \ref{sec:ext}

\item {\bf Limitations}
    \item[] Question: Does the paper discuss the limitations of the work performed by the authors?
    \item[] Answer: \answerYes{} %
    \item[] Justification: Sections~\ref{sec:inkheart}, \ref{sec:m4}, and \ref{sec:incheart_heter}

\item {\bf Theory assumptions and proofs}
    \item[] Question: For each theoretical result, does the paper provide the full set of assumptions and a complete (and correct) proof?
    \item[] Answer: \answerYes{} %
    \item[] Justification: The assumptions and the proofs are in Section~\ref{sec:ab_homo}, \ref{sec:ab_heter} and in the appendix.

    \item {\bf Experimental result reproducibility}
    \item[] Question: Does the paper fully disclose all the information needed to reproduce the main experimental results of the paper to the extent that it affects the main claims and/or conclusions of the paper (regardless of whether the code and data are provided or not)?
    \item[] Answer: \answerYes{} %
    \item[] Justification: Section~\ref{sec:experiments}

\item {\bf Open access to data and code}
    \item[] Question: Does the paper provide open access to the data and code, with sufficient instructions to faithfully reproduce the main experimental results, as described in supplemental material?
    \item[] Answer: \answerYes{} %
    \item[] Justification: In the supplementary materials.

\item {\bf Experimental setting/details}
    \item[] Question: Does the paper specify all the training and test details (e.g., data splits, hyperparameters, how they were chosen, type of optimizer) necessary to understand the results?
    \item[] Answer: \answerYes{} %
    \item[] Justification: Section~\ref{sec:experiments}

\item {\bf Experiment statistical significance}
    \item[] Question: Does the paper report error bars suitably and correctly defined or other appropriate information about the statistical significance of the experiments?
    \item[] Answer: \answerNo{} %
    \item[] Justification:

\item {\bf Experiments compute resources}
    \item[] Question: For each experiment, does the paper provide sufficient information on the computer resources (type of compute workers, memory, time of execution) needed to reproduce the experiments?
    \item[] Answer: \answerYes{} %
    \item[] Justification: Section~\ref{sec:experiments}
    
\item {\bf Code of ethics}
    \item[] Question: Does the research conducted in the paper conform, in every respect, with the NeurIPS Code of Ethics \url{https://neurips.cc/public/EthicsGuidelines}?
    \item[] Answer: \answerYes{} %
    \item[] Justification: We have reviewed the code of ethics, and we are confident that our paper is in compliance with it.

\item {\bf Broader impacts}
    \item[] Question: Does the paper discuss both potential positive societal impacts and negative societal impacts of the work performed?
    \item[] Answer: \answerNA{} %
    \item[] Justification: Our work considers a mathematical problem from machine learning.
    
\item {\bf Safeguards}
    \item[] Question: Does the paper describe safeguards that have been put in place for responsible release of data or models that have a high risk for misuse (e.g., pre-trained language models, image generators, or scraped datasets)?
    \item[] Answer: \answerNA{} %
    \item[] Justification: Our work considers a mathematical problem from machine learning.

\item {\bf Licenses for existing assets}
    \item[] Question: Are the creators or original owners of assets (e.g., code, data, models), used in the paper, properly credited and are the license and terms of use explicitly mentioned and properly respected?
    \item[] Answer: \answerYes{} %
    \item[] Justification: Section~\ref{sec:experiments}

\item {\bf New assets}
    \item[] Question: Are new assets introduced in the paper well documented and is the documentation provided alongside the assets?
    \item[] Answer: \answerYes{} %
    \item[] Justification: The code for the experiments is in the supplementary materials.

\item {\bf Crowdsourcing and research with human subjects}
    \item[] Question: For crowdsourcing experiments and research with human subjects, does the paper include the full text of instructions given to participants and screenshots, if applicable, as well as details about compensation (if any)? 
    \item[] Answer: \answerNA{}%
    \item[] Justification: 

\item {\bf Institutional review board (IRB) approvals or equivalent for research with human subjects}
    \item[] Question: Does the paper describe potential risks incurred by study participants, whether such risks were disclosed to the subjects, and whether Institutional Review Board (IRB) approvals (or an equivalent approval/review based on the requirements of your country or institution) were obtained?
    \item[] Answer: \answerNA{} %
    \item[] Justification: 

\item {\bf Declaration of LLM usage}
    \item[] Question: Does the paper describe the usage of LLMs if it is an important, original, or non-standard component of the core methods in this research? Note that if the LLM is used only for writing, editing, or formatting purposes and does \emph{not} impact the core methodology, scientific rigor, or originality of the research, declaration is not required.
    \item[] Answer: \answerNA{} %
    \item[] Justification: 

\end{enumerate}

\end{document}